\documentclass[12pt, openany]{book}%
\usepackage{ amsfonts, amssymb}
\usepackage{eucal}
\usepackage{amsmath}
\usepackage[all]{xy}
\usepackage{amsthm}
\usepackage{enumerate}
\usepackage{graphicx}
\usepackage{ifthen}
\usepackage{marginnote}
\DeclareMathOperator{\spa}{span}
\DeclareMathOperator{\ad}{ad}

\DeclareMathOperator{\im}{im}
\DeclareMathOperator{\End}{End}
\DeclareMathOperator{\sign}{sign}
\DeclareMathOperator{\supp}{supp}

\usepackage[top=2.5cm,bottom=2.5cm,right=2.5cm,left=3.25cm]{geometry}
\theoremstyle{definition}
\newtheorem{defi}{Definition}[section]

\theoremstyle{plain}
\newtheorem{theo}[defi]{Theorem}

\newtheorem{lem}[defi]{Lemma}
\newtheorem{cor}[defi]{Corollary}

\newtheorem{prop}[defi]{Proposition}

\theoremstyle{remark}
\newtheorem{rem}[defi]{Remark}

\newtheorem{expl}[defi]{Example}

\newcommand{\gradla}{\ensuremath{\underline{\mathbf{LA}}_\Gamma}}
\newcommand{\gce}[1]{\ensuremath{\mathfrak{gce}(#1)}}
\newcommand{\uce}[1]{\ensuremath{\mathfrak{uce}(#1)}}

\newcommand{\che}{\sp{\scriptscriptstyle\vee}} 
\newcommand{\card}[1]{\ensuremath{\mathrm{card}({#1})}}
\newcommand{\str}{\ensuremath{\mathfrak{der}}}
\newcommand{\instr}{\ensuremath{\mathfrak{ider}}}
\newcommand{\HF}{\ensuremath{\mathrm{HC}}}
\newcommand{\TKK}{\ensuremath{\mathrm{TKK}}}

\newcommand{\Hom}{\ensuremath{\mathrm{Hom}}}
\newcommand{\uTKK}{\ensuremath{\mathrm{ULE}}}
\newcommand{\uider}{\ensuremath{\mathfrak{uider}}}

\newcommand{\tw}{\ensuremath{ \tilde \wedge}}
\newcommand{\ud}{\ensuremath{ \mathfrak{ud}}}
\begin{document}
\sloppy


\author{Angelika Welte}
\frontmatter
\begin{titlepage}
\begin{center}
%


%
\ \\
\ \\
\ \\
{\huge \bfseries
Central Extensions of Graded Lie Algebras
\\}
\ \\
\ \\
{\large Angelika Welte}
\ \\
\ \\
\ \\ 
{\large \today} \\
\ \\
\ \\
\textsc{\Large
Thesis submitted to the \\
Faculty of Graduate and Postdoctoral Studies\\
In partial fulfillment of the requirements 
For the Doctoral degree in\\
Mathematics 
\ \\
}
\ \\
\begin{center}
\textsc{\Large
Ottawa-Carleton Institute for Graduate Studies and Research in Mathematics and
Statistics}
\end{center}
\ \\
\ \\
\ \\
\ \\
\ \\
\ \\
\ \\
\ \\
\ \\
\begin{center}
\copyright Angelika Welte, University of Ottawa, Canada, 2009
\end{center}
\footnotesize{\ }
%

%

%
\end{center}
\end{titlepage}

\ \\
\ \\
\ \\
\ \\
\ \\
\ \\
\ \\
\begin{center}
\textbf{Abstract}
\end{center}
\ \\
\ \\
\ \\
In this thesis we describe the universal central extension of two important classes of so-called root-graded Lie algebras defined over a commutative associative unital ring $k.$ Root-graded Lie algebras are Lie algebras which are graded by the root lattice of a locally finite root system and contain enough $\mathfrak{sl}_2$-triples to separate the homogeneous spaces of the grading. Examples include the infinite rank analogs of the simple
finite-dimensional complex Lie algebras. \\ In the thesis we show that in general the universal central extension of a root-graded Lie algebra $L$ is not root-graded anymore, but that we can measure quite easily how far it is away from being so, using the notion of degenerate sums, introduced by van der Kallen. We then concentrate on root-graded Lie algebras which are
graded by the root systems of type $A$ with rank at least 2 and of type $C$. For them one can use the theory of Jordan algebras.
\\
 Given a Jordan algebra $J$, we establish a functorial construction which produces a Lie algebra from $J$, called the universal Tits-Kantor-Koecher algebra of $J$. We are led to study the derivation algebras of Jordan algebras and alternative algebras. Under mild assumptions on the base ring $k$, it is proven that the Albert algebra (a Jordan algebra) and the octonion algebra (an alternative algebra) have derivation algebras which are isomorphic to exceptional Lie algebras of type $F_4$ and $G_2$ respectively. We also show that certain root-graded Lie algebras which are defined by the Albert algebra resp. the octonion algebra are simply connected, i.e., coincide with their central extensions. 
\chapter*{Acknowledgements}
I would like to thank my thesis supervisor, Dr. Erhard Neher. Not only did he provide the intellectual support for my research, he also proved to possess infinite patience. He  helped me order my thoughts, listened to my vague ideas and often explained my own results to me.  He was also unfailing when it came to pointing out my various typographical and mathematical errors. It has been  a great experience to work with him, I just hope he thinks the same.    \\
\ \\
I am grateful to my parents. It was not easy for them to see me leave for Canada. However, they have never criticized my decision and gave me all the loving and caring support possible.  \\
To my friends here and across the ocean. To someone special. You were simply there for me. Thank you. \\
\ \\
I also acknowledge financial support from the Department of Mathematics and Statistics and the University of Ottawa.\\

\tableofcontents

\mainmatter
\chapter{Introduction}
\section{Motivation}
Let us start by looking at two classical results from the theory of semisimple Lie algebras over the field of complex numbers.
 \paragraph{Root systems.}
 A finite (reduced) root system can be  defined as a pair consisting of a $\mathbb Q$-vector space $E$ and a finite subset $R$ of $E$ such that 
\begin{itemize}
\narrower
\item[] \rm{(i)} $0 \in R$ and $\spa_{\mathbb Q} R = E,$

\item[] \rm{(ii)} for every $\alpha \in R^\times := R\setminus \{0\}$ there exists an $\alpha \che$ in $E^*$ such that $\langle \alpha, \alpha \che\rangle =2$ and $s_\alpha(R) = R$ where the reflection $s_\alpha$ is defined by $s_\alpha(x) = x -\langle x, \alpha\che  \rangle \alpha $ for $x \in E,$
\item[] \rm{(iii)} $\langle \beta, \alpha \che\rangle \in \mathbb Z$ where $\langle \beta, \alpha \che\rangle = \alpha \che(\beta)$, $\alpha, \beta \in R.$
\end{itemize}
If  for $\alpha \in R$, the only multiples of $\alpha$ in $R$ are $-\alpha, 0, \alpha,$ then $(R, E)$ is called \emph{reduced.}
Usually (\cite{BouLie2}), $0$ is not assumed to be a root, but for our purposes the requirement $0 \in R$ is more convenient. Also, sometimes (iv) is omitted. \\
If $R$ can be written as union $R = R_1 \cup R_2$ of two root systems such that $R_1 \cap R_2 = \{0\}$ and $\langle R_1, R_2 \che\rangle = 0,$ then $R$ is called reducible. Otherwise $R$ is called irreducible. A root system is finite, if $R$ is finite as set. \\
Finite  reduced root systems are classified in \cite{Hum} and in \cite{BouLie2}. There are four infinite series of reduced irreducible root systems, usually denoted by $A_n, B_n, C_n$ and $D_n$ and five exceptional irreducible root systems, $E_6, E_7, E_8$,$F_4$ and $G_2.$ If (iv) in the definition of a root system is left out, there is another infinite series, denoted by $BC_n.$ The structure theory of simple finite dimensional Lie algebras over $\mathbb C$ is essentially the theory of these root systems.  The root system allows us to construct the Lie bracket on a so-called \emph{Chevalley basis} of the Lie algebra $L$ and the structure coefficients with respect to this basis are integers, which are (up to a sign) given by the root system. Furthermore, every root system corresponds to a unique Dynkin diagram and a unique Cartan matrix, with irreducible Dynkin diagrams belonging to irreducible root systems. From the Dynkin diagram we can then obtain a presentation of $L$ by generators and relations. 
 Root systems (combinatorial discrete objects) therefore hold the key to the structure of all simple finite dimensional Lie algebras over $\mathbb C.$ \\
\paragraph{Whitehead's Lemma and central extensions.}
Let $L$ be a Lie algebra.  A \emph{central extension} $f: L'\rightarrow L$ of $L$ is  an exact sequences of Lie algebras  
$$
\xymatrix{& 0 \ar[r]&  \ker(f) \ar[r]&  L' \ar[r]^f & L \ar[r]&  0
}
$$
with the property that $\ker(f) \subseteq Z(L')$. If $L'$ is a \emph{perfect} Lie algebra, i.e.  $L' = [L', L']$, then $f: L' \rightarrow L$ is called a \emph{covering}. \\
Central extensions form a category and if $L$ is perfect, then $L$ admits a unique \emph{universal central extension} $ u : \mathfrak{uce}(L) \rightarrow L$. This means that $u: \mathfrak{uce} \rightarrow L$ is a central extension such that there is a unique morphism of central extensions from the central extension $u : \mathfrak{uce}(L) \rightarrow L$ to any other central extension $f: L ' \rightarrow L$ of $L.$ For more details on how to construct a universal central extension see Definition~\ref{uce_def} in this thesis.  The map that takes a perfect Lie algebra $L$ to the Lie algebra $\mathfrak{uce}(L)$ is a functor. A Lie algebra $L$ is called \emph{simply connected} if $L$ is perfect and $u: \mathfrak{uce}(L) \rightarrow L$ is an isomorphism of Lie algebras. 
We can formulate Whitehead's Lemma (see for example~\cite[Cor 7.9.5]{weibel}) as follows:
\begin{lem}
Every simple finite dimensional complex Lie algebra $L$ is simply connected. 
\end{lem}
 Mathematicians were interested in generalizing simple finite dimensional Lie algebras. 
One direction which is of lesser importance for the thesis, is the path chosen by Kac and Moody. In 1968, almost at the same time, they modified the presentations that had given the simple finite dimensional Lie algebras. The resulting Lie algebras are nowadays known as \emph{Kac-Moody algebras}. Unfortunately, the connection with root systems as defined above is lost in the transition to Kac-Moody algebras. 
\section{Background material}
\paragraph{Root-graded Lie algebras.}
From now on we will define all algebraic structures over a commutative associative unital ring  $k$, the base ring. Also, $\mathcal Q(R)$ will denote the abelian group generated by $R$ inside the vector space $E$.  We call $\mathcal Q(R)$ the root lattice. 
\begin{defi}[\cite{Neh3}]
Let $R$ be a reduced root system. An \emph{$R$-graded Lie algebra} is a $k$-Lie algebra together with a $\mathcal Q(R)$-grading
such that
\begin{enumerate}[(i)] 
\item  $L = \bigoplus_{\alpha \in R}L_\alpha,$ where $L_\alpha$ is a submodule of $L$,  called the \emph{homogeneous space of degree $\alpha,$} i.e., the support of the $\mathcal Q(R)$-grading is contained in $R.$
\item for every  non-zero $\alpha \in R$  there exists an element $e_\alpha \in L_{\alpha}$ and an element $f_\alpha \in L_{-\alpha}$ such that $\ad h$, $h := [f_\alpha, e_\alpha]$, acts diagonally on $L$: 
$$\ad h|L_\beta = \langle \beta, \alpha\che \rangle \mathrm{id}_{L_\beta}\quad \mbox{for all }\beta \in R.$$
\item $\sum_{ 0 \neq \alpha \in R}[L_{\alpha}, L_{-\alpha}] = L_{0}.$
\end{enumerate}
\label{r_delta_graded_def_inro}
\end{defi} 
The following questions come up quite naturally for a root-graded Lie algebra.
  \begin{itemize}
  \item How can we recognize the root system $R$ when we are given a root-graded Lie algebra? 
  \item What happens to Whitehead's Lemma? 
  \begin{itemize}
   \item How far away is $L$ from being simply connected or centreless? 
   \item Can we describe all  central coverings of $L$? 
   \end{itemize}
   \end{itemize}
   Those questions have been addressed by many authors before, among them Allison, Benkart,  Faulkner, Gao, Kassel, Loday, Moody, Neher, Seligman and Smirnov.  We will shortly summarize the  publications which have inspired or helped us the most. 
\subparagraph{Gradings and central extensions.} 
Let $L$ be a Lie algebra and $\Gamma$ a group. It was proven in \cite{Neh2}, that if $L$ is $\Gamma$-graded then $u: \mathfrak{uce}(L) \rightarrow L$ is $\Gamma$-graded as well and  $u$ is a $\Gamma$-graded Lie algebra morphism. \\
As a consequence, if $L$ is $R$-graded, then $\mathfrak{uce}(L)$ is $\mathcal{Q}(R)$-graded, where $\mathcal{Q}(R)$ is the root lattice.  
\\
The first definition of a root-graded Lie algebra can be found in the paper \cite{BM92} by Berman and Moody. However, already in Seligman's  book \cite{Seligman} the concept appears more or less hidden. The case of an $A_1$-graded Lie algebra even goes back to Tits \cite{TitsKK}, Kantor \cite{TKantorK} and Koecher \cite{TKKoecher}. Seligman also provides a good part of the machinery which was later used to classify root-graded Lie algebras. In particular, he uses Jordan algebras, alternative algebras and associative algebras to describe what he calls Lie algebras of type $R$  (where $R$ is a root system).\\
Therefore, once root-graded Lie algebras were defined, the project was to classify these Lie algebras and describe them in terms of other algebraic structures, called \emph{coordinate algebras}, in a way much similar to Seligman's approach. 
\paragraph{Root graded Lie algebras in characteristic zero.}
This project was carried out over fields of characteristic zero. We commonly call classification theorems for root-graded Lie algebras recognition theorems.  Berman and Moody were able to prove recognition theorems for root graded Lie algebras of type $A_n, n \geq 2,$ $D_n, n \geq 4$ and $E_6$, $E_7$ and $E_8$ in the same paper where the definition appears. In 1994, Benkart and Zelmanov proved recognition results for the remaining reduced root systems in \cite{BZ1}. At about the same time, Neher provided results over rings instead of fields in \cite{Neh1996}.\\
It became clear that even if one knew the coordinate algebra and the root system, it was only possible to retrieve the Lie algebra up to a central quotient, i.e., a quotient by a central ideal. Since all root-graded Lie algebras (over fields of characteristic $0$) are perfect, it was clear that one had to find the central coverings. For reduced root systems, the answer was given by Allison, Benkart and Gao in \cite{ABG2}. For the non-reduced types not equal to $BC_1,$ the project was accomplished in \cite{ABG} by the same authors. The recognition theorems for type $BC_1$ finally can be found in \cite{AllisonBC} and the central extensions are described in \cite{BS}. \\
This finalized the project for Lie algebras over a field of characteristic $0.$ However, the question was still open for finite fields and arbitrary base rings. 
\paragraph{Known results for more general base rings.}
Neher's work in \cite{Neh1996} makes it clear that we can switch from a field of characteristic $0$ to some more general base ring and the theory will still be very rich. In this paper, Jordan algebras are an  important instrument and their theory is well developed over arbitrary base rings. There is also an article by Berman, Gao, Krylyuk and Neher (\cite{BGKN}) which gave us some inspiration in the case of $A_2$-graded Lie algebras. \\
When dealing with central extensions, it is for many reasons advantageous to start with a centreless Lie algebra. Neher provided in \cite{Neh1996} a construction of centreless root-graded Lie algebras as Tits-Kantor-Koecher algebras of certain Jordan pairs. Working over an arbitrary ring $k$, he finds realizations over $k$ for all types except $F_4, G_2$ and $E_8.$  His work already includes various examples for central extensions of those root-graded Lie algebras, but it is in general not clear if they are universal or not. We would like to point out that also Benkart and Smirnov in \cite{BS} replace the Jordan pairs by more general Jordan-Kantor pairs while working over a field of characteristic $\neq 2,3.$ Also, they construct some central extensions of the aforementioned Tits-Kantor Koecher algebras. It is remarkable that, while the Tits-Kantor Koecher construction is not functorial, they introduce a generalized Tits-Kantor-Koecher construction which is functorial as a map from Jordan-Kantor pairs to Lie algebras. \\
If we go back to the finite dimensional simple Lie algebras over $\mathbb C,$ then all of them have adequate analogs over any ring (``forms''). In \cite{vdK} van der Kallen was able to describe the universal central extensions  of those forms over any ring. In particular he proves that Whitehead's Lemma is false over the integers. Certain $\mathbb Z$-forms (for the expert, the simply connected Chevalley forms) of simple finite dimensional Lie algebras over $\mathbb C$  are in general not simply connected. In particular, van der Kallen is able to prove that additional homogeneous spaces  of non-zero degree show up in the universal central extension which lie in the centre. This is in sharp contrast to the result for root-graded Lie algebras in characteristic $0,$ where the kernel of the universal central extension always has degree $0$ with respect to the grading by the root lattice. Van der Kallen named the additional weights in the support of the central extension degenerate sums.
Degenerate sums can be described as follows
\begin{center}
A \emph{degenerate sum} is an element $\gamma $ of the root lattice $\mathcal Q (R)$ which is a sum of two linearly independent roots and has the property that $\langle \gamma, \alpha \che \rangle \neq \pm 1$ for all $\alpha \in R$. \end{center}
In fact, for $\gamma$ a degenerate sum, we either have $ \langle \gamma, R\che \rangle \subset  2 \mathbb Z$ or  $\langle \gamma,R \che \rangle \subset 3 \mathbb Z. $ \\
In \cite{GS07}, Gao and Shang proved, without using degenerate sums explicitly, that the degenerate sums  of the root systems  $A_2$ and $A_3$ will show up in the support of the universal central extension, if we compute the universal central extensions of $\mathfrak{sl}_3(D)$  and $\mathfrak{sl}_4(D)$ for an associative algebra $D$ over a field of characteristic $2$ or $3.$ The structure of the universal central extension was obtained for $\mathfrak{sl}_n(D),$ $n \geq 5$ in \cite{KasLo}. Hence with Gao and Shang's work, we have a complete understanding of the universal central extension of $\mathfrak{sl}_n(D), n \geq 3$ over a field of arbitrary characteristic. \\

\section{A short guide to this thesis}
In our thesis, we work towards describing the central coverings of root-graded Lie algebras over base rings. 
We will now give an overview of our main results. 
\paragraph{Tits-Kantor-Koecher algebras.}
Chapter 4 and part of Chapter 5 concern Jordan pairs or more generally Jordan-Kantor pairs. For the purpose of this introduction we will only talk about Jordan pairs. 
A Jordan pair is a pair $J = (J^+, J^-)$ together with  quadratic maps $Q^\sigma : J^{\sigma} \to \mathrm{End}(Q^{-\sigma}), a \mapsto Q_a$, such that the following identities hold in all scalar extensions of $k$: for $a, c \in J^\sigma$, $b \in J^{-\sigma}$
\begin{eqnarray*}
D_{a,b}Q_{a} &=& Q_{a}D_{b,a} , \\
D_{Q_a b, b } &=& D_{a, Q_b a}, \\
Q_{Q_a b} &=& Q_a Q_b Q_a.
\end{eqnarray*} 
Here $Q^\sigma_{a,c} b  = D^\sigma_{a, b} c$ denotes the linearization of the quadratic map $Q^\sigma.$ The sign $\sigma$ is omitted to avoid repetitions. Sometimes we also write $Q(a) := Q_a, Q(a, c) := Q_{a, c}$ or $D(a, b) := D_{a, b}. $\\ 
We define in Chapter~\ref{TKK_chapter} two  Lie algebras, namely the derivation algebra $\instr(J)$ and the universal derivation algebra $\uider(J).$ 
With these definitions, the $k$-modules $\TKK(J) = J^ + \oplus \instr(J) \oplus J^-$ and $\uTKK(J) = J^ + \oplus \uider(J) \oplus J^-$ can be endowed with Lie algebra structures such that $\TKK(J)$ is a centreless  Lie algebra and $\uTKK(J)$ is a central extension of $\TKK(J)$. Both Lie algebras are $\mathbb Z$-graded with support $\{-1, 0, 1\}.$ 
This provides us with a useful generalization of the (classical) Tits-Kantor-Koecher construction. 
In particular, $\uTKK(-)$ is functorial and we prove (see Corollary~\ref{JKP_central_ext_kernel_cor} ) that the following holds:
\begin{quote}
Let $J$ be a Jordan pair, $L = \TKK(J)$ and endow $\mathfrak{uce}(L)$ with the canonical $\mathbb Z$-grading obtained from the $\mathbb Z$-grading on $L$. Then there is a Lie algebra isomorphism 
$\uce{L}^0 \rightarrow \uider(J).$
\end{quote}
(We use superscripts for the $\mathbb Z$-grading in order to avoid confusion with gradings by $\mathcal{Q}(R)$).
Our work generalizes \cite[Sec. 5]{BS} to arbitrary base rings. 
\paragraph{Degenerate sums.}
Our next result (Proposition~\ref{torsion_bijection_prop}) describes the support of the universal central extensions of a root-graded Lie algebra in the root lattice. In complete analogy with \cite{vdK} we prove the following
for a  perfect root-graded Lie algebra  $L$ and $u: \mathfrak{uce}(L) \rightarrow L$ a $\mathcal{Q}(R)$-graded central extension: If $\gamma \in \mathcal{Q}(R),$ then: 
\begin{itemize}
\item  $\mathfrak{uce}(L)_\gamma \neq \{0\}$ implies that $\gamma$ is either a root or a degenerate sum. 
\item If there is $\alpha \in R$ such that $ \langle \gamma , \alpha \che \rangle \in k^\times,$ then the restriction map $ u : \mathfrak{uce}(L)_\alpha \rightarrow L_\alpha$ is an isomorphism. In particular, this holds if $\gamma \in R^\times.$ 
\item If $\gamma$ is a degenerate sum, then either $2\mathfrak{uce}(L)_\gamma = \{0\}$ or $3\mathfrak{uce}(L)_\gamma = \{0\}.$
\end{itemize}
This extends van der Kallen's result to root-graded Lie algebras. 
We study how our results unfold when they are applied to the most common models for Lie algebras graded by root systems of type $A$ and $C$ and of rank at least $2.$ So this excludes type $A_1.$  
\paragraph{Type A.}
Fix a finite index set $K$, $\card K \geq 3,$ a partition $\{1\} \cup J$ of $K$, and
a unital alternative $k$-algebra $D$, which is associative, if $\card K \geq 4.$ The rectangular matrix Jordan pair $\mathbf M(\{1\}, J, D)$ of size $1 \times J$ is the pair of modules
$$(V^+, V^-) = (\mathrm{Mat}(1, J, D),\mathrm{Mat}(J, 1, D))$$
with quadratic operator given by $$Q_x(y) = xyx \quad \mbox{for  } (x,y) \in V.$$
It is well-known that $L = \TKK(V)$ is an $ A_{\card K + 1}$-graded Lie algebra (see for instance \cite{Neh1996}). Hence we can use our results about the universal central extensions of Tits-Kantor-Koecher algebras to understand the universal central extension of $L$.\\
The structure of the universal derivation algebra of $\TKK(V)$ is given in Proposition \ref{uider_alt_decomp_prop}.
As far as the module structure  of the universal central extension is concerned, we have the following description
\begin{itemize}
\item
 The $k$-module $\uce L_0$ is isomorphic to $$ \frac{(D \otimes D)  \oplus \bigoplus_{j \in J} (D)_{j}}{U}$$  where $D_{j}$, $j \in J,$ are copies of $D$ and $U$ is the submodule generated by:
\begin{eqnarray*}
a \otimes b - b\otimes a, \\
ab\otimes c + bc \otimes a + ca \otimes b - (a, b,c)_{j_0} -(a, b,c)_{j},  
\end{eqnarray*} for distinct $j, j_0$ in $J$ and any $a, b, c \in D.$
Here $(a, b, c) = (ab)c - a(bc) $ is the \emph{associator} of three elements $a, b$ and $c$ in $D$.
\item If $\card K \geq 5,$ then no degenerate sums occur.
\item If $\card K  = 4,$ then a total of $6$ degenerate sums occur and  
$$ \mathfrak{uce}(L)_\gamma  \cong D / (\langle 2D, [D, D]\rangle_{\mathrm{ideal}})$$
for $\gamma$  a degenerate sum.
See Theorem~\ref{deg_sum_theo} for details.
\item If $\card K = 3,$ then again $6$ degenerate sums occur and
$$ \mathfrak{uce}(L)_\gamma  \cong D / (\langle 3D, D[D, D], (D, D, D), \{(ad.c + a.dc + a.cd) b : a, b, c, d \in D \}\rangle_{k-mod})$$
for $\gamma$ a degenerate sum.
See Theorem~\ref{deg_sum_theo_A_3} for details.
\end{itemize}
The Lie product on the universal central extension is described by Lemma~\ref{Lemma_1_for_A_2} and Lemma~\ref{Lemma_1_for_A_3}.  Since $\mathfrak{uce}(L)_0$ can be viewed as a subalgebra of $\uider(V),$ the Lie structure on $\uTKK(V)$ also gives the structure of $\mathfrak{uce}(L)_0.$\\  
If we assume that $D$ is associative and that $k$ is a field, then  our result implies  the main results of the paper \cite{GS07}. Progress has been made in so far that we are not restricted to fields anymore and that we have also treated the case of an alternative algebra. 
\subparagraph{Octonion algebras.}If we specialize the coordinate algebra to an octonion algebra $\mathbb O$  over the ring $k$ and assume that $1/3$ lies in the ring $k,$ then we can prove that $L = \TKK(\mathrm{Mat}(1, 2, \mathbb O),\mathrm{Mat}(2,1, \mathbb O))$ is simply connected, see Theorem~\ref{simply_Connected_oct_theo}. Here we do not assume that $\mathbb O$ is split or reduced and in this generality the result has been unknown. As a beautiful by-product we prove that $\mathrm{Der}(\mathbb O)$ is a Lie algebra of type $G_2.$  The definition of a Lie algebra of type $G_2$ is given in Definition~\ref{type_g_defi}. It had just recently been shown  by Loos, Petersson and Racine (\cite[p.966]{lopera}) that  the derivation algebra  of $\mathbb O$ is finitely generated projective of rank $14$ and this was of course a strong indicator that it should be of type $G_2.$ 
\paragraph{Type C.}
When working with Lie algebras graded by root systems of type $C,$  we assume $1/2 \in k.$ The presence of $1/2$ guarantees that $L$ is perfect and that we do not have to take degenerate sums into consideration (Corollary~\ref{ker_ud_supporr_cor}). 
It is possible to go back and forth between unital Jordan algebras and Jordan pairs.  
and the  following  holds for $V= (J, J)$ the Jordan pair associated to a Jordan algebra $J$: 
$$ \TKK(V) = \TKK(J), \mbox{  and}$$
$$\uider(V) \cong J \oplus J * J $$
where  $J * J$ is the quotient of $J \otimes J$ modulo the submodule $M$ generated by 
  \begin{eqnarray*}
  a \otimes (b  c) + b \otimes (c  a) + c \otimes (a  b),\\
  a \otimes b + b \otimes a.
  \end{eqnarray*}
 We view this  as the universal version of the well-known decomposition 
$\instr(V) = L_{J} \oplus [L_{J}, L_{J}]$ where $V$ is the Jordan pair given by the Jordan algebra $J.$
Proposition~\ref{one_half_jordan_hom_H_2_prop} states that then $\uTKK(V)$ is a  universal covering of $\TKK(V).$ The kernel of the universal central extension $u: \mathfrak{uce}(\TKK(J)) \rightarrow \TKK(J)$ is given by $\{\sum a_i * b_i : \sum [L_{a_i}, L_{b_i}] = 0 \} \subset J*J.$ 
\\
The  Jordan algebra of hermitian matrices with entries in a unital algebra $D$ with nuclear involution ($ \mathcal H_n(D, -), n \geq 3$) is of particular significance. The algebra $D$ is alternative if $n = 3$ and associative if $n \geq 4.$ 
The hermitian matrix Jordan algebra admits a Peirce decomposition with respect to a set of $n$ orthogonal idempotents. This  defines a $C_n$-grading on its Tits-Kantor-Koecher algebra $\TKK(J)$. \\
We can further decompose $J*J$ with respect to the idempotents. This is carried out in Section~\ref{idempotent_homology}. The results enable us to compute the centre of $\uce{\TKK(J)},$  see Proposition~\ref{C_N_prop_Center}.
\subparagraph{Albert algebras.}
Finally, we consider another important example, namely the Albert algebra $ \mathcal A = \mathcal H_3(\mathbb O, -)$ where $\mathbb O$ is an octonion algebra over $k$.  Assuming that $1/2$ and $1/3$ lie in in the ring $k$, we prove  that $\mathrm{IDer}(\mathcal A)$ is simply connected and of type $F_4.$ See Definition~\ref{Albert_Alg_Def_ring} for a explanation of type $F_4.$
\paragraph{An even shorter guide to this thesis.}
Chapter~\ref{fundamental_concepts_chapter} introduces some of the  \textbf{non-associative structures} that frequently show up in the thesis, for example alternative algebras and Jordan algebras. \\
People who are interested in \textbf{graded central extensions}, should read Chapter~\ref{central_extensions_chapter}. 
We define universal central extensions and introduce some interesting ways to construct central extensions. For graded central extensions the functor $\mathfrak{gce}$, introduced in Section~\ref{gce_Section} is quite useful. 
 The results in this chapter hold for a larger class of Lie algebras than root-graded Lie algebras. \\
 A reader who is familiar with \textbf{Jordan algebras} might find Chapter~\ref{TKK_chapter} a good starting point. The universal \textbf{Tits-Kantor-Koecher algebras} are defined in Section~\ref{TKK_Section}. \\
 Somebody who looks for information on \textbf{root systems and root graded algebras}, could begin with  Chapter~\ref{root_systems_chapter}. This chapter contains the definitions of root systems and root-graded Lie algebras.  The first part uses very little Jordan theoretic methods. There we work towards understanding the support of the universal central extension in the root lattice. In the second part of Chapter~\ref{root_systems_chapter}  we look at \textbf{grids} which are roughly to Jordan pairs what root systems are to Lie algebras. In order to understand all the proofs, you might have to go back and read up on some of the material in earlier chapters. \\ The thesis contains a good deal of statements about  \textbf{trialities, alternative algebras  and derivations of alternative algebras}. These can be found in Chapter~\ref{fundamental_concepts_chapter}, Section~\ref{derivations_section}, Lemma~\ref{der_alt_explicit_form_lem}, Subsection~\ref{alt_coordinates} and Subsection~\ref{Octonion_alg}. Some of the statements are our own, some are ``folklore'', some are lesser known results of other authors, but all are referenced to the best of our knowledge.  \\
Those who have seen the \textbf{Freudenthal Magic Square},  might appreciate Sections~\ref{Octonion_alg} and \ref{Albert_example}. We prove that two entries, namely the first two ones in the last row, of the Magic Square work over very general base rings. 
The proofs of this rely on Chapter~\ref{TKK_chapter}, Section~\ref{Jordan_graded_Lie_alg} and Section~\ref{derivations_section}, and the article \cite{lopera}.
  \\ Chapter~\ref{AC_chapter} combines all of the previous chapters in an effort to understand two important examples where the root system are of type $A$ or $C.$  Section~\ref{derivations_section} establishes the facts about alternative algebras which are needed throughout. The next three sections describe central coverings of  \textbf{$A$-graded Lie algebras}, and the last section deals with the same questions for \textbf{$C$-graded Lie algebras}.

\chapter{Algebraic structures}
We recall some definitions and facts about algebraic objects which will play a central role in this thesis. 
\label{fundamental_concepts_chapter}

\section{Categories of algebras}

\subsection*{Algebras}
Let $k$ be a unital commutative associative ring with unit $1$, called the \emph{base ring}.
All modules are assumed to be over $k$ unless indicated otherwise. We will be mostly interested in categories where all the sets of morphisms  $\mathcal{Mor}(X, Y)$ are actually $k$-modules and all the objects are $k$-modules.  Besides categories where the objects are algebras, we will need certain categories of pairs which will however be defined while we go along. 
\subsubsection*{Graded linear algebras}

\begin{defi}  
\label{linear_algebra_definition} 
A  (linear) $k$-algebra is a $k$-module $A$ together with a bilinear operation $\cdot : A \times A \rightarrow A$, called the \emph{product}.  We sometimes will write also $\mu(a,b) := a \cdot b$ and
$L_a(b) = R_b(a) = a \cdot b.$
We write $(A, \cdot)$ for the $k$-module $A$ with product $\cdot$ or $A$ if the product is understood.
The collection of $k$-algebras forms a category and the morphisms are the $k$-linear maps $f : A \rightarrow B$ such that $f(a \cdot b) = f(a)f(b).$  
These morphisms are called \emph{$k$-algebra morphisms}.
The category of $k$-algebras will be denoted by $\mathbf{Alg}_k$.
\end{defi}
\begin{defi} A \emph{unital}  $k$-algebra is a $k$-algebra $A$ together with an element $1_A \in A$ such that 
$$ L_{1_A} = R_{1_A} = \mathrm{id}_A.$$
Unital  $k$-algebras form a category with  morphisms all those  algebra morphisms $f : A \rightarrow B$ such that $f(1_A)= 1_B.$ 
An element $a \in A$, $A$ a unital $k$-algebra, is called \emph{invertible,} if there is $a^{-1} \in A$ such that 
$a a^{-1} = a^{-1}a = 1_A.$

\label{algebra_with_involution_defi}
For  $(A, \cdot)$,  an object of $\mathbf{Alg}_k$ with product $\cdot$, the $k$-algebra $(A, \cdot_{\mathrm{op}})$ where $a  \cdot_{\mathrm{op}} b = b \cdot a$ is called the \emph{opposite algebra} of $A$ or $A^{\mathrm{op}}.$ 
The morphisms $\mathcal{Mor}(A, A^{\mathrm{op}})$ are called anti-morphisms. 
An \emph{involution} on $A$ is an element of $\mathcal{Mor}(A, A^{\mathrm{\mathrm{op}}})$ which has order $2$ as module morphism. 
\end{defi}

\begin{defi} \label{commutator_defi} \label{associator_defi}
Let $A$ be a $k$-algebra. The \emph{commutator} of $a$ and $b$ in $A$ is
$$[a, b] = a\cdot b - b \cdot a .$$
The \emph{associator} of $a$, $b$ and $c$ in $A$ is
$$(a, b, c) = (ab)c - a(bc). $$
\end{defi}

\begin{defi} Let $A$ be a $k$-algebra. Then the \emph{nucleus} of $A$ is the $k$-module
$$ \mathrm{Nuc}(A) = \{a \in A: (a,x, y) = (x ,y, a) = (y, a, x)= 0, \mbox{ for all } x, y \in A \} $$
and the \emph{centre} of $A$ is
$$\mathrm{Cent}(A) = \{ a \in \mathrm{Nuc}(A): [x,a] = 0 \mbox{ for all } x\in A \}.$$
A map $f: A \rightarrow A$ is called \emph{central} resp. \emph{nuclear}, if $f(a) = a$ implies that $a \in \mathrm{Cent}(A)$ resp. $a \in \mathrm{Nuc}(A).$ 
\label{nucleus_defi}
\label{centre_defi}
\end{defi}

\begin{defi}
\label{alternative_algebra_definition}
A  $k$-algebra is an \emph{alternative algebra} if for all $a, b \in A$: 
$$ (a \cdot a) \cdot b = a\cdot (a \cdot b), \quad b \cdot (a \cdot a) = (b \cdot a) \cdot a .$$
Alternative algebras form a full subcategory of the category of  algebras, i.e., every algebra morphism between alternative algebra $f: A \rightarrow B$ is a morphism of alternative algebras. 
\end{defi}

In every alternative algebra, the following identities hold true (see \cite[p. 936]{lopera}):
 
\begin{eqnarray}
\left [L_a, L_b \right ] &=& L_{[a,b]} -2[L_a, R_b] \label{alt_id_1} \\
\left [R_a, R_b \right] &=& - R_{[a,b]} -2[L_a, R_b] \label{alt_id_2} \\
\left [[L_a, R_b], L_c \right] &=& L_{(a,b,c)} -[L_{[a,b]}, R_c]
\label{alt_id_3}
\\ \left [[L_a, R_b], R_c \right ] &=& R_{(a,b,c)} + [R_{[a,b]}, L_c]
\label{alt_id_4}
\end{eqnarray}

\begin{defi}
\label{associative_algebra_definition}
A  $k$-algebra is an \emph{associative} algebra if for all $a, b,c \in A$: 
$$ (a \cdot b) \cdot c = a \cdot (b \cdot c).$$
Associative algebras form a full subcategory of the category of algebras. 
\end{defi}

\begin{defi}
\label{commutative_algebra_definition}
A $k$-algebra is a \emph{commutative algebra} if for all $a, b\in A$: 
$$ a \cdot b= b \cdot a.$$
Commutative algebras form a full subcategory of the category of  algebras. 
\end{defi}

\begin{defi} \label{Lie_algebra_Definition}
A  $k$-algebra $L$ is a \emph{Lie algebra} if for all $a$, $b$  and $c$ in $L$
$$a a = 0, \quad a (bc) + b(ca) + c (ab) = 0 .$$
In the case of a Lie algebra, we denote the product usually by $[a, b].$ This conflicts with previous notation, but it will not cause any problems in our work. \\
\label{Liealgebra_centre_definition}
The \emph{(Lie) centre of a Lie algebra} is the set
$$Z(L) = \{a \in L : ax = 0 \mbox{ for all } x \in L \}$$
\end{defi}
\begin{rem} If $L$ does not have $2$-torsion, then $Z(L)= \mathrm{Cent}(L).$
\end{rem}

\subsection*{Base ring extension}
A unital ring homomorphism $k \rightarrow K$ is called a \emph{base ring extension}. If $M$ is a $k$-module, then there is  a canonical map from $\mathrm{End}_k(M)$ into $\mathrm{End}_K(M \otimes_k K)$ mapping $T$ to  $T_K$ such that $T_K(m \otimes_k a) := T(m) \otimes_k a$ for all $m \in M$ and $a \in K.$  \\
  If the extension $k \rightarrow K$ is flat, then  $T \mapsto T_K$ is injective. If in addition $M$ is finitely generated as a $k$-module, then  $ \mathrm{End}_k(M) \otimes_k K \rightarrow   \mathrm{End}_K(M \otimes_k K)$ is injective.  If $k \rightarrow K$ is flat and $M$ is finitely presented as a $k$-module, then $\mathrm{End}_k(M) \otimes_k K \rightarrow   \mathrm{End}_K(M \otimes_k K)$ is bijective. If $M$ is finitely generated and projective over $k,$ then $\mathrm{End}_k(M) \otimes_k K \rightarrow   \mathrm{End}_K(M \otimes_k K)$ is bijective for all base changes $k \rightarrow K$.

\subsubsection*{Quadratic structures}

\begin{defi}
Let $M$ be a $k$-module. A \emph{quadratic} operator on $M$ is a map $Q : M \rightarrow \mathrm{End}_k(M). x \mapsto Q_{x}$ with the property that 
$$Q_{\lambda x} = \lambda^2 Q_{x}, \quad \forall x \in M, \; \forall \lambda \in k.$$
Every quadratic map can be linearized to obtain a bilinear symmetric map
$$ Q_{x, y} = Q_{x + y} - Q_{x} - Q_{y}. $$ One can show that quadratic maps extend to $M \otimes_k K$ for every extension $k \rightarrow K.$ Therefore, if $Q : M \rightarrow \mathrm{End}_k(M)$ is a quadratic operator, then $$Q  \otimes _k K: M \otimes_k K \rightarrow \mathrm{End}_K(M \otimes_k K)$$ is also a quadratic operator. 
If $1/2 \in k,$ bilinear symmetric operators and quadratic operators are in one to one correspondence. 
\label{quadratic_operator_definition}
\end{defi}

\begin{defi} The category of (quadratic) Jordan algebras has as objects  pairs $(J, U)$ where $J$ is a $k$-module and $U : J \rightarrow \mathrm{End}_k(J), a \mapsto U_a$ is a quadratic operator. We require the following identities  to hold in all scalar extensions of $k$: \\ For $a, b, c \in J$, 
\begin{eqnarray}
V_{a,b}U_{a} &=& V_{a}V_{b,a}, \\
V_{U_a b, b } &=& V_{a, U_b a},\\
U_{U_a b} &=& U_a U_b U_a. 
\end{eqnarray}
Here  we define $U_{a, c} := U(a+ b) - U(a) -U(c) $ and $V_{a, b}c := U_{a, c}b$ for any $a, b, c \in J.$

If there is an element $1_J$ in $J$ such that $U_{1_J} = \mathrm{id}_J$, then $J$ is called a \emph{unital quadratic Jordan algebra.  }
Since the identities defining a quadratic Jordan algebra are required to hold  in all scalar extensions, the pair $(J \otimes_k K, U \otimes_k K)$ is a quadratic Jordan algebra for every base ring extension  $k \rightarrow K.$
\label{quadratic_Jordan_algebra_definition}

\end{defi}

\begin{rem}If $1/2 \in k,$ then every unital quadratic Jordan algebra gives  a unital Jordan algebra by setting
$$2 a b = U_{a,b} 1. $$ 
\end{rem}

\subsection*{Derivations}
If $M$ is a $k$-module, then $\mathrm{End}_k(M)$ is an associative unital algebra. It is easy to check that $[S, T] = S \cdot T- T \cdot S $ defines on $\mathrm{End}_k(M)$ a Lie bracket. The resulting Lie algebra is denoted by $\mathrm{End}_k(M)^{-}.$
For any algebra $A$ we define inside the Lie algebra $\mathrm{End}_k(A)^-$ the \emph{Lie multiplication algebra} $\mathfrak{L}(A)$ as the subalgebra generated by the left and right multiplications. \\
Even if $A$ is unital (see Definition~\ref{linear_algebra_definition}), it is always possible to form the unital hull of $A$, that is the $k$-module $\hat A = 1 k \oplus A$ together with the map $\hat \mu (\alpha \oplus a, \beta \oplus b) = \alpha\beta \oplus (\alpha b + \beta b + \mu (a,b))$ which will result in a unital algebra $\hat A$ with unit $1 \oplus 0.$ 
\begin{defi} A \emph{derivation} of $A$ is an  element $\Delta \in \End_k(A)$ such that 
$\Delta(ab) = \Delta(a)b + a\Delta(b).$ Every derivation $\Delta$ of $A$ extends uniquely to a derivation of $\hat A$ denoted by $\hat {\Delta}$ by setting $\hat \Delta(1 \oplus 0) = 0.$ The $k$-module of all derivations of $A$ is denoted by $\mathrm{Der}(A).$ 
With this convention \emph{multiplication derivations} are the following: 
$$ \mathrm{MulDer}(A) := \{\Delta \in \mathrm{Der}(A) : \hat \Delta \in \mathfrak{L}(\hat A)\} \subset \mathrm{Der}(A)$$
where $\mathfrak{L}(\hat A)$ is the Lie algebra generated by all left and right multiplications. 
 \label{Der_defi}
\label{Mul_Der_Defi}
\end{defi}
It is not difficult to show that $\mathrm{Der}(A)$ is a Lie subalgebra of $\mathrm{End}_k(A)$: Let $\Delta_1$ and $\Delta_2$ be derivations. 
Then $\Delta_1(\Delta_2(ab)) = \Delta_1(\Delta_2(a)b + a\Delta_2(b)) = \Delta_1\Delta_2(a)b + \Delta_2(a)\Delta_1(b) + \Delta_1(a)\Delta_2(b) + a\Delta_1\Delta_2(b), $
hence
$[\Delta_1, \Delta_2](ab) = [\Delta_1, \Delta_2](a)b  + a[\Delta_1, \Delta_2](b).$

\subsection*{Graded structures}
\begin{defi}
Let $M$ be a $k$-module and $\Gamma$ a group. If there is a family of submodules $(M_\gamma)_{\gamma \in \Gamma}$ of $M$  such that $M = \bigoplus_{\gamma \in \Gamma}M_\gamma,$ then $M$ is a $\Gamma$-graded module. The $\Gamma$-graded modules form a subcategory of the category of $k$-modules.  \\
Let $M$ and $N$ be two $\Gamma$-graded modules.  Define $\mathrm{grHom}_k(M,N) = \bigoplus_{\gamma \in \Gamma} \mathrm{grHom}_k(M,N)_\gamma $ where $\mathrm{grHom}_k(M,N)_\gamma$ consists of all $f \in \mathrm{Hom}_k(M,N)$ such that $f(M_\delta) \subset N_{\delta + \gamma}$ for all $\delta \in \Gamma.$
 A $\Gamma$-graded morphism is an element $f \in \mathrm{grHom}_k(M,N).$ 
 \end{defi}
\begin{defi} A  $\Gamma-$\emph{graded algebra} $A$ is an algebra such that $A$ is a $\Gamma$-graded module and such that for $a_\gamma \in A_\gamma,$ the operators $L_{a_\gamma}$ and $R_{a_\gamma}$ are elements of $\mathrm{grHom}_k(A, A)_{\gamma}$.  The morphisms $\mathcal{Mor}_{\mathrm{grAlg}}(A, B)$ are the elements of $\mathrm{grHom}_k(A, B) \cap \mathcal{Mor}_{\mathrm{Alg}_k}(A, B).$ 
\end{defi}
\subsection*{Notation}
Although this might be out of place here, we would like to introduce the following notation which we frequently use: Let $S$ be a set and $s_1, s_2, \ldots, s_k$ a family of elements of $S.$ Then  $s_1, s_2, \ldots, s_k \neq$  means that $s_1, s_2, \ldots, s_k$ are pairwise distinct. This notation is common in Jordan theory. 
\chapter{Central extensions and gradings}
Unless stated otherwise, $k$ will be an arbitrary commutative associative ring, the base ring. This assumption holds for the whole thesis. 
\label{central_extensions_chapter}
\section{The category of graded central extensions}

\subsection*{Central extensions}
Let $L$ be a Lie algebra over $k$. 
By definition, a \emph{central extension} $f: L'\rightarrow L$ of $L$ is  an exact sequence of Lie algebras  
$$
\xymatrix{& 0 \ar[r]&  \ker(f) \ar[r]&  L' \ar[r]^f & L \ar[r]&  0
}
$$
with the property that $\ker(f) \subseteq Z(L')$. If $L'$ is a perfect Lie algebra, i.e.,  $L' = [L', L']$, then $f: L' \rightarrow L$ is called a \emph{covering} and we refer to the homomorphism $f$ as  the \emph{covering map} of the central extension. \\
A \emph{homomorphism of central extensions} from the central extension $f : L' \rightarrow L$ to the central extension  $g : K \rightarrow L$ is a Lie algebra homomorphism $h :L' \rightarrow K$ such that the following diagram commutes:
$$
\xymatrix{
& 0 \ar[r]&  \ker(g) \ar[r]&  L' \ar[r]^f\ar[d]^h & L\ar[d]^{\mathrm{id}} \ar[r]&  0 \\
& 0 \ar[r]&  \ker(f) \ar[r]&  K \ar[r]^g & L \ar[r]&  0.
}
$$
A central extension $f : L' \rightarrow L$ is called \emph{universal}, if for every central extension $g : K \rightarrow L,$ there is a unique homomorphism of central extensions $h : L' \rightarrow K$.

\subsection*{Central extensions from $2$-cocycles}
\begin{defi}Let $L$ be a Lie algebra and $M$ a $k$-module. A $k$-bilinear map $\psi : L \times L \rightarrow M$ is called a \emph{$2$-cocycle}, if
\begin{itemize}
\item[\rm{(i)}] $\psi$ is alternating, 
\item[\rm{(ii)}] $\psi(x, [y,z]) + \psi(y, [z, x]) + \psi(z, [x, y]) = 0$ for all $x, y, z  \in L.$
\end{itemize}
\end{defi}
\begin{defi}Let $L$ be a Lie algebra and $\psi : L \times L \rightarrow M$ a $2$-\emph{cocycle} on $L$ with values in  a $k$-module $M$. Then 
$L \oplus_{\psi} M$ with bracket defined by $[(x, m), (y, m')] = ([x, y], \psi(x,y))$  is  a central extension of $L$. We have an exact sequence
$$0 \rightarrow M \rightarrow L \oplus_{\psi} M \rightarrow L $$ where the homomorphism $L \oplus_{\psi} M \rightarrow L$ is simply projection on the first coordinate. This central extension will be referred to as the \emph{central extension of $L$ by $(M, \psi)$ or the central extension of $L$ by $M$ via $\psi.$}
\label{central_ext_by_cocyle}
\end{defi}
\begin{expl}
Every Lie algebra $L$ admits the trivial $2$-cocycle, $\psi = 0$ for any module $M$. More concretely, for $M$ a $k$-module, the direct sum $L \oplus M$ carries a  Lie algebra product 
$$[(x, v), (y, w)]_{L \oplus M} = ([x,y]_L, 0) $$
where the bracket on the right hand side is the Lie bracket on $L$. Then 
\begin{eqnarray*}
f: L\oplus M &\rightarrow & L \\
(x, v) &\mapsto & x
\end{eqnarray*}
is a central extension. The corresponding exact sequence is
$$
\xymatrix{& 0 \ar[r]&  M \ar[r]&  M \oplus L \ar[r]^f & L \ar[r]&  0.
}
$$
\end{expl}
\begin{expl} Every Lie algebra  $K$ which has a non-trivial centre is obviously a central extension in the following sense: If 
 $J \subset Z(K)$ is a central ideal of $K,$ then we have a central extension $f : K \rightarrow K/ J$, where the homomorphism $f$ is given by the quotient map. 

\end{expl}

\begin{lem}[The central trick]Let $f: K \rightarrow L$ be a central extension. If $f(x)= f(x')$ and $f(y) = f(y'),$ then $[x,y] = [x', y'].$
\label{central_trick}
\end{lem}
\begin{proof} We have $x' \in x + \ker f$, $y ' \in y + \ker f$. But since $\ker f \subset Z(K)$, it follows that $\ad x  = \ad x'$ and also $\ad y = \ad y'$ which proves the claim.
\end{proof}

\subsection*{Central extensions from actions}
The next lemma is slight generalization of \cite[Lemma 5.6]{BS}, not assuming $1/2 \in k$ and it is also a consequence of the more general \cite[Lemma 3.6]{lopera}. We include a proof for the convenience of the reader. 
\begin{lem}Let $L$ be a Lie algebra and $Q$ an $L$-module. If there is a $k$-linear map $\lambda : Q \rightarrow L$ such that for all $x, y \in Q$ \label{Lie_alg_from_mod}
\begin{itemize}
\item[\rm(i)] $\lambda(x).x = 0,$
\item[\rm(ii)]$[\lambda(x), \lambda(y)] = \lambda(\lambda(x).y).$
\end{itemize}
Then
\begin{itemize}
\item[\rm(i)]  $Q$ is a Lie algebra under the product $[x,y] = \lambda(x).y$ and $\lambda : Q \rightarrow L$ is a Lie algebra homomorphism. 
\item[\rm(ii)]  If $\lambda$ is surjective, then $\lambda : Q \rightarrow L$ is a central extension. 
\end{itemize}
\end{lem}
 \begin{proof} The bracket $[\cdot, \cdot]$ is well-defined and bilinear because $\lambda$ and the action of $L$ on $Q$ are well-defined linear maps. Also (i) is equivalent to the condition $[\cdot,\cdot]$ is alternating. Let $x, y, z \in Q$. Then $[[x, y], z] = [\lambda(x). y, z] = \lambda(\lambda(x).y).z  = [\lambda(x), \lambda(y)]. z= \lambda(x).\lambda(y). z - \lambda(y).\lambda(x).z = [x, [y,z]] - [y, [x,z]].$ This shows that the Jacobi identity holds for this bracket product. \\
 If $\lambda(x) = 0,$ then $\lambda(x).y = 0$ which shows (ii).
 \end{proof}

There is an explicit way to construct a central extension for the derived algebra $[L, L]$ which we describe now.
\begin{defi} Let $L$ be a Lie algebra. 
We will denote by $\mathfrak{uce}(L)$ the object constructed in the following way: \\
Let $\mathcal{B}$ be the submodule of $L \wedge L$ generated by: 
\begin{equation}  
x\wedge [y,z] + y \wedge [z,x] + z \wedge [x,y]; \quad x,y, z \in L .\end{equation}
Put
$$\mathfrak{uce}(L) = (L \wedge L)/\mathcal{B}.$$
and
$$ \left< x, y \right> = x\wedge y + \mathcal{B} \in \mathfrak{uce}(L) .$$
Note that we have a well-defined linear map $u : \mathfrak{uce}(L) \rightarrow L$ given by $u(\langle x,y \rangle) = [x,y].$

\label{uce_def}
\end{defi}

\begin{lem} The $k$-module $\mathfrak{uce}(L)$ is a Lie algebra with bracket given by
$$ [X, Y] = \langle u(X), u(Y) \rangle, \quad for \; X, Y \in \mathfrak{uce}(L) $$
The linear map  $u : \mathfrak{uce}(L) \rightarrow L$ is a central extension of $[L,L ].$ \label{uce_is_ce}
\end{lem}
\begin{proof} The Lie algebra $L$ acts canonically on $L \wedge L$ by $a.(x \wedge y) = [a,x] \wedge y + x \wedge [a,y]$ for $a \in L$
Moreover, this factors to an action on $\langle L, L\rangle, $ since  
\begin{eqnarray*}
\lefteqn{a.(x\wedge [y,z] + y \wedge [z,x] + z \wedge [x,y])} \\
&=& [a,x] \wedge [y,z] + x \wedge[a, [y,z]] \\
&&+  [a,y] \wedge [z,x] + y \wedge[a, [z,x]] \\
&&+  [a,z] \wedge [x,y] + z \wedge[a, [x,y]] \\
&=& [a,x] \wedge [y,z] - x \wedge[y, [z,a]] - x \wedge [z,[a,y]] \\
&&+  [a,y] \wedge [z,x] - y \wedge[z, [x,a]] - y \wedge [x,[a,z]]\\
&&+  [a,z] \wedge [x,y] - z \wedge[x, [y,a]] - z \wedge [y,[a,x]] \\
&=& [a,x] \wedge [y,z] - y \wedge[z, [x,a]] - z \wedge [y,[a,x]]  \\
&&+  [a,y] \wedge [z,x]  - x \wedge [z,[a,y]] - z \wedge[x, [y,a]]\\
&&+  [a,z] \wedge [x,y]  - x \wedge[y, [z,a]] - y \wedge [x,[a,z]] \\
&\in& \mathcal B
\end{eqnarray*}
The map $u : \mathfrak{uce}(L) \rightarrow [L,L]$ defined by $\langle x, y \rangle \mapsto [x,y] $ is well-defined (these are just the axioms for a Lie algebra product). Moreover,  
$$u(\langle x, y \rangle ).(\langle x, y \rangle ) = \langle[[x,y], x], y \rangle+  \langle x ,  [[x,y], y] \rangle = \langle [x,y],  [x,y] \rangle = 0 $$
It is also clear that
\begin{eqnarray*}
u(u(\langle x,y \rangle ).(\langle a, v \rangle )) 
&=& u( \langle [[x,y,]a], v \rangle +  \langle a, [[x,y],v] \rangle)\\
&=& \left [[[x,y],a], v \right] + \left[a,[[x,y],v] \right]\\
&=& \left[ [x,y], [a,v]\right]\\
\left[ u(\langle x,y \rangle, u((\langle a, v \rangle ))\right]
\end{eqnarray*}
Thus by Lemma~\ref{Lie_alg_from_mod}, $u: \mathfrak{uce}(L) \rightarrow [L,L]$ is a central extension.
\end{proof}

For the proof of the following statements  see \cite{Neh2}. We only included the proof of \ref{uce_is_ce} because it is a neat application of Lemma \ref{Lie_alg_from_mod}.
\begin{prop}
The central extension $u : \mathfrak{uce}(L) \rightarrow L$ is universal if $L$ is perfect, and in this case $\mathfrak{uce}(L)$ is  also perfect.
\end{prop}
\begin{theo}
A Lie algebra $L$ has a universal central extension $g: K \rightarrow L$ if and only if $L$ is perfect. In this case, the central extensions $g: K \rightarrow L$ and $u: \mathfrak{uce}(L) \rightarrow L$ are isomorphic. 
\end{theo}

\begin{defi}
Two perfect Lie algebras $L$ and $L'$ are \emph{centrally isogeneous}, if $\mathfrak{uce}(L) \cong \mathfrak{uce}(L').$
\end{defi}
\begin{lem} Let $f : L' \rightarrow L$ be a universal central extension and $g : K \rightarrow L'$ a covering. Then $g$ is an isomorphism.\label{uni_cov}
\end{lem}
Our goal is to find generators and relations for the universal central extensions of certain Lie algebras. In view of this, the next lemma explains how to obtain a generating set for $\mathfrak{uce}(L),$ if generators of $L$ are given: 
\begin{lem}\label{generating_lemma}
Let $L$ be a perfect  Lie algebra which is generated by $X \subset L$ and let $f : \tilde L \rightarrow L$ be a covering. Then any pre-image $\tilde X$ of $X$ under $f$ generates $\tilde L$. 
\end{lem}
\begin{proof}
For $\tilde X$ with $f(\tilde X) = X$ let $K$  be the subalgebra of $\tilde L$ generated by $\tilde X.$  Since $X \subset f(K)$, the image of $K$ under $f$ contains a set of generators for $L.$ Therefore $f$ is surjective. It follows that
$$f(K) = L \mbox{ and } \tilde  L = K + \ker f.$$
We can use that $\tilde  L$ is a covering to conclude
$$\tilde L = [\tilde L, \tilde L] = [K + \ker f, K + \ker f] = [K, K] \subset K.$$
Obviously $K \subset \tilde  L$ and we have equality: 
$K = \tilde L. $

\end{proof}

\subsection*{Graded Central Extensions}
\label{gradings}
Throughout let $\Gamma$ be an abelian group. 
\begin{prop}[\cite{Neh2}]\label{graded_cover_neh}
Suppose that $L$ is $\Gamma$-graded. Then the central extension $\mathfrak{uce}(L)$ of L is also $\Gamma$-graded and 
$$\mathfrak{uce}(L) = \bigoplus_{\gamma \in \Gamma} \mathfrak{uce}(L)_\gamma, \quad \mbox{where} \quad \mathfrak{uce}(L)_\gamma = \sum_{\delta \in \Gamma} \left< L_\delta, L_{\gamma - \delta}\right>.$$
The canonical map $u : \mathfrak{uce}(L) \rightarrow L$ is a homomorphism of graded Lie algebras. If in addition $L$ is perfect and $L_0 =\sum_{\gamma \neq 0} [L_\gamma, L_{-\gamma}],$ then
$$\mathfrak{uce}(L)_0 = \sum_{\gamma \neq 0}[\mathfrak{uce}(L)_\gamma, \mathfrak{uce}(L)_{-\gamma}]. $$
\end{prop}
The above result asserts in particular that
$$ \supp L \subset \supp \mathfrak{uce}(L) \subset \supp L + \supp L.$$

The following lemma will be useful later on: 
 \begin{lem}
 \label{central_ext_lem}Let $L$ and $K$ be $\Gamma$-graded Lie algebras and assume that $L$ is generated by $\{L_\gamma: \gamma \neq 0\}.$
If $f: L \rightarrow K$ is a graded homomorphism such that $\ker f \subset L_0,$ then $ \ker f \subset Z(L).$ 
and $f$ is a central extension.
\end{lem}
\begin{proof} Clearly, such an $f$ is always an epimorphism. In general, $\ker f$ is a graded ideal of $L$, for all $z \in \ker f$ and $x \in K$, $[z, x] \in \ker f.$ The Lie algebra $L$ is graded, so $[z, x] \in L_0 \cap L_\alpha.$ Thus $[z, L_\alpha ] = \{0\}$ for $\alpha \neq  0.$ By assumption  $\{L_\gamma: \gamma \neq 0\}$ generates $L$, hence this implies $[z, L] = \{0\}.$ 
\end{proof}

\section{A construction for a graded central extension}
\label{gce_Section}
\begin{defi}Let $\Gamma$ be an abelian group. The category $\underline{\mathbf{LA}}_\Gamma$ has as objects all $\Gamma$-graded $k$-Lie algebras which satisfy
\begin{equation}L_0 = \sum_{\gamma \neq 0}[L_\gamma, L_{-\gamma}] \label{gradla} \end{equation}
 and as morphisms the $\Gamma$-graded Lie algebra homomorphisms.
\end{defi}
We will give some examples. 
\begin{expl}
Let $L$ be a $\Gamma$-graded Lie algebra. Its \emph{core} is  
$$L_c = \sum_{\gamma \neq 0}[L_\gamma, L_{-\gamma}] \oplus \sum_{\gamma +\delta \neq 0} [L_\gamma, L_\delta].$$ 
It follows that $L_c$ is an ideal of $L$ which is $\Gamma$-graded and clearly $(L_c)_0 = \sum_{\gamma \neq 0}[(L_c)_\gamma, (L_c)_{-\gamma}]$. So $L_c$ is an object of $\underline{\mathbf{LA}}_\Gamma.$
\end{expl}
\begin{expl}
Assume that $L$ is perfect and an object of $\underline{\mathbf{LA}}_\Gamma$. According to Proposition~\ref{graded_cover_neh}, if $u : \mathfrak{uce}(L) \rightarrow L$ is a universal central extension of $L$, then $\mathfrak{uce}(L)$ is also an object of $\underline{\mathbf{LA}}_\Gamma$ and $u$ is a morphism in $\underline{\mathbf{LA}}_\Gamma$. 
\end{expl}
\begin{expl} Let $A$ be a unital commutative associative $k$-algebra. If $L$ is an object of \gradla,  then $ A \otimes_k L $ is again an object in this category. It is $\Gamma$-graded in the obvious way with
$(A \otimes L)_\gamma =  A \otimes  L_\gamma, \mbox{ for all }\gamma \in \Gamma.$
Note that 
$$ A \otimes_k L_0 = \sum_{\gamma \neq 0} [ 1 \otimes_k L_\gamma,   A\otimes_k  L_{-\gamma}].$$
If $L$ is perfect then $ A \otimes_k L$ is easily seen to be perfect as well. 
\end{expl}
\begin{expl}
Let $\Delta$ be a finite irreducible root system and $\Gamma =  Q(\Delta)$ the root lattice. Examples for objects in $\underline{\mathbf{LA}}_\Gamma$ are the finite dimensional simple Lie algebras, and the $\Delta$-graded algebras in the sense of \cite{BM92}. The more general $(R; S, \Lambda)$-graded Lie algebras introduced in \cite{Neh3} also fall under this concept. The cores of extended affine Lie algebras provide also examples of objects in $\gradla.$ 
\end{expl}
An object of $\underline{\mathbf{LA}}_\Gamma$ need not be perfect as the next examples illustrates:
\begin{expl}
Let $\mathfrak{a}$ be a Lie algebra such that $\mathfrak{a} = \mathfrak{c} \oplus \mathfrak{h}$ for a non-trivial ideal $\mathfrak c$  satisfying:
\begin{eqnarray*}
\left[\mathfrak{h}, \mathfrak{h}\right] &=& \mathfrak{c},\\
\left[\mathfrak{a}, \mathfrak{c}\right] &=& \{0\}.
\end{eqnarray*}
For $\Gamma = C_2$, the cyclic group of order $2$ we have a $\Gamma$-grading on $\mathfrak{a}$:

$$\mathfrak{a}_1 = \mathfrak{h},\quad \mbox{and}\quad
\mathfrak{a}_0 = \mathfrak{c}. $$
The Lie algebra $\mathfrak{a}$ is in general not perfect and $\mathfrak{a}_0 = [\mathfrak{a}_1, \mathfrak{a}_1].$
\end{expl}

\begin{defi}\label{pres}Let $L$ be an object in $\underline{\mathbf{LA}}_\Gamma$.  
We denote by $\gce{L}$ the Lie algebra given by the following presentation: \\
\textbf{Generators}: $\{z(x),x \in L_\gamma  : \gamma \neq 0\}$. \\
\textbf{Relations}:  For $x_{\pm \gamma}\in L_{\pm \gamma}$, $y_\gamma \in L_\gamma,$ $x_{\epsilon} \in L_{\epsilon}$ and  $x_\delta \in L_\delta$ with $\gamma, \epsilon,  \delta \neq  0: $
\begin{itemize}
\narrower
\item[] (T0) $ z(s x_\gamma + t y_\gamma) = s z(x_\gamma) + t z(y_\gamma) \mbox{ for }s, t \in k. $
\item[](T1) If $0\neq \gamma + \delta $ then 
\begin{subequations}
\begin{equation}
 \left[z(x_\gamma), z(x_\delta) \right] = z([x_\gamma,  x_\delta]) \mbox{ if }   \gamma + \delta \in \supp L. 
 \label{T1a}
 \end{equation}
 \begin{equation}
  \left[[z(x_\gamma), z(x_\delta)], z(x_{\epsilon}) \right] =  0  \mbox{ if } \gamma +\delta \notin \supp L. 
  \label{T1b}
  \end{equation}
\end{subequations}
\item[](T2) $ \big[[z(x_\gamma), z(x_{-\gamma})] , z(y_{\gamma})\big] = z([[x_\gamma,  x_{-\gamma}],y_\gamma ]) .$
\end{itemize} 

\end{defi}
\begin{rem} By (T0), $z(0_{L_\gamma}) = z(0_k \cdot 0_{L_\gamma}) = 0_kz(0_{L_\gamma}) = 0_{\mathfrak{gce}(L)}. $ By (T1) the element  $[z(x_\gamma), z(x_\delta)]$ is central if $ 0 \neq \gamma +\delta \notin \supp L. $
\end{rem}
Let $X = \bigcup_{\gamma \neq 0}L_\gamma\subset L.$
Note that the set map
\begin{eqnarray*}
z : X & \rightarrow & \gce L\\
x & \mapsto &  z(x)
\end{eqnarray*}
is well-defined.
We denote the image of $X$ in $\gce L$ by $\mathfrak{X}$. It is clear that $\mathfrak{X}$ generates $\gce L$.
In the following we will abbreviate an element $[z(x), z(y)]$ where $x \in L_{\gamma}$, $y  \in L_{-\gamma}$, $\gamma \neq 0,$ by $g(x,y)$. 
Whenever we write $z(x)$ or $g(x,y)$ it is understood that $x\in L_\gamma$ and $y \in L_{-\gamma}$ for some $0 \neq \gamma \in \supp_L $. Likewise when we use the symbol $z(x_\gamma)$ or $g(x_\gamma, x_{-\gamma})$ it is implicit that $x_{\pm \gamma} \in L_{\pm \gamma}$ and that $\gamma \neq 0$. 

\begin{lem} 
Let $0\neq \sum_{i =1}^3 \gamma_i \in \supp L$ and all $\gamma_i \neq 0$. Then for any $x_i \in L_{\gamma_i}$  \label{three_elem}
\begin{equation}[z(x_1),[z(x_2), z(x_3)]] = z([x_1, [x_2, x_3]]).\end{equation}
\end{lem}
\begin{proof}
\bf{Case 1.} \normalfont  $\gamma_2 + \gamma_3 \neq 0$.\\ If $\gamma_2 + \gamma_3 \in \supp_\Gamma(L),$ then  $[z(x_2), x(x_3)] = z([x_2, x_3])$ by (T1).  Thus by (T1) again: 
$$[[z(x_1),[z(x_2), z(x_3)]] = \big[z(x_1), z([x_2, x_3])\big]= z([x_1, [x_2, x_3]]). $$
If $\gamma_2 + \gamma_3 \notin \supp L$ then $[x_2, x_3] = 0$ and $[z(x_2), z(x_3)]$ is central, hence
$$[[z(x_1),[z(x_2), z(x_3)]] = 0 = z([x_1, [x_2, x_3]]) = z(0) = 0. $$
\bf{Case 2.} \normalfont  $\gamma_2 + \gamma_3 = 0$, $\gamma_1 + \gamma_2 \neq 0$ and $\gamma_1 + \gamma_3 \neq 0$.\\ The Jacobi identity yields
$$[z(x_1), [z(x_2), z(x_3)]] =[z(x_2), [z(x_1), z(x_3)]] + [z(x_3), [z(x_2), z(x_1)]]. $$
Now we may use the result of the previous case on both summands: 
$$[z(x_1), [z(x_2), z(x_3)]] =z([x_2, [x_1, x_3]]) + z([x_3, [x_2, x_1]]), $$
and (T0) allows to write this expression as
$$[z(x_1), [z(x_2), z(x_3)]] = z([x_2, [x_1, x_3]]) + [x_3, [x_2, x_1]]),$$
which, by the Jacobi identity, is equal to $z([x_1, [x_2, x_3]])$.\\
\bf{Case 3.} \normalfont $\gamma_2 + \gamma_3 = 0$ and w.l.o.g. $\gamma_1 + \gamma_2 =0$.\\ 
Then $\gamma_1 + \gamma_2 + \gamma_3 = \gamma_1 = \gamma_3=:\gamma$ and $\gamma_2 = -\gamma$ and by (T2) we obtain
$$[z(x_1), [z(x_2), z(x_3)]] = z([x_1, [x_2, x_3]]). $$
\end{proof}

\subsection*{The grading of $\gce{L}$.}

\begin{prop}Let $L$ be an object in \gradla. Define $S = \supp_{\Gamma}(L) \cup \{0\},$ $S^\times  = S \setminus \{0\}$.  Then the Lie algebra $\gce{L}$ is an object of $\underline{\mathbf{LA}}_\Gamma$ with homogeneous components given by
$$\gce{L}_\gamma = z(L_\gamma):=\{z(x_\gamma): x_\gamma \in L_\gamma\} \quad \mbox{if } 0\neq \gamma \in \supp L,$$ 
$$ \gce{L}_0 =\tilde{H}:= \spa\{g(x,y): x \in L_\gamma, y \in L_{-\gamma},  \gamma\in S^\times  \},$$
$$\gce{L}_\gamma = \sum_{\delta, \gamma - \delta \in S^\times} [z(L_{\delta}), z(L_{\gamma - \delta})]  \quad \mbox{if } 0\neq \gamma \in (S + S) \setminus S .$$
$$ \gce{L}_\gamma = 0 \quad \mbox{if }\gamma \in \Gamma \setminus (S  + S).$$
We have a $\underline{\mathbf{LA}}_\Gamma$-epimorphism uniquely defined by 
\begin{eqnarray*}
\eta : \gce L & \rightarrow & L \\
z(x_\gamma) & \rightarrow &  x_\gamma, \quad x_\gamma \in L_\gamma, \gamma \neq 0. 
\end{eqnarray*}
\label{grading_prop}
 \end{prop}
 \begin{proof}
 The ideal generated by the relations (T0)-(T2) is $\Gamma$-graded since its generators are homogeneous elements of the free Lie algebra on $\mathfrak X,$ $\mathrm{Lie}(\mathfrak X )$; therefore the quotient $\gce{L}$ has a canonical grading which is induced from the grading on $\mathrm{Lie}(\mathfrak X)$ such that $z(x_\gamma)$ lies in $\gce{L}_\gamma.$
 Relation (T0) implies that $z(L_\gamma)$ is a  $k$-module on generators $\{z(x): x \in L_\gamma \}$. \\
Define 
$$\tilde{N} = \bigoplus_{\gamma \neq 0} z({L}_\gamma).$$
$$\tilde Z_\gamma = \sum_{\delta, \gamma - \delta \in S^\times} [z(L_{\delta}), z(L_{\gamma - \delta})], \quad \gamma \in (S + S)\setminus S.$$
\begin{equation}\tilde  Z = \sum_{\gamma \in (S +S) \setminus S} \tilde Z_ \gamma.  \label{Z_definition}\end{equation}
By relation (T1)(b), the module $\tilde Z$ lies in the centre of $\gce L. $
Consider the subspace of $\mathfrak{gce} (L)$ given by 
$$  \tilde L = \tilde{H} + \tilde{N} + \tilde Z \subset \mathfrak{gce}(L).$$

Let $ z(x_\gamma)$ be a generator of $\gce L$, $x_\gamma \in L_\gamma$ where $\gamma \in S^\times,$  and  let $n \in \tilde N$. 
We may assume that $n = z(y_\delta),$ $y_\delta \in L_\delta$ where $\delta \in S^\times.$
Thus
$$[z(x_\gamma), n] \in \begin{cases} \tilde H & \gamma + \delta = 0\\
 \tilde N & \gamma + \delta \in S^\times \\
 \tilde Z_{\gamma + \delta} & \gamma + \delta \in (S + S)\setminus S
 \end{cases} $$
 Thus $[z(x), \tilde N] \subset \tilde L.$
 For $g(y, w)\in \tilde{H},$  Lemma~\ref{three_elem} shows that $$[ g(y, w), z(x_\gamma)  ] \in z(L_\gamma)\subseteq \tilde{N}$$ since
$[[z(y), z(w)],z(x) ]  = z([[y, w] ,x ]) \in z(L_\gamma).$\\
The submodule $\tilde Z$ is central by (T1)(b) since it is spanned by homogeneous elements  whose degrees lie in $(S + S)\setminus S$, thus 
$$[z(x), \tilde Z] = \{0\}. $$
The $k$-module $\tilde{N} \subset \tilde L$ contains the generators of $\gce{L}$ and since  we have just shown that $\ad z(x)$ maps $\tilde L$ into itself for any generator $z(x)$, $x \in L_\delta$, it follows that $\tilde L = \mathfrak{gce}(L).$
The sum $\tilde{N} + \tilde{H} + \tilde Z$ is direct by construction of the $\Gamma$-grading on $\gce L$. We have therefore proven that 
$$\mathfrak{gce}(L) = \tilde N \oplus \tilde H \oplus \tilde Z. $$
 At this point we know the map  $\eta : z(L_\gamma) \rightarrow L_\gamma$, given by $z(x) \rightarrow x,$ is surjective and maps $\gce{L}_\gamma$ onto $L_\gamma$. 
Since $L$ is generated by $\eta(\mathfrak X),$ the map $\eta$ extends to a map $L(\eta) : \mathrm{Lie}(\mathfrak{X}) \rightarrow L.$ Moreover, $L(\eta)$ factors through the relations in (T1) and (T2), thus we obtain by factorization a uniquely defined Lie algebra epimorphism $\eta : \gce L \rightarrow L$ with the property that $\eta(z(x)) = x$ for all $x \in L_\gamma,$ $\gamma \neq 0. $  
\end{proof}

\begin{cor} With the same notation as in the proof of  Proposition~\ref{grading_prop}, 
$$\tilde Z \subset Z(\gce L) $$ In particular, if $0 \neq \gamma \in (S + S)\setminus S,$ then $\gce{L}_\gamma \subset Z(\gce L).$ 
\label{Z_is_central_Cor}
\end{cor}
\begin{proof} By definition, the submodule $\tilde Z$ is spanned by $[z(x), z(y)]$ where $x \in L_\alpha$, $y \in L_\beta$ and where $\alpha \in S, \beta \in S$ and $\alpha + \beta \in (S+ S)\setminus S.$ According to relation (T1)(b) the element $[z(x), z(y)]$ is central in $\gce L.$
\end{proof}

We can now  obtain an improved version of Lemma~\ref{three_elem}. \\
\textbf{Notation:} For elements $x_1,\ldots, x_n$ in a Lie algebra we define
$$[x_n, \ldots, x_1] = \ad x_n \ad x_{n-1}\ldots \ad x_2(x_1). $$
\begin{lem}
\label{calc_in_root_spaces}
\label{n_elem_lem} 
For elements $x_k \in L_{\gamma_k}$, $1\leq k \leq n$ and $0\neq \sum_{k= 1}^n{\gamma_k} \in \supp L$
$$\big[z(x_n), z(x_{n-1}),  \ldots, z(x_1)\big] = z\Big (\big [x_n,x_{n-1} \ldots, x_1 \big]\Big). $$
\end{lem}
\begin{proof}
The proof is by induction on $n$. The formula is true for $n \leq 3$ by Lemma~\ref{three_elem}. Assume that it is proved for all $n < N$. Consider
$$\big[z(x_N),[ z(x_{N-1}), \ldots, z(x_1)]\big] $$
If $\sum_{j = 1}^{N-1}\gamma_{j} \neq 0$ we either have $\sum_{j = 1}^{N-1}\gamma_{j} \in \supp L$ and in this case by (T1)
$$\big [z(x_N),[ z(x_{N-1}), \ldots, z(x_1)]\big] = z\big ([x_N, \ldots, x_1]\big ),$$
or $\sum_{j = 1}^{N-1}\gamma_{j} \notin \supp L$ and then
$[ z(x_{N-1}), \ldots, z(x_1)]$ is central  by Corollary~\ref{Z_is_central_Cor} and $[x_{N-1}, \ldots, x_1] = 0$ so that
$$\big [z(x_N),[ z(x_{N-1}), \ldots, z(x_1)]\big] = 0 = z\big ([x_N, \ldots, x_1]\big ) .$$
This finishes the case $\sum_{j = 1}^{N-1}\gamma_{j} \neq 0.$
So assume that $\sum_{j = 1}^{N-1}\gamma_{j} = 0$. Then $[ z(x_{N-1}), \ldots, z(x_1)] \in \gce{L}_0,$ and hence can be written as linear combination
$$\big [ z(x_{N-1}), \ldots, z(x_1)\big] = \sum_j g(u_j, v_j )$$
for $u_i \in L_{\gamma_j},v_i \in L_{-\gamma_i}$, some $\gamma_i \neq 0$. 
By Lemma~\ref{three_elem} 
$$\big[z(x_N),g(u_i, v_i)\big] = z\Big(\big[x_N,[u_i, v_i]\big]\Big)  $$
and summation over $i$ yields
$$\sum_i \big[z(x_N),g(u_i, v_i)\big] = z\big([x_N, \ldots, x_1]\big). $$
\end{proof}
\begin{prop} Let $L$ and $L'$ be objects in $\underline{\mathbf{LA}}_\Gamma$ such that $\supp_{\Gamma} L = \supp_{\Gamma} L.$ If $f: L \rightarrow L'$ is an epimorphism in $\underline{\mathbf{LA}}_\Gamma,$ then
the set map $z(x)\rightarrow z(f(x))$, $x \in L_\gamma$, $\gamma \neq 0$ extends uniquely to an $\underline{\mathbf{LA}}_\Gamma$-epimorphism $\gce f : \gce L \rightarrow \gce{L'}$ such that the following diagram commutes:
$$\xymatrix{ \gce L \ar[d]_\eta \ar[rr]^{\gce f} && \gce{L'}\ar[d]^\eta\\
 L \ar[rr]^f && L'
}$$

\end{prop}
\begin{proof}  By assumption, $\gce{L'}$ is generated by 
$$X' := \{z((L')_\gamma): \gamma \neq 0\} .$$
Denote by $\mathcal{I}$ (resp. $\mathcal{I}'$) the ideal in $\mathrm{Lie}(X)$ (resp. $\mathrm{Lie}(X')$) generated by (T0)-(T2). 
Since $f$ is graded we have a set map
\begin{eqnarray*}
f : X_\gamma & \rightarrow & (X')_\gamma\\
w(x) & \mapsto &w(f(x)).
\end{eqnarray*}
where $w(x)$ for $x \in X$ resp. $X'$  denotes the image of $x$ in $\mathrm{Lie}(X)$ resp. $\mathrm{Lie}(X').$
By the universal property of $\mathrm{Lie}(X),$ $f$ extends to a homomorphism $L(f): \mathrm{Lie}(X) \rightarrow \mathrm{Lie}(X')$. We prove that $\mathcal{I}$ is mapped into $\mathcal{I}'$. 
For (T0) there is nothing to check since Lie algebra homomorphisms are always linear maps. For (T1)  there are two cases. Throughout $x \in L_\gamma,$ $y \in L_\delta$, $\gamma, \delta, \gamma + \delta \neq 0.$ 
 \begin{itemize}
\item[-] Let $\gamma + \delta \in \supp_\Gamma(L)  =  \supp_{\Gamma}(L').$  Then 
 $$L(f)\left( [w(x), w(y)] - w([x,y]) \right) = [w(f(x)), w(f(y))] - w(f([x,y]))$$ which is relation (T1)(a) for elements in $X'$ with $f(x) \in (L')_\gamma$ and $f(y) \in (L')_\delta.$
  \item[-]Assume $\gamma + \delta \notin \supp_{\Gamma}(L)$ and thus $\gamma + \delta \notin \supp_{\Gamma}(L').$
 $$L(f)[w(x), w(y)] = [w(f(x)), w(f(y))] = w([f(x), f(y))] $$
 which is central in $\mathfrak{gce}(L')$ by (\ref{T1b}).  
 \end{itemize}
 Next,
  \begin{eqnarray*} 
   && L(f)\big[[w(x_\gamma), w(x_{-\gamma})] , w(y_{\gamma})\big]- w([[x_\gamma,  x_{-\gamma}],y_\gamma ])\\ &=& \big[[w(f(x_\gamma)), w(f(x_{-\gamma}))] , w(f(y_{\gamma}))\big] - w(f([[x_\gamma,  x_{-\gamma}],y_\gamma ]))
  \end{eqnarray*}
   since $f$ is a Lie algebra homomorphism. 
  This proves that $L(f)$ factors through (T2). 
 Since the ideals $\mathcal{I}$ and $\mathcal{I}'$ are $\Gamma$-graded it follows that $\gce{f}$ is also graded, hence a morphism in $\underline{\mathbf{LA}}_\Gamma.$ Since $z(X')$ generates $\gce{L'}$ if follows that $\gce{f}$ is surjective.  
 \end{proof}
 \begin{lem}
 \label{central_ext_lem_gce}
If $f: L \rightarrow L'$ is an $\underline{\mathbf{LA}}_\Gamma$-epimorphism such that $\ker f \subset L_0$ then $f$ is a central extension.
\end{lem}
\begin{proof}This is immediate from Lemma~\ref{central_ext_lem}. 
\end{proof}

\begin{prop}Let $L \in \underline{\mathbf{LA}}_\Gamma.$ 
\begin{itemize}
\label{central_ext_prop}
\item[\rm{(a)}]
$\eta : \gce{L} \rightarrow L$ is a graded central extension of $L$.  
\item[\rm{(b)}]  If $L$ is perfect, then $\gce{L}$ is perfect and $\eta : \gce{L} \rightarrow L$ is a covering. \label{perfect_lem}
\end{itemize}
\end{prop}
\begin{proof} 
 By Proposition~\ref{grading_prop}, $\eta$ is an epimorphism in $\underline{\mathbf{LA}}_\Gamma$. Thus $\ker \eta$ is graded. We have $\ker \eta \subseteq \mathfrak{gce}(L)_0 \oplus \tilde {\mathcal Z}$  with $\tilde {\mathcal Z}$   defined by (\ref{Z_definition}).  Note that $\tilde {\mathcal Z}$ is a central ideal in $\gce L$ which is contained in $\ker \eta$. Therefore $\gce L$ is a central extension, if the quotient map $\eta' : \gce L/ \tilde {\mathcal Z} \rightarrow L$  is a central extension. But $\ker \eta' \subset \gce L_0$ and  Lemma~\ref{central_ext_lem} gives that $\eta '$ is central extension. This proves (a).\\
  It suffices to show that each generator is contained in $[\gce{L}, \gce{L} ]$. The Lie algebra $L$ is perfect. Thus for any $x_\gamma \in L_\gamma, \gamma \neq 0$ there are $x_i \in L_{\delta_i}, y_i \in L_{\gamma- \delta_i}$ such that 
$$x_\gamma = \sum_{i = 1}^n [x_i,  y_i]. $$
By (T0) it suffices to prove that $z([x, y])\in [\gce{L}, \gce{L}]$ where $x \in L_\delta$ and $y \in L_{\gamma - \delta}.$
Assume first that $\delta \neq 0$ and $\gamma - \delta \neq 0$. Then it is an immediate consequence of (T1) that 
$$[z(x), z(y)] = z([x,y]) \in [\gce{L}, \gce{L}] .$$
For the case $\delta = 0$ observe that $L_0 = \sum_{\gamma \neq 0} [L_\gamma, L_{-\gamma}]$. Thus each $x\in L_0$ can be written as $x = \sum_i [x_i,y_i]$ with $x_i \in L_{\gamma_i}$, $y_i \in L_{-\gamma_i}$ and $\gamma_i \neq 0$. By Lemma~\ref{three_elem} we obtain
$$z([x, y]) = z(\sum_i[ [x_i,y_i], y])= \sum_i[ [z(x_i),z(y_i)], z(y)]. $$
Therefore $z([x,y]) \in [\gce{L}, \gce{L}]$. \\
Combining this with (a) it follows that $\eta$ is a perfect central extension, i.e., a covering. 
 \end{proof}
At this point we know by Proposition~\ref{central_ext_prop} that $\eta : \gce{L} \rightarrow L$ is a central extension of $L$. Naturally the question arises under which conditions $\eta$ is a universal central extension. 

\begin{cor}\label{hull_cor}
If $L \in \underline{\mathbf{LA}}_\Gamma$ is perfect, then the Lie algebras $\gce{\mathfrak{uce}(L)}$ and $\mathfrak{uce}(L)$ are isomorphic. 
\end{cor}
\begin{proof}The epimorphism $\eta: \gce{\mathfrak{uce}(L)} \rightarrow  \mathfrak{uce}(L)$ is a covering. Hence we can apply Lemma~\ref{uni_cov} and conclude that $\eta$ is an isomorphism. 
 \end{proof}
\begin{theo}\label{uni_central_cov_theo} Let $L$ be an object in $\underline{\mathbf{LA}}_\Gamma$ which is perfect. Assume that the restriction $u|_{\mathfrak{uce}(L)_\gamma} : {\mathfrak{uce}(L)_\gamma} \rightarrow L_\gamma$ is bijective for all $0\neq \gamma \in \supp L$.
Then
$$\mathfrak{uce}(L)\cong \gce{L}.$$
\end{theo}
\begin{proof} By the universal property of $\mathfrak{uce}$ the central extension $f : \mathfrak{uce}(L) \rightarrow \mathfrak{gce}(L)$ is also universal (Corollary 3.8 in \cite{Neh3}).
We define a map 
$$\sigma : \bigcup_{\gamma \in \supp L}z(L_\gamma)  \rightarrow \bigcup_{\gamma \in \supp L}\mathfrak{uce}(L)_\gamma, \quad \sigma(z(x_\gamma))\rightarrow u^{-1}(x_\gamma). $$
Then $\sigma$ is well-defined and bijective since the restriction of $u$ (resp. $\eta$) to $\mathfrak{uce}(L)_\gamma$ (resp. $\gce L_\gamma$) is bijective for $\gamma \in \supp L$. 
The image of $\sigma$ generates $\mathfrak{uce}(L)_\lambda$:  
$$\mathfrak{uce}(L)_\lambda = \sum_{\delta \in \supp L}[\mathfrak{uce}(L)_{\lambda -\delta}, \mathfrak{uce}(L)_\delta] $$
and 
$$\mathfrak{uce}(L)_0 = \sum_{0 \neq \lambda \in \supp L}[\mathfrak{uce}(L)_\lambda, \mathfrak{uce}(L)_{-\lambda}].$$
  We claim that the elements $\{\sigma(z(x_\gamma)): 0\neq \gamma \in \supp L\}$ constitute a subset of $\mathfrak{uce}(L)$ fulfilling relations (T0)-(T2). \\
(T0) is clear because the restriction of $f$ to $\mathfrak{uce}(L)_\gamma$, $0\neq \gamma \in \supp L$ is in particular a $k$-module isomorphism. \\
Let $x \in L_\gamma$, $y \in L_\delta$. \\
Case 1): $\gamma + \delta \in \supp L$. 
 $$\sigma(z([x,y ])) - [\sigma(z(x)), \sigma(z(y))] \in \ker f \cap \mathfrak{uce}(L)_{\gamma+ \delta} = \{0\}.$$
 Hence
$$\sigma(z([x,y ])) = [\sigma(z(x)), \sigma(z(y))].$$
Case 2): $\gamma + \delta \notin \supp L$. Then 
$$[f^{-1}(z(x)), f^{-1}(z(y))] \in \ker f \subset Z(\mathfrak{uce}(L)).$$ 
For (T2) we use the central trick,  let $ x, u \in L_\gamma$, $y \in L_{- \gamma}$, then
$$[\sigma(z(x)), \sigma(z(y))] - \sigma([z(x), z(y)]) \in \ker f \subset Z(L).$$  
Thus we may choose any pre-image of $[z(x), z(y)]$ and the following will hold: 
$$[[\sigma(z(x)), \sigma(z(y))], \sigma(z(u))] = [\sigma([z(x), z(y)]), \sigma(z(u))].$$
Moreover, $$[\sigma([z(x), z(y)]), \sigma(z(u))]- \sigma([[z(x), z(y)],z(u)]) \in \ker f \cap \mathfrak{uce}(L)_\gamma = \{0\}.$$
Thus we have Lie algebra homomorphism 
$$\sigma: \gce L \rightarrow \mathfrak{uce}(L). $$
Since the restriction of $\sigma \circ f$ (resp. $f \circ \sigma$) is the identity on a generating set of $\mathfrak{uce}(L)$ (resp. \gce L) it follows that $f$ is invertible as a Lie algebra homomorphism with inverse $\sigma$. Thus $f$ is an isomorphism. 
 \end{proof}
 

\chapter{Tits-Kantor-Koecher constructions}

\label{TKK_chapter}
\section{The universal inner derivation algebra}
\begin{defi}A \emph{Jordan-Kantor Pair} is a quadruple $$P= (J, M) = ((J^{-}, J^+), (M^{-}, M^{+}))$$ of $k$-modules together with quadratic maps $Q^\sigma : J^\sigma \rightarrow \Hom(J^{-\sigma}, J^\sigma)$, linear maps $\circ : J^\sigma \rightarrow \End(M^{-\sigma}, M^{-\sigma})$, bilinear operators $V^{\sigma} : M^\sigma \times M^{\sigma} \rightarrow \End(M^{\sigma})$ and bilinear maps $\kappa : M^\sigma \times M^\sigma \rightarrow J^\sigma$ such that
\begin{enumerate}
\item $J$ is a Jordan pair with quadratic maps $Q^\sigma$, i.e., the following identities hold in all scalar extensions of $k$: for $a, c \in J^\sigma$, $b \in J^{-\sigma}$
\begin{eqnarray}
D_{a,b}Q_{a} &=& Q_{a}D_{b,a} \label{JP1}, \\
D_{Q_a b, b } &=& D_{a, Q_b a}, \label{derid1}, \label{JP2}\\
Q_{Q_a b} &=& Q_a Q_b Q_a \label{JP3}
\end{eqnarray} 
where $Q_{a , c}b =:D_{a, b }c  $ is the linearization of the quadratic map $Q$. We also abbreviate $D_{a,b}c$ by $\{a, b, c\}.$

\item $M$ is a \emph{special $J$-module} with respect to $\circ$. In terms of identities this means:   $$\{a, b, c\}\circ y = a \circ(b \circ (c \circ y)) + c \circ (b \circ (a\circ y))$$ for all $a,c  \in J^\sigma$, $b \in J^{-\sigma}$, $y \in M^{-\sigma}.$
\item For all $(x,y) \in M$, $(z, w) \in M:$
\begin{equation}[V_{x, y}, V_{z, w}] = V_{V_{x,y} z, w} - V_{z, V_{y,x}w} \label{KP0} \end{equation}  
We also abbreviate $V_{x,y}z$ by $\{x, y, z\}.$
\item $\kappa(x, x)$ = 0  for $x \in M^\sigma$.
\item The maps and operators are compatible in the following sense:\\ for all $a, c  \in J^\sigma$, $b \in J^{-\sigma}$, $x,u, z \in M^\sigma$, $y, w \in M^{-\sigma}$: 
\begin{eqnarray} 
\kappa(x, z)\circ y &=& \{x,y,z\} - \{z,y,x\} \label{JK4} \\
\kappa(x,z) \circ (b \circ u) &=& \{z, b \circ x, u\} - \{x, b\circ z, u\} \\
\kappa(b \circ x, y) \circ z &=& b \circ \{x,y,z\} - \{y,x, b\circ z\}\\
\{ \kappa(x,u), b , a \} &=& \kappa(a \circ (b \circ x), z)  + \kappa(x, a \circ (b \circ z)) \label{derid2}\\
\{a, \kappa(y, w), c\} &=& \kappa(a \circ w, c\circ y) + \kappa(c \circ w, a \circ y) \label{derid3}\\
\kappa(\kappa(z,u) \circ y, x) &=& \kappa(\{x,y,z\}, u) + \kappa(z, \{x,y,u\}) \label{KP1}
\end{eqnarray}
\end{enumerate}
Given two Jordan-Kantor pairs $P = (J, M)$ and $P' = (J',M')$, a Jordan-Kantor homomorphism from $P$ to $P'$ is a quadruple of $k$-linear maps
$ f = (f_J^+,f_J^- ,f_M^+, f_M^-)$ such that for all $a \in J^\sigma,$ $x, z \in M^\sigma,$ $y \in M^{-\sigma}$ 
\begin{equation}  f_J^\sigma Q^\sigma(a) = {Q'}^\sigma(f_J^\sigma(x))f_J^{-\sigma}, \quad f_M^\sigma V_{x,y} = {V'}^\sigma_{f_M^\sigma(x), f_M^{-\sigma}(y)}f_M^\sigma ,\end{equation}
 $$\kappa' (f_M^\sigma(x), f_M^\sigma(z)) = f_J^\sigma (\kappa(x,z)), \quad  f_M^\sigma(a \circ y ) = f_J^\sigma(a) \circ' f_M^{-\sigma}(y).  $$
Jordan-Kantor pairs form a catergory with morphisms the Jordan-Kantor homomorphisms. 
\end{defi} \begin{rem}
Jordan-Kantor pairs were introduced by Benkart and Smirnov in \cite[3.1]{BS}, but there it is only required that $\kappa$ is anti-commutative. This is clearly equivalent to $\kappa(x,x) = 0$ for all $x$ if $1/2 \in k$. Since the authors work over a base ring containing $1/2$ our definition is an appropriate extension of theirs. Likewise Benkart and Smirnov use different identities to define a Jordan pair which can be shown to be equivalent to  \ref{JP1} , \ref{JP2} and \ref{JP3} if $1/2 \in k$ and $P$ does not have $3$-torsion. 
\end{rem}
\begin{expl} Clearly, Jordan pairs are a subcategory of the category of Jordan-Kantor pairs. All the constructions which follow can therefore also be carried out for a Jordan pair. 
\end{expl}
\begin{expl} If $1/2$, $1/3 \in k,$ then every Kantor pair embeds into a Jordan-Kantor pair, see \cite[7.4]{BS}. 
\end{expl}

The $k$-module $P  = J^- \oplus J^+ \oplus M^+ \oplus M^-$ can be endowed with a $5$-grading in the following manner: $P^{\pm 2} = J^{\pm 1}$, $P^{\pm 1} = M^{\pm 1}$ and $P^{\pm 0} = \{0\}$. 
 We denote by $\mathcal E$ the subring $$\mathcal E := \End_k(P)_0 = \{T \in \End_k(P): T.P^i \subset P^i, -2 \leq i \leq 2\} ,$$ the ring of all $k$-endomorphisms of $P$ that preserve the $5$-grading. Every element $T$ of $\mathcal E$ can be thought of as a block diagonal matrix $$T = \left (\begin{array}{cccc}T_{-2} &&&\\  &T_{-1}&& \\ &&T_1& \\  &&&T_2\\   \end{array} \right ), \mbox{ where }T_i \in \End{P^i}.$$
For the sake of brevity we will write $ T = (T_{-2}, T_{2})$ if $T_{|M} = 0$ and $ T = (T_{-1}, T_{1})$ if $T_{|J} = 0$. If $p \in P^{i}$  and $T  \in \mathcal E$ then $T.p := T_i.p.$
\begin{defi} Let $(a, b) \in J$ and $(x, y) \in M$. We define $\delta(a, b) \in \End_k(P)$ and $v(x,y) \in \End_k(P)$ by
\begin{eqnarray*}\delta(a, b)|_J = (D_{a, b}, - D_{b,a})\\
\delta(a, b)(x, y) = (a \circ (b \circ x), - b \circ (a \circ y))\\
v(x,y)|_M = (V_{x,y}, - V_{y,x})\\
v(x,y)(a, b) = (\kappa(a \circ y, x), -\kappa(b \circ x, y))
\end{eqnarray*}

\label{extended_ops_JKP_defi}
\end{defi}

\begin{rem}For all $(a, b) \in J$ and $(x, y) \in M$ the linear operators $\delta(a, b)$ and $v(x,y)$ are elements of $\mathcal E$.
\end{rem}
\begin{defi} The \emph{derivation algebra} of $P$ denoted by $\str(P)$ consists of all $T$ in $\mathcal E$ which have the following properties: 
\begin{eqnarray}
[T, D_{a, b}] =  D_{T.a, b} + D_{a, T.b} && [T, V_{x, y}] = V_{T.x, y} + V_{x, T.y} \label{str1}\\ 
T.(a \circ y) = T.a \circ y + a \circ T.y & & T.\kappa(x,z) = \kappa(T.x, z) + \kappa(x, T.z) \label{str2}
\end{eqnarray}
where $a \in J^\sigma$, $b \in J^{-\sigma}$, $x,z \in M^\sigma$ and $y \in M^{-\sigma}$.
It follows from (\ref{derid1}), (\ref{KP0}), (\ref{derid2}) and (\ref{derid3}) that for all $T \in \str(T)$
\begin{equation}
[T, \delta(a, b)] = \delta(T.a,b) + \delta(a, T.b), \quad [T, v(x,y)] = v(T.x, y) +v(x, T.y)\label{str11}. \end{equation}
It is easily checked that $\delta(a, b)$ and $v(x,y)$ are derivations (see \cite{BS}). 
Hence \str(P) is a Lie subalgebra of $\mathcal E^-$ which contains all $\delta(a, b)$ and $v(x,y).$ 
 By  \ref{str11}, the submodule spanned by all $\delta(a, b)$ and $v(x, y)$ is an ideal in $\str(P),$ called the \emph{inner derivation algebra}, $\instr(P)$.
 \end{defi}
\label{subsection_uce}

\begin{defi} Let $I(P)$ be the submodule of the direct sum $(J^+ \otimes J^-) \oplus (M^+ \otimes M^-)$ which is generated by the following elements\\
\label{uider_JkP_Def}
\begin{eqnarray}
x \otimes \{yxy\} - \{xyx\} \otimes y, \label{HCP11}\\ 
\{xy u\} \otimes  w + \{uwx\} \otimes y - u \otimes \{yx w\} - x \otimes \{wu y\}, \label{HCP12} \\
   a \otimes \{bab\} - \{aba\} \otimes b \label{HCP112},  \\ 
  \{abc\} \otimes d - c \otimes \{bad\} + \{cda\} \otimes b - a \otimes \{bcd\}\label{HCP113} \\
\kappa(u, x) \otimes b - x \otimes (b \circ u) + u \otimes (b\circ x) \label{HCP13}\\
a \otimes \kappa(y, w) - (a \circ y) \otimes w + (a \circ w ) \otimes y  \label{HCP131}
\end{eqnarray}
for $(a, b), (c,d) \in J$, $(x,y), (u,w) \in M$.\\
The quotient $$ (J^+ \otimes J^-) \oplus (M^+ \otimes M^-)/ I(P) \mbox{ is denoted by }\mathfrak{uider}_{JKP}(P) \mbox{ or also by } P \diamond P $$ and the cosets $x\otimes y + I(P)$ and $a \otimes b + I(P)$ by $x\diamond y$ (resp. $a \diamond b$).\\
In the notation $\mathfrak{uider}_{JKP}$ the subscript will often be omitted; it is our intention not to cause any confusion by doing so. This also holds for notation that is ``derived'' from this one. \\
  We define the \emph{cyclic homology} of the Jordan-Kantor pair as $$\HF(P) = \left \{\sum a_i \diamond b_i + \sum x_j \diamond y_j: \sum \delta({a_i, b_i}) + \sum v({x_j, y_j}) = 0 \right \}.$$

\end{defi}
\begin{rem} 
\label{one_half_remark}
Our definition of $\HF(P)$ is equivalent to \cite[5.3]{BS} if $1/2 \in k$: The relation (\ref{HCP112}) is omitted in their paper. It is easily seen that (\ref{HCP112}) and (\ref{HCP113}) are equivalent whenever we have $1/2 \in k$. The same holds for (\ref{HCP11}) and(\ref{HCP12}). 
\end{rem}
\begin{rem}
 The module $I(P)$ is generated by elements which lie either in $(J^+ \otimes J^-)$ or in $(M^+ \otimes M^-).$ Therefore the module $\uider(P)$ has  a decomposition $(M  \diamond M) \oplus (J \diamond J)$ where  $J  \diamond J  : = J^+ \otimes J^-/(I(P) \cap (J^+ \otimes J^-))$ and $M  \diamond M  : = M^+ \otimes M^-/(I(P) \cap (M^+ \otimes M^-)).$
 \end{rem}
\begin{lem}
\label{uider_is_ider_module}
The $k$-module $\uider(P)$ is a module for the Lie algebra $\str (P)$ with action defined as follows (for $T \in \str (P), a \diamond b \in J \diamond J, x \diamond y \in M \diamond M$):
$$T.(a \diamond b) = T.a \diamond b  + a \diamond T.b \mbox{ and } T.(x \diamond y) = T.x \diamond y  + x \diamond T.y$$
\end{lem}
\begin{proof}
The $k$-module $(J^+ \otimes J^-) \oplus (M^+\otimes M^-)$ has a canonical $\mathcal E$-module structure, given by $T(a \otimes b) = T.a \otimes b + a \otimes T.b$ and $T(x \otimes y) = T.x \otimes y + x.\otimes T.y$, and this defines an $\mathcal E^-$-action on $(J^+ \otimes J^-) \oplus (M^+\otimes M^-)$. It suffices therefore to show that every element $T \in \str(P)$ leaves the submodule $I(P)$ invariant, or, equivalently that every spanning element of $I(P)$ is mapped by $T$ into $I(P)$. 
 Note that $\{aba\} \otimes b - a \otimes \{bab\} = \delta(a, b)(a \otimes b)$ and $\{xyx\} \otimes x - x \otimes \{yxy\} = v(x,y)(x \otimes y).$ The relations  (\ref{HCP12}) resp. (\ref{HCP113})  can be rewritten as $\delta(a, b)(c\otimes d)+ \delta(c,d)(a \otimes b)$ resp. $v(x,y)(w \otimes z) + v(z, w).(x \otimes y).$  
This observation greatly simplifies the calculation: 
\begin{eqnarray*}
T.(\delta(a, b)(a \otimes b)) &=& [T, \delta(a, b)].(a \otimes b) + \delta(a, b)T.(a \otimes b)\\
&=& \delta(T.a, b)(a \otimes b) + \delta(a, T.b)(a \otimes b) \\ &&+ \delta(a, b)(T.a \otimes b) + \delta(a, b).(a  \otimes T.b)\\
&=& \delta(T.a, b)(a \otimes b)  + \delta(a, b)(T.a \otimes b) \\ &&+ \delta(a, b).(a  \otimes T.b) + \delta(a, T.b).(a  \otimes b) \in I(P)
\end{eqnarray*}
The computation for $v(x, y). (x \otimes y) \in M$ is the same.\\
For elements of the form (\ref{HCP12}): 
\begin{eqnarray*}
T.(\delta(a, b)(c \otimes d) + \delta(c,d)(a \otimes b)) &=& [T, \delta(a, b)].(c \otimes d) + \delta(a, b)T.(c \otimes d) \\ &&+ [T,\delta(c,d)](a \otimes b) + \delta(c,d)T(a \otimes b)\\
&=& \delta(T.a, b)(c\otimes d) + \delta(c,d)(T.a, b) \\ &&+ \delta(a, T.b)(c \otimes d) + \delta(c,d)(a \otimes T.b)\\
&&+ \delta(a, b)(T.c \otimes d) + \delta(T.c, d)(a \otimes b) \\ &&+ \delta(a, b)(c \otimes T.d) + \delta(c, T.d)(a \otimes b) \in I(P)
\end{eqnarray*}
It follows in an analogous fashion that $T(v(x,y).(w \otimes z) + v(z, w).(x \otimes y))\in I(P)$. For an element of the form (\ref{HCP13}): 
\begin{eqnarray*} && T.(\kappa(z,x)\otimes b - x\otimes (b \circ z) + z\otimes (b \circ x)  ) \\ &=& T.\kappa(z,x) \otimes b + \kappa(z,x) \otimes T.b  - T.x \otimes (b\circ z) \\ &&- x \otimes T.(b \circ z) + T.z \otimes ( b \circ x) + z \otimes T.(b \circ x)\\
&=& \kappa(T.x, z) \otimes b - T.x \otimes (b\circ z) + z \otimes (b \circ T.x) \\
&&+ \kappa(x, T.z) \otimes b - x \otimes (b \circ T.z) + T.z \otimes (b \circ x)\\
&&+ \kappa(z,x) \otimes T.b - x \otimes (T.b \circ z) + z \otimes (T.b \circ x) \in I(P)  
\end{eqnarray*}
Therefore the algebra $\str P$ acts on $I(P)$ and thus there is a well-defined $\str(P)$ action on $\uider(P).$
\end{proof}
The following Lemma is identical to \cite[Prop. 5.18]{BS} in case $1/2, 1/3 \in k.$
\begin{lem} \label{structure_of_uider} Let $P$ be a Jordan-Kantor pair. Then 
\begin{itemize} 
\item[\rm{(i)}] The $\instr(P)$-module $\uider(P)$ is a Lie algebra with bracket defined by
\begin{eqnarray*}
\left [a \diamond b, c \diamond d\right] &=& \delta(a, b).c \diamond d + c \diamond \delta(a, b).d\\
\left [x \diamond y, u \diamond v\right] &=& v(x,y).u \diamond v + u \diamond v(x,y).v\\
\left [a \diamond b, x \diamond y\right] &=& \delta(a, b).x \diamond y + x \diamond \delta(a, b).y\\
\left [x \diamond y, a \diamond b\right] &=&  v(x,y).a \diamond b + a \diamond v(x,y).b \\
\end{eqnarray*}
\item[\rm{(ii)}]  The $k$-linear map $\ud_{JKP} :\uider(P) \rightarrow \instr(P)$ given by linear extension of $\ud_{JKP}: a \diamond b \mapsto \delta(a, b)$, $x \diamond y \mapsto v(x,y)$ is a central extension of Lie algebras. 
\end{itemize}
\label{instr_aleg_is Lie_lage_lem}
\end{lem}
\begin{proof}
We have seen in (\ref{str1}) that $\instr(P)$ is an ideal of $\str(P)$ with Lie bracket defined by $[T, \delta(a,b)]  = \delta(T.a, b) + \delta(a, T.b)$  and $[T, v(x,y)]  =  v(T.x, y) + \delta(x, T.y)$ for $T \in \instr(P).$
By restriction, $\uider(P)$ is a Lie algebra module for $\instr(P)$. Moreover, let $(a, b), (c,d) \in J$ and $(x,y), (u,w) \in M,$ then  
$[\delta(a, b), \delta(a,b)] = \delta(\delta(a,b)a, b) + \delta(a, \delta(a,b)b) = 0$ similarly, $v(v(x,y)x, y) + v(x, v(x,y)y) = 0,$ $[\delta(a,b), \delta(c,d)] + [\delta(c,d), \delta(a,b)] = [v(x,y), v(u,w)] + [v(u,w), v(x,y)] = 0.$ Further, (\ref{derid2}) and (\ref{derid3}) show that $\delta(a, \kappa(y,w)) = v(a \circ y, w) - v(a \circ w, y)$ and $\delta(\kappa(u,x), b) = v(x, b \circ u) + v(u, b \circ x).$ \\ 
According to these observations, there is a well-defined linear map $\ud : \uider(P) \rightarrow \instr(P)$ obtained by extending: $a \diamond b \mapsto \delta(a, b)$, $x \diamond y \mapsto v(x,y).$ The bracket defined in (i) can be expressed as $[p,q] = \ud (p).q$ for $p, q \in \uider(P).$ It suffices therefore to check the two conditions of Lemma~\ref{Lie_alg_from_mod}. By (\ref{HCP11}) $\ud (a\diamond b).(a\diamond b) = \ud (x \diamond y).(x \diamond y) = 0$ and similarly  (\ref{HCP12}) is equivalent to $ \ud (a\diamond b).c\diamond d +  \ud (c\diamond d).a\diamond b =  \ud(x \diamond y).(w \diamond z) + \ud(w \diamond z).(x \diamond y) = 0.$  \\
Moreover, 
$  \ud(x \diamond y).a \diamond b  =  \kappa(a \circ y, x) \diamond b  -a \diamond \kappa(b \circ x, y) $
and by (\ref{HCP13}) and (\ref{HCP131}) this is equal to  $x \diamond  (b \circ (a \circ y)) - a \circ  y \diamond b \circ x +a \circ  y \diamond b \circ x - a \circ (b \circ x) \diamond y. $  By definition of $\delta(a, b)$ this equals $-x \diamond \delta(a, b) y - \delta(a, b)x \diamond y = -\delta(a, b)(x \diamond y)$. Hence, $\ud(x \diamond y).a \diamond b  = -\ud(a \diamond b).x \diamond y$ as required. \\
It suffices to check $[\ud(p), \ud(q)] = \ud(\ud(p).q)$ for elements $p, q$ in $J \diamond J,$ $M \diamond M$ or $P.$ \\
First case: $p\in J \diamond J$, $ q \in  M \diamond M.$ 
\begin{eqnarray*}
\left[\ud(a \diamond b), \ud(x\diamond y)\right] &=& [\delta(a, b), v(x,y)]\\
&=& \delta(\delta(a, b)x, y) + \delta(x, \delta(a,b).y)\\
&=& \ud(\delta(a,b).x \diamond y + x \diamond\delta(a,b).y)\\
&=& \ud(\delta(a, b).(x\diamond y))\\
&=& \ud(\ud(a \diamond b).(x \diamond y))
\end{eqnarray*}
Second case: $p,q \in (J\diamond J)$ or $p, q\in (M\diamond M)$. 
\begin{eqnarray*}
\left[\ud(a \diamond b), \ud(c\diamond d)\right] &=& [\delta(a, b), \delta(c,d)]\\
&=& \delta(\delta(a, b)c, d) + \delta(c, \delta(a,b).d)\\
&=& \ud(\delta(a,b).c \diamond b + c \diamond\delta(a,b).d)\\
&=& \ud(\delta(a, b).(c\diamond d))\\
&=& \ud(\ud(a \diamond b).(c\diamond d))
\end{eqnarray*}
It is also a straightforward verification that by (\ref{HCP13}),  we have $\delta(a, b)(x \diamond y) + v(x,y)(a \diamond b) = 0. $
For the other cases replace $\delta$ by $v$, $(a, b)$ by $(x,y)$ and $(c,d)$ by $(z,w)$. The calculations are identical in both cases. \\
We can apply Lemma~\ref{Lie_alg_from_mod} and conclude that $\ud_{JKP} : \uider(P) \rightarrow \instr(P)$ is a central extension of Lie algebras.
 \end{proof}

\begin{defi} If $P$ is a Jordan-Kantor pair, then $\ud_{JKP} : \uider(P) \rightarrow \instr(P)$ will from now on always denote the map defined in Lemma~\ref{structure_of_uider}. 
\end{defi}

\begin{prop} \label{uider_functor_prop} Let $f: P \rightarrow Q$ be a homomorphism of Jordan-Kantor pairs. Then
\begin{eqnarray*}
\uider(f) : \uider(P) & \rightarrow & \uider(Q) \\
a \diamond b & \mapsto  & f(a) \diamond f(b) \in J \\
 x \diamond y  &\mapsto  & f(x) \diamond f(y) \in M
 \end{eqnarray*}
is a Lie algebra homomorphism and
the assignment $\mathfrak{uider}_{JKP} : P \rightarrow \uider(P),$  $\uider : f \rightarrow \uider(f)$ 
is a covariant functor from the category of Jordan-Kantor pairs to the category of Lie algebras. 
\end{prop}
\begin{proof} By Lemma~\ref{structure_of_uider}, $\mathfrak{uider}(P)$ is a Lie algebra. Let $(f_J, f_M): P \rightarrow Q$ be a morphism of Jordan-Kantor pairs where $ P = (J_P, M_P)$ and $Q= (J_Q, M_Q).$  
 By the universal property of the tensor product $f$ extends uniquely to a map $f \otimes f = (f_J^+ \otimes f_J^- , f_M^+ \otimes f_M^-) : J_P^+ \otimes J_P^- +  M_P^+ \otimes M_P^- \rightarrow  J_Q^+ \otimes J_Q^- +  M_Q^+ \otimes M_Q^- .$ 
It is sufficient to prove that $(f \otimes f)(I(P)) \subset I(Q)$, see Definition~\ref{uider_JkP_Def}. 
Throughout $(a, b), (c,d) \in J_P$, $(x,y), (u,w) \in M_P.$
The map  $f: P \rightarrow Q$ is a homomorphism, therefore
\begin{eqnarray*}
f(\delta(a, b)c) \otimes f(d) &=& \delta(f(a), f(b)) f(c) \otimes f(d)\\
 f(c \otimes  \delta(a, b)d) &=& f(c) \otimes \delta(f(a), f(b))f(d) \\
f(v(x, y)z) \otimes f(w) &=& v(f(x), f(y)) f(z) \otimes f(w)\\ 
f(z \otimes  v(x, y)w)&=& f(z) \otimes v(f(x), f(y))f(w) 
\end{eqnarray*}
These identities show that  $f$ maps the elements defined in (\ref{HCP11}) and (\ref{HCP12}) into $I(Q)$ and  that $f$ sends elements of the form (\ref{HCP112})  and (\ref{HCP113}) to elements of $I(Q)$.
By the same argument, the following equalities hold true:
\begin{eqnarray*}
f(a) \circ f(y) = f(a \circ y),\,&& f(b) \circ f(x)= f(b \circ x) \\
f(\kappa(u,x)) = \kappa(f(u),\,&& f(x)) , f(\kappa(y,w)) = \kappa(f(y), f(w))
\end{eqnarray*} 
Hence it is easy to see that $f$ respects also relations  of the form (\ref{HCP13}). 
 By factorization, we obtain a well-defined  module homomorphism $\uider(f) : \uider(P) \rightarrow \uider(Q).$
 
 It remains to show that $\uider(f)$ is a Lie algebra homomorphism. For the product on $\uider(P)$ and $\uider(Q)$ defined by the equations in Lemma~\ref{structure_of_uider}, we easily check $[f(a) \diamond f(b), f(c) \diamond f(d)] = [\uider(f)(a \diamond b), \uider(f)(c \diamond d)] = \uider(f) \left ([a \diamond b, c \diamond d] \right ),$ and likewise, $[f(a) \diamond f(b), f(x) \diamond f(y)] = [\uider(f)(a \diamond b), \uider(f)(x \diamond y)] = \uider\left ( [a \diamond b, x \diamond y]\right ),$ $[f(u) \diamond f(w), f(x) \diamond f(y)] = [\uider(f)(u \diamond w), \uider(f)(x \diamond y)] = \uider \left ( [u \diamond w, x \diamond y] \right ).$ \\
 The Lie bracket on $\uider(P)$ is given by $[X, Y] = \ud(X).Y$ where the inner derivation algebra acts canonically on the tensor product and factors through $I(P),$ see Proposition~\ref{structure_of_uider}. Thus $\uider(f)([X,Y]) = \uider(f)(\ud(X).Y) = f(\ud(X)).f(Y) = \ud(f(X)).f(Y) = [\uider(f)(X), \uider(f)(Y)] $ and this proves that $\uider(f)$ is  a homomorphism of Lie algebras. Covariance of the functor is easily checked.

   \end{proof}
The results for the special case of a Jordan pair are as follows.
\begin{defi}Let $V$ be a Jordan pair. The \emph{universal derivation algebra} of $V$ is the $k$-module $\uider(V) := V^+\otimes V^-/I(V)$ where $I(V)$ is the $k$-submodule generated by the elements
$$\delta(x, y)(u \otimes v) + \delta(u, v)(x \otimes y),\quad\mbox{and}\quad \delta(x,y)(x \otimes y), \;(x,y), \,(u,v) \in V. $$
The coset $x \otimes y + I(V)$ is denoted by $x \diamond y.$

\label{uider_JP_defi}
\end{defi}
\begin{cor} Let $V$ be a Jordan pair. The $k$-module $\uider(V)$ is a Lie algebra under the product $[x\diamond y, u \diamond v] = \delta(x,y)(u\diamond v).$ \\
Moreover, the map $\ud : \uider(V) \rightarrow \instr(V)$ defined by linear extension of $x \diamond y \mapsto \delta(x,y)$ is a central extension of Lie algebras with kernel $$\mathrm{HC}(V) := \{\sum x_i \diamond y_i : \sum \delta(x_i, y_i) = 0\}.$$
The assignment $\mathfrak{uider}_{JP} : V \rightarrow \uider(V),$  $\uider : f \rightarrow \uider(f)$ 
is a covariant functor from the category of Jordan pairs to the category of Lie algebras. 
\label{JP_uider_cor}
\end{cor}

\section{(Universal) Tits-Kantor-Koecher algebras}
\label{TKK_Section}
\subsection{Tits-Kantor Koecher algebras}

 

\begin{defi}\label{gen_uTKK_Defi}
Let $P = (J,M)$  be a Jordan-Kantor pair and let $D^0$ be a Lie algebra with bracket $[\;, \; ]^0$ such that 
$$\xymatrix{ \uider(P) \ar[r]^{f} &  D^0 \ar[r]^{g} & \instr(P) \\}$$
is a sequence of central extensions with $g \circ f = \ud.$
Define on 
$\hat{L}(D^0) := P \oplus  D^0 =  J^+ \oplus M^+ \oplus D^0 \oplus M^- \oplus J^- $ a product by bilinear extension of
\begin{eqnarray}
\left [f(X'), f(Y') \right] &=& [f(X'), f(Y')]^0, \\
\left [(a, x), (b,y) \right] &=& a \circ y + f(a \diamond b) + f(x \diamond y) - b \circ x,  \label{bracket_to_zero_pos1}\\
 &=& -\left [(b, y), (a,x) \right],  \label{bracket_to_zero_neg1}\\
\left [f(X'), (a,x) \right ] &=& (\ud(X').a, \ud(X').x), \\
										 & =& - \left [(a, x), f(X') \right ], \\
\left [f(X'), (b,y) \right ] &=& (\ud(X').b, \ud(X').y), \\
 & =& - \left [(b, y), X \right ], \\
\left[(a,x), (c,u) \right] &=& (0, \kappa(x,z)), \\
\left[(b,y), (d,w) \right] &=& (0, \kappa(y,w)), 
\end{eqnarray}
for $(a, b), (c,d) \in J$, $(x,y), (u,w) \in M$ and $X', Y' \in \uider(P).$
 
\end{defi}
\begin{rem} Since $f$ is a Lie algebra epimorphism  and the multiplication in $P$ and the induced action of $\uider(P)$ on $P$ are all well-defined it follows that the product is indeed well-defined. 
Bilinearity of the product  is also clear. This gives  $\hat L(D^0)$ an algebra structure with $5$-graded product $[\,,\,].$\\
 \label{well_Def_ULE_Rem}
\end{rem}

We are now ready to define the Tits-Kantor-Koecher algebra of a Jordan-Kantor pair $P$: 
\begin{defi} For $D^0= \instr (P)$, $f = \uider(P), g = \mathrm{id}$, the algebra  of Definition~\ref{gen_uTKK_Defi} is called the \emph{Tits-Kantor-Koecher algebra} of $P$ or $\TKK(P).$\\
For $D^0= \uider (P)$, $f = \mathrm{id}, g = \uider(P)$ the algebra  of Definition~\ref{gen_uTKK_Defi} is called the \emph{ universal Tits-Kantor-Koecher algebra} of $P$ or $\uTKK(P).$  Here $\uTKK$ stands for \textbf{u}niversal \textbf{L}ie \textbf{e}nvelope. Then
$$\uTKK(P) = P \oplus \uider(P) .$$
For $h: P \rightarrow Q$ a morphism of Jordan-Kantor pairs define
\begin{eqnarray*}
\uTKK(h) : P \oplus \uider(P) \rightarrow Q \oplus \uider(Q)\\ 
\uTKK(h) := (h, \uider(h))
\end{eqnarray*}
Since $\uider$ is a functor, it is clear that $\uTKK( \;)$ is a functor with respect to the algebra structure defined in Definition~\ref{gen_uTKK_Defi}. In particular, $\uTKK(h)$ as defined above is a morphism. 
 \end{defi}

\subsection{The universal property and the functor $\uTKK$}

We would like to establish central extensions of $\TKK(V)$ and their relationship to the central extension $\uTKK(V) \rightarrow \TKK(V).$ It will turn out that $\uTKK(V)$ has a universal property which is closely related to the universal property of $\uce{\TKK(V)},$ if $V$ is perfect in the sense of Definition~\ref{perfect_JKP_defi}.

 
The following proposition extends \cite[Thm. 5.16]{BS}.

\begin{prop} 
\label{ULE_TKK_lem}
Let $P$ be a Jordan-Kantor pair
\begin{itemize}
\item[\rm(i)]
The algebra $\hat{L}(D^0) $ is a  $5$-graded Lie algebra. 
\item[\rm(ii)] There are unique  graded central extensions $\hat f : \uTKK(P) \rightarrow \hat L(D^0)$ and
$\hat g : \hat L(D^0) \rightarrow \TKK(P)$ such that $\hat f |_P = \hat g |_P  = \mathrm{id}_P.$ 
\item[\rm(iii)] The map $\uTKK$ is a covariant functor from the category of Jordan-Kantor pairs to the category of Lie algebras. \\ In particular, for every homomorphism $f: P \rightarrow Q$ of Jordan-Kantor pairs, the homomorphism $\uTKK(f) : \uTKK(P) \rightarrow \uTKK(Q)$ is the unique map that renders the diagram below commutative
$$
\xymatrix{P  \ar[r] \ar[d]^{f}&  \uTKK(P) \ar[d]^{\uTKK(f)} \\
Q  \ar[r]& \uTKK(Q) .
}
$$

\end{itemize}
\end{prop}

\begin{proof}
(i) We have to show that the product $[\,,\, ]$ is  alternating and fulfills the Jacobi identity. 

Since $\hat L(D^0)$ is a graded algebra with respect to $[\,,\, ]$ it suffices to consider homogeneous elements. The space $D^0$ is by definition a Lie algebra in its own right, so we may restrict to pairs or triplets of elements where at least one of the elements is not of degree $0.$\\
Then Definition~\ref{gen_uTKK_Defi} explicitly states that the bracket between elements in $P$ and $D^0$ is anti-commutative. For an element $ p = ((a, b), (x,y)) \in P$ 
\begin{eqnarray*}
\left[p, p \right] &=& [(a, b), (a, b)] + [(a, b), (x,y)]\\ 
&&+ [(x,y), (a, b)] + [(x,y), (x,y)]\\ 
&=& [\delta(a, b)a, b] + [a, \delta(a ,b)b] + [\delta(a, b)x, y] + [x, \delta(a,b).y ]\\
&&+ [\delta(x, y)a, b] + [a, \delta(x,y).b ] + [v(x,y).x, y] + [x, v(x,y)]\\
&=& f(\delta(a, b)(a \diamond b)) + f(\delta(a, b)(x \diamond y) ) \\
&&+ f(\delta(x,y)(a \diamond b)) + f(v(x,y)(x \diamond y))
\end{eqnarray*}
and this is zero, since $f : \uider(P) \rightarrow D^0$ is well-defined.  Thus $[p,p] = 0.$\\
For the Jacobi identity it suffices again to consider homogeneous elements.  
It suffices to check the Jacobi identity on $\uTKK(P),$ since $\hat L(D^0)$ is an algebra homomorphism.
Let $X, Y \in \mathcal D_0$ and choose $X'$ and $Y'$ in $\mathfrak{uider}(P)$ such that $\ud(X') = g(X) = T$ and $g(Y') = Y = S.$ 

 The module $P$ is a Lie algebra module for $g(D^0) = \instr(P)$, thus if $\deg X '= \deg Y' = 0$ and $r \in P$ then $[[X,Y], r] = g([X',Y']).r = g(X').g(Y').r - g(Y').g(X').r = [T, [S,r]] - [S,[T.r]]$ which shows the Jacobi identity in this case. \\ We do not have to check the Jacobi identity for elements whose degree sum up to an integer smaller than $-2$ or greater than $2,$  since, in this case, all the three terms are already zero.
The remaining cases are: 
\begin{itemize}
\item[\rm{(i)}] $(x,a) \in P^{\sigma}, (z, c) \in P^{\sigma}$, $X \in \mathcal D_0$ as above,
\item[\rm{(ii)}] $(x,a)\in P^{\sigma},(y,b) \in P^{-\sigma}$,$X \in \mathcal D_0$ as above,
\item[\rm{(iii)}] $(x,a) \in P^{\sigma}, (z,c) \in P^{\sigma}$, $(y,b) \in P^{-\sigma}$. 
\end{itemize}
\begin{itemize}
\item[\rm{(i)}] $(x,a) \in P^{\sigma}, (z, c) \in P^{\sigma}$, $X \in \mathcal D_0$ as above.
\begin{eqnarray*}
\left[[(x,a), (z,c)], X\right] &=& [\kappa(x,z), X] \\
&=& - [X, \kappa(x,z)]
=  -T.\kappa(x,z)\\
&=& -(\kappa(T.x, z) - \kappa(z, T.x) \\
&=& [(x, a)[(z, c), X]] - [(z, c)[(x,a), X]]
\end{eqnarray*}
In the last equation it was possible to re-introduce $a$ and $c$ because the last line does not depend on the choice of the elements in $J^\sigma.$\\
\item[\rm{(ii)}]  $(x,a)\in P^{\sigma},(y,b) \in P^{-\sigma}$,$X \in \mathcal D_0$ as above.
\begin{eqnarray*}
\left[X, [(x,a), (y,b)]\right] &=& [X,  (-b \circ x, 0)_{-\sigma} + f(a \diamond b) + f(x \diamond y) + (a \circ y, 0)_\sigma]\\
&=& -T.(b \circ x) + [X, f(a \diamond b)] + [X, f(x \diamond y)] + T.(a \circ y)\\
&=& -(T.b \circ x) - (b \circ T.x) + f(T.a \diamond  b +  a\diamond T.b) \\ 
&&+ f(T.x\diamond y  + x \diamond T.y) + T.a \circ y + a \circ T.y\\
&=& [(x,a),(T.y, b) + (y, T.b) ] - [(y,b), (T.x, a) + (a, T.x)]\\
&=& [(x,a),[X, (y,b)]] - [(y,b),  [X, (x,a)]]
\end{eqnarray*}
\item[\rm{(iii)}] Let $(x,a) \in P^{+}$, $(z, c) \in P^{+}$, $t  =(y,b) \in P^{-}$, the case where all three signs are multiplied by $-1$ can be obtained in the same fashion. 
\begin{eqnarray*}
\left[(x,a), [(y,b), (z,c)] \right] &=& [(x,a), (b \circ z, 0)  - f(c \diamond b) - f(z \diamond y) - (c \circ y, 0)]\\
 &=& \delta(c,b)a + v(z,y)a - \kappa(x, c \circ y) + a \circ (b \circ z) \\ &&+ v(z,y)x + \delta(c,b)x + f(x \diamond(b \circ z)) \\
 \left[(z,c) [(x,a), (y,b)]\right] &=&  -  \delta(a,b).c - v(x,y).c +  \kappa(z, a \circ y)  - c \circ (b \circ x) \\ &-& \delta(a,b).z -  v(x,y).z  - f(z \diamond(b \circ x)) \\
 \left[(y,b), [(z,c), (x,a)]\right] &=& [(y,b), \kappa(x,z)]\\
 &=& -f( \kappa(x,z) \diamond b )- \kappa(x,z) \circ y
\end{eqnarray*} 
Thus the homogeneous components of $ [(x,a), [(y,b), (z,c)]]  + [(y,b), [(z,c), (x,a)]] + [(z,c) [(x,a), (y,b)]]$ are as  follows: In degree $2$: 
$\delta(c,b)a + v(z,y)a -  \delta(a,b).c - v(x,y).c - \kappa(x, c \circ y) +   \kappa(z, a \circ y) = v(z,y)a  - v(x,y).c - \kappa(x, c \circ y) +   \kappa(z, a \circ y) = 0$ by Definition~\ref{extended_ops_JKP_defi} and (\ref{JK4}). 
In degree $1$, one obtains 
$v(z,y)x + \delta(c,b)x + x \diamond(b \circ z) - z \diamond(b \circ x) + a \circ (b \circ z)- c \circ (b \circ x) - \delta(a,b).z -  v(x,y)z  - \kappa(x,z) \circ y.$ Again by (\ref{JK4}) we have that $0 = v(z,y)x -  v(x,y)z  - \kappa(x,z) \circ y,$ and further $ \delta(c,b)x  - c \circ (b \circ x) + a \circ (b \circ z) - \delta(a,b).z = 0$ by definition of $\delta.$ Lastly, in degree $0$, $f(x \diamond(b \circ z) - z \diamond(b \circ x) - \kappa(x,z) \diamond b) = 0$ by (\ref{HCP13}).
\end{itemize}
(ii) Let $f = \mathrm{id}$. Then  by (i) $\uTKK(P) = \hat L(\uider(P))$ is a Lie algebra. If $g = \mathrm{id},$ then (i) shows that $\TKK(P)$ is a Lie algebra.  \\
Since $P$ generates $\hat L(D^0),$ it follows that $\hat f$ and $\hat g$ are onto. It is also clear that $\hat f$ and $\hat g$ are $5$-graded, with $\ker \hat f = \ker f \subset \uTKK(P)_0$ and  $\ker \hat g = \ker g \subset D^0.$ Thus Lemma~\ref{central_ext_lem} implies that  the homomorphisms are central extensions. \\
(iii) Remark~\ref{well_Def_ULE_Rem} already states that $\uTKK(h)$ is an algebra morphism. Combined with (ii) this gives that $\uTKK(h)$ is a Lie algebra homomorphism. Since $\uider$ is covariant and preserves  the identity morphism, it follows that $\uTKK$ is also covariant and that $\uTKK(\mathrm{id}_P) = (\mathrm{id}_P, \uider(\mathrm{id}_P) ) =  \mathrm{id}_{\uTKK(P)}.$
\end{proof}
\begin{defi} Proposition~\ref{ULE_TKK_lem} gives a uniquely determined central extension
$$\hat \ud : \uTKK(P) \rightarrow \TKK(P) $$
such that $\ker \hat \ud = \HF(P).$
\label{universal_TKK_ce_def}
\end{defi}

\begin{defi}
\label{JKP_type_defi}
Let $L$ be a $5$-graded Lie algebra.  We say that $L$ is of \emph{Jordan-Kantor-type} if the pair $(L_2, L_{-2})$ can be endowed with the structure of a Jordan pair $J$  such that
 $J^\sigma = L_\sigma,$  $ D^\sigma(a, b) = \ad[a,b]|_{L_\sigma},$ and $L_0 = [L_{-1}, L_{-1}] + [L_{-2}. L_{2}].$ \\
 \end{defi}

\begin{prop} If $L$ is of Jordan-Kantor type, then  $((L_{2}, L_{-2}), (L_1, L_{-1}))$ is a Jordan-Kantor pair, with respect to 
$v(x,y) = \ad[x,y]$ for $x \in L_\sigma, y \in L_{-\sigma}$ and  $\kappa(x,u) = [x,u]$ for $x,u \in L_\sigma.$
\\
Let $L$ be a $5$-graded Lie algebra such that $L_0 = [L_{-1}, L_{-1}] + [L_{-2}. L_{2}].$ If  $1/2 \in k$ and $L$ does not have $3$-torsion, then $L$ is of Jordan-Kantor type. 
\end{prop}
\begin{proof} Let $\sigma = \pm.$ 
If $1/2 \in k,$ then $Q^\sigma_{a} = -1/2 (\ad(a))^2,$ $a \in L_{\sigma 2}$ is a quadratic map from  $L_{\sigma2}$ into $\Hom(L_{-\sigma2}, L_{\sigma2}).$ For $a,c \in L_{\sigma2}$ and $b \in L_{-\sigma 2}$ the linearized operator acts as
\begin{eqnarray*}
Q^\sigma_{a,c}(y) &=& -1/2 (\ad(a + c))^2(b) + -1/2 (\ad(a))^2(b) - -1/2 (\ad(c))^2(b)\\
&=& -1/2([a+ c,[a + c, b ]] - [a, [a,b]] - [c,[c,b]])\\
&=& - 1/2([a,[ c, b ]] + [c,[a, b ]]  )
\end{eqnarray*}
Since $[a,c] = 0$
this equals 
$$ -1/2( [c, [a,b]] - [[a,b], c]) = \ad[a,b](c).$$
Thus $D_{a,b}(c) = Q_{a,c}(b) = \ad[a, b].$ 
The Jacobi identity  implies in particular that for elements $a, c \in L_{\sigma2}, b,d \in L_{-\sigma 2}$
$$[\ad[a, b], \ad[c,d]] = \ad [c, \ad[a, b]d  ] + \ad[\ad[a, b]c, d].$$
Since $1/2 \in k$ and $L$ does not have $3$-torsion in particular $L$ does not have $6$-torsion
Thus by \cite{loosJP}, $(L_{2}, L_{-2})$ is a Jordan pair with quadratic operators $Q^\sigma.$ 
The remaining identities for a Jordan-Kantor pair are quite straightforward computations which require only the axioms for a Lie algebra.  
\end{proof}
\begin{defi} If $L$ is of Jordan-Kantor type, then we denote the Jordan-Kantor pair $((L_{2}, L_{-2}), (L_1, L_{-1}))$ by $P_L.$
\end{defi}

See \cite[Thm. 5.18]{BS} for the following result, if $1/2$, $1/3$ are in $k.$
\begin{cor}
Let $L$ be a $5$-graded Lie algebra of Jordan-Kantor type. 
\begin{itemize} 
\item[(i)]
There are uniquely defined graded central extensions 
\begin{eqnarray*}
\hat f: \uTKK(P_L) & \rightarrow  & L \\
\hat g : L  &\rightarrow & \TKK(P_L)
\end{eqnarray*}
such that $\hat g \circ \hat f = \hat \ud$ and $\hat f|_{P_L} = \mathrm{id}_{P_L}.$
\item[(ii)]
The Lie algebra $\uTKK(P)$ is universal in the following sense:  For any homomorphism of Jordan-Kantor pairs $h : P \rightarrow P_L$ there exists a unique extension $\tilde h $ of $h$ to a graded Lie algebra homomorphism $\tilde h : \uTKK(P) \rightarrow L.$
\label{uTKK_universal_prop}
\end{itemize}
\end{cor}

 \begin{proof}
If such a $\hat f$ exist, then  it is uniquely determined by the condition: $\hat f|_{P_L} = \mathrm{id}_{P_L}.$ In this case, we also have  $\hat g|_{P_L} = \mathrm{id}_{P_L}.$
 \begin{itemize}
\item[(i)] In view of Lemma~\ref{ULE_TKK_lem} it suffices to show that there is a chain of central extensions $\uider(P_L) \stackrel{f}{\rightarrow} L_0 \stackrel{g}{\rightarrow} \instr(P_L)$ such that $g \circ f = \ud.$ 
If $L$ is of Jordan-Kantor type then 
$L_0 = [L_2, L_{-2}] + [L_1, L_{-1}]$ and thus $L_0$ is a quotient of $ L_2 \otimes L_{-2} + L_1 \otimes L_{-1}.$
Since the Lie bracket is bilinear and alternating, there is a well-defined epimorphism of the following form
\begin{eqnarray*}
 \tilde {f} :  L_2 \otimes L_{-2} + L_1 \otimes L_{-1} \rightarrow L_0,\\
 a \otimes b \mapsto [a, b], && (a, b) \in J_L, \\
 x \otimes y \mapsto [x,y], && (x,y) \in M_L.
 \end{eqnarray*}
Let $J_L = (L_{2}, L_{-2})$ and $M_L = (L_1, L_{-1})$ be the Jordan resp. the Kantor components of $L.$
Then, for $(a, b) \in J_L$ and $(x,y) \in M_L$, the Jordan-Kantor pair structure on $(J_L, M_L)$ gives $\ad_L[a, b] = \delta(a, b)$ and $\ad_L[x,y] = v(x,y)$ when restricted to $J_L \oplus M_L.$ 
For $(a, b)$ and $(c,d)$ in $J_L$, this gives us the following identities: 
\begin{eqnarray*}
0 &=&  [[a, b], [a, b]] \\
&=& [ \delta(a,b)a, b] + [a, \delta(a,b)b] \\
&=& [\{a,b,a\}, b] - [a, \{b,a,b\}] \\
0 &=& [[a, b], [c,d]] + [[c, d], [a,b]]\\
&=& [\delta(a, b)c, d] + [c, \delta(a, b)d] + [\delta(c, d)a, b] + [a, \delta(c, d)b]\\
&=& [\{a, b, c\} ,d] - [c, \{b,a, d\}]   + [\{c, d, a\} ,b] + [a, \{c,d,b \}] 
\end{eqnarray*}
and likewise for $(x,y)$ and $(u,w)$ in $M_L$
\begin{eqnarray*}
0 &=&  [[x, y], [u, w]] \\
&=& [\{x,y,x\}, y] - [x, \{y, x, y\}] \\
0 &=& [[x, y], [u,w]] + [[u, w], [x,y]]\\
&=& [\{x, y, u\} ,w] - [u, \{y,x, w\}]   + [\{w, u, x\} ,y] + [x, \{u,w,y \}]. 
\end{eqnarray*}
The calculations above show that $\tilde f$ factors through the submodule generated by (\ref{HCP11}) to (\ref{HCP113}).
It follows similarly  that 
\begin{eqnarray*}
0 &=&  \left [a,[y,w]\right ] -[a \circ y, w] + [a \circ w, y],\\
0 &=&  [[u,x], b] + [u, b \circ w]  - [x, b \circ u].
\end{eqnarray*}
Therefore, $\tilde f$ also factors through (\ref{HCP13}) and we have a Lie algebra epimorphism 
\begin{eqnarray*}
f : \uider((J_L, M_L)) \rightarrow L_0, \\
a \diamond b \mapsto [a, b], && (a, b) \in J_L, \\
 x \diamond y \mapsto [x,y], && (x,y) \in M_L.
 \end{eqnarray*}
For any  finite number of elements $((a_i, b_i), (x_i, y_i)) \in P,$  $\sum_{i} [a_i, b_i] + [x_i, y_i] = 0$ implies that $\ad(\sum_{i} [a_i, b_i] + [x_i, y_i])|_{P_L} = \sum \delta(a_i, b_i) + v(x_i, y_i)  = 0.$ 
Hence there is a well-defined Lie algebra homomorphism:
 \begin{eqnarray*}
  g: L_0 \rightarrow \instr((J_L, M_L)),\\
\left [a, b \right] \mapsto \delta(a, b), && (a, b) \in J_L, \\
\left [x,y \right] \mapsto \delta(x,y), && (x,y) \in M_L.
 \end{eqnarray*}
 It is obvious that $g \circ f = \ud.$ Therefore $\ker f \subset Z(L)$. Assume $z \in \ker g.$ Then $\ad z|_{P} = 0,$ since $g$ is onto and $\instr(P)$ acts faithfully on $P.$ However, $P$ generates $L,$ hence $ z \in Z(L_0).$ Hence $f$ and $g$ are central extensions.  \\
 Using the notation in Lemma~\ref{ULE_TKK_lem} with $P = P_L$ and $D^0 = L_0,$ it is easy to see that $L= \hat L(D^0)$. Thus there is a unique Lie algebra homomorphism $ \hat f : \uTKK(P_L) \rightarrow L$ such that $\hat f|_{P_L} = \mathrm{id}.$ Lemma~\ref{ULE_TKK_lem}(iii) gives for every Jordan-Kantor homomorphism $h: P \rightarrow P_L$ the existence of a unique Lie algebra morphism $\uTKK(h) : \uTKK(P) \rightarrow \uTKK(P_L)$ which extends $h : P \rightarrow P_L.$ Thus $\tilde h = \hat f \circ \uTKK(h) $ has the properties  required in (ii). Let $h'$ be another Lie algebra homomorphism $h' : \uTKK(P) \rightarrow L$ such that $h'|_P = h.$ Since $P$ generates $\uTKK(P)$ as Lie algebra, this determines $h'$ on all of $\uTKK(P).$ Therefore $h = \tilde h. $
\end{itemize}
\end{proof}

\begin{defi} A \emph{perfect} Jordan-Kantor pair is a Jordan-Kantor pair $P$ such that for $\sigma  = \pm$: 
\begin{eqnarray*}
J^\sigma &=& \{J^\sigma J^{-\sigma }J^\sigma \} + \kappa(M^\sigma,M^\sigma), \\
M^\sigma &=&  \{M^\sigma M^{-\sigma }M^\sigma \} + J^{\sigma} \circ M^{-\sigma}.
\end{eqnarray*} \label{perfect_JKP_defi}
\end{defi}
\begin{lem}
The following are equivalent:
\begin{itemize} 
\item[\rm{(i)}] $P$ is a perfect Jordan-Kantor pair.
\item[\rm{(ii)}] $\TKK(P)$ is a perfect Lie algebra.
\item[\rm{(iii)}] $\uTKK(P)$ is a perfect Lie algebra.
\item[\rm (iv)] There is a perfect Lie algebra $L$ of Jordan-Kantor  type such that $P_L = P$ (as Jordan -Kantor pair). 
\end{itemize}
\label{perfect_TKK_lem}
\end{lem}
\begin{proof} We prove (i) $\implies$ (iii) $\implies$ (ii) $\implies$ (i), and (iii) $\implies$ (iv) $\implies$ (ii). 
Assume that $P$ is a perfect Jordan-Kantor pair. 
As Lie algebra $\uTKK(P)$ is generated by $\sum_{i \neq 0}P^i$. It suffices therefore to show that $P^i \subset [\uTKK(P), \uTKK(P)]$ for $-2 \leq i \leq 2$. If $i =  \sigma2, $, $\sigma = \pm,$ then $\{J^\sigma J^{-\sigma} J^{\sigma}\} + \kappa(M^\sigma, M^\sigma) = [J^\sigma, [J^{-\sigma}, J^{\sigma}]] + [M^\sigma, M^\sigma] = P^{\sigma 2} \subset [\uTKK((P)), \uTKK(P)]$. For $ i = \sigma 1$ a similar argument works:  $\{M^\sigma M^{-\sigma} M^{\sigma}\} + J^{-\sigma} \circ M^\sigma) = [M^\sigma [M^{-\sigma} M^{\sigma}]] + [J^{-\sigma}, M^\sigma] = P^{\sigma 1} \subset [\uTKK(P),\uTKK(P)].$ Thus (i) implies (iii). The implication (iii) $\implies$ (ii) is clear,  since $\TKK(P)$ is an epimorphic image of $\uTKK(P).$ \\
If $\TKK(P)$ is perfect then $$J^\sigma = [[J^\sigma, J^{-\sigma}], J^\sigma ] + [[M^\sigma, M^{-\sigma}], J^\sigma] + [M^\sigma, M^\sigma] \subset \{J^\sigma J^{-\sigma} J^\sigma \} + \kappa(M^\sigma, M^\sigma)$$ and
$$M^\sigma = [[J^{-\sigma}, J^\sigma], M^\sigma] + [[M^{-\sigma}, M^\sigma], M^\sigma] + [J^\sigma, M^{-\sigma}] = \{M^\sigma M^{-\sigma }M^\sigma \} + J^{\sigma} \circ M^{-\sigma}.$$ Thus $P$ is perfect, i.e (ii) $\implies$ (i).\\
Now (iii) implies (iv) since $\uTKK(P)$ is such a Lie algebra. Assume (iv). By \ref{uTKK_universal_prop}, $L$ is a central extension of $\TKK(P)$ whence $\TKK(P)$ is perfect which is (ii).  
\end{proof}

Every Jordan-Kantor pair $P$  comes with an involution, the ``flip'' that exchanges the positive and the negative parts. The Jordan-pair obtained in this way is called the \emph{opposite pair} of $P$ and it is denoted by $P^{\mathrm{op}}.$ It is very easy to check that the following holds.

\begin{prop} 
\label{opp_TKK_prop}
Let $P$ be a Jordan-Kantor pair and $P^{\mathrm{op}}$ the opposite pair of $P$. Then 
\begin{itemize}
\item[\rm(i)] $\TKK(P) \cong \TKK(P^{\mathrm{op}}),$
\item[\rm(ii)] $\uTKK(P) \cong \uTKK(P^{\mathrm{op}}),$
\item[\rm(iii)] $P$ is perfect if and only if $P^{\mathrm{op}}$ is perfect. 
\end{itemize}
\end{prop}

\subsubsection{The central extension $\mathfrak{uce}(L)$ vs.  the central extension $\uTKK(P)$}
Recall the construction of $\uce L$ where $L$ is a $\Gamma$-graded Lie algebra which we described in Definition~\ref{uce_def}. 
\begin{cor}
\label{JKP_central_ext_kernel_cor} 
Let $P$ be a perfect Jordan-Kantor pair, $L = \TKK(P)$ and endow $\mathfrak{uce}(L)$ with the canonical $\mathbb Z$-grading obtained from the $\mathbb Z$-grading on $L$. Then there is a Lie algebra isomorphism 
$\uce{L}_0 \rightarrow \uider(P)$, given by $\langle a , b\rangle  \mapsto a \diamond b$, $\langle x, y\rangle \rightarrow x \diamond y $ for $(a, b) \in J, (x, y) \in M.$  In particular, for  $u : \mathfrak{uce}(L) \rightarrow L, $ we have $$(\ker u)_0\cong \mathrm{HC}(P).$$
\end{cor}
\begin{proof}
By the universal property of $\mathfrak{uce}(L)$ there is a unique homomorphism $f$ such that the following diagram commutes
\begin{equation*}
\xymatrix{\mathfrak{uce}(L) \ar[r]^f\ar[d] & \uTKK(P) \ar[d]^{\ud_{JKP}}\\
[L, L] \ar[r] & \TKK(P)
}.
\end{equation*}
Now restrict to the zero component. For any Lie algebra $K$ in this diagram we have $[K, K]_0 = K_0$. This gives a surjective Lie algebra homomorphism $f_0 : (\mathfrak{uce}(L))_0 \rightarrow \uTKK(P)_0 = \uider(P)$ which is uniquely determined by $\langle x, y\rangle \rightarrow x \diamond y$ and $\langle a, b\rangle \rightarrow a \diamond b$. 
Define $g: x \otimes y \rightarrow x \wedge y$ and $g: a \otimes b \rightarrow  a \wedge b$ 
for $(x,y) \in M, (a, b) \in J.$ 
It suffices to prove that $g$ respects the relations in $I(P)$ : 
For relations of the form $ x \otimes \{yxy\} - \{xyx\} \otimes y $ or $ a \otimes \{bab\} - \{aba\} \otimes b $ this follows from 
\begin{eqnarray*}
g(a \otimes \{bab\} - \{aba\} \otimes b ) &=&  a \wedge [[b,a],b]  -  b \wedge [[a, b],a]. \\
&= &  [a, b] \wedge [a, b] = 0
\end{eqnarray*}
 Similarly, a relation of the form (\ref{HCP113})
 is mapped into $$[a, b] \wedge [c,d] + [c,d] \wedge [a, b],  $$
 which is in the same coset as $0$ in $\uce L_0.$
 If we consider a relation (\ref{HCP13}) then under $g$ it lies in the coset of
 $$[z,x]  \wedge b + [b,z] \wedge x + [x,b] \wedge z $$
 which is the coset of $0$ in $\uce  L.$
 Thus there is a well-defined linear map $g_0 : \uider(P) \rightarrow J^{+} \wedge J^{-}  \oplus M^+ \wedge M^{-}.$ Since $\TKK(P)$ is of Jordan-Kantor type and perfect, Lemma~\ref{graded_cover_neh} implies that $g_0$ maps $\uider(P)$ indeed onto $\mathfrak{uce}(L)_0.$ 
  Then $f_0$ and $g_0$ are well-defined, linear and inverse to each other. 
  \end{proof}

\subsection*{Base change}
\begin{prop}Let $P = (J,M)$ be a Jordan-Kantor pair of $k$ such that $J^\sigma$ and $M^\sigma$  are finitely generated as $k$-modules. If $k\rightarrow K$ is a faithfully flat base change, then 
$\uTKK(P \otimes_k K) = \uTKK(P) \otimes_k K.$
\label{TKKalghalfthirddfunc}
\end{prop}
\begin{proof} Since Jordan-Kantor pairs are given by identities of degree at most $2$ in every variable, the category of Jordan-Kantor pairs  is stable under base changes. Thus $ P_K = P \otimes_k K$ is also a Jordan-Kantor pair.
For any base change the map $\omega :\instr(P) \otimes_k K \to \instr(P \otimes_k \otimes K)$, given by  $\delta(a, b) \otimes \alpha  \mapsto \delta(a \otimes \alpha, b\otimes 1)$, $v(x,y) \otimes \alpha \mapsto v(x \otimes \alpha, y \otimes 1)$, is surjective onto the inner derivations of $P \otimes_k K$. Note that $\omega$ is the restriction of the canonical map $ \Omega: \mathrm{End}_k(P) \otimes_k  K  \to \mathrm{End}_K(P \otimes_k K).$
If $k \rightarrow K$ is a faithfully flat base change, then $\Omega$ and thus also $\omega$ is injective and hence $\instr(P \otimes_k K ) = \instr(P) \otimes_k K$ which proves the proposition. 
\end{proof}

\begin{rem}The result also holds for Jordan pairs and Kantor pairs. 
\end{rem}

\subsection{Jordan pairs}
If $V$ is a Jordan pair, then we have already defined $\uider_{JP}(V),$ see Definition~\ref{uider_JP_defi}. 
If the Jordan-Kantor pair is a Jordan pair $V$, that is when $M = 0,$ the $\pm 1$ components of $\uTKK$ and $\TKK$ vanish. These Lie algebras are thus only $3$-graded once we have defined $\uTKK(V)^{\pm 1} = \uTKK(V)_{\pm 2}$ and $\TKK(V)^{\pm 1} = \TKK(V)_{ \pm 2}.$ 
\begin{defi} Let $L= L^{-1} \oplus L^{0} \oplus L^{1}$ be  a $3$-graded Lie algebra, i.e., $L$ is $\mathbb Z$-graded and $\supp_{\mathbb Z}L \subset \{-1, 0 ,1\}.$ Then $L$ is \emph{Jordan $3$-graded or of Jordan type} if  the pair of modules $(L^1, L^{-1})$ can be endowed with the structure of a Jordan pair in such a way that $V^\sigma = L^\sigma,$  $ D^\sigma(x, y) = \ad[x,y]|_{L^\sigma}$ and $L^0 = [L^{1}, L^{-1}].$ \label{jordan_type_Defi}
\end{defi}

\begin{rem} This definition is consistent  with Definition~\ref{JKP_type_defi}. If in Definition~\ref{JKP_type_defi} $L_{\pm \sigma} = 0$ then we obtain a Lie algebra of Jordan type by setting $L^\sigma = L_{2\sigma}$ (with the notation used there). 
\end{rem}


\begin{cor}
Let $L= L^{-1} \oplus L^{0} \oplus L^{1}$ be a $3$-graded Lie algebra of Jordan-Kantor type and $V_L = (L_1, L_{-1})$ a Jordan pair.  
\begin{itemize} 
\item[ \rm (i)]
There are uniquely defined graded central extensions 
\begin{eqnarray*}
\hat f: \uTKK(V_L) & \rightarrow  & L \\
\hat g : L  &\rightarrow & \TKK(V_L)
\end{eqnarray*}
such that $\hat g \circ \hat f = \hat \ud.$
\item[\rm (ii)] Let $V$ be any Jordan pair. 
The Lie algebra $\uTKK(V)$ is universal in the following sense:  For any homomorphism of Jordan pairs $h : V \rightarrow V_L$ there exists a unique extension $\tilde h $ of $h$ to a graded Lie algebra homomorphism $\tilde h : \uTKK(V) \rightarrow L.$
\label{uTKK_universal_prop_JP}
\end{itemize}
\end{cor}

Perfect Jordan pairs are those that are perfect as Jordan-Kantor pairs, i.e.

\begin{center}
A Jordan pair $V$ is perfect if and only if $\{V ^\sigma, V^{-\sigma} V^{\sigma} \} = V^{\sigma}$ for $\sigma = \pm. $
\end{center}
We immediately obtain from Lemma \ref{perfect_TKK_lem} the following corollary: 
\begin{cor}
The following are equivalent:
\begin{itemize} 
\item[\rm{(i)}] $V$ is a perfect Jordan pair.
\item[\rm{(ii)}] $\TKK(V)$ is a perfect Lie algebra.
\item[\rm{(iii)}] $\uTKK(V)$ is a perfect Lie algebra. 
\item[\rm (vi)] There is a perfect Lie algebra $L$ of Jordan type such that $(L^1, L^{-1}) = V$ (as Jordan pairs).
\end{itemize}
\label{perfect_TKK_cor}
\end{cor}

\chapter{Root systems and root graded Lie algebras}
\label{root_systems_chapter}

\section{Locally finite root systems}
Throughout this section, $\mathbb K$ is a field of characteristic $0$.
\begin{defi}[\cite{lfrs}]
A \emph{(locally finite) root system} is a pair $(R, E)$ consisting of a $\mathbb K$-vector space $E$ and a subset $R \subset E$ satisfying the following properties: 
\begin{itemize}
\label{root_sys_defi}
\narrower
\item[] \rm{(i)} $0 \in R$ and $\spa_{\mathbb K} R = E,$
\item[] \rm{(ii)} $R$ is \emph{locally finite}, i.e., $R \cap F$ is finite for every finite dimensional subspace $F$ of $E,$ 
\item[] \rm{(iii)} for every $\alpha \in R^\times := R\setminus \{0\}$ there exists an $\alpha \che$ in $E^*$ such that $\langle \alpha, \alpha \che\rangle =2$ and $s_\alpha(R) = R$ where the reflection $s_\alpha$ is defined by $s_\alpha(x) = x -\langle x, \alpha\che  \rangle \alpha $ for $x \in E,$
\item[] \rm{(iv)} $\langle \beta, \alpha \che\rangle \in \mathbb Z$ for all $\alpha, \beta \in R$, where $0\che = 0.$
\end{itemize}
\end{defi}
The \emph{rank} of a root system $(R, E)$ is the dimension of the underlying vector space $E$. Every finite root system (i.e., $R$ is finite in the category of sets) trivially has finite rank and whenever a root system has finite rank, then the root system itself is finite. Otherwise the intersection in (ii) with $F = E$ would be infinite, and $R$ would not be locally finite. \\
In most cases the underlying vector space $E$ is of little importance, so we will write $R$ instead of $(R, E).$ Instead of saying ``underlying vector space,'' $E$ is also called the \emph{ambient space} of the root system. \\ 
One can show that the element $\alpha \che$ in (iii) of the definition is uniquely determined. It  is called the \emph{coroot} of $\alpha$. The coroots form a root system in $E\che := \spa_{\mathbb K}\{\alpha\che : \alpha \in R\}.$ This root system is called the \emph{coroot system } of $R$. 
A subset $S \subset R$ is called a $\emph{subsystem}$ of $R$ if it contains $0$ and if it is closed under all reflections $s_\alpha$, $\alpha \in S$. Indeed, a subsystem $S \subset R$ is again a root system with ambient space $F = \spa_{\mathbb K}(S) \subset E.$ Therefore, as any well-behaved category, root systems are closed under taking ``subobjects''. Let $(R, E)$ and $(S, F)$ be root systems. It is somewhat tricky to find an adequate notion of a morphism  between  $(R,E)$ and $(S,F)$ (see \cite[p.23]{lfrs}), but for the purposes of our work it is enough to define an isomorphism of root systems as a vector space isomorphism $f: E \rightarrow F$ such that $f(R) = S.$ We denote by $\mathrm{Aut}(R)$ the group of isomorphisms from $R$ to $R$. It is not difficult to verify that all reflections $s_\alpha$ are automorphisms. \\
The \emph{weights} of $R$ are the set $\mathcal P(R) = \{p \in (E\,\che)^*: \langle R\che,\, p\rangle \subset \mathbb Z \}.$ The \emph{root lattice} $\mathcal Q(R)$ is the $\mathbb Z$-span of $R$ in $E.$ The sets $\mathcal P(R)$ and $\mathcal Q(R)$ carry both the structure of an abelian group where the group operation is the vector addition in $E$ resp. $(E\,\che)^*$ . Note that $\mathcal P(R)$ is not a subset of $E$. However, there is an embedding of $\mathcal Q(R)$ into $\mathcal P(R)$ which is indeed a homomorphism of abelian groups. This embedding is given by assigning to $w \in \mathcal Q(R)$ the linear form $\alpha\che \mapsto \langle w, \alpha \che \rangle.$
\begin{defi}
Let $(R, E)$ be a root system. The \emph{Weyl group} of $R$, denoted $\mathcal W(R),$ is the subgroup of $\mathrm{Aut}(E)$ generated by all $s_\alpha$, $\alpha \in R$.
\end{defi}
Any finite root system $(R, E)$ in the sense of \cite{BouLie2} is a root system as defined above (after including $0$) and the definition of the Weyl group made here is consistent with the usual notion. A special feature of a finite root system $R$ is that $R$ has a root basis $B$, i.e., a  linearly independent subset $B = \{\alpha_n, 1\leq n \leq r\}$ of $R$  such that any root $\alpha$ in $R$ can be written as
$$\alpha = \sum_{i =1}^r  k_i \alpha_i$$
where $k_i \in \mathbb Z$ and either all $k_i \leq 0$ or all $k_i \geq 0.$ This concept generalizes obviously to the setting of arbitrary root systems, although it is not true that every (infinite) root system admits a root basis.\\
 Two non-zero roots $\alpha$ and $\beta$ are called \emph{connected} if there are roots $\beta_1, \ldots, \beta_k$ such that 
$\langle \alpha, \beta_1 \che\rangle \neq 0$, $\langle \beta, \beta_k \che\rangle \neq 0$  and $\langle \beta_i, \beta_{i+1}\che \rangle \neq 0$ for all $1 \leq i \leq (k-1).$ It is easy to show that connectivity is an equivalence relation on $R$ and therefore we obtain a partition of $R$ in equivalence classes, the \emph{connected components}. A root system is called \emph{connected} or \emph{irreducible}, if all its non-zero roots are connected to each other. In \cite{lfrs} the authors prove for an irreducible root system $R$ that
$$ R \mbox{ has a root basis} \iff R \mbox{ is either finite or countably infinite}. $$
The elements of a root basis are also called \emph{simple roots}. \\
If $(R,E)$ is a root system, then there exists an inner product $( \cdot | \cdot)$ on $E$ such that $(x | s_\alpha y) = (s_\alpha x| y)$ for all $x,y \in E$ and $\alpha \in R$ (see \cite[Thm 4.2]{lfrs}). With respect to this inner product 
$$\langle \alpha, \beta \che \rangle = 2 \frac{(\alpha | \beta)}{(\alpha |\alpha)}. $$ The \emph{length} of a root is defined as $(\alpha | \alpha). $  Of course, an inner product with the property $\langle \alpha, \beta \che \rangle = 2 \frac{(\beta | \alpha)}{(\alpha| \alpha)}$ is only unique up to a non-zero constant.
It is a non-trivial  fact that the set $$ \{(\alpha | \alpha) : \alpha \in R\} $$ is bounded from above and below and that in fact we can choose the constant in such a way that 
$$ \{(\alpha | \alpha) : \alpha \in R\} \subset \{1, 2 ,3 ,4\} .$$
The integer $(\alpha | \alpha )$ is the \emph{length of the root} $\alpha.$ 
\begin{defi}A root system is called \emph{reduced} if $\mathbb Z \alpha \cap R \subset\{-\alpha, 0, \alpha\}$ for all $\alpha \in R.$
\end{defi}
If the root system is reduced and irreducible, then only two root lengths occur and we can distinguish between \emph{long and short roots}.  In this case, assuming that short roots have length $1,$ there are no roots of length $4$.  If there is only one root length, then all roots are called long and have according to convention length $2$. If all the roots have the same length, the root system is called \emph{simply laced} and if two root lengths occur in a root system, then it is called \emph{doubly laced.} \\

\begin{defi}Let $R$ be a root system with root basis $B= \{\alpha_i : i \in I\}.$ 
For $\lambda, \mu$ in $\mathcal Q(R),$
$\lambda = \sum_i k_i \alpha_i$ and $\mu = \sum_i m_i \alpha_i,$ we define
$$\lambda \prec \mu \iff \lambda \neq \mu \quad \mbox{and} \quad k_i \leq m_i \quad \mbox{for all}\;i \in I.$$
Then $\prec$ is a partial order on $\mathcal Q(R)$ which depends on the choice of $B$. 
\end{defi}
The subset $R \subset \mathcal Q(R)$ decomposes as
$$ R = (-R_+)\;  \dot{\cup}\; \{0\} \;\dot{\cup}\; R_{+} $$
where $R_+ = \{\lambda \in R: 0\prec \lambda\}.$\\
For $\lambda$ as above $$\supp(\lambda) := \{\alpha_i: k_i \neq 0\}\subset B.$$
\begin{lem} Let $(R, E)$ be a root system and $F \subset E$ a subspace. Then $(R \cap F, F)$ is a subsystem. \label{subsystem_lem}
\end{lem}
\begin{proof} Set $S = R \cap F.$ Since $ 0 \in F,$ we have $0 \in S.$  Pick $\alpha, \beta \in S.$ Then $s_{\alpha}(\beta) = \beta  - \langle \alpha, \beta \che \rangle \alpha \in \mathbb Z \alpha + \mathbb Z \beta \subset Q(S) \cap R \subset S.$ Thus $S$ is closed under reflections $s_\alpha,$ $\alpha \in S.$ According to the definition of a subsystem, $S$ is a  subsystem of $R.$
\end{proof}

\begin{lem}Let R be a root system and $\gamma \in \mathcal Q(R).$
There is a finite subsystem $S$ of $R$ such that $\gamma \in Q(S)$. If $R$ has a root basis $B,$ then one can choose a root basis $B'$ of $S$ such that $B'\subset B$. \end{lem} 
\begin{proof}
Since $\gamma \in \mathcal Q(R),$ there is a finite set $S' \subset R,$ such that $\gamma \in \spa_{\mathbb Z}S'.$ Let $F = \spa_{\mathbb K}(S')$ and  $S = R \cap F.$ By local finiteness, $S$ is finite and, by Lemma~\ref{subsystem_lem}, $S$ is a finite subsystem of $R.$\\  
If $R$ has  a  root basis $B,$ then we may assume that the set $S'$ defined above is contained in $B.$ Let $\alpha \in R \cap F.$ Then, since $S'$ is contained in the root basis $B$ and in addition spans $Q(S),$ every $\beta \in S$ can be uniquely written as 
$ \beta = \sum_{\alpha \in S' }  k_\alpha \alpha $ where either all $k_\alpha \geq 0$ or all $k_\alpha \leq 0.$
Hence $S'$ is  a root basis. 
\end{proof}

\paragraph{Root strings.} 
Let $\alpha, \beta \in R$. Then the subspace $\mathbb K \alpha + \mathbb K \beta$ in $E$  is at most $2$-dimensional and closed under reflections $s_{k_1\alpha + k_2 \beta}$, hence the intersection $S = R \cap (\mathbb K \alpha + \mathbb K \beta)$ is a finite root system of rank at most $2$. We define the \emph{$\beta$-string} through $\alpha$ as the set
$$S_{\alpha, \beta} = \{\alpha + k\beta: k \in \mathbb Z\} \cap R.$$
Let $d = \min\{k \in \mathbb Z : \alpha + k\beta \in R\}$ and $u = \max\{k \in \mathbb Z : \alpha + k\beta \in R\}$.
Note that, since $S$ is finite, $d$ and $u$ are also finite. From the theory of finite root systems it follows that for two roots $\alpha$ and $\beta$ the following hold true:
\begin{itemize}
\item[(i)] For all integers $k$, $d \leq k \leq u$ the element $\alpha + k\beta$ is in $R$ (root string property).
\item[(ii)] $u - d = \langle \beta, \alpha\che \rangle \leq 4 .$
\end{itemize}

The following lemma is well-known for finite root systems, see for example  \cite[Lemma 2.3]{vdK}
  \begin{lem} Let $R$ be an irreducible reduced root system and let $\alpha$ and $\beta$ be independent roots. Then there is $\gamma \in R$ such that $S = \spa\{\alpha, \beta , \gamma\} \cap R$ is an irreducible  reduced root system. If the rank of $R$ is greater  than or equal to $3,$ then $S$ can be chosen to be of rank $3$. 
  \label{irred_rank_3_lem}
  \end{lem}
\begin{proof}
If the rank of $R$ is  $2,$ then $S = \spa \{\alpha, \beta\} = R$ is an irreducible reduced root system. Assume that the rank of $R$ is $3$ or greater.\\ If $\alpha$ and $\beta$ are not connected,  then  there is a third root $\gamma \in R$ such that $\gamma$ connects $\alpha$ and $\beta$ (see \cite[p.26]{lfrs} or also \cite[p.14]{vdK}). Since $\alpha$ and $\beta$ are not connected, $S = (\mathbb K \alpha \oplus \mathbb K \beta \oplus \mathbb K \gamma) \cap R$ is  a root system of rank $ 3.$  Assume that $S = S_1 \oplus S_2$ is a decomposition such that $S_1$ and $S_2$ are not connected.  Without loss of generality $\alpha \in S_1,$ thus $\beta \in S_1$ and also $\gamma \in S_1.$ It follows that $S_2 \subset \{0\}$ thus $S$ is irreducible and reduced. \\
If $\alpha$ and $\beta$ are connected, then $S' = (\mathbb K \alpha + \mathbb K \beta) \cap R$ is an irreducible root system. Since $R$ is connected, there is a root $\gamma \in R$ such that $\langle \gamma , (S')\che \rangle \neq 0. $ Thus
 $S = (\mathbb K \alpha \oplus \mathbb K \beta \oplus \mathbb K \gamma) \cap R$ is  an irreducible and reduced  root system of rank $3.$ 
\end{proof}

\begin{lem}\label{hum_ex}Assume that $R$ is an irreducible root system with root basis  $B= \{\alpha_i: i \in J\}.$
Let $\lambda \in \mathcal Q(R)$ such that either $\lambda \prec 0$ or $0 \prec \lambda.$
Then exactly one of the following holds:
\begin{itemize}
\item[(i)] there is $\sigma \in \mathcal{W}(R)$ such that $\sigma \lambda = n \alpha$ for some $\alpha \in B$ and $n \in \mathbb{Z}\setminus \{0\},$
\item[(ii)]there is $\sigma \in \mathcal{W}(R
)$ such that $\sigma \lambda = \sum_{i = 1}^n k_i' \alpha_i$ and at least one $k_i' < 0$ and at least one $k_i' > 0.$
\end{itemize}
\end{lem}
\begin{proof}
Let $S$ be a finite irreducible subsystem of $R$ which contains $\lambda$ and has root basis $B' \subset B$ (see Lemma~\ref{subsystem_lem}). The Weyl group $\mathcal W (S)$ is a subgroup of $\mathcal W (R)$, see \cite[Thm 5.7]{lfrs}. Lemma 10 in \cite[VIII.4.3]{BouLie3} asserts the existence of an element $\sigma \in \mathcal W (S)$ with the desired property. 
\end{proof}

\subsection{Classification}

For finite root systems we will use the standard notation as it can be found for example in \cite{BouLie2} or \cite{Hum}.
\paragraph{Examples of infinite root systems.} 
Let $I$ be a non-empty set and $X = \mathbb K^{(I)} = \bigoplus_{i \in I}\mathbb K \varepsilon_i$ the free $\mathbb K$-vector space on the set $I$. 
\begin{expl}Define a linear functional $t : X \rightarrow \mathbb K$ by $t(\epsilon_i) = 1$ for all $i \in I$ and let $\dot X = \ker t$. Then
$\dot A_I = \{\varepsilon_i - \varepsilon_j : i, j \in I\} $
is an  irreducible reduced root system in $\dot X$. If $\card I < \infty$ then $\dot A_I \cong A_{\card I -1}$. \label{A_I}
\label{type_A_expl}
\end{expl}
\begin{expl}$B_I = \{\pm (\varepsilon_i \pm \varepsilon_i): i\neq j \in I \} \cup \{\varepsilon_i: i \in I\}\cup \{0\}$ is an irreducible reduced root system in $X$. If $\card I < \infty$ then $B_I \cong B_{\card I}$. \label{B_I}
\end{expl}
\begin{expl}$C_I = \{\pm (\varepsilon_i \pm \varepsilon_i): i\neq j \in I \} \cup \{2\varepsilon_i: i \in I\}\cup \{0\}$ is an irreducible reduced root system in $X$. If $\card I < \infty$ then $C_I \cong C_{\card I}.$ In particular if $\card I = 1$ then $C_I = A_1.$ \label{C_I}
\end{expl}
\begin{expl}$D_I = \{\pm (\varepsilon_i \pm \varepsilon_i): i\neq j \in I \} \cup \{0\}$ is a reduced root system in $X$ if $\card I \neq 1$. If $\card I > 2$ then $D_I$ is irreducible.  If $\card I < \infty$ then $D_I \cong D_{\card I}$. \label{D_I}
\end{expl}
\begin{expl}$BC_I = B_I \cup C_I$ is an irreducible root system in $X$ which is non reduced.  If $\card I < \infty$ then $BC_I \cong BC_{\card I}$. \label{BC_I} \\
\end{expl}
The following result classifies irreducible root systems up to isomorphism.
\begin{theo}[\cite{lfrs}]
Let $R$ be an irreducible root system. If $R$ is finite, then $R$ is isomorphic to one of $A_n$, $B_n$, $C_n$, $BC_n$, $D_n$, $E_6$, $E_7$, $E_8$, $F_4$ or $G_2$. \\
If $R$ is infinite, then $R$ is isomorphic to one of $\dot{A}_I$, $B_I$, $C_I$, $BC_I$ or $D_I$ for an infinite index set $I$.\\ Moreover $R$ is reduced if and only if $R \not \cong BC_I$. 
\label{loc_fin_irred_class_thm}
\end{theo}
We are accustomed to refer to finite root systems as being of type $X$ where $X$ is one of $A, B,BC,  C,D,E,F$ or $G.$ We use the same language for a root system of infinite rank. By \cite[p.65]{lfrs}, an infinite irreducible root system $R = X_I = \
\lim_{\rightarrow} X_\lambda$ is the direct limit of finite root systems $X_\lambda$ of type $X.$ It follows that $R$ has finite subsystems of type $X$ of arbitrarily high rank.

\section{Degenerate sums}
Later in this section, we will define the main objects of interest for the rest of the thesis, namely root-graded Lie algebras, Definition~\ref{r_delta_graded_def}. It is convenient to look first at root systems as combinatorial or geometric objects. Similar to \cite{vdK} which deals with a certain special case of root-graded Lie algebras defined over the integers, we can show that a root system $R$ contains a huge amount of information about $R$-graded Lie algebras and their behaviour under central extensions.\\
Throughout this section $R$ is a reduced root system, $\mathcal Q(R)$ the root lattice and $\mathcal P(R)$ the weight lattice. 
 We will include here some more details for  the proofs in \cite{vdK}, but  essentially follow the very same strategy and make use of (mostly combinatorial) properties of the root system. Before we give the definition of a degenerate sum we will prove the following lemma: 

\begin{lem} Let $\gamma \in \mathcal P(R), \gamma \neq 0$. Then there is a uniquely determined maximal positive integer $n_\gamma$ such that 
$\langle \gamma, R \che\rangle \subset n_\gamma \mathbb Z.$ 
\end{lem}
\begin{proof} Let $n_\gamma$ be the greatest common divisor of all $n \in \{\langle \gamma, \alpha\che \rangle : \alpha \in R \}$. Then $\langle \gamma, R  \che \rangle \subset n_\gamma \mathbb Z$ and $n_\gamma$ is the greatest integer with this property.
\end{proof}
\begin{defi}The integer $n_\gamma$ is called the \emph{divisor} of the weight $\gamma.$
\end{defi}
\begin{defi} 
 A \emph{degenerate sum with respect to $n$} is an element $\gamma $ of the root lattice $\mathcal Q (R)$ which is a sum of two linearly independent roots and has divisor $n > 1$. The set of all degenerate sums of divisor $n$ of a root system $R$ is denoted by $\mathbb{DS}(R, n)$ or, if $n$ is understood from the context, $\mathbb{DS}(R)$. \\
 A \emph{degenerate pair} (of divisor $n$) is a pair of two non-zero linearly independent roots whose sum is a degenerate sum (of divisor $n$).
 The set of degenerate pairs (of divisor $n$) of a root system $R$ will be denoted $\mathbb{DP}(R, n)$ or, if $n$ is understood, $\mathbb{DP}(R)$. 
 When we refer to a degenerate sum $\gamma$ we mean an element $ \gamma = \alpha + \beta \in \mathbb{DS}(R)$, $(\alpha, \beta) \in \mathbb{DP}(R)$. Sometimes also the terminology ``$\gamma = \alpha + \beta$ is a degenerate sum'' is used which means that $(\alpha, \beta) \in \mathbb{DP}(R).$ 
 \label{deg_sum_defi}
\end{defi} 
\begin{lem}[\cite{vdK}]Let $\gamma = \alpha + \beta$ be a degenerate sum of divisor $n_\gamma$ and assume  $(\alpha | \alpha) \leq (\beta | \beta).$ Then
\begin{itemize}
 \item[\rm(i)]$n_\gamma$ is either $2$ or $3$, and $n_{\gamma} = \langle \gamma, \beta \che\rangle.$
 \item[\rm(ii)] $\gamma \in n_\gamma \mathcal P(R) \cap \mathcal Q(R).$
  \item[\rm(iii)] If $n_\gamma = 2$, then $\langle \alpha , \beta\che \rangle = 0.$
 \item[\rm(iv)] Let $S \subset R$ be a subsystem containing $\alpha$ and $\beta$. Then $(\alpha, \beta) \in \mathbb{DP}(S, n_\gamma).$
 \end{itemize}\label{vdK_2.6}
 \end{lem}
 \begin{proof}(i) Let $\alpha$ and $\beta$ be two linearly independent roots and let $(\alpha | \alpha) \leq (\beta | \beta)$. Then 
 $-1 \leq  \langle \alpha, \beta\che \rangle \leq 1$ and therefore $1 \leq  \langle \alpha + \beta, \beta\che \rangle \leq 3.$ Hence $n_\gamma \in \{1 ,2 ,3\}.$ As $n_\gamma | \langle \gamma , \beta \che \rangle$  we have the equality $n_\gamma = \langle \gamma, \beta \che \rangle.$ Clearly, $\gamma \in \mathcal Q(R)$ and $\frac{1}{n_\gamma} \gamma \in \mathcal P(R)$. This shows (ii).\\
 For  (iii) assume $n_\gamma = 2,$ then $\langle \alpha + \beta, \beta\che \rangle \in 2\mathbb Z$, hence $\langle \alpha + \beta, \beta\che \rangle = 2 = \langle \beta, \beta\che\rangle + \langle \alpha,  \beta\che \rangle = 2 + \langle \alpha, \beta \che \rangle.$\\
 (iv) A degenerate sum in $\mathcal{Q}(R)$ is trivially also degenerate in $\mathcal Q(S)$. It follows from (i) that $\langle \alpha + \beta, \beta\che \rangle = n_\gamma$. Since $\{\alpha, \beta\} \subset S$ this proves (iv).
 \end{proof}

 \begin{lem}[\cite{vdK}]
\label{root_deg_sums_intersection_lem} Let $R$ be an irreducible reduced root system of rank at least $2$. 
\begin{itemize} 
\item[\rm(i)] If $R \neq C_I$, then $R \cap \mathbb{DS}(R) = \emptyset. $
\item[\rm(ii)] If $R = C_I$, then $R \cap \mathbb{DS}(R) = \{\pm 2 \epsilon_i :i \in I  \}$
\end{itemize}
\end{lem}
 \begin{proof}
 Let $\gamma \in R \cap \mathbb{DS}.$ Since $ \langle \gamma, \gamma \che \rangle  = 2$ we have $n_\gamma = 2$ by Lemma~\ref{vdK_2.6}(i). Because $\langle \alpha, \beta \che \rangle  \subset \{0, \pm 1, \pm 2 ,\pm 3\}$ for two roots $\alpha, \beta$ in a root system $R$, (see \cite[VI,  1.3]{BouLie2}), it follows that $\langle \gamma, \alpha \che \rangle  \subset \{0,  \pm 2\}$ for all $\alpha \in R.$ Let $S \subset R$ be an irreducible subsystem of rank $2$ containing  $\gamma.$ Thus $S \cong A_2, B_2 = C_2 $ or $G_2.$ The condition $\langle \gamma, S \che \rangle \subset 2 \mathbb Z$ implies $S = B_2 = C_2$ and that $\gamma$ is a long root.  In particular, $R$ is not simply laced. If the rank of $R$ is $2,$ then $R = C_2$.   If the rank of $R$ is $3$ or greater, we can choose an irreducible subsystem $T \subset R$ of rank $3$  containing $\gamma.$ The condition $\langle \gamma, T \che \rangle \subset 2 \mathbb Z$ rules out $ T \cong B_3,$ whence $T \cong C_3.$ It now follows that $R$ does not contain a subsystem of type $F_4.$ Hence $R \cong C_I$ for some set $I.$  \\
 It remains to verify that $\pm 2 \epsilon_i$, $i \in I,$ is a degenerate sum. First $\pm 2 \epsilon_i = \pm (\epsilon_i - \epsilon_j + \epsilon_i + \epsilon_j)$ is a representation of the root in question as  a sum of two linearly independent roots. Moreover, for $j, k \in I$ the integer $\pm \langle 2 \epsilon_i, \pm (\epsilon_j \pm \epsilon_k )\che \rangle = \pm 2(\delta_{ij} \pm \delta_{ik})$ is divisible by $2$ which proves that $2 \epsilon_i \in 2 \mathcal P(R).$ It follows that $2 \epsilon_i$ is degenerate sum and at the same time a root. 
 \end{proof}

\begin{lem}
If $R = \bigoplus_{i \in I} R_i$ is a reducible root system with irreducible reduced components $R_i,$ then $\mathbb{DS}(R) = \bigcup\mathbb {DS}( R_i) $ unless 
all $R_i$ are of type $C$. In the latter case $\mathbb{DS}(R) = \bigcup \mathbb{DS}( R_i) \cup \bigcup_{i\neq j} (({\mathbb{DS}}( R_i) \cap R_i) + ({\mathbb{DS}}( R_j)) \cap R_j).$
\end{lem}
\begin{proof}
For two indices $i\neq j$ we always have $\langle \mathcal Q(R_i), R_j\che \rangle  = \{0\}$, so that $\bigcup \mathbb{DS}(R_i) \subset \mathbb{DS}(R). $
Let $\alpha+ \beta = \gamma$ be a degenerate sum in $R$ which is not in $\bigcup \mathbb{DS}(R_i)$ . It follows that $\alpha \in R_i$ and $\beta \in R_j$, for $i\neq j$. By definition for every $\delta \in R$ , $\langle \alpha + \beta, \delta \che \rangle \in n_\gamma\mathbb Z $. Let $\delta \in R_i$,  then $\langle \alpha + \beta, \delta \che \rangle  = \langle \alpha , \delta \che \rangle $, thus $\alpha$ must be a degenerate sum in $R_i$ and likewise $\beta$ is degenerate in $R_j$. But we have seen that the only root systems which have roots that are degenerate sums are those of type $C_I$. Let $I$ be the index set of $R_i = C_I$ and $J$ the index set of $R_j  = C_J$. Since the choice of indices $i$ and $j$ was arbitrary, the degenerate sums in $R$ are contained in $\mathbb{DS}(R) = \bigcup \mathbb{DS}( R_i) \cup \bigcup_{i\neq j} (({\mathbb{DS}}( R_i) \cap R_i) + ({\mathbb{DS}}( R_j)) \cap R_j).$ The other inclusion can easily be verified. 
\end{proof}
\subsection{Classification of degenerate sums}
\label{Class_deg_Sums}
For the convenience of the reader this subsection contains a full proof of the classification of degenerate sums in finite root systems. 
Throughout $R$ is assumed to be an irreducible finite reduced root system and $(\cdot | \cdot)$ denotes the inner product introduced in \cite[4.6]{lfrs}. This inner product is normalized in the sense that for any  short root $(\alpha, \alpha) = 1.$ Also we abbreviate $ \| \alpha \| = (\alpha, \alpha)$ and refer to this number as the length of $\alpha.$ Often the length is defined as the square root of $(\alpha, \alpha)$,  but it is easier for us to work with this definition, for example because of the well-known fact (see for instance \cite[Chap.VI, \S 1] {BouLie2}) that the length function normalized in this way can only assume the values $1, 2$ or $3$ if evaluated on a root. 

  
  \subsubsection*{Divisor $3$}
  
  \begin{lem}Let $R$ be root system which has degenerate sums of divisor $3$. 
  Then  $\alpha + \beta$ is a degenerate sum of divisor $3$ if and only if the following two conditions hold: 
  \begin{itemize} 
\item[\rm{(i)}]$R$ is of type $G_2$ or $A_2$, and 
 \item[\rm{(ii)}] the roots $\alpha$ and $\beta$ are both long and $\langle \alpha, \beta \che \rangle = 1.$
  \end{itemize}
  \end{lem}
  \begin{proof} 
  Let the  integer $n$ denote the connection index of $R$, that is the order of the group $\mathcal P(R)/\mathcal Q(R).$\\
  First we show that (i) and (ii) are necessary. Let $\gamma = \alpha + \beta$ be a degenerate sum of divisor $3$.
 Then $\gamma \in 3\mathcal{P}(R) \cap \mathcal Q(R)$ and thus $n^2(\gamma | \gamma) \in  9 \mathbb Z$ and $(\gamma | \gamma) \in \mathbb Z$.  On the other hand
  \begin{eqnarray*}
  (\gamma | \gamma) &=& (\alpha | \alpha) + 2(\alpha | \beta) + (\beta| \beta)\\
  &=& (\alpha | \alpha) + \langle \alpha, \beta\che \rangle (\beta | \beta) + (\beta | \beta).
  \end{eqnarray*}
  With the chosen normalization a root can have a length of at most $3$ and so 
  \begin{eqnarray*}
     (\alpha | \alpha) + \langle \alpha, \beta\che \rangle (\beta | \beta) + (\beta | \beta) \leq 6 + 3\langle \alpha, \beta\che \rangle .
  \end{eqnarray*}
  Moreover, without loss of generality $\|\alpha \| \leq \|\beta \|$ and thus $\langle \alpha, \beta\che  \rangle\leq 1$. Hence the following inequality holds true: 
  $$0 \leq (\gamma, \gamma) \leq 9, $$
  and equality holds if and only if $\alpha$ and $\beta$ both have length $3$ and  $\langle \alpha, \beta\che \rangle = 1.$ This can only occur if $R = G_2$. Otherwise there are no roots of length $3$. For the remaining consideration we may therefore assume that the inequality is strict.  In this case the largest power of $3$ that could divide $(\gamma, \gamma)$ is $3^1.$ Since at the same time $9 | n^2(\gamma, \gamma)$, by uniqueness of prime factor decomposition, it follows that $3 |n^2$ and $3|n$. There are only two types of roots systems such that the index of the root lattice in the weight lattice is divisible by $3$ and these are $A_{3m - 1}, m \geq 1$ and $E_6$ (see for instance \cite[Planches I - IX]{BouLie2}). Both root systems are simply laced. Note that $A_{3m-1}$ has never rank $3$, so it is already clear that root systems of rank $3$ do not have any degenerate sums with respect to $3$. If the rank of $R$ is greater  than or equal to $3$,  we can pass to an irreducible subsystem $S$ of rank $3$ which contains $\alpha$ and $\beta$, (see Lemma~\ref{irred_rank_3_lem}). But then $\gamma$ must be degenerate in $S$ as well, contradicting the fact that there are not any root systems of rank $3$ which have degenerate sums with respect to $3$. It follows that the rank of $R$ is $2$ which leaves $R = A_2$ ($m = 1$) as the single candidate.
  The integer $\langle \alpha + \beta, \beta \che\rangle$ has to be divisible by $3$,  hence  $3 = \langle \alpha + \beta, \beta \che\rangle =  \langle \alpha, \beta \che\rangle + \langle \beta, \beta \che\rangle$ and $\langle \alpha, \beta \che\rangle = 1.$
  This shows that the conditions (i) and (ii) are necessary. \\
  Next assume that $R= A_2$ or $G_2$ and that $\alpha$ and $\beta$ fulfill (ii). By (ii) the angle between $\alpha$ and $\beta$ is $\pi/3. $
In $G_2$ two short roots enclosing an angle of $\pi/3$ will sum up to a root and we have seen before in Lemma~\ref{root_deg_sums_intersection_lem} that for $R = G_2$ or $R = A_2$ the intersection $R \cap \mathbb{DS}(R)$ is empty. Also, a long and a short root never form an angle of $\pi/3$ in $G_2$.  Therefore, $\alpha$ and $\beta$ are long. The long roots in $G_2$ form a root system of type $A_2$ and for all long roots $\alpha ,\beta$ and every short root $\gamma$ the integer $\langle \alpha + \beta, \gamma \che \rangle$ is divisible by $3$.  In order to find the elements of $\mathcal Q(G_2) \cap 3 \mathcal P(G_2)$ which are degenerate sums, it thus suffices to classify degenerate sums in $A_2$. The degenerate sums in $G_2$ can then be obtained by rotating and scaling. \\ So let $R = A_2$ and $\alpha = \varepsilon_i - \varepsilon_j \in R$. The two roots that form an angle of $\pi/3$ with $\alpha$ are $\beta = \varepsilon_i - \varepsilon_k$ and $\beta ' = \varepsilon_k - \varepsilon_j,$ $\{i,j,k\} = \{1, 2 ,3\}.$ The sum of $\alpha$ with any of these is of the form $\pm (2 \epsilon_i - \epsilon_j -\epsilon_k).$ It remains to show that $\pm (2 \epsilon_i - \epsilon_j -\epsilon_k) \in 3\mathcal P(R)$. Pick a root $\epsilon_n - \epsilon_m$, $1 \leq n \neq m \leq 3$. We consider the case where $n \neq i, m \neq i$ and here it may even be assumed that $j = n$ and $m = k$
 $$ \langle 2 \epsilon_i - \epsilon_j -\epsilon_k, (\epsilon_n - \epsilon_m)\che \rangle = 0.$$
 In the other case (without loss of generality $n = i$ and $ m = j$) one obtains
 $$ \langle 2 \epsilon_i - \epsilon_j -\epsilon_k, (\epsilon_n - \epsilon_m)\che \rangle = 2 + 1 = 3.$$
 In both cases the result is divisible by $3$ which shows that $2 \epsilon_i - \epsilon_j -\epsilon_k$ is a degenerate sum. To pass from the root system $A_2$ to $G_2$ it suffices to rotate the degenerate sums about the origin by $\pi/6$ and scale each of them by $\sqrt{3}.$ 
     \end{proof}
    For convenient reference, we summarize the classification in the divisor $3$ case in the table below.
  \begin{equation}\label{tab:DegSums3}
\begin{tabular}[h]{|l|l|l|}
\hline
Type & $n_\gamma$ & Degenerate sums \\
\hline
$A_2$ & $3$ & $\pm\{ \varepsilon_1 - 2 \varepsilon_2 + \varepsilon_3, \varepsilon_1 + \varepsilon_2 - 2\varepsilon_3, 2\varepsilon_1 - \varepsilon_2 - \varepsilon_3 \} $\\
\hline
$G_2$ & $3$ & $\{3(\varepsilon_i - \varepsilon_j) : 1 \leq i\neq j \leq 3 )\}$\\
\hline
\end{tabular}
\end{equation}
    \subsubsection*{Divisor $2$}
  \begin{lem}[Proof of Proposition 2.12 in \cite{vdK}]Let $R$ be a  finite irreducible reduced root system and $\gamma \in 2 \mathcal{P}(R) \cap \mathcal{Q}(R)$, $\gamma \neq 0$. If there is a long root $\alpha'$ such that $(\gamma| \gamma) \leq 2(\alpha', \alpha')$, then  $\gamma$ is in the  $\mathcal W(R)$-orbit of $2 \omega$ for a fundamental weight $\omega.$ 
  \label{deg_sum_as_double_fund_weight_lem}
  \end{lem}
  \begin{proof} Fix a root basis $B = \{\alpha_j : 1 \leq j \leq \mathrm{rank} R \}$ of $R.$
 First assume that $R$ is not of type $C_\ell$, $\ell \geq 1.$ Let $ 0 \neq\gamma \in 2\mathcal{P}(R) \cap \mathcal Q(R)$ be  arbitrary, but fixed. Since for $R \neq C_I,$ the divisor of a root is $1,$ and it follows that $\gamma \notin R.$   Without loss of generality we may assume that $\gamma$ is a dominant weight since the set $ 2 \mathcal{P}(R) \cap \mathcal{Q}(R)$ is invariant under the Weyl group and every weight is conjugate to a dominant one. Therefore it is possible to express $\gamma$ as $2 \sum n_i \omega_i = \sum m_j \alpha_j$ where the $\omega_i$ are the fundamental weights with respect to $B,$ and the coefficients $m_i, n_i$ are non-negative integers.\\
 Claim: \begin{itemize}
 \item[\rm{(i)}] $m_j > 0$ for all $j$, 
 \item[\rm{(ii)}] $0 \leq n_i \leq m_i - 1$ for all $i$.
 \end{itemize}
 Since $\gamma \neq 0$ there is at least one $m_i$ which is non-zero, hence greater than or equal to $1$. Let $k$ be adjacent to $i$ in the Dynkin diagram of $R$ and consider the set $K = \{j : j \mbox{ is adjacent to }k\}.$ Clearly, $ i \in K.$ The weight $\gamma$ is dominant, hence 
 $$0 \leq \langle \gamma, \alpha_k \che \rangle  = \sum_{j \in K}m_j \langle \alpha_j, \alpha_k \che \rangle + 2 m_k.$$
 In every Dynkin diagram, if $j$ and $k$ are adjacent, then $\langle \alpha_j, \alpha_k \che \rangle \leq -1 .$ Thus  
 $$ 0 \leq - \sum_{j \in K}m_j + 2m_k.$$ By assumption, $- \sum_{j \in K}m_j \leq -m_i \leq -1,$ whence
 $2m_k  \geq 1, $ and thus $m_k > 0.$ The Dynkin diagram is connected, so  we can apply induction and obtain $m_j > 0$ for all $j.$ \\
 Next, since $n_i$ is the coefficient of the $i$th fundamental weight, 
\begin{equation*}
2n_i = \langle \gamma, \alpha_i \che\rangle = 2 m_i + \sum_{j \neq i}m_j \langle \alpha_j, \alpha_i \che \rangle
\end{equation*}
We have shown that all $m_j$ are greater than zero. Moreover, $\langle \alpha_j, \alpha_i \che \rangle < 0$ if $i$ and $j$ are neighbours in the Dynkin diagram,  $\langle \alpha_j, \alpha_i \che \rangle = 0$ if $i$ and $j$ distinct and not adjacent, and $\langle \alpha_i, \alpha_i \che \rangle = 2$ for all $i,$
Hence $\sum_{j \neq i}m_j \langle \alpha_j, \alpha_i \che \rangle < 0$
and $2n_i < 2 m_i. $ Since $n_i$ and $m_i$  are non-negative integers, it follows that $1 \leq n_i + 1\leq m_i.$
Let $\alpha'$ be a long root. Then, using that $(\omega_j, \alpha_i) = 0$ for $j \neq i$
\begin{eqnarray*} 2(\alpha'| \alpha') \geq (\gamma | \gamma) &=& \sum m_i(\gamma | \alpha_i) = 2\sum \sum n_j m_i (\omega_j | \alpha_i) \\ &=&  \sum m_i n_i (\alpha_i | \alpha_i) \geq \sum (n_i + 1)n_i (\alpha_i | \alpha_i)\\ &=& \sum n_i^2 (\alpha_i | \alpha_i) + \sum n_i (\alpha_i| \alpha_i).\end{eqnarray*}
Note that there is at least one $i$ such that $n_i \geq 1.$ \\
 We first consider  the case that $R$ is simply laced. Then we can divide the inequality by $(\alpha' |\alpha')$ which yields
 $$  0 \neq \sum n_i^2 + n_i \leq 2 .$$
 This inequality holds if and only if $n_i = 1$ for some $i$ and $n_j = 0$ for $j \neq i.$ It follows that $\gamma = 2 \omega_i$ for a fundamental weight $\omega_i.$\\ 
If $R$ is doubly laced, then a long root has length either $2$ or $3$. 
Assume that the longest root has length $2$, i.e., $R = B_\ell$, $\ell > 2$ or $F_4$ (type $C_\ell$ was excluded earlier).
There are only two non-trivial cases where this inequality holds. Either $n_i = 1$ for some $i$ and $n_j = 0$ for $i \neq j$ or there are distinct short simple roots $\alpha_i$ and $\alpha_j$ such that $n_i = n_j = 1$ and $n_k = 0$ for $k, i,j$ pairwise distinct.   
The second case can only occur when $R = F_4.$ However, in this case the two short roots are adjacent in the Dynkin diagram, leading to $\gamma = \alpha_i + \alpha_j \in R.$ We have seen above that the intersection of $2 \mathcal P(F_4)$ with the root system $F_4$ is trivial, whence a contradiction to $\gamma \in 2 \mathcal P(F_4).$ Therefore the only possibility is that $\gamma = 2 \omega_i$ for a fundamental weight $\omega_i.$

Let now the longest root be of length $3$, i.e., $R= G_2$. The condition on the coefficients $n_i$, $i = 1, 2$ yields
$$   \sum_{i = 1}^2 (n_i + 1)n_i (\alpha_i, \alpha_i) \leq 6.$$

Let $\alpha_1$ be the short simple root in $B$ and $\alpha_2$ the long simple root in $B.$ 
If $n_1 = 1$ and $n_2 = 0$ or $n_2 = 1$ and $n_1 = 0,$ then the inequality holds and $\gamma = 2 \omega_i$, for a fundamental weight $\omega.$ Also, since the inequality is strict in the case $n_2 = 1$ and $n_1 = 0,$ we may assume now $n_2 = 0.$ The only case left is then $n_1 = 2$ and $n_2 = 0.$ Since $\omega_1$ is in fact equal to $2\alpha_1 + \alpha_2$ this gives $m_1 = 2,$ contradicting $n_1 < m_1.$ \\
Lastly, in case $R = C_\ell$, then $\{\epsilon_i : 1 \leq i \leq \ell\}$ is an orthogonal basis for the weight lattice $\mathcal{P}(R)$.
  We have $\gamma \in 2 \mathcal P(R),$ if and only if $\gamma = 2 \sum \lambda_i \epsilon_i$  for integers $\lambda_i.$ The length of such an element is $ 2\sum m_i^2 $ and we need to find those that have length at most $4.$ This is only possible if at most two coefficient are non-zero and then these coefficients have to equal $\pm 1.$ It follows that 
 $\gamma $ is conjugate to $2 \epsilon_1$ or $2(\epsilon_1 + \epsilon_2)$ and which both are equal to twice a fundamental weight. 
  \end{proof}
  \begin{prop}[Propostion 2.12 \cite{vdK}] \label{deg_sums_char_by_length_prop} 
  If $\gamma$ is a degenerate sum of divisor $2$, then there is a long root $\alpha'$ such that $(\gamma | \gamma) \leq 2(\alpha'| \alpha').$ \label{semi_char_prop}
  \end{prop}
 \begin{proof} If $\gamma$ is degenerate sum with respect to $2,$ then $\gamma = \alpha + \beta$ for $(\alpha | \beta) = 0$ by Lemma~\ref{vdK_2.6}(iii). 
Thus
 $$(\gamma | \gamma) = (\alpha| \alpha) + (\beta | \beta) \leq 2 \max\{(\alpha | \alpha), (\beta | \beta)\} \leq 2\max_{\alpha ' \in R}\{(\alpha' | \alpha')\}. $$
 
\end{proof}
 Lemma~\ref{deg_sum_as_double_fund_weight_lem} allows us to use the following strategy to find degenerate sums with respect to $2$: 
\begin{enumerate}
\item Compute the intersection of the doubled weight lattice with the root lattice.
\item Find an element $\gamma$ in this intersection of length shorter than  or equal to twice the length of  a long root.  
\item If such an element exists, there is a fundamental weight $\omega$ such that $\gamma$ is conjugate to $2\omega.$
\item Write $2 \omega = \alpha + \beta$ for two independent roots. 
\item Compute the Weyl group orbit of all $(\alpha, \beta)$ in step 4 to obtain the degenerate pairs in the orbit of $\gamma$.
\item Go back to step 2 and if there is  $\gamma$ of length shorter than or equal to twice the length of a long root, which does not lie in the orbits of the already found degenerate sums, repeat steps 3 through 6. 
\end{enumerate}
This strategy was as well suggested in \cite{vdK}. 
\subsection{Classification for the finite types and divisor $2$}

The goal is now to classify degenerate sum of index $2$ in finite root systems. Note that for convenience, in step 2 of the algorithm above we might choose another normalization of the inner product, most frequently this will be done in such a way that the length and the norm in the underlying Euclidean vector space coincide.  We will do the classification case-by-case. For each case, the degenerate sums will be listed in a table for quick reference later on. 
\subsubsection*{Type $A_\ell$, $\ell > 1$}
\begin{itemize}
\item[] $R = \{\epsilon_i - \epsilon_j: 1 \leq i \neq j \leq \ell + 1\},$
\item[] $\mathcal{Q}(R) \subset \bigoplus_{i = 1}^{\ell + 1} \mathbb Z \epsilon_i, $ consisting of those vectors whose coordinates sum up to zero, 
\item[] $\mathcal{P}(R) = \mathcal{Q}(R) + \mathbb Z ( \epsilon_1 + \frac{1}{\ell + 1}(\epsilon_1 + \ldots + \epsilon_{\ell + 1})) ,$
\end{itemize}
 The $i$th fundamental weight is $\sum_{j = 1}^i (1- \frac{i}{\ell + 1})\epsilon_j + \sum_{j > i}(- \frac{i}{\ell + 1}\epsilon_j).$ In particular for each $j$ the coefficient of $\epsilon_j$ in $\omega_i$ is non-zero. The length of a (long) root is $2.$\\
An element $\gamma$ in the root lattice, which is not a root, can have length less than or equal to $4$ if and only if one of the following holds: Either $\gamma = 2(\epsilon_i - \epsilon_j)$ or $\gamma = \epsilon_i- \epsilon_j + \epsilon_k - \epsilon_l $ for $i,j,k,l$ pairwise distinct. The  element $\omega = \epsilon_i - \epsilon_j$ is not conjugate to  a fundamental weight unless $\ell = 1$ and this contradicts the assumption $\ell > 1.$ The element $\omega = 1/2(\epsilon_i- \epsilon_j + \epsilon_k - \epsilon_l )$ is conjugate to a fundamental weight if and only if $\ell = 3.$
Then $\frac{\gamma}{2}$ is conjugate to  $\omega = 1/2 (\epsilon_1 + \epsilon_2 - \epsilon_3 - \epsilon_4)$, and $\epsilon_1 + \epsilon_2 - \epsilon_3 - \epsilon_4$ is also a sum of two linearly independent roots, hence a degenerate sum. The orbit of $2\omega$ gives all degenerate sums. 
 
\begin{equation}
\begin{tabular}[h]{|l|l|l|}
\hline
Type & $n_\gamma$ & Degenerate sums \\
\hline
$A_3$ & $2$ & $ \pm \{\varepsilon_1 - \varepsilon_2 + \varepsilon_3 - \varepsilon_4, \varepsilon_1 - \varepsilon_2 - \varepsilon_3 + \varepsilon_4, \varepsilon_1 + \varepsilon_2 - \varepsilon_3 - \varepsilon_4 \}$\\
\hline
\end{tabular}
	\label{tab:DegSumsA}
\end{equation}
\subsubsection{Type $B_\ell$, $\ell > 2$}
\begin{itemize}
\item[] $R = \{\pm \epsilon_i , \pm \epsilon_i \pm \epsilon_j: 1 \leq i \neq j \leq \ell\},$
\item[] $\mathcal{Q}(R) = \bigoplus_{i = 1}^\ell \mathbb Z \epsilon_i, $
\item[] $\mathcal{P}(R) = \mathcal{Q}(R) + \mathbb Z 1/2(\epsilon_1 + \ldots + \epsilon_\ell),$
\item[] $2 \mathcal{P}(R) \cap \mathcal{Q}(R) = \{\sum m_i \epsilon_i: m_i - m_j \in 2 \mathbb Z, \;\forall \, 1 \leq i, j\leq \ell \}.$
\end{itemize}
The description of $ \mathcal P(R)$ and $\mathcal Q(R)$ can for instance be found in \cite{BouLie2}. The elements in $2 \mathcal P(R)$ are precisely those of the form $\sum_{i = 1}^\ell (2 m_i + b)\epsilon_i$ where $m_i \in \mathbb Z$ and $b = 0$ or $b = 1.$ Equivalently, the difference of two coefficients is divisible by $2.$
The length of a long root is $2$. A degenerate sum  $\gamma$ must satisfy at least the following:  $\gamma = \sum m_i \epsilon_i$ such that $ 1\leq \sum m_i^2 = (\gamma, \gamma) \leq 4$ and the coefficients are pairwise congruent modulo $2$. Thus the coefficients of a degenerate sum are  $0,\pm 1$ or $\pm 2$. If one $m_i = \pm 2$, then all others are zero. Hence $\gamma$ is conjugate to $2\epsilon_1$ and since $2\epsilon_1 = \epsilon_1 - \epsilon_ i + \epsilon_ i +\epsilon_1$, this is indeed a degenerate sum. \\
If one of the $m_i = \pm 1$,  all the others have to equal $\pm 1$ as well since otherwise they would not be congruent modulo $2$. All such elements are conjugate under the Weyl group and we may assume that $m_i = 1$ for all $1 \leq  i \leq \ell$. The length of such an element is $\ell$ and hence it can only be a degenerate sum if $3 \leq \ell \leq 4$ and in this case $  1/2 \sum_{i = 1}^\ell \epsilon_i$ is a fundamental weight and its double $\sum_{i = 1}^\ell \epsilon_i$ is the sum of two non-proportional roots. 
 \begin{equation}
\begin{tabular}[h]{|l|l|l|}
\hline
Type & $n_\gamma$ & Degenerate sums \\
\hline
$B_3$ & $2$ & $\pm \{2 \varepsilon_i : 1 \leq i \leq 3\} \cup \{\pm \epsilon_1 \pm \epsilon_2 \pm \epsilon_3\}$\\
\hline
$B_4$ & $2$ & $\pm \{2 \varepsilon_i : 1 \leq i \leq 4\} \cup \{\pm \epsilon_1 \pm \epsilon_2 \pm \epsilon_3 \pm \epsilon_4\}$\\
\hline
$B_I, \mathrm{card}(I)> 4$ & $2$ & $\pm \{2 \varepsilon_i, i \in I\}$\\
\hline
\end{tabular}
	\label{tab:DegSumsB}
\end{equation}

\subsubsection*{Type $C_\ell$, $\ell \geq 2$}
\begin{itemize}
\item[] $R = \{ \pm \epsilon_i \pm \epsilon_j: 1 \leq i , j \leq \ell\},$
\item[] $\mathcal{Q}(R) = \{\sum_{i = 1}^\ell m_i \epsilon_i:  m_i \in \mathbb Z,  \sum m_i \in 2 \mathbb Z\},$
\item[] $\mathcal{P}(R) = \bigoplus_{i = 1}^\ell \mathbb Z \epsilon_i,$
\item[] $2 \mathcal{P}(R) \cap \mathcal{Q}(R) = \bigoplus_{i = 1}^\ell 2 \mathbb Z \epsilon_i.$
\end{itemize}
The elements of $2 \mathcal{P}(R)$ are those vectors that have only even coefficients. Evidently, the sum over the coefficients will be even as well, so that $2 \mathcal{P}(R) \cap \mathcal{Q}(R) = \bigoplus 2 \mathbb Z \epsilon_i.$
The length of a long root is $4$ (after appropriate scaling).  An element $\gamma$ of $2 \mathcal{P}(R) \cap \mathcal{Q}(R)$ can only be a degenerate sum if its length is $4$ or $8$, i.e., $\gamma$ equals $\pm 2 \epsilon_i$ or $2(\epsilon_i \pm \epsilon_j),\, i\neq j.$ Both can be written as sum of two independent roots and thus they are degenerate sums. 
 \begin{equation}
\begin{tabular}[h]{|l|l|l|}
\hline
Type & $n_\gamma$ & Degenerate sums \\
\hline
$C_I, \mathrm{card}(I) \geq 2$ & $2$ & $\pm  \{2\varepsilon_i: i \in I\} \cup \pm \{ 2(\varepsilon_i \pm \varepsilon_j): i\neq j \in I\}$ \\
\hline
\end{tabular}
	\label{tab:DegSumsC}
\end{equation}

\subsubsection*{Type $D_\ell$, $\ell > 3$}
\begin{itemize}
\item[] $R = \{ \pm \epsilon_i \pm \epsilon_j: 1 \leq i \neq j \leq \ell\},$
\item[] $\mathcal{Q}(R) = \{\sum m_i \epsilon_i:  \sum m_ i \in 2 \mathbb Z\},$
\item[] $\mathcal{P}(R) = \mathcal{Q}(R) + \mathbb Z 1/2(\epsilon_1 + \ldots + \epsilon_\ell),$
\item[] $2 \mathcal{P}(R) \cap \mathcal{Q}(R) = \begin{cases} 2 \mathcal P(R), & \ell \in 2 \mathbb Z, \\
 2 \mathcal Q(R),&  \ell \in  2 \mathbb Z + 1.  \end{cases} $
\end{itemize}
An element in $2 \mathcal P(R)$ is of the form $\gamma = \sum_{i = 1}^\ell (2m_i + b)\epsilon_i$ where $\sum m_i \in 2 \mathbb Z $ and $b = 0$ or $b = 1.$ If $\ell$ is even, then  the sum over the coefficients is divisible by $2$ and $\gamma$ is contained in the root lattice, hence $2 \mathcal P(R) \cap \mathcal Q(R) = 2 \mathcal P(R) $ for $\ell$ even. If $\ell$ is odd, then $\gamma \in 2 \mathcal P(R) \cap \mathcal Q(R)$ if and only if $b = 0$ and then $2 \mathcal P(R) \cap \mathcal Q(R) = 2 \mathcal Q(R)$, $\ell$ odd. 
The length of the $i$th fundamental weight $\omega_i$ is $i$ for $1 \leq i \leq \ell - 2$ and the other two fundamental weights have length $\ell/4$, \cite[Planche IV]{BouLie2}. Thus, since we have to find the fundamental  weights such that their length multiplied by $4$ is smaller than or equal to $4,$ the first candidate is $\epsilon_1$ and if $\ell \neq 4,$ this is the only one.  Indeed $ \gamma = 2 \epsilon_1 = \epsilon_1 + \epsilon_2 + \epsilon_1 - \epsilon_2$ is a degenerate sum. The Weyl group orbit of $\gamma$ is $\pm 2 \epsilon_i,$ $1 \leq i \leq \ell.$ 
If $\ell = 4,$ then if $\gamma  = 2\omega_3$ or $\gamma = 2 \omega_4,$ the length of $\gamma$ is $4.$ In fact, $\omega_3$ and $\omega_4$ are conjugate under the Weyl group, so that it suffices to observe that $2 \omega_4 = \epsilon_1 + \epsilon_2 + \epsilon_3 + \epsilon_4$ is a degenerate sum. The Weyl group acts on $2 \omega_4$ by flipping the signs arbitrarily.
  \begin{equation}
\begin{tabular}[h]{|l|l|l|}
\hline
Type & $n_\gamma$ & Degenerate sums \\
\hline
$D_4$ & $2$ & $\pm \{2 \varepsilon_i: 1\leq i\leq 4\} \cup \{\pm \varepsilon_1 \pm \varepsilon_2 \pm \varepsilon_3 \pm \varepsilon_4\},$ \\
\hline
$D_I, \mathrm{card}(I) > 4$ & $2$  & $\{\pm 2\varepsilon_i, i \in I\}.$ \\ 
\hline
\end{tabular}
	\label{tab:DegSumsD}
\end{equation}

\subsubsection{Type $E$}
\begin{itemize}
\item[] $E_8 = \{ \pm \epsilon_i \pm \epsilon_j: 1 \leq i \neq j \leq 8\} \cup \{\frac{1}{2}\sum (-1)^{\nu(i)}\epsilon_i: \sum \nu(i) \in 2 \mathbb Z\},$
\item[] $\mathcal{Q}(R) = \{\sum m_i \epsilon_i:   2m _i \in \mathbb Z, \sum m_ i  \in 2 \mathbb Z, m_i - m_j \in \mathbb Z\},$
\item[] $\mathcal{P}(R) = \mathcal{Q}(R),$
\item[] $2 \mathcal{P}(R) \cap \mathcal{Q}(R) = 2 \mathcal{P}(R) = \{\sum m_i \epsilon_i: m_i \in \mathbb Z, \sum m_i \in 4 \mathbb Z, m_i - m_j \in 2\mathbb Z\}.$
\end{itemize} First assume $R = E_8,$ then
the connection index of the root system is $1$, so the weight lattice equals the root lattice and $2 \mathcal{P}(R) \cap \mathcal{Q}(R) = 2 \mathcal{P}(R).$
The candidates for degenerate sums have squared length less than or equal to $4$ and lie in $2 \mathcal{P}(R)$. Since the rank of the lattice is greater than $8$ and all entries must be congruent modulo $2$ this leaves only one option: $m_i = \pm 2$ for exactly one index $i$. But then $\sum m_ i = \pm 2$ which is not in $4 \mathbb Z.$ Thus there are no elements of the required length in $2\mathcal{P}(R)$, hence no degenerate sums for type $E_8.$ Similarly, since the weight lattices  resp the root lattices of $E_6$ and $E_7$ are contained in the weight lattice  resp. the root lattice of $E_8,$ there are no degenerate sums for $E_6$ and $E_7$ either.  
\subsubsection*{Type $F_4$}
\begin{itemize}
\item[] $R = \{ \pm \epsilon_i, \pm \epsilon_i \pm \epsilon_j: 1 \leq i \neq j \leq 4\} \cup \{1/2\sum (-1)^{\nu(i)}\epsilon_i : \nu(i) \in \mathbb Z\},$
\item[] $\mathcal{P}(R) = \mathcal{Q}(R) = \bigoplus \mathbb Z \epsilon_i + \mathbb Z 1/2(\epsilon_1 + \ldots + \epsilon_4),$
\item[] $2 \mathcal{P}(R) \cap \mathcal{Q}(R) = 2 \mathcal P(R) = \{\sum m_i \epsilon_i: m_i - m_j \in 2 \mathbb Z, \;\forall \, 1 \leq i, j\leq 4 \}.$
\end{itemize}
The root system $B_4$ sits inside $F_4$ and their weight lattices coincide. Therefore the elements in the double weight lattice which have length at most $4$ are the same in both cases. Since $B_4 \subset F_4$, this allows us to conclude $\mathbb{DS}(B_4) \subset \mathbb{DS}(F_4)$. The other inclusion is trivial, since a degenerate sum for $F_4$ will be a degenerate sum for any subsystem. The summands have to contained in the subsystem in question by definition of a degenerate sum. 
Therefore $\mathbb{DS}(B_4) =  \mathbb{DS}(F_4)$.
\begin{equation}
\begin{tabular}[h]{|l|l|l|}
\hline
Type & $n_\gamma$ & Degenerate sums \\
\hline
$F_4$ & $2$ & $\pm \{2 \varepsilon_i : 1 \leq i \leq 4\} \cup \{\pm \epsilon_1 \pm \epsilon_2 \pm \epsilon_3 \pm \epsilon_4\}$\\
\hline
\end{tabular}
	\label{tab:DegSumsF}
\end{equation}
\subsubsection{Type $G_2$}
\begin{itemize}
\item[] $R = \pm \{\epsilon_i - \epsilon_j, 2 \epsilon_i - \epsilon_j - \epsilon_k : \{i, j , k\} = \{ 1, 2, 3\} \},$
\item[] $\mathcal{P}(R) = \mathcal{Q}(R) \cong  \mathbb{Z}^3 \cap \{m_1\epsilon_1 + m_2\epsilon_2 + m_3\epsilon_3: m_1 + m_2 + m_3 = 0\}, $
\item[] $2\mathcal{P}(R) \cong 2 \mathbb{Z}^3 \cap \{m_1\epsilon_1 + m_2\epsilon_2 + m_3\epsilon_3: m_1 + m_2 + m_3 = 0\}.$ 
\end{itemize}
Since short roots have length $1$ in $G_2,$ then the length function in terms of coordinates is given by $ \|m_1\epsilon_1 + m_2\epsilon_2 + m_3\epsilon_3\| = \frac{1}{2}(m_1^2 + m_2^2 + m_3^2) $, so a candidate for a degenerate sum has to be of the form $\gamma = 2 m_1 \epsilon_1 + 2 m_2 \epsilon_2 + 2 m_3 \epsilon_3$ such that
$ (m_1^2 + m_2^2 + m_3^2) \leq 6$  and the sum over the coordinates is $0$. Thus $\gamma$ can only be conjugate to $2(\epsilon_1 - \epsilon_2)$ or $2(\epsilon_1 + \epsilon_2 - 2\epsilon_3)$. The latter cannot be written as sum of two independent roots, but for the former we get $2(\epsilon_1 - \epsilon_2) = 2\epsilon_1  - \epsilon_2 - \epsilon_3 + (-\epsilon_3 - \epsilon_1)$, so it is indeed a degenerate sum. 
  \begin{equation}
\begin{tabular}[h]{|l|l|l|}
\hline
Type & $n_\gamma$ & Degenerate sums \\
\hline
$G_2$ & $2$ & $\{ 2(\varepsilon_i - \varepsilon_j) : 1 \leq i\neq j \leq 3 )\}$ \\
\hline
\end{tabular}
	\label{tab:DegSumsG}
\end{equation}
The classification of degenerate sums for finite reduced irreducible root systems can be extended to possibly infinite ones. The key observation is that most of the results of \cite[\S 2]{vdK} depend only on local properties of the root system and generalize therefore to locally finite root systems. 
\begin{lem}
 \label{loc_deg_Sum_lem}
 Let $\gamma = \alpha + \beta$ be a degenerate sum in $R$ where $R$ is irreducible, reduced and of rank at least 5. Then there is a finite subsystem $S$ of $R$ containing $\alpha$ and $\beta$ such that $\gamma$ is a degenerate sum in $\mathcal Q(S)$. Moreover we may assume that the type of $R$ equals the type of $S$ and that the rank of $S$ is at least $5$.
 \end{lem}
 \begin{proof}For finite $R$ there is nothing to show. If $R$ is infinite this is a consequence of \cite[Lemma 8.3]{lfrs}.
  \end{proof}
\begin{cor}The following table contains all degenerate sums for $R$ an irreducible reduced root system. 
 \begin{equation}
\begin{tabular}[h]{|l|l|l|}
\hline
Type & $n_\gamma$ & Degenerate sums \\
\hline
$A_2$ & $3$ & $  \pm\{ \varepsilon_1 - 2 \varepsilon_2 + \varepsilon_3, \varepsilon_1 + \varepsilon_2 - 2\varepsilon_3, 2\varepsilon_1 - \varepsilon_2 - \varepsilon_3 \} $\\
\hline
$A_3$ & $2$ & $ \pm \{\varepsilon_1 - \varepsilon_2 + \varepsilon_3 - \varepsilon_4, \varepsilon_1 - \varepsilon_2 - \varepsilon_3 + \varepsilon_4, \varepsilon_1 + \varepsilon_2 - \varepsilon_3 - \varepsilon_4 \}$\\
\hline
$B_3$ & $2$ & $\pm \{2 \varepsilon_i : 1\leq i \leq 3\} \cup \{\pm \epsilon_1 \pm \epsilon_2 \pm \epsilon_3\}$\\
\hline
$B_4$ & $2$ & $\pm \{2 \varepsilon_i : 1\leq i \leq 4\} \cup \{\pm \epsilon_1 \pm \epsilon_2 \pm \epsilon_3 \pm \epsilon_4\}$\\
\hline
$B_I, \mathrm{card}(I)> 4$ & $2$ & $\pm \{2 \varepsilon_i, i \in I\}$\\
\hline
$C_I, \mathrm{card}(I) \geq 2$ & $2$ & $\pm  \{2\varepsilon_i: i \in I\} \cup \pm \{ 2(\varepsilon_i \pm \varepsilon_j): i\neq j \in I\}$ \\
\hline
$D_4$ & $2$ & $\pm \{2 \varepsilon_i: 1\leq i\leq 4\} \cup \{\pm \varepsilon_1 \pm \varepsilon_2 \pm \varepsilon_3 \pm \varepsilon_4\}$ \\
\hline
$D_I, \mathrm{card}(I) > 4$ & $2$  & $\{\pm 2\varepsilon_i, i \in I\}$ \\ 
\hline
$F_4$ & $2$ & $\pm \{ 2 \varepsilon_i: 1 \leq i \leq 4 \} \cup \{ \pm \varepsilon_1 \pm \varepsilon_2 \pm \varepsilon_3 \pm \varepsilon_4 \}$\\
\hline
$G_2$ & $2$ & $\{ 2(\varepsilon_i - \varepsilon_j) : 1 \leq i\neq j \leq 3 )\}$ \\
\hline
$G_2$ & $3$ & $\{3(\varepsilon_i - \varepsilon_j) : 1 \leq i\neq j \leq 3 )\}$\\
\hline
\end{tabular}
	\label{tab:DegSums}
\end{equation}
\end{cor}
\begin{proof}
If $R$ is finite, then this is \cite[Table 1, p.13]{vdK} or can be obtained by merging Tables 1 to 7. Assume now that $R = X_I$ is infinite and that $\gamma = \alpha + \beta$ is a degenerate sum of divisor $n_\gamma$. By Lemma~\ref{loc_deg_Sum_lem} there is a finite subsystem $S$ of type $X$  and rank at least $5$ such that $\alpha, \beta \in S$ and $\gamma = \alpha + \beta$ is a degenerate sum in $\mathcal Q(S)$. It follows that $X \neq \dot{A}$, since for $A_n$, $n \geq 5$ no degenerate sums occur. Without loss of generality we may assume $X = B, C$ or $D$. The classification of the finite cases tells us that for $4 < \mathrm{rank}(S) < \infty$ all degenerate sums $\gamma = \alpha + \beta$ have divisor $2$, so infinite root systems will also only have degenerate sums of divisor $2$ and the set of degenerate sums for type $X_I$ is contained in $\bigcup_{J \subset I, 4 <\card J < \infty} \mathbb{DS}(X_J)$ : 
\begin{eqnarray*}
\mathbb{DS}(B_I)& \subset &\pm \{2 \varepsilon_i, i \in I\},\\
\mathbb{DS}(C_I)& \subset & \pm \{2 \varepsilon_i, i \in I\} \cup \pm \{ 2(\varepsilon_i \pm \varepsilon_j): i\neq j \in I\},\\
 \mathbb{DS}(D_I)& \subset & \pm \{2 \varepsilon_i, i \in I\}.
\end{eqnarray*}
Assume that $\gamma = \alpha + \beta$ is a degenerate sum for $X_J$, $4 < \card J <\infty$, $J \subset I$. Then the roots $\alpha$ and $\beta$ are also linearly independent in $X_I$ and the argument above gives us that $n_\gamma =2$. It remains to prove that $\langle \gamma, \alpha\che \rangle \in 2 \mathbb Z$ for all $\alpha \in X_I$. \\
Let $X = B$ or $D$. Then $\gamma = 2\varepsilon_i $ some $i \in I$ and $\langle 2 \varepsilon_i, \alpha\che \rangle = 2\langle  \epsilon_i, \alpha\che \rangle $ where $\langle  \epsilon_i, \alpha\che \rangle  \in \{0, \pm 1\}$ for all roots $\alpha$. 
Thus $\langle \gamma, \alpha\che \rangle$ is twice an integer.\\
If $X =C$ then a degenerate sum $\gamma$ is either a long root or twice a short root. Since the long roots are of the form $2 \epsilon_i$  $i \in I$ the same argument as for types $B$ and $C$ shows that $\langle \gamma, R \rangle \subset 2 \mathbb Z$. If $\gamma = 2\beta$, $\beta \in R$ then $\langle \gamma, \alpha\che \rangle = 2 \langle \beta, \alpha \che\rangle \in 2 \mathbb Z$ for all roots $\alpha$. 
So equality holds in all three inclusions above. 
\end{proof}
\begin{rem} We note the following mysterious facts:
\begin{eqnarray*}
A_2 \cup \mathbb{DS}(A_2) &=& G_2, \\
B_I \cup  \mathbb{DS}(B_I) &=& BC_I, \card I > 4, \\
D_I \cup  \mathbb{DS}(D_I) &=& C_I, \card I > 4.\\
\end{eqnarray*}
\end{rem}
\subsection{Degenerate sums and central extensions of R-graded Lie algebras}

\begin{defi}[\cite{Neh3}]
Let $R$ be a reduced root system. An \emph{$R$-graded Lie algebra} is a $k$-Lie algebra together with a $\mathcal Q(R)$-grading
such that
\begin{enumerate}[(i)] 
\item  $L = \bigoplus_{\alpha \in R}L_\alpha,$ where $L_\alpha$ is a submodule of $L$,  called the \emph{homogeneous space of degree $\alpha.$}
\item For every $\alpha \in R^\times$ the homogeneous space $L_\alpha$ contains an invertible element $e_\alpha$ of the $\mathcal Q(R)$-graded Lie algebra $L$, i.e., there exists an element $f_\alpha \in L_{-\alpha}$ such that $\ad h_{\alpha}$, $h_{\alpha} := [f_\alpha, e_\alpha]$, acts diagonally on $L$: 
$$\ad h_{\alpha}|L_\beta = \langle \beta, \alpha\che \rangle \mathrm{id}_{L_\beta}\quad \mbox{for all }\beta \in R.$$
\item $\sum_{\alpha \in R^\times}[L_{\alpha}, L_{-\alpha}] = L_{0}.$
\end{enumerate}
\label{r_delta_graded_def}

Note that $(e_\alpha, h_\alpha, f_\alpha)$ is an $\mathfrak{sl}_2$-triple in the sense of \cite{BouLie3}.\\
\end{defi} 

We would like to include either in the definition of a root-graded Lie algebra $L$ as axiom that $L$ is perfect or show that perfectness follows from the axioms. This holds ``almost'' as we will see soon.

\begin{lem} 
\label{root_on_root_div_one}
Let $R$ be a reduced irreducible root system, $R \neq C_I$ (so in particular $R \neq A_1$ and $R \neq B_2$). Then for every root $\alpha \in R^\times,$ there is a root $\beta \in R^\times$ such that 
\begin{itemize}
\item[\rm(i)] $ \langle \alpha, \beta \che \rangle =  -1, $
\item[\rm (ii)] $(\alpha + \mathbb Z \beta) \cap R = \{\alpha, \alpha + \beta\}.$
\end{itemize}
\end{lem}
\begin{proof} Depending on the type of $R$ any two connected roots in $R$ are contained in a subsystem of type $A_2$, $G_2$ or $B_3.$ The case of $R = A_2$ follows by inspection, we may just pick one of the two roots which form an angle $\pi/3$ with $\alpha.$ 
If $\alpha$ is a long root in $B_3$ or $G_2$ then there is a subsystem of  type $A_2$ containing $\alpha$ and the argument for type $A_2$ applies. Let $\alpha$ now be a short root in $G_2,$ the  root $\beta$ can again be found by inspection.  In fact $\beta$ is long and the angle between $\alpha$ and $\beta$ is $5\pi/6.$
If $\alpha$ is short in $B_3,$ then the existence of $\beta$ can again be shown by inspection, the angle between $\alpha$ and $\beta$ is $\pi/4$ in this case. 
\end{proof}

\begin{lem} Let $R$ be a root system, $R \neq C_I$ and let $L$ be an $R$-graded Lie algebra. Then, for every $\alpha \in R^\times,$ there is $\beta \in R^\times$ such that $\alpha + \beta \in R^\times$ and
$$ [L_{-\beta}, L_{\alpha + \beta}] = L_{\alpha}. $$
\label{not_type_C_generation_lem}
\end{lem}
\begin{proof} Since $R \neq C_I,$  by Lemma~\ref{root_on_root_div_one} there is a  root $\beta,$ such that $(\alpha + \mathbb Z \beta) \cap R = \{ \alpha, \alpha + \beta\}$ and $\langle  \alpha, \beta \che \rangle = -1.$  Consider the linear maps $\ad e_\beta : L_{\alpha} \rightarrow L_{\alpha + \beta}$ and $\ad f_{\beta} : L_{\alpha + \beta} \rightarrow L_{\alpha}$.  By assumption  $\alpha - \beta$ is not a root, hence for $x_\alpha \in L_{\alpha}$
$$\ad f_{\beta} \ad e_\beta. x_\alpha =  \ad[f_\beta, e_\beta]x_\alpha = -\ad h_{\beta}x_{\alpha} = - \langle  \alpha, \beta \che \rangle x_{\alpha} = x_{\alpha}. $$
Similarly for $x_{\alpha + \beta} \in L_{\alpha + \beta}$ since $\alpha + 2 \beta \notin R:$
$$\ad e_{\beta} \ad f_{\beta} x_{\alpha + \beta} = \langle \alpha + \beta, \beta \che \rangle x_{\alpha + \beta} = (-1 + 2)x_{\alpha + \beta} = x_{\alpha + \beta.} $$
Therefore, $\ad f_{\beta} : L_{\alpha + \beta} \mapsto L_{\alpha} $  is an isomorphism and  in particular onto. 
\end{proof}
\begin{prop}
\label{K_R_turns_rogral_perfect}
 Let $R$ be a reduced root system and $L$ an $R$-graded Lie algebra. The Lie algebra $L$ is perfect, if
\begin{itemize}
\item[\rm(i)] $R$ has no irreducible component of type $C_I,$ or
\item[\rm(ii)] $1/2 \in k.$ \end{itemize}
\end{prop} 
\begin{proof} If there is no irreducible component of type $C_I,$ then (i) follows by Lemma~\ref{not_type_C_generation_lem}. If there is an irreducible component of type $C,$ then the roots $2\epsilon_i$ have divisor $2$ and therefore $L_{2\epsilon_i} \in [L,L]$ if $1/2 \in k. $
\end{proof}
For a Lie algebra of type $C,$ the assumption $1/2 \in k$ is necessary. It was shown in \cite[Proposition 2.2]{vdK} that Lie algebras of the form $\mathfrak{sp}(2n, \mathbb Z)$ are never perfect.

It will be convenient to have the following notation at hand:
\begin{defi} For a reduced root system $R$ we denote by $k_R$ a commutative associative unital ring such that 
$1/2 \in k_R$ if $R$ has an irreducible component of type $C_I ,\card I \geq 1$ (recall $C_1 = A_1$ and $C_2 = B_2$) and $k = k_R$ otherwise.  
\label{root_base_ring_Defi}
\end{defi}

\begin{prop} Let $R$ be an irreducible reduced root system, let $L$ be an $R$-graded Lie algebra over $k_R$ and  let $f: K \rightarrow L$ be a $\mathcal Q(R)$-graded central covering. If $0\neq \gamma \in \supp_{\mathcal Q(R)}(K),$ then either
\begin{itemize}
\item[\rm (i)] $\gamma \in R^\times$ in which case $f : K_\gamma \rightarrow L_\gamma$ is a bijection, or
 \item [\rm(ii)] $\gamma = \alpha + \beta$ is a degenerate sum,  $K_\gamma \subset \ker f$ and $n_\gamma K_\gamma =0.$
 \end{itemize}\label{torsion_bijection_prop}
\end{prop}
\begin{proof} Under the assumption on the base ring $k_R$, $L$ is always perfect. 
By Lemma~\ref{generating_lemma}, $K$ is generated as a Lie algebra by $\bigoplus_{\alpha \in R^\times}K_\alpha.$ 
If $\gamma \in \supp K,$ then $K_\gamma = \sum_{\alpha +\beta  = \gamma}[K_{\alpha}, K_{\beta}].$ In fact, $K_\gamma = \sum_{\alpha \in R, \gamma - \alpha \in R}[K_\alpha, K_{\gamma - \alpha}],$ since a homogeneous space $K_\beta$, $\beta \notin R,$ lies in the kernel of $f$ and hence in the centre of $K.$ 

Let $\gamma \in \supp K$ and
 let $ [x_{\alpha}, y_{ \beta}] \in K_\gamma$ where $x_{\alpha} \in K_\alpha, y_{\beta} \in K_\beta.$ Pick $t' \in f^{-1}(h_{\delta})$, $\delta \in R^\times$. Then, $f([t', y_{\beta}]) = [f(t'), f(y_\beta)] = [h_{\delta}, f(y_\beta)]= \langle \beta, \delta \che \rangle f(y_\beta)$ since $f$ is $\mathcal Q(R)$-graded. Likewise $f([t', x_\alpha]) = \langle \alpha, \delta \che \rangle f(x_\alpha).$ 
The central trick gives
\begin{eqnarray*}
[t', [x_\alpha, y_\beta]] &=& [[t', x_\alpha], y_\beta] + [ x_\alpha,[t', y_\beta]]\\
&=&\langle \alpha, \delta \che \rangle  [x_\alpha, y_\beta] + \langle \beta, \delta \che \rangle [x_\alpha, y_\beta] = \langle \alpha + \beta, \delta \che \rangle [x_\alpha, y_\beta]. 
\end{eqnarray*}
Since $K_{\gamma}$ is spanned by $[x_\alpha, y_\beta]$ with $\alpha + \beta = \gamma,$ this means that $\ad t'$ is diagonalizable on $K.$ By assumption on $k_R,$  for  $\gamma \in R^\times$ there is $\delta \in R^\times$ such that $ \langle \gamma, \delta \che \rangle \in k^\times. $ Choose $x_\gamma \in K_\gamma \cap  \ker f \subset Z(K).$ Then
$$0 = \langle \gamma, \delta \che \rangle^{-1}[t', x_{\gamma} ] = \langle \gamma, \delta \che \rangle ^{-1}\langle \gamma, \delta \che \rangle  x_\gamma = x_{\gamma} $$
whence $K^\gamma \cap \ker f = \{0\}.$ This proves (i). \\
If $\gamma \notin \supp L,$ then since $f :K \rightarrow L$ is graded, $K_\gamma $  is central. Pick $t' \in f^{-1}(h_\delta)$ where $\delta \in R^\times, $ then 
$ 0 = [t', K_\gamma] = \langle \gamma, \delta \che \rangle K_\gamma,$ whence $K_\gamma$ has $n_\gamma$-torsion. 
Assume that $\gamma = \alpha + \beta \in \supp_{\mathcal Q(R)}K\setminus R$. The goal is to prove that $\gamma$ is a degenerate sum. \\
 Assume that $\gamma = \alpha ' + \alpha '$ for a non-zero root $\alpha'$ and this is the only way to express $\gamma$ as sum of two non-zero roots. By assumption on $k_R$ there is a  either a root $\beta \in R,$ such that $ \langle \alpha', \beta \che  \rangle = 1$ or if $R = C_I,$ $ \langle \alpha', \beta \che  \rangle = 2.$  If $R \neq C_I,$ then $ \langle \alpha'+ \alpha' , \beta \che  \rangle  = 2$ and hence $n_\gamma = 2.$  If $R = C_I,$ then $n_\gamma$ is a divisor of $2  \langle \alpha' ,(\alpha')\che \rangle  = 4.$ Hence  $n_\gamma$ is invertible in $k_R,$ and $K_\gamma$ is trivial. We may from now on assume that $R \neq C_I.$
 Recall from the proof of (i) that the  graded covering $K$ is generated by  the spaces $K_{\alpha}, \alpha \in R^\times$, see Lemma~\ref{generating_lemma}.
  Under the assumption that $\gamma = \alpha ' + \alpha '$ for a non-zero root $\alpha'$ and that this is the only way to express $\gamma$ as sum of two roots, it follows that $K_{\gamma} = [K_{\alpha'}, K_{\alpha'}] + [K_{0}, K_{\gamma}].$ However, $K_{\gamma}$ lies in the center of $K,$ so we have $K_{\gamma} = [K_{\alpha'}, K_{\alpha'}]$. According to Lemma~\ref{not_type_C_generation_lem}, there are non-zero roots $\beta'$ and $\gamma'$ such that $L_{\alpha'} = [L_{\beta'}, L_{\gamma'}].$ Part (i) implies, that the restriction of $f$ to $K_{\beta'} + K_{\gamma'} + K_{\alpha'}$ is a bijection, and hence  $K_{\alpha'} = [K_{\beta'}, K_{\gamma'}].$ The Jacobi identity gives
 $$[K_{\alpha'}, K_{\alpha' }] = [[K_{\beta'}, K_{\gamma'}], K_{\alpha'}] =  [K_{\beta'}, [K_{\gamma'}, K_{\alpha'}]] + [K_{\gamma'}, [K_{\beta'}, K_{\alpha'}]] . $$
 This space can only be non-zero if at least one of $\beta' + \alpha'$ and $\gamma' + \alpha'$ lies in $R^\times.$ Then $(\beta', \gamma' + \alpha')$ or $(\gamma', \beta' + \alpha')$ is a degenerate pair, contradicting the assumption that it is not possible to express $\gamma$ as a degenerate sum. Thus, by contradiction, (ii) is true. 
\end{proof}

\section{Jordan graded Lie algebras}
\label{Jordan_graded_Lie_alg}
Throughout this section, $k$ is a unital commutative associative ring. 
\subsection{Idempotents and grids}All the results and proofs can be found in \cite{three_graded_triple_neher} (the passage from triple systems to pairs is straightforward.)
\begin{defi} Let $V = (V^+, V^-)$ be a Jordan pair with quadratic operator $Q = (Q^+, Q^-)$. An \emph{idempotent} is a pair $(e^+, e^-) \in V$ such that for $\sigma = \pm$,  $$ Q_{e^\sigma}^\sigma e^{-\sigma} = e^\sigma, $$ or shorter $Q_ee = e. $\\
If $e$ is an idempotent of $V,$ we define the Peirce space with respect to $e$ as: 
$$V_2^\sigma (e) = \im(Q_{e^\sigma}), V_1^\sigma(e) = \ker (\mathrm{id} - D(e^\sigma, e^{-\sigma})), V_0^\sigma(e) = \ker(D(e^\sigma, e^{-\sigma})) \cap \ker(Q_{e^{-\sigma}}).$$\label{idempotent_def}
\end{defi}
As usual, we will often suppress the superscript $\sigma$ whenever it is clear from the context. Also to avoid subscripts, one often write $Q(a)$ instead of $Q_a$ and likewise $Q(a, c)$ and $D(a, b)$ for $a, c \in V^\sigma$ and $b \in V^{-\sigma}. $
The immediate consequence of Definition~\ref{idempotent_def} is the proposition below. 
\begin{prop}Let $V$ be a Jordan pair and $e$ an idempotent. For $\sigma = \pm$ the following hold:
\begin{itemize} 
\item[\rm(i)]$\ker(Q^\sigma_{e^\sigma}) = V_1^\sigma(e) \oplus V_0^\sigma(e),$ 
\item[\rm(ii)] $V_i^\sigma(e) \subset \{x \in V^\sigma : \{e^\sigma, e^{-\sigma}, x\} =  ix\},$ with equality for $i = 1$ and for all $i$, if $1/2 \in k^\times.$
\end{itemize}
\end{prop}

\begin{defi}\label{cog_relations} Let $V$ be a Jordan pair. 
Two idempotents $e$ and $f$ are called 
 \begin{itemize}
\item[\rm(i)] \emph{orthogonal}, if $e \in V_0(f)$ and $f \in V_0(e)$, denoted by $e \bot f$,
\item[\rm(ii)] \emph{collinear}, if $e \in V_1(f)$ and $f \in V_1(e)$, denoted by $e \top f$,
\item[\rm(iii)] \emph{associated}, if $e \in V_2(f)$ and $f \in V_2(e)$, denoted by $e \approx f$.
\item[\rm{(iv)}] If  $e \in V_1(f)$ and $f \in V_2(e)$, then $e$ \emph{governs} $f$, denoted by $e \vdash f$ and $f \dashv e$. 
\end{itemize}
We will refer to the relations $\bot, \top, \approx, \vdash$ and $\dashv $  as \emph{cog relations}. It is useful to note that
$$ e \approx f \iff V_2(e) = V_2(f) \iff V_i(e) =V_i(f) \mbox{ for }i \in \{0,1,2\}.$$
\end{defi}

\begin{defi}A root system $R$ is \emph{3-graded} if there exists a partition $R  = R_{-1}  \cup R_0 \cup R_{1}$ satisfying
\begin{itemize}
\item[\rm(i)] $-R_{1} = R_{-1}$, 
\item[\rm(ii)]$(R_{i} + R_{j}) \cap R \subset R_{i+ j},$ where $R_{i+ j} = \emptyset$ if $ i+j  \notin \{-1, 0, 1\}$, 
\item[\rm(iii)] $(R_1 + R_{-1}) \cap R = R_0.$
\end{itemize}
 \label{3_grad_def}
\end{defi}
\begin{prop} See \cite[17.8 and 17.9]{lfrs}. An irreducible root system $R$ has a $3$-grading  if and only if $R$ is of type ${A}_I$, $B_I$, $C_I$, $D_I$, $E_6$ or $E_7$. 
In this case,
\begin{itemize}
\item $R_1$ spans $\mathcal Q(R)$, 
\item For any two $\alpha \neq \beta \in R_1$ one of the following holds: 
\begin{itemize}
\item $\alpha$ is orthogonal to $\beta$, $\langle\alpha , \beta  \check{\;} \rangle = 0$, denoted by $\alpha \perp \beta,$
\item $\alpha$ governs $\beta$, $\langle \alpha, \beta  \check{\;} \rangle = 1$, $\langle \beta, \alpha \check{\;} \rangle = 2,$ denoted by $
\alpha \vdash\beta$, 
\item $\beta$ governs $\alpha$, $\langle \beta, \alpha \check{\;} \rangle = 1$, $\langle \alpha, \beta \check{\;} \rangle = 2,$ denoted by $\alpha \dashv \beta$, 
\item $\alpha$ and $\beta$ are collinear, $\langle \alpha, \beta \check{\;} \rangle = 1 = \langle \beta, \alpha \check{\;} \rangle$, denoted by $\alpha \top \beta$. 
\end{itemize} 
\end{itemize}
\end{prop}

\begin{lem} Assume that $(R, R_1)$ is  a $3$-graded root system. Let $0 \neq \delta$ in $R_0$. Then there are $\alpha, \beta \in R_1$ such that $\delta = \alpha - \beta$, $\langle \beta, \alpha\check{\;} \rangle  = 1$ or $\langle \alpha, \beta \check{\;} \rangle  = 1$ and for all $\gamma \in R$. 
\begin{equation}  \langle \gamma, (\alpha -\beta)\che  \rangle = \langle \alpha, \beta \che \rangle \langle \gamma, \alpha \check{\;} \rangle - \langle \beta, \alpha \che \rangle \langle \gamma, \beta \che \rangle  .\label{root_inner_formula} \end{equation}

\end{lem}
\begin{proof} Let $\gamma \in R_0$.  By (iii) of Definition~\ref{3_grad_def}, we can write $\gamma = \alpha - \beta$ for $\alpha, \beta \in R_1$ such that $\alpha \mathcal R  \beta$ and $\mathcal R \neq \perp.$  Therefore $\mathcal R = \vdash, \dashv$ or $\top.$ Hence $\langle \beta, \alpha\che \rangle  = 1$ or $\langle \alpha, \beta \che \rangle  = 1.$ Switching the roles of $\alpha$ and $\beta$ in (\ref{root_inner_formula}) amounts to multiplying both sides by $-1$, therefore without loss of generality  $\langle \alpha, \beta \che \rangle  = 1.$
So either $\alpha$ governs $\beta$ or $\alpha$ and $\beta$ are collinear. In this case
$s_\beta(\alpha) = \alpha - \langle \alpha, \beta \che \rangle \beta = \alpha - \beta$ and thus $s_{\alpha - \beta} = s_{s_\beta(\alpha)} = s_\beta s_\alpha s_\beta.$ With this formula for every $\gamma \in R:$ 
\begin{eqnarray*}\langle \gamma, (\alpha - \beta) \che \rangle = \langle \gamma, (s_\beta \alpha) \che \rangle = \langle s_\beta \gamma, \alpha  \che \rangle = \langle \gamma -\langle \gamma, \beta \che \rangle\beta , \alpha  \che \rangle \\ = 1\cdot  \langle \gamma, \alpha \che \rangle - \langle \beta, \alpha \che \rangle \langle \gamma, \beta \che \rangle = \langle \alpha, \beta \che \rangle \langle \gamma, \alpha \che \rangle - \langle \beta, \alpha \che \rangle \langle \gamma, \beta \che \rangle.
\end{eqnarray*}
\end{proof}
\begin{defi}\label{grid_definition} Let $(R, R_1)$ be a $3$-graded root system. A family $\mathcal E = \{e_\alpha : \alpha \in R_1\}$ is called a \emph{covering $(R, R_1)$-grid} if the following hold: 
\begin{itemize}
\item[\rm(i)]If $\alpha, \beta \in R_1,$ then  $e_\alpha \mathcal R e_\beta$ if and only if  $\alpha \mathcal R  \beta.$
In this case we define the \emph{joint Peirce spaces} as $V_\alpha = \bigcap_{\beta \in R_1} V_{\langle  \alpha , \beta \che\rangle }(e_\beta), \alpha \in R_1.$ 
\item[\rm(ii)] $V = \sum_{\alpha \in R_1}V_\alpha.$
\end{itemize}
\end{defi}
\begin{prop} The sum of the joint Peirce spaces is always direct, i.e.,  if $\mathcal E$ is a covering $(R, R_1)$ grid then 
$V = \bigoplus_{\alpha \in R_1}V_\alpha.$ Moreover for $\alpha, \beta, \gamma  \in R_\sigma$ 
$$ \{V_\alpha^\sigma, V_{\beta}^{-\sigma}, V_\gamma^\sigma\} \subset V_{\alpha - \beta + \gamma}^\sigma $$
with $\{V_\alpha^\sigma, V_{\beta}^{-\sigma}, V_\gamma^\sigma\}= \{0\}$ if $\alpha - \beta + \gamma  \notin R_\sigma.$
\end{prop}
The following condition on $R$ will hold: 

\begin{center}
\textbf{$(R, R_1)$  is a $3$-graded root system.
}
\end{center}

Recall  the definition of a Lie algebra of Jordan type (Definition~\ref{jordan_type_Defi}), Definition~\ref{r_delta_graded_def},  the definition of $k_R$ from Definition~\ref{root_base_ring_Defi} and the definition  of the functor $\uTKK$ and its universal properties (see Corollary~\ref{uTKK_universal_prop}).
\begin{theo} 
\label{Root_graded_grid_theo} Let $L$ be a Lie algebra of Jordan type over $k= k_R$ and assume that the Jordan pair $V = (L_+, L_-)$ is covered by a grid $\mathcal E$ of type $(R, R_1)$. 
 Then \begin{itemize}
 \item[\rm{(i)}] 
There are uniquely defined central extensions
$f : \uTKK(V) \rightarrow L $ and $g: L \rightarrow \TKK(V)$ such that $f|_V = g|_V = \mathrm{id}|_V. $
If $L$ is perfect, then these central extensions are coverings. 

\item[\rm{(ii)}]If $[L_\alpha, L_{-\beta}] = 0$ whenever $\alpha, \beta \in R_1$ and $\alpha \perp \beta,$ then 
 $L$ is an $R$-graded Lie algebra with homogeneous components
 $$L_{\sigma \alpha} = V_\alpha^\sigma, \; for \; \alpha \in R_1,$$ 
 $$L_{\gamma} = [L_\alpha, L_{-\beta}], \quad \gamma = \alpha  - \beta; \; \alpha, \beta \in R_1, \, \alpha -\beta \in R_0,$$
 $$L_0 = \sum_{\alpha \in R_1}[L_\alpha, L_{-\alpha}]. $$
 \item[\rm(iii)] The Lie algebra $\TKK(V)$ is $R$-graded.  
\end{itemize} \label{rogra_la_to_grd_jp}
\end{theo}
\begin{proof}  
The existence and uniqueness of the central extensions in (i) are Corollary~\ref{uTKK_universal_prop_JP}. 
 Part (ii) is Theorem 2.6 in \cite{Neh1996} and part (iii) is  Theorem 2.5 \cite{Neh1996}. Thus by Corollary~\ref{perfect_TKK_cor}, $\uTKK(V)$ is also perfect and hence $L$ is perfect too. 
\end{proof}

\begin{lem}\cite[2.5]{Neh1996} \label{grid_space_iso_lem} Let $V$ be a Jordan pair over $k_R$ which is covered by an $(R, R_1)$-grid and $L$ of Jordan type with $(L_+, L_-) = V.$ If $\alpha \neq  \beta \in R_1,$ then 
\begin{eqnarray*} \langle \alpha, \beta \che \rangle = 1  & \implies & L_{\alpha - \beta} \cong L_\alpha, \\
\langle \beta, \alpha \che \rangle = 1  & \implies & L_{\alpha - \beta} \cong L_{-\beta}.
\end{eqnarray*}
  The isomorphisms are given by $\ad e_\beta^-$ and $\ad e_\alpha^+$ respectively. 
\label{iso_from_grid_lem}
\end{lem}
\begin{proof} 
We only consider the first case since both are proved in exactly  the same manner. Denote $e^\pm  = e_\beta^\pm$. Let $x \in L_{\alpha}$. Then 
$\ad( e^+)\ad(e^-).x = x$ so that $\ad (e^-) : L_\alpha \rightarrow L_{\alpha - \beta}$ is left invertible with left inverse $\ad (e^+)$. For the next step it is crucial that $\alpha - 2\beta \notin R$ since the $\beta$-string through $\alpha$ is only $\{\alpha, \alpha - \beta\}$. Thus if $z \in L_{\alpha - \beta}$ then $[e^-, z] = 0$ and
$$\ad e^-. \ad e^+. z  = - \ad[e^+, e^-].z = z$$ which shows that $\ad e^-$ is invertible. 
\end{proof}

 \begin{prop} Let $V$ be a Jordan pair which is covered by an $(R_1, R)$ grid and $L = \TKK(V).$
If $f: K \rightarrow L$ is a  $\mathcal Q(R)$-graded  central covering of $L$ and $\alpha, \beta \in R_1,$ then 
$$ \langle \alpha, \beta \che \rangle = 1  \implies K_{\alpha - \beta} \cong K_{\alpha} \cong L_\alpha. $$ 
$$\langle \beta, \alpha \che  \rangle = 1  \implies  K_{\alpha - \beta} \cong K_{-\beta} \cong L_{-\beta}. $$
\label{iso_from_grid_lem_ce}
\end{prop}
\begin{proof} 
This follows from Proposition~\ref{torsion_bijection_prop} and Lemma~\ref{iso_from_grid_lem}.
\end{proof}



\subsection{Degenerate sums revisited}
It now makes sense to ask what kind of additional structure an $(R_1, R)$-grid on a Jordan pair induces on its universal derivation algebra $\uider(V).$ 
Recall that for a Jordan pair $V$, $\uider(V)$ was defined as $V^+ \otimes V^-$ modulo the relations
\begin{eqnarray}
&\delta(x, y)(u \otimes v) + \delta(u, v)(x \otimes y),& \label{HCP1_op_version}\\
& \delta(x,y)(x \otimes y),& \label{HCP2_op_version}
\end{eqnarray} 
where $(x,y), (u,v) \in V$ and $\delta(x,y) = (D_{x,y}, -D_{y,x}).$
We will abbreviate the relation $\delta(x, y)(u \otimes v) + \delta(u, v)(x \otimes y)$ by $A(x,y,u,v)$ and the relation $\delta(x,y)(x,y)$ by $B(x,y).$ Then $2B(x,y) = A(x,y,x,y)$ and the symmetries  \\
\begin{equation}A(x,y,u,v) = A(u,v, x,y) = A(x,v, u,y)\end{equation}
hold. 
Linearize (\ref{HCP2_op_version}):
\begin{eqnarray}
B(x,y) &=& \{xyx\} \otimes y - x \otimes \{yxy\} \label{quadratic_B}\\
B'(x,u;y) &=& 2\{xyu\} \otimes y - x \otimes \{yuy\} - u \otimes \{yxy\} = A(x,y, u, y) \label{quadratic_linear_B} \\
B''(x,u, y, v) &=&  A(x,y, u, v)  + A(x,v, u, y) = 2A(x,y, u,v) \label{bilinear_B}
\end{eqnarray}

Assume now that $V = \bigoplus_{\alpha \in R_1} V_\alpha$ is covered by a grid. 
The root grading of $V$ induces a $\mathcal Q(R)$-grading on the tensor product 
$$V^+ \otimes V^- = \bigoplus_{\rho, \mu \in \mathcal Q(R)} V_\rho^+ \otimes V_\mu^- =  (V^+ \otimes V^-)_0 \oplus (\bigoplus_{\rho \neq \mu \in \mathcal Q(R)} V_\rho^+ \otimes V_\mu^-) ,$$ where 
$(V^+ \otimes V^-)_{\mu} = \sum_{\gamma \in R_1} V_\gamma^+ \otimes V_{\gamma - \mu}^-.$

By (\ref{bilinear_B}) and (\ref{quadratic_linear_B}) we only need to consider $B(x,y)$ for $x$ and $y$ homogeneous. 
Thus $\uider(V)$ is equal to $V^+\otimes V^-$ modulo
\begin{itemize}
\item[(1)] $A(x,y, u, v)$ for $x,y,u,v$ homogeneous,
\item[(2)] $B(x,y)$ for $x,y$ homogeneous. \label{B_hom_graded} We immediately discuss the possible cases for $B(x_{\alpha}, y_{\beta})$ with $x_{\alpha} \in V^+_{\alpha}, y_{\beta} \in V_{\beta}^{-}.$
Let $x_\alpha \in V_\alpha^\sigma,$ $y_\beta \in V_\beta^{-\sigma}.$ For a root system grading this gives $3$ cases:
\begin{itemize}
\item[(2a)] $\alpha = \beta:$ $B(x_\alpha, y_\alpha) \in V_{\alpha}^+ \otimes V_{\alpha}^-$
\item[(2b)] $\alpha \vdash \beta:$ $B(x_\alpha, y_\beta) = 2Q(x_\alpha)y_\beta \otimes y_\beta \in V_{2\alpha - \beta}^+ \otimes V_{\beta}^-$
\item[(2c)] $\alpha \dashv \beta:$ $B(x_\alpha, y_\beta) = - x_\alpha \otimes 2Q(y_\beta)x_\alpha \in V_{\alpha}^+ \otimes V_{2\beta - \alpha}^-$
\end{itemize}
There are no other cases since $Q(x_\alpha)(y_\beta) = 0$ if $\alpha \top \beta$ or $\alpha \perp \beta.$
\end{itemize}
Since $$\{V_\alpha^{\sigma},V_\beta^{-\sigma}, V_\gamma^\sigma \} = 0, \;\mbox{if} \; \alpha - \beta + \gamma \notin R_{\sigma}$$
we can assume  for the discussion of $A(x,y, u,v)$ with $x \in V_\alpha^\sigma, y \in V_\beta^{-\sigma}, z \in V_\gamma^\sigma, v \in V_\delta^{-\sigma}$ that at least one of the following elements of $\mathcal Q(R)$ lies in $R_{1}:$
\begin{equation} \epsilon_1 = \alpha - \beta + \gamma, \, \epsilon_2 = \beta - \alpha + \delta, \,\epsilon_3 = \gamma - \delta + \alpha, \, \epsilon_ 4 = \delta - \gamma + \beta. \label{epsilon_Defi}\end{equation}
We define 
$$A(\alpha, \beta, \gamma, \delta) := A(V_\alpha^+,V_\beta^-, V_\gamma^+ , V_\delta^-)$$
for all $\alpha, \beta, \gamma, \delta \in R_1$
and $$B(\alpha, \beta) := B(V_\alpha, V_\beta)$$ for all $\alpha, \beta \in R_1.$ For $\mu \in \mathcal Q(R)$ let $$A(\mu) := \sum_{\alpha - \beta + \gamma - \delta = \mu}A(\alpha, \beta, \gamma,\delta), \quad B(\mu) = \sum_{2 (\alpha - \beta) = \mu }B(\alpha, \beta). $$
Then the submodule $I$ generated by all $A(\alpha, \beta, \gamma, \delta)$ and $B(\alpha, \beta)$ is equal to the submodule generated by all $A(\mu)$  and $B(\mu) , \mu \in \mathcal Q(R).$ We have a canonical grading on this submodule, where the homogeneous component of degree $\mu \in \mathcal Q(R)$ is given by $A(\mu) + B(\mu).$
\begin{prop}
The Lie algebra $\uider(V)$ is $\mathcal Q(R)$-graded with homogeneous components
\begin{equation}\uider(V)_\mu = (V^+ \otimes V^-)_{\mu}/ (A(\mu) + B(\mu)) = \sum_{\alpha \in R_1}(V_{\alpha}^+ \diamond V_{\alpha - \mu}^-) \label{uider_graded_components}, \end{equation}
and the epimorphism $\ud: \uider(V) \rightarrow \instr(V)$ is $\mathcal Q(R)$-graded. 
\label{uider_QR_grading_lem}
\end{prop}
\begin{proof}
By definition, $\uider(V)$ is the quotient of $V^+ \otimes V^-$ by the submodule generated by all $A(x,y, u,v)$ and $B(x,y).$ We have seen above that this submodule is $Q(R)$-graded with homogeneous components $A(\mu) + B(\mu).$ This implies (\ref{uider_graded_components}). The homogeneous components of $\instr(V)$ are $\instr(V)_{\mu} = \sum_{\alpha \in R_1} \delta(V_\alpha^+, V_{\mu - \alpha}^-).$ Under the Lie algebra epimorphism $\ud : \uider(V) \rightarrow \instr(V),$ $\ud(x_\alpha \diamond y_{\beta}) = \delta(x_{\alpha}, y_{-\beta}) \in \instr(V)_{\alpha - \beta}.$ Thus $\ud$ is graded. 
\end{proof}

We will now describe the kernel of $\ud$, see Corollary \ref{ker_ud_supporr_cor}. So far, very little use has been made of the explicit classification of $3$-gradings which can for instance be found in \cite[17.8 and 17.9]{lfrs}, but for what follows we will really have to look at each $3$-grading individually.  The notation for $3$-graded root systems can also be found in \cite{lfrs}.

\label{deg_sum_in_defect}
Assume that $V$ is defined over $k_R.$
By the above, $ \ud : \uTKK(V) \rightarrow \TKK(V) $ is a $\mathcal Q(R)$-graded central covering with kernel $\ker \ud = \mathrm{HC}(V)$ and $\TKK(V)$ is $R$-graded (Theorem~\ref{Root_graded_grid_theo}). We know by Proposition~\ref{torsion_bijection_prop}, that with respect to the $\mathcal Q(R)$-grading
$\HF(V) =  \bigoplus_{\gamma \in \mathbb{DS}(R) } (\HF(V))_\gamma \oplus (\HF(V))_0.$
where $\mathbb{DS}(R)$ was defined in Definition~\ref{deg_sum_defi}.
In addition, $ \ud : \uTKK(V) \rightarrow \TKK(V) $ is a $3$-graded central extension with respect to the $3$-grading $(R, R_1).$ With respect to this $3$-grading, $\HF(V)$ has degree $0,$ so it follows that, if $ \alpha - \beta = \gamma \in \mathbb{DS}(R),$  and $\HF(V)_\gamma \neq \{0\},$ the   degenerate pair (see Definition~\ref{deg_sum_defi}) $(\alpha, -\beta)$ lies in $R_1 \times -R_1.$ \\
Degenerate sums occur either with divisor $2$ or $3.$ Assume that $\alpha - \beta = \gamma$ is a degenerate sum and $n_\gamma = 3.$ Then $R= A_2$ or $R = G_2.$ Since the root system $G_2$ is not $3$-graded, we can restrict to the case $n_\gamma = 3$ and $R = A_2.$ In this  case, the grid is of type $A_2^{coll}$ and  without loss of generality $R_1 =   \{\epsilon_i -\epsilon_j, \epsilon_i -\epsilon_k \}$  for  $\{i,j,k\} = \{1 ,2, 3\}.$ Moreover, $R_1 - R_1 =  \{\epsilon_k -\epsilon_j, 0, \epsilon_j -\epsilon_k \}$. Comparison with table \ref{tab:DegSums} tells us that none of these three elements is a degenerate sum. Thus  if $n_\gamma = 3,$ then $\gamma$ is not in the support of $\HF(V).$ We may therefore assume from now on that $n_\gamma = 2.$ 

In the following if $\gamma \in \mathcal Q(R)$ and $L$ a $\mathcal Q(R)$-graded Lie algebra, then $L_\gamma$ will mean the $\gamma$-component with the respect to the $\mathcal Q(R)$-grading and if $L$ is a $\mathbb Z$-graded Lie algebra, then $L^n,$ $n \in \mathbb Z,$ is the $n$-homogeneous component. In this sense $\uTKK(V)^0 = \uider(V) = \bigoplus_{\gamma \in \mathcal Q(R)} \uider(V)_\gamma$ and $(\uTKK(V))_\gamma = \uider(V)_\gamma$ for $\gamma \in  (R_1 + R_{-1}).$\\
By  Proposition~\ref{torsion_bijection_prop} the following holds: 
\begin{center}
If $1/2 \in k,$ then  $\ker \ud \subset \uider(V)_0.$ \label{deg_sum_in_defect_one_half}
\end{center}
Assume that $\mathrm{HC}(V)_\gamma \neq 0$ and $\gamma \neq 0.$
Since $n_\gamma = 2,$ we know that there must be a degenerate pair $(\alpha, -\beta),$ $\alpha, \beta \in R_1$, such that $\gamma = \alpha - \beta$ and $\alpha \perp \beta$ (see Lemma~\ref{vdK_2.6}). 
\begin{itemize}
\item[-] $A^{coll}.$ Here, any two distinct roots in $R_1$ are collinear (hence the name). Thus a degenerate sum $\gamma$ cannot lie in $R_1 + R_{-1}$ and we obtain
\begin{prop}
If $V$ is covered by a collinear grid $A^{coll},$ then $\ker \ud \subset \uider(V)_0. $
\label{deg_sum_in_defect_coll}
\end{prop}
 \item[-] $A_I^J.$ We know, see Table~\ref{tab:DegSums}, that only $A_1 = C_1$, $A_2$ and $A_3$ admit degenerate sums. The case $A_1$ falls under the case $C_1.$ For $A_2,$ the rectangular grid is isomorphic to a collinear grid and this case has been treated.  The root system $A_3$ admits $2$ non-isomorphic $3$-gradings, $A_3^{coll}$  and $ A_{3}^2$, The first one was already considered.  Assume therefore without loss of generality  that $\card K = 4$, $I = \{1, 3\}$ and $J = \{2, 4\}.$ Then 
 $R_1 = \{\epsilon_1 - \epsilon_2,  \epsilon_1 - \epsilon_4, \epsilon_3 - \epsilon_2,   \epsilon_3 - \epsilon_4\}. $
 Every element of $\mathbb{DS}(A_3)$  is of the form $\sum (-1)^{\nu_i} \epsilon_i$ where $\prod_{i= 1}^4 (-1)^{\nu_i} = 1.$
 If $\alpha - \beta \in R_1 + R_{-1},$ $\alpha \neq \beta,$ then $\alpha \perp \beta$ if and only if $\alpha = \epsilon_i - \epsilon_j$ and $ \beta = \epsilon_m - \epsilon_n$ with $i \neq m, \, j \neq n.$ Thus $\alpha - \beta = \epsilon_i - \epsilon_j - \epsilon_m + \epsilon_n$ is a degenerate sum (see Table~\ref{tab:DegSums}). We obtain every $\gamma,$ which is a degenerate sum, as such a difference except $\pm( \epsilon_1 - \epsilon_2 + \epsilon_3 - \epsilon_4).$ 
 \label{deg_sum_in_defect_rect}
\begin{prop}
If $V$ is covered by a rectangular grid $A_{3}^2,$ then $\ker \ud \subset \uider(V)_0 + \sum_{j, k \in J, j\neq k } \uider(V)_{\epsilon_1 - \epsilon_j -  \epsilon_3 + \epsilon_k} +  \sum_{j, k \in J, i\neq k } \uider(V)_{-\epsilon_1 + \epsilon_j +  \epsilon_3 - \epsilon_k}  .$ 
\end{prop}
 \begin{prop}
If $V$ is covered by a rectangular grid $\dot A_{K}^J,$ $\card K > 4,$ then $\ker \ud \subset \uider(V)_0 .$ 
\end{prop}
 \item[-]\label{deg_sum_in_defect_odd_quad} $B_I^{qf}.$ The positive roots are $\{ \epsilon_{\infty} \pm  \epsilon_i\} \cup \{\epsilon_\infty\}$ and $\alpha, \beta \in R_1$ are orthogonal if and only if $\alpha = \epsilon_\infty \pm \epsilon_i$  and $\beta = \epsilon_\infty \mp
 \epsilon_i. $  Then $\alpha - \beta = \pm 2  \epsilon_i$  is a degenerate sum. 
 \begin{prop}
If $V$ is covered by an odd quadratic grid $B_{I}^{qf},$ then $\ker \ud \subset \uider(V)_0 + \sum_{i \in I}  \uider(V)_{\pm 2 \epsilon_i} .$ 
\end{prop}
 
 \item[-]\label{deg_sum_in_defect_even_quad} $D_I^{qf}.$ The positive roots are $\{ \epsilon_{\infty} \pm  \epsilon_i\} $ and $\alpha, \beta \in R_1$ are orthogonal if and only if $\alpha = \epsilon_\infty \pm \epsilon_i$  and $\beta = \epsilon_\infty \mp
 \epsilon_i. $  Then $\alpha - \beta = \pm 2  \epsilon_i$  is a degenerate sum. 
 \begin{prop}
If $V$ is covered by an even quadratic grid $D_{I}^{qf},$ then $\ker \ud \subset \uider(V)_0 + \sum_{i \in I}  \uider(V)_{\pm 2 \epsilon_i} .$ 
\end{prop}
\item[-] \label{deg_sum_in_defect_alt}  $D_I^{alt}.$ The positive roots are $\{ \epsilon_i + \epsilon_j, i, j \in I \} $ and $\alpha, \beta \in R_1$ are orthogonal if and only if $\alpha = \epsilon_i  +  \epsilon_j$  and $\beta = \epsilon_n  +  \epsilon_m,$ with $\{i,j\} \cap \{n,m\} = \emptyset.$ Comparing with Table~\ref{tab:DegSums}  yields that $\alpha - \beta = \epsilon_i  +  \epsilon_j - \epsilon_n  - \epsilon_m$ is a degenerate sum if and only if $\card I = 4.$ 
Thus
\begin{prop}
If $V$ is covered by an alternating  grid $D_{4}^{alt},$ then $\ker \ud \subset \uider(V)_0 + \sum_{i, j, m,n \neq, i<j, m <n}   \uider(V)_{\epsilon_i + \epsilon_j - \epsilon_n - \epsilon_m}.$ (where we choose any total order $<$ on $I$).
\end{prop}
\begin{prop}
If $V$ is covered by an alternating grid $D_{I}^{alt},$ $\card I > 4$, then $\ker \ud \subset \uider(V)_0.$ 
\end{prop}

\item[-] \label{deg_sum_in_defect_even_caley_albert}  Bi-Caley grid and Albert grid. The root systems of type $E$ do not admit any degenerate sums (Table~\ref{tab:DegSums}),  therefore
\begin{prop}
If $V$ is covered by a an Bi-Caley grid or by an Albert grid, then $\ker \ud \subset \uider(V)_0.$ 
\end{prop}

\item[-] $C_I^{her}.$ 
 In the case where $R =  C_I$, we always assume that $1/2 \in k_R$ therefore  
\begin{prop}
If $V$ is covered by a grid a hermitian grid $C_I^{her},$ then $\ker \ud \subset \uider(V)_0.$ \label{deg_sum_in_defect_herm}
\end{prop}
However, it might be useful to compute the intersection $\mathbb{DS}(R) \cap (R_1 + R_{-1}).$
 If $V$ is covered  by hermitian grid, then two roots $\alpha, \beta \in R_1$ are orthogonal, if and only if $\alpha = \epsilon_i + \epsilon_j$, $\beta = \epsilon_m + \epsilon_n$ with $\{i,j\} \cap \{m,n\} = \emptyset.$ Thus
$\alpha - \beta = 2(\epsilon_i - \epsilon_m)$ or $\alpha - \beta = \epsilon_i + \epsilon_j - \epsilon_m - \epsilon_n$ for $i,j,m,n \neq.$ A look at Table~\ref{tab:DegSums}, reveals that only the first ones are  degenerate sum. Thus 
$$(R_1 + R_{-1}) \cap \mathbb{DS}(R) = \{2 (\epsilon_ i - \epsilon_j): i \neq j \}.$$

\end{itemize}

\begin{cor} Let $k =  k_R$ and $V$ a Jordan pair which is covered by an  $(R, R_1)$-grid. The kernel of $\ud : \uider(V) \rightarrow \instr(V)$, that is $\HF(V)$, is $\mathcal Q(R)$-graded and
$$\supp_{\mathcal Q(R)}(\HF(V)) \subset \big( \mathbb{DS}(R) \cap (R_1 + R_{-1}) \big)\cup \{0\}. $$
We have $\HF(V) = \HF(V)_0$ if
\begin{itemize}
\item[\rm(i)] $1/2 \in k,$  thus in particular for $R = C_I,$ or
\item[\rm(ii)] $\mathbb{DS}(R) \cap (R_1 + R_{-1}) = \emptyset$ and this happens if and only if
$R_1 = \dot A^{coll},$ or $R = \dot A_I$, $D_I$, $E_6$, $E_7,$ $E_8,$ $\card I > 4.$   
\end{itemize}
\label{ker_ud_supporr_cor}
\end{cor}

In  \cite[Lemma 2.5]{Neh1996}, the space $D =  \sum_{\alpha, \beta \in R_1, \alpha \perp \beta } \subset [L^1, L^{-1}]$ was called the defect set. In the same paper, the author proves that this space is central and has $2$-torsion. 
 In view of our Proposition~\ref{torsion_bijection_prop}, those facts are immediate (see Corollary~\ref{LemmaNeh1996}), since the weights  of homogeneous components of the defect are degenerate sums (and not just any difference of orthogonal positive roots). We have therefore obtained a more precise version of the cited result which we can prove knowing only that the central covering $L \rightarrow \TKK(L^1, L^{-1})$ is $\mathcal Q(R)$-graded. 
 \begin{cor}[Lemma 2.5 \cite{Neh1996}] Let $L$ be a Lie algebra of Jordan type over $k$, such that $(L^+, L^-)$ covered by an $(R_1, R)$-grid. If $D = \sum_{\alpha \perp \beta, \alpha, \beta \in R_1} [L_\alpha, L_{-\beta}]$ then 
 $$D \subset Z(L), \quad 2D = 0. $$
 Moreover, if $\alpha \perp \beta, \alpha, \beta \in R_1$ and $[L_\alpha, L_{-\beta}] \neq \{0\},$ then $\alpha- \beta$ is a degenerate sum of divisor $2.$
 \label{LemmaNeh1996}
 \end{cor}
\begin{proof} Let $V =(L^+, L^-).$ By Theorem~\ref{Root_graded_grid_theo}, there is a graded central covering  $g: L \rightarrow \TKK(V)$ and $\TKK(V)$ is root-graded. Thus, since $\alpha \perp \beta, \alpha, \beta \in R_1$ implies that $\TKK(V)_{\alpha - \beta} = 0,$ we have $D \subset \ker f \subset Z(L).$ Moreover, by Proposition~\ref{torsion_bijection_prop}, $2D = 0,$ since every $\gamma$ in the support of $D$ is a degenerate sum of divisor  $2.$
\end{proof}

\begin{expl}
We would like to state explicitly, that for $L = \TKK(V)$ 
the universal central extension $u:\mathfrak{uce}(L) \rightarrow L$ is in general $5$-graded and in addition $(\ker u)_{\pm 2}$ and even $(\ker u)_{\pm 1}$ can be non-trivial. 
Consider for instance a Jordan pair which is covered by  a grid of type $B_{\{1,2\}}^{qf}.$ The $3$-grading is given by $R_1 = \{ \epsilon_{3} \pm  \epsilon_i, i= 1,2\} \cup \{\epsilon_3\}.$ The root system $B_3$ has degenerate sums with respect to $2$ and the corresponding weights are 
$ \{2 \varepsilon_i : 1\leq i \leq 3\} \cup \{\pm \epsilon_1 \pm \epsilon_2 \pm \epsilon_3\}.$
 Thus the support of the universal central extension is contained in $S = BC_3 \cup \{\pm \epsilon_1 \pm \epsilon_2 \pm \epsilon_3\}$ with the following $5$-grading which is induced by the $3$-grading on $R:$
 \begin{eqnarray*}
 S_0 &=& R_0 \cup \{\pm 2 \epsilon_1, \pm 2 \epsilon_2 \},\\
 S_{1} &=&  R_{1} \cup  \left \{ \pm \epsilon_j + \epsilon_3 \pm \epsilon_i : \{i,j, 3\} = \{1, 2, 3\} \right\},  \\
 S_{-1} &=&  R_{-1} \cup \{ \pm \epsilon_j - \epsilon_3 \pm \epsilon_i : \{i,j, 3\} = \{1, 2, 3\} \}, \\
 S_2 &=& \{2 \epsilon_3 \}, \\
 S_{-2} &=& \{-2 \epsilon_3 \}.
 \end{eqnarray*}
 Let $L = \mathfrak{so}(3, \mathbb Z)$ with its canonical $Q(B_3)$-grading. Then $V = (L^-, L^+)$ is covered by a grid of type $B_{\{1,2\}}^{qf}.$ In \cite{vdK} it is proven that the support of $u : \mathfrak{uce}(L) \rightarrow \TKK(V)$ is all of $S$ and that the kernel of $u$ has support $\mathbb{DS}(B_3) \cup \{0\}.$\\
 \end{expl}
 It will turn out that the collinear grid has particularly nice properties. To see this we need the following result \cite[Prop. 18.9]{lfrs}: \\
\begin{prop}
\label{sum_of_three_roots_prop}
Suppose $\alpha, \beta, \gamma \in R$ and $\alpha \neq \beta \neq \gamma.$ Then the following are equivalent:
\begin{itemize}
\item[\rm(a)] $\alpha - \beta + \gamma \in R_1,$
\item[\rm(b)] One of the following holds
\begin{itemize}
\item[\rm(i)] $\alpha \not \perp \beta \not \perp \gamma \perp \alpha,$ or
\item[\rm(ii)] $\alpha \top \gamma,$ and $\alpha \vdash \beta \dashv \gamma,$ or
\item[\rm(iii)] $(\alpha , \beta, \gamma)$ is a collinear family and $\alpha - \beta + \gamma \in R_1,$ or
\item[\rm(iv)] $\alpha = \gamma \vdash \beta.$
\end{itemize}
\end{itemize}
\end{prop}
\begin{rem} Case (b)(iii) can only occur if the covering grid is of type $C_I^{herm}.$
\label{C_is_impure_remark}
\end{rem}

The zero component of $\mathfrak{uider}(V)$ with respect to the root grading is the quotient of $(V^+ \otimes V^-)_0$ modulo the span of all relations 
$$A(\alpha, \beta, \gamma, \delta) \cap (V^+ \otimes V^-)_0, B(\alpha, \beta) \cap (V^+ \otimes V^-)_0. $$ 
For $B(\alpha, \beta)$, (\ref{B_hom_graded}) implies that we only need to consider $B(\alpha, \alpha).$ Next we discuss the $A$-relations: 
\begin{eqnarray*}
A(\alpha, \beta, \gamma, \delta) &\subset& V^{+}_{\epsilon_1} \otimes V^-_\delta + V^{+}_{\epsilon_3} \otimes V^-_\beta \\
 &+& V_{\alpha}^+ \otimes V_{\epsilon_4}^- + V_{\gamma}^+ \otimes V_{\epsilon_2}^-   
\end{eqnarray*}
where the elements $\epsilon_i$ are as defined in (\ref{epsilon_Defi}). This is homogeneous with respect to the $\mathcal \mathcal Q(R)$-grading. Suppose $\epsilon_1 = \alpha - \delta + \gamma = \delta.$ Then $\gamma = \beta - \alpha + \delta,$ $\delta - \gamma + \beta = \alpha.$ Hence $A(\alpha, \beta, \gamma, \delta)$ lies either in $\mathfrak{uider}(V)_0$ or in a homogeneous component outside of it. Therefore we need to determine 
$$A(\alpha, \beta, \gamma, \delta), \quad \alpha - \beta + \gamma = \delta. $$
This leads to three cases : (I) $\alpha = \beta, \gamma = \delta;$ (I') $\alpha = \delta, \beta = \gamma;$ (equivalent to (I) by symmetry) and (II) \begin{equation}\alpha \neq \beta \neq  \gamma, \;\alpha - \beta + \gamma = \delta. \label{case_2_sum_zero} \end{equation}


\begin{cor} If $V$ is covered by a grid of type $A^{coll},$ then $A(0) = \sum_{\alpha, \beta \in R_1}A(\alpha, \beta, \alpha, \beta)$ and $B(0) = \sum_{\alpha \in R_1}B(\alpha ,\alpha).$
\label{collinear_zero_hom_Cor}
\end{cor}
\begin{proof}
We have to determine the span of all relations 
$$A(\alpha, \beta, \gamma, \delta) \cap (V^+ \otimes V^-)_0, B(\alpha, \beta) \cap (V^+ \otimes V^-)_0 $$
where $B(\alpha, \beta) := B(V_\alpha^+, V_\beta^-). $
For $B(\alpha, \beta)$, (\ref{B_hom_graded}) implies that we only get $B(\alpha, \alpha). $ If $\alpha \neq \beta \neq  \gamma,$
then $(\alpha, \beta, \gamma)$ is collinear, hence by Remark~\ref{C_is_impure_remark}, $\alpha - \beta + \gamma \notin R_1.$ Therefore,  if $A(\alpha, \beta, \gamma, \delta) \subset (V \otimes V)_0$ then without loss of generality $\alpha = \gamma$ and $\beta = \delta.$
\end{proof}

\chapter{Types A and C}
\label{AC_chapter}
In this chapter we will describe  the universal central extension of $\TKK(V)$ for two important examples, namely for the rectangular matrix Jordan pairs and for the hermitian matrix Jordan pairs. It is known that over a field of characteristic $\neq 2,3$ the Tits-Kantor-Koecher algebras of those pairs are precisely the centreless root-graded Lie algebras of types $A$ and $C$ respectively (see \cite{Neh1996}). Although over a general base ring, we can only prove that $\TKK(V)$ is $A$-graded resp. $C$-graded (and not that every such root-graded Lie algebra is of this form), this gives us some insight into a possible theory of root-graded Lie algebras over $k$. It is also a straightforward generalization of the approach which was chosen in \cite{vdK} where the author looks at Lie algebras of the form $k \otimes_{\mathbb Z} \mathfrak g_\mathbb{Z}$ for a simply connected Chevalley form $\mathfrak g_{\mathbb Z}$ of  a simple finite dimensional  Lie algebra.
\section{Derivations }
\label{derivations_section}
If $L$ is a root-graded Lie algebra, see Definition~\ref{r_delta_graded_def}, its structure can often be described by another algebra $D$, which is called the coordinate algebra. There is a close connection between derivations of $D$ and derivations of certain Jordan pairs which we explore in this chapter. A good reference for alternative algebras which are needed throughout is \cite{prebookalt}.
  In the paper \cite{lopera} the authors approach derivations of an alternative algebra $D$ in a way which is well suited to our setting.\\ 
Recall  Definition \ref{Mul_Der_Defi} for the definition of the Lie multiplication algebra of an alternative algebra $D$. 
 If $D$ is an alternative (possibly commutative or associative) unital algebra, it can be shown (see \cite{lopera}) that a multiplication derivation  of $D$ is an  element $\Delta \in \End_k(A)$ of the form
$$\Delta = L_a - R_a  + \sum_{i = 1}^n [L_{a_i}, R_{b_i}]$$ for some $a, a_i, b_i \in A$ such that
$$ 3a + \sum_{i = 1}^n[a_i, b_i] \in \mathrm{Nuc}(A).$$ 
This will be considered as the definition of a multiplication derivation whenever alternative algebras are concerned. We will define special types of multiplication derivations following \cite[2.5]{lopera}. There the reader can also find proofs that those concepts are well-defined.
\begin{defi}
An \emph{inner derivation} is a multiplication derivation  $L_a - R_a  + \sum_{i = 1}^n [L_{a_i}, R_{b_i}]$ such that $3a + \sum_{i = 1}^n[a_i, b_i] = 0. $ An \emph{associator derivation} is an inner derivation where $a = 0.$ The submodule of \emph{standard derivations} is spanned by elements of the form $SD(a,b) = L_{[a, b]} - R_{[a, b]} + 3 [L_a, R_b].$ Lastly, \emph{commutator derivations} are those inner derivations where $\sum [L_{a_i}, R_{b_i}] = 0.$ We denote, with obvious names, 
\begin{eqnarray*}
\mathrm{Der}(D) &:=& \{\Delta \in \End_k(D) :\Delta(ab) = \Delta(a)b + a\Delta(b) , a,b \in D \}, \\
\mathrm{MulDer}(D) &:=& \spa_k \{L_a - R_a  + \sum_{i = 1}^n [L_{a_i}, R_{b_i}]: 3a + \sum_{i = 1}^n[a_i, b_i] \in \mathrm{Nuc}(D)\},\\
\mathrm{IDer}(D) &:=& \spa_k \{L_a - R_a  + \sum_{i = 1}^n [L_{a_i}, R_{b_i}]: 3a + \sum_{i = 1}^n[a_i, b_i] = 0.\},\\
\mathrm{AssDer}(D) &:=& \spa_k \{ \sum_{i = 1}^n [L_{a_i}, R_{b_i}]:  \sum_{i = 1}^n[a_i, b_i] = 0.\},\\
\mathrm{StanDer}(D) &:=& \spa_k\{L_{[a, b]} - R_{[a, b]} + 3 [L_a, R_b]: a,b \in D\},\\
\mathrm{ComDer}(D)&:=& \spa_k \{L_a - R_a  + \sum_{i = 1}^n [L_{a_i}, R_{b_i}]: 3a =0, \sum_{i = 1}^n[a_i, b_i] = 0.\}.
\end{eqnarray*}
\end{defi}
The following proposition is part of  \cite[Prop. 3.7]{lopera}.
\begin{prop}
Let $D$ be an alternative $k$-algebra. 
\begin{itemize}
\item[\rm (i)] The submodules $\mathrm{MulDer}(D) , \mathrm{IDer}(D) ,\mathrm{StanDer}(D),\mathrm{ComDer}(D)$ are ideals of the Lie algebra $\mathrm{Der}(D)$.
\item[\rm (ii)] If $3 [D, D] = [D, D]$ (in particular if $D$ is commutative), then
$$\mathrm{IDer}(D) = \mathrm{AssDer}(D) + \mathrm{StanDer}(D) + \mathrm{ComDer}(D).$$ 
\item[\rm (iii)] If $3 D = D,$ then 
$$\mathrm{IDer}(D) =  \mathrm{StanDer}(D) + \mathrm{ComDer}(D).$$ 
\item[\rm (iv)] If $1/3 \in k,$ then 
$$\mathrm{IDer}(D) =  \mathrm{StanDer}(D), \;\mathrm{ComDer}(D) = 0.$$ 
\item[\rm (v)] If $3 D = 0,$ then
$$\mathrm{IDer}(D) = \mathrm{AssDer}(D) + \mathrm{ComDer}(D).$$
\end{itemize}
\label{lopera_der_proposition}
\end{prop}

\subsubsection{Trialities}
\label{triality_subsubsection}
Throughout let $D$ be an alternative unital algebra, see Definition~\ref{alternative_algebra_definition}. 
\begin{defi}
 A triple of elements $(t_1, t_2, t_3)$, $t_i \in \End_k(D)$ is called an (additive) \emph{triality} of $D$ if:
$$t_1(ab) = t_2(a)b + at_3(b)$$
for $a, b \in D.$
The trialities form a Lie subalgebra of $\End_k(D)^3$ which is denoted by $\mathcal T(D).$
\end{defi}
\begin{rem}\label{der_is_trial_remark}
Every derivation $D \in \mathrm{Der}(D)$ allows us to define a triality by setting $t_i = D$ for all $1 \leq i\leq 3.$ Conversely, every triality with $t_1 = t_2  = t_3$ corresponds to a derivation $D = t_i.$
\end{rem}
\begin{rem} There exist multiplicative trialities, but since we are interested in Lie algebras and not in groups, the additive version is better adapted to our purposes. Sometimes the objects which we call trialities  are called (Lie) related triples in the literature and the expression (local, additive) triality is reserved for orthogonal trialities. See Definition~\ref{orth_trial_defi} and the paragraph following it for more information.  
\end{rem}

\begin{lem} Let $(t_1, t_2, t_3)$ be a triality. Then $t_1 = t_2 = t_3,$ if and only if $t_1(1) = t_2(1) = t_3(1)= 0.$

\label{inner_trial_char_zero}
\end{lem}
\begin{proof} We already know that $t_1 = t_2 = t_3$ implies that $t_1$ is  a derivation. 
For every derivation $\Delta$, $\Delta(1) = 0$, so $\Rightarrow$ is clear. 
Assume $t_i(1) = 0$ for all $i.$
For all $x \in D$, we have $t_1(x) = t_2(x)1 + xt_3(1) = t_2(1)x + 1t_3(x)$ by triality. Thus
$t_1(x)= t_2(x) = t_3(x).$
\end{proof}

\begin{expl} The triples $\lambda(a) = (L_a , L_a + R_a, -L_a )$ , $\rho (b) = (R_b, -R_b, L_b + R_b)$  $a, b \in D.$
are trialities. Indeed,  in every alternative algebra 
$$-(a, x, y) = L_a(xy) - (L_ax)y = (R_ax)y  - x(L_ay) = (x,a, y)$$  and
$$( x, y, b) = R_b(xy) - x(R_by) = -(R_bx)y  + x(L_by) = -(x, b, y).$$
Since $\mathcal T(D)$ is a Lie algebra, $\sigma(a, b) := [\lambda(a), \rho(b)]$ is a triality. More explictly, it is of the form 
$(t_1, t_2,t_3 ) = [\lambda(a), \rho(b)],$ where
\begin{eqnarray*}
t_1 &=& [L_a, R_b], \\ 
t_2 &=& -[L_a, R_b] - [R_a, R_b] = -[L_a, R_b] + 2[L_a, R_b]+ R_{[a, b]} = R_{[a, b]} + [L_a, R_b],\\
t_3 &=& -[L_a, L_b] - [L_a, R_b] = 2[L_a, R_b] -[L_a, R_b] - L_{[a, b]} = -L_{[a, b]} + [L_a, R_b].
\end{eqnarray*}
Thus $$ \sigma(a, b) = ([L_a, R_b], [L_a, R_b],[L_a, R_b]) + (0, R_{[a, b]}, -L_{[a, b]}).$$
\label{rho_lambda_expls}
\end{expl}

\begin{defi} \label{inner_trial_def} 
A triality  $(t_1, t_2, t_3)$ is called an \emph{inner triality}, if it lies in the subalgebra  $\mathcal T_0$ generated by the trialities 
$\lambda(a) = (L_a , L_a + R_a, -L_a )$, $\rho (b) = (R_b, -R_b, L_b + R_b),$  $a, b \in D.$
\end{defi}
\begin{lem} A triality $(t_1, t_2, t_3)$ is inner if and only if there exist $a, b \in D$ and finitely many pairs $(a_i, b_i) \in D^2$ such that
$$(t_1, t_2, t_3) = \lambda(a) + \rho(b) + \sum_{i} \sigma(a_i, b_i).$$ 
Moreover,  \begin{subequations}
\begin{equation}
[\lambda(c), \sigma(a, b)] = \sigma([a, b], c) - \lambda((a, b,c )), \quad  [\rho(c), \sigma(a, b)] = \sigma(c, [a, b]) - \rho((a, b,c )), \label{lambda_action}
\end{equation}
\begin{equation} [\rho(a), \rho(b)] = -2 \sigma(a, b) - \rho([a,b]), \quad [\lambda(a), \lambda(b)] = -2\sigma(a,b) + \lambda([a,b])\label{rho_action}
\end{equation}
\end{subequations}
 \end{lem}
\begin{proof} 
It is enough to show that the identities (\ref{lambda_action})  and (\ref{rho_action}) hold. It follows then that the submodule of $\mathcal T_0$ spanned by all $\lambda(a)$, $\rho(b)$ and $\sigma(a, b)$ is a subalgebra which contains the generators of $\mathcal T_0,$ i.e., this submodule is already all of $\mathcal T_0.$
Recall the identities (\ref{alt_id_1}) to (\ref{alt_id_4}) which will frequently be used in this proof. 
The components of 
$[\rho(a), \rho(b)] = [(R_a, -R_a, L_a + R_a), (R_b, -R_b, L_b + R_b)]$ are
\begin{eqnarray*}
t_1 = [R_a, R_b] &=& -2 [L_a, R_b]-R_{[a, b]}  \\ 
t_2 = -2 [L_a, R_b]-R_{[a, b]} &=& -2 [L_a, R_b]- 2 R_{[a, b]} + R_{[a, b]}\\
t_3 =  [L_a + R_a, L_b + R_b] &=&  -2 [L_a, R_b] + L_{[a, b]} - R_{[a, b]}
\end{eqnarray*}
which follows by (\ref{alt_id_1}) and (\ref{alt_id_2}). 
Thus $[\rho(a), \rho(b)] = -2 \sigma(a, b) - \rho([a, b])$
and similarly we show  $[\lambda(a),\lambda(b) ] = -2\sigma(a,b) + \lambda([a,b]).$\\
If  
$(s_1, s_2, s_3) = [\lambda(c), \sigma(a, b)] ,$
then by (\ref{alt_id_3}) and (\ref{alt_id_4}):
\begin{eqnarray*}
s_1 = [L_c, [L_a, R_b]] &=&  [L_{[a, b]}, R_c] - L_{(a,b, c)}\\ 
s_3 = -[L_c , -L_{[a, b]} + [L_a, R_b]] &=&  [L_c , L_{[a, b]}] - [L_c, [L_a, R_b]]\\
&=& L_{(a, b, c)} - [L_{[a, b]}, R_c] + L_{[[a, b]c]} + 2[L_{[a, b]}, R_c] \\
&=&  [L_{[a, b]}, R_c] + L_{[[a, b]c]} + L_{(a, b, c)}.
\end{eqnarray*}
Since the first and the third component of a triality uniquely determine the second component, it follows that $[\lambda(c), \sigma(a, b)] = \sigma([a, b], c) - \lambda((a, b,c ))$ and similarly for $\lambda(c)$ replaced by $\rho(c).$
\end{proof}

The following proposition should be known, but, lacking a reference, we include a proof. 
\begin{prop} A triality $(\Delta, \Delta, \Delta)$ is inner if and only if $\Delta$ is an inner derivation. The inner derivations embed into $\mathcal T_0$ using the map $\Delta \mapsto (\Delta, \Delta, \Delta).$ 
\label{inner_der_is_inner_trial_prop}
\end{prop}

\begin{proof}
Assume that $\Delta$ is a derivation such that $(\Delta, \Delta,\Delta)$ is an inner triality. 
Then there are  elements $a, b$, $a_i$,$b_i$, $1 \leq i \leq n$ in $D$ such that 
\begin{eqnarray*}
L_a + R_b +  \sum_{i = 1}^n [L_{a_i}, R_{b_i}] &=& \Delta\\
L_a + R_a -R_b + \sum_{i}  R_{[a_i,b_i]} + \sum_{i = 1}^n [L_{a_i}, R_{b_i}]&=& \Delta\\
-L_a + R_b + L_b + \sum_{i = 1}^n-L_{[a_i,b_i]} +  [L_{a_i}, R_{b_i}]   &=& \Delta
\end{eqnarray*}
Note that such a $\Delta$ is automatically a Lie multiplication derivation. 
Since $\Delta$ is a derivation, we get $\Delta(1) = 0.$ This implies 
\begin{eqnarray*}
a + b  &=& 0\\
2a -b + \sum_{i = 1}^n  [a_i,b_i]  &=& 0\\
-a + 2b  - \sum_{i = 1}^n [a_i,b_i]    &=& 0
\end{eqnarray*}
This gives $a = -b$
and the second and the third equation are equivalent to
\begin{eqnarray*}
3b -  \sum_{i}[a_i,b_i]&=& 0 \\	
\end{eqnarray*}
Thus
\begin{eqnarray*}
a   &=& -b\\
3a + \sum_{i}  [a_i,b_i]  &=& 0\\
\end{eqnarray*}
Therefore  $\Delta = L_a - R_a +  \sum[L_{a_i}, R_{b_i}])$ with $3a + \sum_{i}  [a_i,b_i]  = 0$ which means that $\Delta$ is inner. 
 Conversely, if $\Delta = L_a -R_a + \sum([L_{a_i}, R_{b_i}])$ is such that $3a + \sum[a_i, b_i] = 0,$ then  $\Delta$ can be obtained as any component of the inner triality
$$ (t_1, t_2 , t_3) = \lambda(a) -\rho(a) + \sum \sigma( a_i,  {b_i}). $$ 
Indeed, 
$t_1  = L_a -R_a + \sum([L_{a_i}, R_{b_i}])$, $t_2 = L_a + R_a + R_a + \sum L_{[a_i, b_i]} + \sum([L_{a_i}, R_{b_i}]) = 
L_a + R_{2a} - R_{3a} + \sum([L_{a_i}, R_{b_i}]) = \Delta$. Since a triality is uniquely determined by any two of its components, it follows that $t_3 = \Delta$ as well. 
Therefore $(t_1, t_2,t_3)$ is an inner triality.  The Lie product of two inner derivations is again an inner derivation. hence the map $\Delta \rightarrow (\Delta, \Delta,\Delta)$ is a Lie algebra monomorphism onto the submodule of diagonal elements in $\mathcal T_0.$
\end{proof}

\begin{defi}  Assume that $D$ is an alternative algebra with involution $\bar{\;}.$ Define a $k$-linear map on $\mathrm{End}_k( D)$ by 
$$\bar{T}(x) = \overline{T(\bar x)} $$
where $\bar{\;}$ is the involution on $D.$ 
\end{defi}
\begin{lem}\label{Lie_alg_auto_from_alt_involution} 
The map $T \mapsto \bar{T}$ is a Lie algebra automorphism $\mathrm{End}_k(D) \rightarrow \mathrm{End}_k(D)$. Moreover, if $\Delta$ is a derivation, then so is $\overline \Delta$ and if $(t_1, t_2, t_3)$ is a triality, then $(\bar t_1, \bar t_3, \bar t_2)$ is a triality as well. 
\end{lem}
\begin{proof}
Let $T, S \in \mathrm{End}_k(D)$ and $a \in D.$ Then
\begin{eqnarray*}
& & [\bar T, \bar S](x) = \bar T(\bar S(x)) - \bar S (\bar T (x))
= \bar T (\overline{S(\bar x)}) - \bar S (\overline{T(\bar x)})  \\
&=& \overline{T(S(\bar x))} - \overline{S(T(\bar x))} 
= \overline{T(S(\bar x)) - S(T(\bar x))}
= \overline{[T, S]}(x).
\end{eqnarray*}
This proves that $\bar{\;}$ is a Lie algebra homomorphism. 
Moreover,  $\overline{\bar T}(x) = \overline{\overline {T(\overline{\bar x})}}= T(x)$ so that $\bar{\;}$ is of order $2$ and hence in particular invertible. 
If $\Delta$ is a derivation, then  \\
$\overline{\Delta}(xy) = \overline{\Delta(\bar y)\bar x + \bar y \Delta(\bar x)} = \overline{\Delta}(x) y + x \overline{\Delta}(y),$ thus $ \overline{\Delta}$ is a derivation.
If $(t_1, t_2, t_3)$ is a triality, then
\begin{eqnarray*}
\overline{t_1}(xy) &=& \overline{t_1(\bar y \bar x)} \\
&=& \overline{t_2(\bar y) \bar x +  \bar y t_3(\bar x)} \\
&= & x \overline{t_2}(y) +  \overline{t_3}(x)y
\end{eqnarray*}  
which proves that $(\bar t_1, \bar t_3, \bar t_2)$ is a triality. 
\end{proof}

\begin{lem} Assume that $1/3 \in k$ and let $\mathrm{Der}(D)$ be the Lie algebra of derivations of the alternative algebra $D$ and let $\mathcal T$ be the triality algebra of $D$. Then there is an isomorphism of $k$-modules
$$ h : \mathrm{Der}(D) \oplus D^2 \rightarrow \mathcal T, (\Delta,a,b) \mapsto (\Delta, \Delta, \Delta) + \lambda(a) - \rho(b) . $$
The inverse is given by $$ h^{-1} : (t_1, t_2, t_3) \mapsto (t_1 - L_a + R_b, a, b) $$
where $3a =2 t_2(1) + t_3(1)$, $3 b =  -t_2(1) - 2t_3(1).$\\
This map induces an isomorphism  $\mathrm{StanDer}(D) \oplus D^2 \rightarrow \mathcal T_0$ where $\mathcal T_0$ denotes the inner trialities. 
\label{trial_der_alt_lemma} \label{inner_Der_inner_trial}
\end{lem}
\begin{proof} 
By Remark~\ref{der_is_trial_remark} and Example~\ref{rho_lambda_expls}, the  map $h$ is well-defined.
The formulas for $a$ and $b$ are equivalent to $t_2(1) = 2a + b$ and $t_3(1) = -a - 2b.$ Also note that for every triality $(t_1, t_2, t_3)$ we have $t_1(x) = t_2(x) + xt_3(1)$ and $t_1(y) =  t_2(1)y + t_3(y)$ so that 
 \begin{equation}t_2 = t_1 + R_{a + 2b} \mbox{ and } t_3(y) = t_1 - L_{2a +b} \label{t_2_t_3in_terms_oft_1}.\end{equation}
\begin{eqnarray*} (t_1 - L_a + R_b)(xy) &=& t_2(x)y + xt_3(y) - L_a(xy) + R_b(xy)\\
&=&  (t_1(x) + R_{a + 2b}x)y + x(t_1(y) - L_{2a +b}y)- L_a(xy) + R_b(xy)\\
&=& t_1(x)y +xt_1(y ) +(R_bx)y -  x(L_a y) \\
&&+ (R_ax)y  + (R_bx)y - x(L_ay) - x(L_by)- L_a(xy) + R_b(xy)\\
&=& t_1(x)y +xt_1(y ) +(R_bx)y  - (L_ax)y  + x(R_by) - x(L_a y) \\
&&+ (R_ax)y  + (R_bx)y - x(L_ay) - x(L_by) \\&&- L_a(xy) + R_b(xy) - ( -(L_ax)y + x(R_by))\\
&=& t_1(x)y +xt_1(y ) +(R_bx)y  - (L_ax)y-  + x(R_by) + x(L_a y) \\
&&+ (R_ax)y - x(L_ay)  + (R_bx)y  - x(L_by) \\ &&- L_a(xy)+ (L_ax)y  + R_b(xy) -  x(R_by)\\
&=& t_1(x)y +xt_1(y ) +(R_bx)y  - (L_ax)y-  + x(R_by) - x(L_a y) \\
&&+ (x,a,y) +(x,b y)+ (a,x,y) + (x,y, b)\\
&=&(t_1(x)  - (L_ax)  +(R_bx))y  +x(t_1(y )   + (R_by) - (L_a y) )
\end{eqnarray*}  
Thus $t_1 - L_a + R_b$ is a derivation which proves that $h^{-1}$ is well-defined.  \\
Let $(t_1, t_2, t_3)$ be an inner triality. Then with $a$ and $b$ as in the formula for $h^{-1}$
\begin{eqnarray*}
hh^{-1}(t_1, t_2, t_3) &=& h(t_1 - L_a  + R_b,a, b) \\
&=& (t_1 - L_a + R_b + L_a + R_b,t_1 - L_a + R_b + L_a \\ &&+ R_{a+b}, t_1 - L_a + R_b - L_{a+b} - R_b ) \\
&=& (t_1, t_1 + R_{2a + b}, t_1 - L_{2a + b}) \\
&=& (t_1, t_2, t_3)
\end{eqnarray*}
by (\ref{t_2_t_3in_terms_oft_1}). \\
If $\Delta \in \mathrm{Der}(D)$ and $a, b \in D,$ then 
\begin{eqnarray*}
h^{-1}h(\Delta, a, b) &=& h^{-1}(\Delta + L_a - R_b, \Delta + L_a + R_{a+b}, \Delta - L_{a+b} - R_b).
\end{eqnarray*}
Since $(\Delta + L_a + R_{a+b})(1) = 2a + b$ and $(\Delta - L_{a+b} - R_b)(1) = - a - 2b,$ we obtain 
 $2(\Delta + L_a + R_{a+b})(1) + (\Delta - L_{a+b} - R_b)(1) = 3a$ and $ -(\Delta + L_a + R_{a+b})(1) - 2(\Delta - L_{a+b} - R_b)(1) = 3b$ and therefore
\begin{eqnarray*}
&& h^{-1}(\Delta + L_a - R_b, \Delta + L_a + R_{a+b}, \Delta - L_{a+b} - R_b)\\
&=& (\Delta + L_a -R_b - L_a + R_b, a, b) \\
&=& (\Delta, a, b) 
\end{eqnarray*}
Thus $h$ and $h^{-1}$ are mutually inverse to each other. We could define a Lie algebra structure on $\mathrm{Der}(D) \oplus D \oplus D$ by means of the isomorphism $h,$ but the formulas are not particularly useful for our purposes here. It is however obvious that this Lie bracket would be compatible with the Lie bracket on $\mathrm{Der}(D).$
Assume that $\Delta \in \mathrm{Der}(D)$ is an inner derivation. Since $1/3 \in k$, every inner derivation is a standard derivation and vice versa (Proposition~\ref{lopera_der_proposition}). Then the image of $\Delta$ in $\mathcal T_0$ is an inner triality, see Lemma~\ref{inner_Der_inner_trial}.
  Also,  for all $a, b \in D$, $\lambda(a) - \rho(b) =(L_a - R_b, L_a + R_{a + b}, - L_{a+b} - R_b)$ is an inner triality. Therefore, $h(\mathrm{IDer}(D) \oplus D^2) \subset \mathcal T_0.$ Let $(T_1, T_2, T_3)$ be an inner triality. Then $T_1 + L_a - R_b$ is a derivation (where, of course, $a$ and $b$ are as above) and we only have to show that it is inner.  Since $(T_1, T_2, T_3)$ is an inner triality, $(\Delta_1, \Delta_2, \Delta_3) = (T_1, T_2, T_3) - (L_a - R_b, L_a + R_{a + b}, - L_{a+b} - R_b)$ is also inner. It suffices to check, according to Lemma~\ref{inner_trial_char_zero}, that $\Delta_1(1)= \Delta_2(1)= \Delta_3(1) = 0.$ By triality, $\Delta_1(1)= \Delta_2(1) + \Delta_3(1)$ and $\Delta_2(1)= T_2(1) - 2a - b = 0$ as well as $\Delta_3(1) = T_3(1) +2b + a = 0 $ which proves that $\Delta_1 = T_1 - L_a + R_b$ is a derivation.  By Proposition~\ref{inner_der_is_inner_trial_prop}, $\Delta_1$ is an inner derivation. 
\end{proof}

\section{Lie algebras graded by $A_n$, $n \geq 2$}
\label{rectangular_coordinatization}

\begin{defi} \label{rect_defi}
Fix an index set $K$, $\card K \geq 3,$ a partition $I \cup J$ of $K$, and
a unital alternative $k$-algebra $D$, which is associative if $\card K \geq 4.$ The rectangular matrix Jordan pair $\mathbf M(I, J, D)$ of size $I \times J$ is the pair of modules
$$(V^+, V^-) = (\mathrm{Mat}(I, J, D),\mathrm{Mat}(J, I, D))$$
with quadratic operator given by $$Q_x(y) = xyx \label{quad_op_ass_pair}.$$ The triple products are
$$\{xyz\}  = x(yz) + z(yx), \quad \{yxu\}  = (yx)u + (ux)y, \label{triple_prod_ass_pair}$$
for $(x,y), (z,u) \in V.$
\end{defi}
The brackets are important for the case $\card K = 3.$
For the rest of the section $V =(V^+, V^-) := \mathbf{M}(I,J,D).$ \\
The root system $\dot A_{K}$ (see Example~\ref{type_A_expl} ) has a $3$-grading with
$$R_1 = \{\epsilon_i - \epsilon_j : i \in I, j \in J\}. $$

\begin{prop} \label{rect_grid_la_prop}
\begin{itemize}
\item[\rm(i)]The family $\mathcal E = \{(E_{ij}, E_{ji}) : i \in I, j \in J\}$ is an $(\dot A_{K}, R_1)$ covering grid.
\item[\rm(ii)] The Lie algebra $\TKK(V)$ is $\dot A_{K}$-graded.
\end{itemize}
\end{prop}

\begin{proof} This follows from \cite[3.3]{Neh1996} and \cite[2.7]{Neh1996}, but we include a proof for the convenience of the reader.  
Let $\alpha = \epsilon_i - \epsilon_j, \beta = \epsilon_k - \epsilon_l$ be distinct roots in $R_1.$ There are only two possible relations: Either $\alpha \perp \beta$ or $\alpha \top \beta.$
The relation $\alpha \perp \beta$ is equivalent to $i \neq k$ and $j \neq l.$ Thus by Definition \ref{quad_op_ass_pair}
\begin{eqnarray*}
\{E_{ij} E_{ji} E_{kl}\} = \{E_{kl} E_{lk} E_{ij}\} &=& 0,\\
\{E_{ji} E_{ij} E_{lk}\} = \{E_{lk} E_{kl} E_{ji}\} &=& 0,\\
Q_{E_{ij}}E_{lk} = Q_{E_{kl}}E_{ji} &=& 0,\\
Q_{E_{ji}}E_{kl} = Q_{E_{lk}}E_{ij} &=& 0.
\end{eqnarray*}
Thus, $e_\alpha = (E_{ij}, E_{ji})$ and $e_\beta = (E_{kl}, E_{lk})$ are orthogonal (by Definition~\ref{cog_relations}).\\
If $\alpha \top \beta,$ then, equivalently, either $i = k$ or $l =j.$
Thus \begin{eqnarray*}
\{E_{ij} E_{ji} E_{kl}\} &=& E_{kl}, \\
\{E_{ji} E_{ij} E_{lk}\} &=& E_{lk}, \\
\{E_{kl} E_{lk} E_{ij}\} &=& E_{ij}, \\
\{E_{lk} E_{kl} E_{ji}\} &=& E_{ji}. 
\end{eqnarray*}
Next, let $V_\alpha = \bigcap_{\beta \in R_1} V_{\langle \alpha, \beta \che \rangle }(e_\beta).$
We claim that $(E_{ij} D, E_{ji}D) \subset V_{\alpha} .$
Let $\beta = \epsilon_k - \epsilon_l \in R_1.$ If $\langle \alpha, \beta \che \rangle = 1,$ then $\alpha \top \beta.$ Therefore for all $c \in D$
\begin{eqnarray*}
\{E_{kl} E_{lk} E_{ij}c\} &=& E_{ij}c, \\
\{E_{lk} E_{kl} E_{ji}c\} &=& E_{ji}c, 
\end{eqnarray*}
and therefore $(E_{ij} D, E_{ji}D) \subset V_1(e_\beta).$\\
If $\langle \alpha, \beta \che \rangle = 2 ,$ then $\alpha =  \beta.$ For all $c \in D$
\begin{eqnarray*}
Q_{E_{ij}}E_{ji}c   &=& E_{ij}c,\\
Q_{E_{ij}}E_{ij}c   &=& E_{ji}c,
\end{eqnarray*}
and therefore $(E_{ij} D, E_{ji}D) \subset V_2(e_\alpha).$ The Peirce relations show in fact that $V_\alpha = (E_{ij}D, E_{ji}D)$. Thus $V= \sum_{\alpha  \in R_1}V_\alpha$ and hence by Definition~\ref{grid_definition}, the family $\mathcal E$ is a covering grid. 
Pick  $\alpha, \beta \in R_1$ such that $\alpha \perp \beta.$ Then $\alpha = \epsilon_i - \epsilon_j$, $\alpha = \epsilon_k - \epsilon_l$ with  $k \neq i$ and $j \neq l$. We denote by $e_\alpha a$ the elements $E_{ij}a$ and by $e_{-\alpha}b$ the element $E_{ji}b$ and similarly for $\beta.$\\  For $m \in I$, $n \in j$: 
$$\delta(e_\alpha a, e_{-\beta}b)(E_{mn}c) = \delta_{km}\delta_{jl}E_{in}abc + \delta_{ln}\delta_{ki}E_{mj} $$
which is $0$ since  $k \neq i$ and $j \neq l.$ Similarly on $V^-.$ Thus  $[V^{+}_\alpha, V^{-}_\beta] = \{0\},$ or, equivalently the inner derivations $\delta(e_\alpha a, e_{-\beta}b)$ are trivial. Theorem~\ref{Root_graded_grid_theo} now applies and gives that $\TKK(V)$ is $\dot A_K$-graded.  
\end{proof}

It is also known (see \cite[p.467]{Neh1996}) that the structure of $\TKK(V)$ and $\instr(V)$ only depends on the cardinality of $K$ and the algebra $D$. So from now on we can take \begin{equation} I = \{1\}, J = K\setminus \{1\}.\end{equation} 
In this case the grid covering the Jordan pair $V$ is the collinear grid and thus by \ref{deg_sum_in_defect_coll}, $\ker \ud \subset \uider(V)_0. $

By Lemma~\ref{iso_from_grid_lem},  the map $\ad e_{\beta}^{- \sigma}$ restricted to $V^{\sigma}_{\alpha}$ is an isomorphism onto $L_{\sigma(\alpha - \beta)}$ for all $\alpha, \beta \in R_1, \alpha \neq  \beta.$ Thus the following notation is well-defined.
\begin{defi} Let $\alpha  = \epsilon_1 - \epsilon_i \in R_1, \beta  = \epsilon_1 - \epsilon_j \in R_{1}$, $a, b \in D$, then for $\alpha - \beta  \neq 0$
\begin{eqnarray*}
E_{1i} (a) = x_{\alpha} (a) &:=& E_{1i} a ,\\
E_{j1}(b) = x_{-\beta}(b) &:=& E_{j1} b, \\
E_{ji} (a) = x_{\alpha - \beta}(a) & :=& -\delta (E_{1i}(a), E_{j1}(1)).
\end{eqnarray*}
\label{rectangular_matrix_notation}
\end{defi}
The multiplication in $\TKK(V)$ between elements of non-zero degree in $\mathcal Q(R)$  is given as follows for elements $i,j,k,l \in I $ with $i \neq j$, $ l \neq k$:
\begin{eqnarray*}
\left [E_{ij}a, E_{lk}b \right] &=& \delta_{j,l}E_{ik}(ab) - \delta_{ik}E_{lj}(ba), \quad \text{if} \;\card {\{i, j\} \cap \{k, l\}} \leq 1,\\
\left [E_{ij}a, E_{ij}b \right] &=& 0, \\
\left[E_{1j}a, E_{j1}b \right] & = & \delta(E_{1j}(a), E_{j1}(b)),\\
\left[E_{ij} b, E_{ji}a \right] &= & \delta(E_{1j}(b), E_{j1}(a) ) - \delta(E_{1i}(ba), E_{i1}(1) ),\\
\left [E_{1i}(a), E_{1j}(b) \right] &=& 0,\\
\left [E_{i1}(a), E_{j1}(b) \right] &=& 0.\\
\end{eqnarray*}
The Lie brackets listed above uniquely determine the bracket on all of $\TKK(V).$ 

\subsection*{Computing the kernel of $\uider(V)_0 \rightarrow \instr(V)_0$}

Since $L = \mathrm{TKK}(V)$ is $R$-graded, its universal central extension $u : \mathfrak{uce}(L) \rightarrow L$ is $\mathcal Q(R)$-graded and Proposition~\ref{torsion_bijection_prop} implies:
\begin{itemize}
\item[-] If $\card K = 3$, then 
\begin{equation} \label{A_2_in_description} \ker(u) = \bigoplus_{\alpha \in \mathbb{DS}(A_2, 3)} \mathfrak{uce}(L)_\alpha \oplus (\ker u)_0 \end{equation}
where $ \mathbb{DS}(A_2, 3)$ is the set of degenerate sums of divisor $3$ in $Q(A_2).$
\item[-] If $\card K = 4$, then \begin{equation} \label{A_3_in_description}
\ker(u) = \bigoplus_{\alpha \in \mathbb{DS}(A_3, 2)} \mathfrak{uce}(L)_\alpha \oplus (\ker u)_0 \end{equation}
where $ \mathbb{DS}(A_3, 2)$ is the set of degenerate sums of divisor $2$ in $Q(A_3).$
\item[-] If $\card K > 4$, then
$\ker u =  (\ker u)_0.$
\end{itemize}
Consider the central covering $\hat \ud : \uTKK(V) \rightarrow L.$ By the universal property of $\mathfrak{uce}(L),$ there is a unique epimorphism $ \hat u: \mathfrak{uce}(L) \rightarrow \uTKK(V)$ so that $u = \hat \ud \circ \hat u,$ and $\hat u: \mathfrak{uce}(L) \rightarrow \uTKK(V)$ is a universal central extension. The epimorphism $\hat u$ is automatically graded. 
 \begin{itemize}
\item[\rm{(i)}] $\hat u : \uce L \rightarrow \uTKK(V)$  is a $Q(\dot A_{K})$-graded central extension and a $\mathbb Z$-graded central extension,
 \item[\rm{(ii)}] $\hat u_\alpha$ is an isomorphism for all $\alpha \in \dot A_K$, in particular by Corollary~\ref{collinear_zero_hom_Cor}, Corollary~\ref{ker_ud_supporr_cor} and \ref{deg_sum_in_defect_coll}
 \begin{eqnarray*}
\uce L^0 & \cong & \uider(V)\\
(\ker u)_0 & \cong & \ker (\ud_{JP})   \label{zero_space_theo} 
\end{eqnarray*} 
under the isomorphism 
$$ f: \sum \langle x_i, y_i \rangle \mapsto \sum x_i \diamond y_i, \mbox{ for }(x_i, y_i)\in V, $$ 
where
$$\mathfrak{uider}(V)_0 = \spa_k\{ E_{1j}a \diamond E_{j1}b:  j \in J, a,b \in D \}  \label{collinear_zero_hom_eqn}$$ and
$\ud_{JP}(x \diamond y) = \delta(x,y).$\\

Here upper indices refer to the $\mathbb Z$-grading that is obtained from the $3$-grading on $R$, and lower indices refer to the $\mathcal Q(R)$-grading induced by the $R$-grading on $L.$
\end{itemize}
The result that $\hat u : \mathfrak{uce}(L)_0 \rightarrow \uider(V)_0 $ an isomorphism is part of Corollary~\ref{JKP_central_ext_kernel_cor}.

\begin{lem}  \label{uider_gen_lem_rect}
The module $\uider(V)_0$ is isomorphic to $\bigoplus_{j \in J}(V_{\epsilon_1 - \epsilon_j})^+ \otimes (V_{\epsilon_1 - \epsilon_j})^- $ modulo the relations (where $i \neq j \in  J$ and $a, b,c, d \in D$):
\begin{eqnarray}
&& E_{1j}(a(bc) + c(ba)) \otimes E_{j1}d + E_{1j}(c(da) + a(dc)) \otimes E_{j1}d \\
&&-  E_{1j}(c) \otimes E_{j1}((ba)d + (da)b) - E_{1j}(a) \otimes E_{j1}((bc)d + (dc)b), \nonumber \\
&& \nonumber \\
&& E_{1i}(a(bc)) \otimes E_{i1}d + E_{1j}(c(da)) \otimes E_{j1} b \\ 
&&- E_{1i}c \otimes E_{i1}((da)b) - E_{1j}a \otimes E_{j1}((bc)d), \nonumber \\
&& \nonumber \\
&& 2(E_{1j}(aba) \otimes E_{j1} b + E_{1j}a \otimes E_{j1}(bab)).
\end{eqnarray}
\end{lem}
\begin{proof}
It follows from Corollary~\ref{collinear_zero_hom_Cor} that $\uider(V)_0$ is isomorphic to $\bigoplus_{j \in J}(V_{\epsilon_1 - \epsilon_j})^+ \otimes (V_{\epsilon_1 - \epsilon_j})^- $ modulo the relations:
\begin{eqnarray}
&& \{ E_{1j}a E_{j1}b E_{1i}c \}\otimes E_{i1}d + \{ E_{1i}c E_{i1}d E_{1j}a \}\otimes E_{j1}b \\
&&-E_{1i}c \otimes \{E_{i1}d E_{1j}a E_{j1}b\} - E_{1j}a \otimes \{E_{j1}b E_{1i}c E_{i1}d\}, \nonumber \\
&& \{ E_{1i}a E_{i1}b E_{1i}a \}\otimes E_{i1}b -  E_{1i}a \otimes \{E_{i1}b E_{1i}a E_{i1}b\}, 
\end{eqnarray} where $i$ and $j$ are in $J$ and $a, b, c, d \in D.$
The formulas in the claim are then simply obtained by expanding the triple products. 
\end{proof}
\begin{defi} \label{first_h_defi}
Define the following elements in $\uider(V)_0$ for $a, b \in D$ and distinct $j, j_0 \in J$
\begin{eqnarray*}
h_{j}(a, b) &=& E_{1, j}a \diamond E_{j,1}b,\\
H_{j}(a, b) &=& h_{j}(a, b) - h_{j}(1, ba),\\
T_{j_0, j}(a, b) &=& 3H_{j_0}(a, b) -h_{j_0}(1 ,[a, b]) - h_j(1, [a, b]).
\end{eqnarray*}
\end{defi}

\begin{lem}

\begin{itemize}
\item[\rm{(i)}]  The element $H_j(a, b)$ does not depend on the choice of $j$ (and will therefore be denoted by $H(a, b)$).
 \item[\rm{(ii)}] 
 The following relations hold between the elements $H(a, b)$ and $h_j(1, c)$, $j \neq j_0:$ 
\begin{eqnarray}
H(a, 1) = H(1,a) &=& 0, \label{h_1_a_rel}\\ 
H(a, b) + H(b,a) &=& 0, \label{h_2_a_rel}  \\
H(a b,c) + H(bc,a) + H(ca,b) - h_{j_0}(1 ,(a ,b,c)) - h_j(1, (a, b,c)) &= & 0 \label{cyclic_sum_h_rel},\\
T_{j_0, j}(a, b) + T_{j_0, j}(b,a) &=& 0, \label{t_1_a_rel}\\
T_{j_0, j}(a, bc) + T_{j_0, j}(b, ca) + T_{j_0, j}(c, ab) &=& 0. \label{cyclic_sum_t_rel}
\end{eqnarray}
\item[\rm(iii)] The elements $H(a, b)$ and  $h_j(c):= h_j(1, c)$, $j \in J, a, b , c \in D$  span $\uider(V)_0.$
\end{itemize}\label{identies_h_lem}
\end{lem}
\begin{proof}
Combining Lemma~\ref{uider_gen_lem_rect} with the notation introduced in Definition~\ref{first_h_defi}, we obtain
$$h_j(a(bc), d) - h_{j}(c, (da)b) - h_{j_0}(a, (bc)d) + h_{j_0}(c (da),b) = 0, $$
or equivalently
$$h_j(a(bc), d) - h_{j}(c, (da)b) =  h_{j_0}(a, (bc)d) - h_{j_0}(c (da), b). $$
Let $a = b = d = 1$ then 
$$ 0 = h_j(c, 1) - h_{j}(c, 1) =  h_{j_0}(1, c) - h_{j_0}(c, 1), $$
thus \begin{equation} h_{j_0}(1, c) = h_{j_0}(c, 1), \label{1_d_shift_h_id} \end{equation}
and also $H_{j_0}(1, d)= H_{j_0}(d, 1) = 0$ whence (\ref{h_1_a_rel}).\\
Next set $ c= 1 = d.$ This yields 
$$ 0 = h_j(a, d) - h_{j}(1, da) =  h_{j_0}(a, d) - h_{j_0}(da,1), $$
and so, combining this with (\ref{1_d_shift_h_id}), it is indeed  true that $H_{j_0}(a, d) = H_{j}(a, d)$ and we can simply write $H(a, d)$ from now on. \\
For (\ref{h_2_a_rel}) consider the case $a = c = 1.$ Then 
$$h_j(b, d) - h_{j}(1, db) =  h_{j_0}(1, bd) - h_{j_0}(d, b) $$
which is equivalent to
$$H(b,d) = -H(d,b). $$ 

Set $d= 1$ in $h_j(a(bc), d) - h_{j}(c, (da)b) =  h_{j_0}(a, (bc)d) - h_{j_0}(c (da), b). $
Then using (\ref{h_2_a_rel}) and (\ref{1_d_shift_h_id}): 
\begin{eqnarray*}
h_j(a(bc), 1) - h_{j}(c, ab) &=&  h_{j_0}(a, bc) - h_{j_0}(c a, b) \\
 h_j(a(bc), 1) - H(c,ab) - h_j((ab)c,1)  &=& H(a, bc) + h_{j_0}(1, (bc)a) - H(ca,b)  \\ && - h_{j_0}(1, b(ca)) \\
H(ab,c) + H(bc,a) + H(ca,b) &=& - h_j ((a, b, c)) - h_{j_0} ((b,c,a))
\end{eqnarray*}
Thus 
$$ H(a, bc) + H(b, ca) + H(c, ab) = h_{j_0}(1 ,(a ,b,c)) + h_{j}(1, (a, b,c)).$$
It is clear that the elements as described in (iii) span $\uider(V)_0$ since every element $E_{1j}a \diamond E_{1j}b$ can be written as
$H(a, b) + h_j(1, ba).$\\
The anti-commutativity of $T(a, b)$ is obvious from the definition as $[a, b] = -[a,b]$ and $H(a, b) = - H(b,a).$\\
In an alternative algebra
$(s(a), s(b), s(c)) = \sign(s)(a,b,c)$ for any element $s$ in  the symmetric group on $\{a, b, c\}$. Thus
\begin{eqnarray*}
\lefteqn{T_{j_0, j}(ab, c) + T_{j_0, j}(bc,a) + T_{j_0, j}(ca,b)} \\
 &=& 3h_{j_0}(1, (a,b ,c)) + 3h_{j}(1, (a,b,c))\\
						&&- \sum_{cyc.}h_{j_0}(1, (ab)c - c(ab)) - \sum_{cyc.}h_{j}(1, (ab)c - c(ab))\\
						&=& 3h_{j_0}(1, (a,b ,c)) + 3h_{j}(1, (a,b,c)) \\ &&- 3h_{j_0}(1, (a,b ,c)) + 3h_{j}(1, (a,b,c))\\
						&=& 0
	\end{eqnarray*}

	\end{proof}

\begin{defi} Let $D$ be any $k$-algebra. 
\label{cyc_hom_sans_one_half_defi}
The module $\tilde \bigwedge(D)$ is defined as the quotient of
$(D \otimes D) \oplus D \oplus D'$ (where $D'$ is a copy of $D$) modulo the submodule generated by
\begin{eqnarray}
& a \otimes b + b\otimes a,&\\
& ab\otimes c + bc \otimes a + ca \otimes b - (a, b,c) - (a, b,c)' & \label{cyc_sum_3_alt}
\end{eqnarray}

We will denote the image of $a \otimes b$ by $a \tilde \wedge b. $ Define also 
$ \mathrm{HC}_1(D) = \{ \sum a_i  \tw  b_i : \sum[a_i, b_i] = 0\}.$
\end{defi}
\begin{rem} For an associative algebra, one can define ``higher cyclic homology'', but this requires the machinery of mixed bicomplexes. For this approach see \cite{Lo}. For our purposes we are content with the definition given above. 
\end{rem}
First it is necessary to establish a number of identities in $\tilde \bigwedge(D)$. 
\begin{lem} The following identities hold in $\tilde \bigwedge(D)$ for all $a, b, c ,d \in D:$

\begin{eqnarray}
2 (a \tw a) &=& 0 \\
1 \tw a &=& 0\\
 \lefteqn{a(bc) \tw d + (bc)d \tw a + c(da) \tw b + (da)b \tw c} \nonumber \\ &=&((da)b)c  - d(a(bc)) \nonumber \\ &&+ (((da)b)c  - d(a(bc)))'. \label{cyc_sum_4_alt}
\end{eqnarray}
\end{lem}
\begin{proof}
Set $b = c = 1$ in (\ref{cyc_sum_3_alt}). Then the associators in  \ref{cyc_sum_3_alt} vanish and the remaining part is equivalent to
$$a \tw 1 + 1 \tw a + a \tw 1  = 0.$$ 
Thus $a \tw 1 = 1 \tw a = 0.$
Next 
\begin{eqnarray*}
&& a(bc) \tw d + (bc)d \tw a + c(da) \tw b + (da)b \tw c \\
& = & a(bc) \tw d + da \tw bc + (bc)d \tw a  -da \tw bc  \\
&&+ c(da)\tw b + (da)b \tw c + bc \tw da - bc \tw da  \nonumber \\
&=& (a, bc, d) + (c,da,b) +  (a, bc, d)' + (c,da,b)' \nonumber
\end{eqnarray*}
In every alternative algebra $(a, bc, d) + (c,da,b) = ((da)b)c  - d(a(bc)) $ which proves the claim. 
\end{proof}

\begin{prop} The $k$-module $\uider(V)_0$ is isomorphic to $$ \frac{(D \otimes D)  \oplus \bigoplus_{j \in J} (D)_{j}}{U}$$  where $D_{j}$, $j \in J$ are copies of $D$ and $U$ is the submodule generated by
\begin{eqnarray}
& a \otimes b - b\otimes a, &\\
& ab\otimes c + bc \otimes a + ca \otimes b - (a, b,c)_{j_0} -(a, b,c)_{j}, & \label{cyc_sum_3_alt_in_prop}
\end{eqnarray} for distinct $j, j_0$ in $J.$
\label{uider_alt_decomp_prop}
\end{prop}
\begin{proof} 
By the universal property of the tensor product there is a well-defined linear  epimorphism  
\begin{eqnarray*}
 \tilde \rho : (D \otimes D)  \oplus \bigoplus_{j \in J} (D)_{j} & \rightarrow & H(D, D) +\sum_{j \in J}h_j(D)\\
 a \otimes b & \mapsto & H(a, b), \\
  c_j& \mapsto & h_j(c).
\end{eqnarray*}
By Lemma~\ref{identies_h_lem} and (\ref{h_2_a_rel}),  $\tilde \rho(U) = 0.$ Hence we also have that

\begin{eqnarray*}
 \rho : \frac{(D \otimes D)  \oplus \bigoplus_{j \in J} (D)_{j}}{U}&\rightarrow & H(D, D) +\sum_{j \in J}h_j(D),\\
 a \otimes b  + U& \mapsto & H(a, b), \\
  (c_j) + U & \mapsto & h_j(c_j).
\end{eqnarray*}
is well-defined, linear and surjective as well.

 We have seen in Lemma~\ref{uider_gen_lem_rect} that $\uider(V)_0$ is isomorphic to $\bigoplus_{j \in J}(V_{\epsilon_1 - \epsilon_j})^+ \otimes (V_{\epsilon_1 - \epsilon_j})^- $ modulo the module $U'$ generated by
\begin{eqnarray}
&& E_{1j}(a(bc) + c(ba)) \otimes E_{j1}d + E_{1j}(c(da) + a(dc)) \otimes E_{j1}d  \label{rect_0_uider_id_2} \\
&& -  E_{1j}(c) \otimes E_{j1}((ba)d + (da)b) - E_{1j}(a) \otimes E_{j1}((bc)d + (dc)b),  \nonumber \\
&& \nonumber \\ 
&& E_{1i}(a(bc)) \otimes E_{i1}d + E_{1j}(c(da)) \otimes E_{j1} b  \label{rect_0_uider_id_1}\\ 
&& -  E_{1i}c \otimes E_{i1}((da)b) - E_{1j}a \otimes E_{j1}((bc)a), \nonumber \\
&& \nonumber \\ 
&& 2(E_{1j}(aba) \otimes E_{j1} b + E_{1j}a \otimes E_{j1}(bab)). \label{rect_0_uider_id_3}
\end{eqnarray}

Consider the linear map $ \tilde \pi : \bigoplus_{j \in J}(V^{\epsilon_1 - \epsilon_j})^+ \otimes (V^{\epsilon_1 - \epsilon_j})^-  \rightarrow D \otimes D  \oplus \bigoplus_{j \in J} (D)_{j}$  given by $E_{1j}a \otimes E_{j1}b \rightarrow a \otimes b + (ba)_{j}$. The map $\tilde \pi$ is clearly well-defined by the universal property of the tensor product. 
We consider the images of the generators of $U'$ under $\pi$. The element (\ref{rect_0_uider_id_1}) is mapped to 
\begin{eqnarray*}
&& a(bc) \otimes d + c(da)\otimes b + (bc)d \otimes a + (da)b \otimes c  \\ &+& (d(a(bc)))_{i} - (((bc)d)a)_j + (b(c(da))_j - ((da)b)c)_{i}. 
\end{eqnarray*}
By (\ref{cyc_sum_4_alt}) this is contained in $U$.\\
The image of (\ref{rect_0_uider_id_2}) is
\begin{eqnarray*}
&&a(bc) \otimes d + c(da)\otimes b + (bc)d \otimes a + (da)b \otimes c \\ &&-\big ((d(a(bc)) - ((bc)d)a) + (b(c(da))) - ((da)b)c \big)_{j_0}\\
&&+ c(ba) \otimes d + a(dc)\otimes b + (dc)b \otimes a + (ba)d \otimes c  \\ &&-\big( (b(c(da)) - ((da)b)c) + (d(a(bc))) - ((bc)d)a \big)_{j_0}.
\end{eqnarray*}

The first line is in the same coset of $U$  as the following (by~\ref{cyc_sum_4_alt}) 
\begin{eqnarray*}
&&a(bc) \otimes d + c(da)\otimes b + (bc)d \otimes a + (da)b \otimes c \\ 
&=& - da \otimes bc -  bc \otimes da + (a, bc ,d)_{j_0} + (a, bc, d)_j \\
 &&+ (da, b, c)_{j_0} + (da, b, c)_j \\
 &=& \big ( (da, b, c)  + (a, bc ,d) \big)_{j_0} + \big( (da, b, c)  + (a, bc ,d)\big)_j
\end{eqnarray*}
and likewise for the third line
\begin{eqnarray*}
&&c(ba) \otimes d + a(dc)\otimes b + (dc)b \otimes a + (ba)d \otimes c \\ 
&=&  - ba \otimes dc -  dc \otimes ba + (c, ba ,d)_{j_0} + (c, ba, d)_j \\
 &&+ (a, dc, b)_{j_0} + (a, dc, b)_j \\
 &=& \big ( (c, ba, d)  + (a, dc ,b) \big)_{j_0} + \big( (c, ba, d)  + (a, dc ,b)\big)_j.
\end{eqnarray*}
If we piece all this together, the  image of (\ref{rect_0_uider_id_2}) is in the same coset as
\begin{eqnarray*}
&&\big ( (da, b, c)  + (a, bc ,d) \big)_{j_0} + \big ( (c, ba, d)  + (a, dc ,b) \big)_{j_0}  \\
&&+\big( (da, b, c)  + (a, bc ,d)\big)_j  + \big( (c, ba, d)  + (a, dc ,b)\big)_j \\
&&-\big ((d(a(bc)) - ((bc)d)a) + (b(c(da))) - ((da)b)c \big)_{j_0}\\
&&-\big( (b(c(da)) - ((da)b)c) + (d(a(bc))) - ((bc)d)a \big)_{j_0}.
\end{eqnarray*}

The first two lines are zero, because of the so-called ``right-bumping identity,'' which holds in all alternative algebras, that 
$$(c, ba, d)  + (a, dc ,b)  + (da, b, c)  + (a, bc ,d) = 0. $$ Secondly, the last two lines vanish since in every alternative algebra $((da)b)c - (d(a(bc)))  + ((bc)d)a) - (b(c(da))) = 0.$ 
Consider the image of the element (\ref{rect_0_uider_id_3}). It is sent to
$$2( aba \otimes b - a \otimes bab - (baba)_i + (baba)_i) = 2( (aba) \otimes b + (bab) \otimes  a). $$
Since the subalgebra of $D$ generated by $\{a, b\}$ is associative we obtain
$$ 2(aba \otimes  b + bab \otimes a ) + U = - 2( ab  \otimes ab)  + U= U$$
Therefore $\tilde \pi(U') \subset U$ and by factorization $\pi$, given by  $\pi \big  (\sum H(a_i ,b_i) + \sum_{j \in J} h_j(1, a_j) \big) \mapsto \sum a_i \tilde \wedge b_i + (a_j)_{j \in J}$, is well-defined.

It remains to show that $\pi$ and $\rho$ are inverse to each other. 
For an element $\sum a_i \tilde \wedge b_i + \sum_{j \in J}(c_j)_{j \in J}$ we obtain $ \pi \circ \rho( \sum a_i \tilde \wedge b_i + \sum_{j \in J}(c_j)_{j \in J}) =  \pi( \sum H(a_i ,b_i) + \sum_{j \in J}h_{j}(1, c_j))  = \sum (a_i \tw b_i) + \sum_{j \in J}(c_j)_{j \in J}$ and likewise for $\rho \circ \pi.$ 
\end{proof}
\begin{rem}
 If $D$ is associative, we have therefore proved that $\uider(V)_0$ contains a free $D$-module of rank $\card J.$  
\end{rem}

It will be important not only to know the action by derivations of $\ud(H(a, b))$ and $\ud(h_j(a))$ on $(E_{1j}D, E_{j1}D),$ but also on the homogeneous components of the $\TKK(V)$ whose degree are roots in $R^0$ (with respect to the $3$-grading on $R$).
This is summarized in the next lemma:

\begin{lem} Let $i,j, m,n \in J$ and $i\neq j$ With the notation as in Definition~\ref{rectangular_matrix_notation} we have \label{H_Action_alt_lem} \label{uider_to_Liemult_lem}
\begin{eqnarray}
\ud(H(a, b))(E_{m,n}(c)) &=& -E_{m,n}([L_a, R_b]c), m \neq n  \label{H_Action_alt}\\
\ud(H(a, b))(E_{1,j}(c)) &=& -E_{1,j}(L_{[a, b]}c -[L_a, R_b]c),  \\
\ud(H(a, b))(E_{j,1}(c)) &=& -E_{j,1}(-R_{[a, b]} +[L_a, R_b])c,  \\
\ud(h_j(a))E_{j,n}(c) &= &  -E_{j,n}(L_a c), j\neq n,   \\
\ud(h_j(a))E_{m,j}(c) &= & E_{m,j}(R_a c), j \neq m,  \\
 \ud(h_j(a))E_{m,n}(c) &= & 0, \; m \neq j, n \neq j, m\neq n\\ 
\ud(h_j)(a)E_{1j}(c) & =& E_{1j}((L_a + R_a)c)\\
\ud(h_j)(a)E_{1i}(c) & =& E_{1j}((L_a)c)\\
\ud(h_j)(a)E_{j1}(c) & =& -E_{1j}((L_a + R_a)c)\\
\ud(h_j)(a)E_{i1}(c) & =& -E_{1j}((R_a)c)
\end{eqnarray}
\end{lem}
\begin{proof}
We may restrict to the case $\card J = 2,$ since for all indices $i,j,m,n$ which are pairwise non-equal, it always holds true that $[\delta(E_{m,n}a,E_{n,m}b ), E_{ij}c] = 0$ in $\TKK(V).$ So from now on $ J =\{2,3\}.$ Then
\begin{eqnarray*}[h_2(a), h_3(b)] &=& \{E_{12}a E_{21} E_{13}b\} \diamond E_{31} \\
&&-  E_{13}b \diamond \{E_{21} E_{12}a E_{31}\}  \\ 
&=& h_3(ab) - h_3(b ,a)\\
&= & H(a,b).
\end{eqnarray*}
Thus, since $\ud: \uider(V) \rightarrow \instr(V)$ is a Lie algebra homomorphism, it suffices to compute the action of $\ad h_2(a)$ and $\ad h_3(b).$ It will then follow that $\ud(H(a, b)) = [\ud(h_2(a)), \ud(h_3(b))].$ The formulas for the action of $\ud(h_j(a)) =  \delta(E_{1i}a, E_{i1}1)$ on $E_{1j}c$ or $E_{j1}c$ follow from (\ref{triple_prod_ass_pair}) and the action of $\ud(H(a, b))$ on these elements is obtained from  the basic identities in alternative algebras, see (\ref{alt_id_1}) and (\ref{alt_id_2}). 
It remains to consider the action on $E_{23}c$, the case of $E_{32}c$ being completely analoguous. 
Then 
$\ud h_2(a). E_{23}c = \ud h_2(a). [E_{21}, E_{13}c] = - [E_{21}(a + a), E_{13} c ] + [E_{21} ,E_{13}(ac)] = -E_{23}(ac)  $ 
and $\ud h_3(a). E_{23}c = \ud h_3(a). [E_{21}, E_{13}c] = - [E_{21}(a) , E_{13} c )] + [E_{21} ,E_{13}(ac + ca)] = -E_{23}(ca) .$
 The remaining formulas can be obtained easily now. 
\end{proof}



\begin{cor} If $D$ is associative or $1/3 
\in k$, then
$$\tilde \bigwedge D \cong \langle D, D\rangle \oplus D_\alpha \oplus D_\beta$$
where $\langle D, D \rangle $ is the quotient of $D \otimes D$ modulo the submodule $U''$ generated by 
$$a \otimes b + b \otimes a, $$
$$ ab \otimes c + bc \otimes a + ca \otimes b. $$
\label{assoc_coord_A_2_cor}
\end{cor}
\begin{proof}If $D$ is associative, then $(a, b, c)_\alpha + (a, b, c)_\beta = 0$, so that there are no relations between elements in $D \otimes D$ and $D \oplus D$. Therefore the sum is direct and the presentation reduces to the claimed one.\\
We assume from now on that $D$ is alternative, $\card J = 2$  and $1/3 \in k.$ Then all elements can be written as a linear combinations 
$\sum_{i = 1}^n T(a_i, b_i) + h_2(a) + h_3(b)$ for $a_i, b_i, a, b \in D.$ Define
$  \tilde \eta : D \otimes D \oplus D \oplus D \rightarrow \tilde \bigwedge D $ by $ \sum a_i \otimes b_i + a + b \mapsto \sum T(a_i, b_i) + 1/3 h_2(a) +  1/ 3 h_3(b).$ This map is well-defined by the universal property of the tensor product and it factors through $U''$ by \ref{t_1_a_rel} and \ref{cyclic_sum_t_rel}. Thus the factor map $\eta$ is well-defined. Conversely define, $\tilde \vartheta : D \otimes D \oplus D \oplus D \rightarrow \langle D, D\rangle \oplus D_\alpha \oplus D_\beta $ by 
$\tilde \vartheta (\sum_{i = 1}^n a_i \otimes b_i + a + b) = 1/3( \sum \langle a_i, b_i \rangle) + (1/3 \sum_{i = 1}^n [a_i, b_i] +  a)_{\alpha} + (1/3\sum_{i = 1}^n [a_i, b_i] + b)_{\beta} .$ Again, the properties of the tensor product ensure that this is well-defined. We claim that $\tilde \vartheta$ factors through $U$. \\
\begin{eqnarray*}
\tilde \vartheta (a \otimes b + b \otimes a) &=& \langle a, b \rangle + \langle a, b \rangle + ([a, b] + [b,a])_{\alpha} + ([a, b] + [b,a])_{\beta} \\
&=&  \langle a, b \rangle + \langle a, b \rangle = 0
\end{eqnarray*}
\begin{eqnarray*}
 \lefteqn{ 3 \tilde \vartheta (ab \otimes c + bc \otimes a + ca\otimes b - (a, b, c )_2 - (a, b, c)_3)} \\
  &=& \langle ab, c \rangle + \langle bc , a \rangle + \langle ca, b \rangle  \\
&&+ ([ab,c] + [bc, a] + [ca, b] - 3 (a, b, c) )_\alpha + ([ab,c] + [bc, a] + [ca, b] - 3 (a, b, c) )_\beta \\
&=& 0 
\end{eqnarray*}
That $[ab,c] + [bc, a] + [ca, b] - 3(a, b, c) = 0$ is a well-known identity which holds in all alternative algebras. 
The indices $2$ and $3$ were introduced to distinguish the two $D$-components. 
Thus we also obtain an epimorphism $\vartheta :  \tilde \bigwedge D :  \langle D, D\rangle \oplus D_\alpha \oplus D_\beta.$ It is now easy to check that $\eta$ and $\vartheta$ are mutually inverse, if one uses the identity $T_{2,3}(a, b) = 3H (a, b) -h_{2}(1 ,[a, b]) - h_3(1, [a, b])$. 
\end{proof}
\begin{rem} For $J$ with $\card J > 2,$ the definition of $\tilde \bigwedge D$ depends on $\alpha$ and $\beta.$ However, we will only use it in the case $\card J = 2,$ so that this ambiguity does not affect the future applications. 
\end{rem}

Before proceeding to our main examples, we record the following fact which in fact holds for any $k$-algebra:

\begin{lem} Let $D$ be an arbitrary unital $k$-algebra. \label{base_change_at_hom} Assume that  $ k \rightarrow K$ is a flat base change. Then
$$ \langle D \otimes_k K,D \otimes_k K \rangle = \langle D, D \rangle\otimes_k K.$$\label{alt_cycle_base_change_lem}
\label{JA_cycle_base_change_lem}
\end{lem}
\begin{proof}
Denote by $C_k$ the $k$-span of all elements $\{ab \otimes c + bc \otimes a + ca \otimes b, a \otimes b + b \otimes a : a,b, c \in D\}$ and analogously define $C_K$ for the algebra $D_K=  D \otimes_k K.$ 
Then the sequence
$$\{0\} \rightarrow C_k \rightarrow D \otimes D \to \langle D ,D \rangle \to \{0\} $$
is exact in the category of $k$-modules. 
Since the base change is flat, we obtain another exact sequence: 
$$\{0\} \rightarrow C_k \otimes_k K \rightarrow (D \otimes D) \otimes_k K \to \langle D ,D \rangle  \otimes_k K \to \{0\}. $$
It is well known that the map 
\begin{eqnarray*}
\Phi : (D \otimes D) \otimes_k K & \rightarrow & (D \otimes_k K)  \otimes_K (D \otimes_k K), \\
(a \otimes b) \otimes \alpha  & \mapsto & (a \otimes_k \alpha) \otimes (b \otimes_k 1)
\end{eqnarray*}
for $a, b \in D$ and $\alpha \in K$ is an isomorphism. Its inverse map is
\begin{eqnarray*}
\Psi :    (D \otimes_k K)  \otimes_K (D \otimes_k K)& \rightarrow & (D \otimes D) \otimes_k K, \\
(a \otimes_k \alpha) \otimes (b \otimes_k \beta)  & \mapsto & (a \otimes b) \otimes \alpha \beta.  
\end{eqnarray*}

It is a straightforward exercise to show that $\Phi(C_k \otimes K) \subset C_K$ and that $\Psi(C_K) \subset C_k \otimes K. $
Hence the $\Phi$ restricted to $C_k \otimes K$ is an isomorphism onto $C_K.$ 
Hence we obtain a morphism of exact sequences
$$
\xymatrix{ & \{0\}  \ar[r] \ar[d] & C_k \otimes_k K \ar[r]\ar[d]^{\Phi}& (D \otimes D) \otimes_k K \ar[r] \ar[d] & \langle D ,D \rangle \otimes_k K \ar[r] \ar[d]& \{0\} \ar[d]\\
 & \{0\}  \ar[r] & C_K  \ar[r] &(D \otimes_k K)  \otimes (D \otimes_k K )  \ar[r] & \langle D \otimes_k K ,D \otimes_k K\rangle   \ar[r] &\{0\} 
 } 
 $$
where the middle and left vertical arrows are isomorphisms. Therefore, so is the map on the right. 
\end{proof}

\section{$A_2$-graded Lie algebras}
\label{alternative_coordinatization}
We fix some notation for the rest of the section: 
\begin{itemize}
\item $R$ is the root system $A_2,$ with the collinear $3$-grading given  by $R_1 = \{\epsilon_1 - \epsilon_2, \epsilon_1 - \epsilon_3\},$
\item  $\mathbb{DS}:= \mathbb{DS}(A_2),$ $\mathbb{DP}:= \mathbb{DP}(A_2),$
\item  $D$ is an alternative unital algebra,
\item $V  = \mathbf{M}(1 ,2, D )$ is the matrix Jordan pair  with coordinates in $D,$ 
\item $L = \TKK(V).$
\end{itemize}
Then $L$ has a $3$-grading and  an $R$-grading which are compatible in the following way: 
$$L = L^{-1} \oplus L^ 0 \oplus L^1, \quad L^{i} = \bigoplus_{\alpha \in R_{i}} L_{\alpha}, \; i \in \{-1, 0, 1\}.$$
Recall that if $u: \mathfrak{uce}(L) \rightarrow L$ is the universal central extension then 
\begin{equation} \label{A_2_in_description_recal;} \ker(u) = \bigoplus_{\alpha \in \mathbb{DS}} \mathfrak{uce}(L)_\alpha \oplus (\ker u)_0. \end{equation}
\subsection{Degenerate sums $A_2$}
The main goal of this subsection is to prove a description of the $k$-module $\mathfrak{uce}(L)_\gamma$ where $\gamma$ is not a root. Since the structure of $L$ only depends on the alternative algebra $D$ there should be a way to find a ``coordinatization'' of $\mathfrak{uce}(L)$ solely in terms of $D$ and its algebraic properties. 

\begin{lem} 
 Let $D$ be a linear \textbf{unital} algebra. 
 The $k$-submodule $D[D, D] + D(D,D, D)$ equals the two-sided ideal $\mathfrak m$ of $D$ which is generated by $[D, D] + (D,D,D)$. \end{lem}
 \begin{proof}
First, using the unit of the algebra 
$$[D, D]D \subset [[D, D], D] + D[D, D] \subset [D,D] + D[D, D] \subset D[D, D]$$
and also  $$D[D, D] \subset [D, [D, D]] + [D, D]D \subset [D,D] + [D, D]D \subset [D, D]D$$
hence \begin{equation}
D[D, D] = [D, D]D. \label{comm_teich_id} 
\end{equation}
The Teichmuller identity (see \cite[p.336]{TasteOfJA}) which holds in every linear algebra states that for any $4$ elements $x,y, z,w$ in a linear algebra
 $$(xy, z, w) - (x, yz, w) + (x,y, zw)  = (x,y,z)w + x(y,z,w).$$
 Therefore
 we have 
 \begin{eqnarray*}
 D(D,D, D) & \subset & (D, D, D) + (D, D, D)D = (D, D, D)D  \\ (D, D, D)D & \subset & D(D, D, D) + (D, D, D) = D(D,D,D). 
 \end{eqnarray*}
Hence \begin{equation} D(D,D, D) = (D, D, D)D.  \label{teich_id} \end{equation}
Denote the ideal in $D$ which is generated by $[D, D] + (D,D,D)$ by $\mathfrak m.$ 
Clearly, $D[D, D] + D(D,D, D) \subset \mathfrak  m.$ We need to show that $D[D, D] + D(D,D, D)$ is an ideal, then it will follow that $D[D, D] + D(D,D, D) =  \mathfrak m.$ \\
By definition of the associator, 
\begin{eqnarray}
D(D[D, D]) & \subset &  (D, D, [D, D]) + D^2[D, D]  \\ & \subset &  (D, D, D) + D[D, D]. \label{sec_2_id_2}
\end{eqnarray}
Since $D$ is unital, $$(D, D, D) + D[D, D] \subset D(D, D,D) + D[D, D].$$ Using (\ref{comm_teich_id}), $(D[D, D])D = ([D, D]D)D.$ By an analogous argument it follows that  $$([D, D]D)D \subset (D, D, D)D + [D, D]D.$$ The two identities (\ref{comm_teich_id}) and (\ref{teich_id}) now imply that $$(D[D, D])D \subset D(D, D, D) + D[D, D].$$ 
So far we have 
\begin{equation} D(D[D, D]) + (D[D, D])D \subset D(D, D ,D) + D[D, D].  \label{commutators_id} \end{equation}
  Furthermore, 
  \begin{eqnarray*}
  D(D(D, D, D)) & \subset & (D, D , (D, D, D)) + D^2(D, D,D) \\ & \subset & (D, D,D) + D(D, D, D) \subset D(D, D,D) \end{eqnarray*}
  which is using again that $D$ is unital. \\
  Lastly, by (\ref{teich_id}):
  \begin{eqnarray*}
  (D(D, D, D))D &= &  ((D, D,D)D)D \subset ((D, D,D), D, D) + (D, D, D)D^2 \\ &\subset & (D, D, D) + (D, D, D)D \\ &=& (D, D, D) + D(D, D, D) = D(D, D, D). \end{eqnarray*}
  Thus
  $$D(D(D, D, D)) + (D(D, D, D))D \subset D(D, D,D) .$$
  Together with (\ref{commutators_id}), this proves that $D[D, D] + D(D,D, D)$ is an ideal.
\end{proof}

\begin{lem} Let $D$ be  a unital algebra.  Then 
$$D([D, D]D) \subset D[D, D] + (D,D,D). \label{sec_2_id_3} $$
If $D$ is associative, then $[D, D]D$ is an ideal and ${D_3}$ is a an algebra. 
\label{lemma_sec_3}
\end{lem}
\begin{proof}
For all unital algebra, we have $D[D, D] = [D, D]D$ as seen above in (\ref{comm_teich_id}),then we can apply (\ref{sec_2_id_2}) to conclude that
$D([D, D]D) = D(D[D, D]) \subset D[D, D] + (D,D,D).$ 
If $D$ is associative and unital, then $(D,D,D) = \{0\}$ and hence 
$D(D[D,D]) + (D[D,D]D)  = D^2[D,D] = D[D, D].$ Therefore  $D[D,D]$ is a two sided ideal. 
\end{proof}

We introduce the following notation: for $a, b, c \in D$, $D$ an algebra 
$$ab.c := (ab)c. $$
This saves some parentheses.

\begin{defi}
If $D$ is an alternative unital $k$-algebra then define $D_3$ to be the  quotient of $D$ modulo the $k$-module spanned by 
$$3D, D[D, D], (D, D, D),   \{(ad.c + a.dc + a.cd) b : a, b, c, d \in D \}.$$

 The coset of $a \in D$ in  $D_3$ is denoted by $\bar a$.
\end{defi}
We note some useful identities which hold in $D_3:$
\begin{eqnarray}
\overline{c.(ab - ba)} &=& 0  \label{D_3_1}\\
\overline{(a, b,c)} &=& 0 \label{D_3_2}\\
\overline{(ab - ba).c} &=& 0  \label{D_3_1_2}\\
\overline{(c.da)b} &=& \overline{(cb.d)a} \label{D_3_3}\\
\overline{a(c.db)} &=& \overline{(c.db)a}  \label{D_3_4} \\
\overline{a(b.dc)} &=& \overline{(cd.b)a}  \label{D_3_5} \\
\overline{ab.c - bc.a} &= & 0  \label{D_3_6}\\
\overline{(c.da)b +  a(c.db) + a(b.dc)} &= & 0 \label{D_3_7} 
\end{eqnarray}
Identities (\ref{D_3_1}) and (\ref{D_3_2}) follow immediately from the definition. The equality in (\ref{D_3_1_2}) is a consequence of $[D, D]D \subset D[D, D] + (D, D,D)$, see Lemma~\ref{lemma_sec_3}. \\
Likewise, $ab.c - bc.a = ab.c - a.bc  = 0$ by (\ref{D_3_1}) and (\ref{D_3_2}).
We show (\ref{D_3_3}):  $\overline{(c.da)b} = c(da.b) = \overline{c(b.da)} = \overline{cb.da} = \overline{(cb.d).a}.$
The equality in  (\ref{D_3_4}) is immediate from the fact that $\overline{[D,D]} = 0.$ Further, in order to show (\ref{D_3_5}) observe that $\overline {(D, D,D)} = 0,$ so $ \overline{a(b.dc)} = \overline{ab.dc} $. Since also $ \overline{D[D, D]} = 0,$ this is equal to $ \overline{ab.cd}.$ Then we use again $\overline {(D, D,D)} = 0$ and  $\overline{D[D, D]} = 0$ for the following two equalities respectively $\overline{ab.cd} = \overline{a(b.cd)} = \overline{(b.cd)a}.$
 Identity (\ref{D_3_7}) follows from (\ref{D_3_3}), (\ref{D_3_4}) and (\ref{D_3_5}) and the assumption that $(cb.d + c.db + cb.d) a = 0$ in $D_3$ by definition. 
We show (\ref{D_3_3}):  $\overline{(c.da)b} = c(da.b) = \overline{c(b.da)} = \overline{cb.da} = \overline{(cb.d).a}.$
The calculations for (\ref{D_3_4}) and (\ref{D_3_5}) are: 
$$ \overline{a(b.dc)} = \overline{(b.dc)a} $$
and $$ \overline{a(b.dc)} = \overline{ab.dc} = \overline{ab.cd} = \overline{a(b.cd)} = \overline{(b.cd)a}.$$

\begin{rem} If $D$ is associative and unital, then $D_3$ is a commutative algebra with product $ \bar  a \bar b = \overline{ab}.$
\end{rem}
\begin{proof} In view of Lemma~\ref{lemma_sec_3}, it is enough to show that for all $a, b, c, d \in D$ the element $(ad.c + a.dc + a.cd) b$ lies in $D[D, D] + 3D.$  We have $(ad.c + a.dc + a.cd).b = (2ad.c + acd).b \in (acd - adc).b + 3D = a.cdb - adcb + 3D \in(cdb - dcb)a + 3D + D[D,D] $
and $(cdb - dcb)a  =  (cd)(ba) - (dcb)a  \in (dc)(ba) - (dcb)a + D[D,D] =  (dc b  - dcb)a + D[D,D] = D[D,D]$, thus $(ad.c + a.dc + a.cd) b \in 3D + D[D, D].$ Therefore, $D_3$ equals the quotient of $D$ modulo the ideal $D[D, D] + 3D.$ This proves the remark. 
\end{proof}

\begin{cor} \label{root_space_iso_lem}Let $u: \mathfrak{uce}(L) \rightarrow L$ be the universal central extension. Then
 $u: \mathfrak{uce}(L)_\alpha \rightarrow L_\alpha$ is a bijection for $ \alpha \in R^\times$. 
\end{cor}
\begin{proof}
This is an immediate consequence of Proposition~\ref{torsion_bijection_prop}. 
\end{proof} 
\begin{defi}Let $1 \leq i \neq j \leq 3$, $a \in D$ and $g : K \rightarrow L$ a $\mathcal Q(R)-$graded covering. By Corollary~\ref{root_space_iso_lem} there are well-defined elements
$$x_{ij}(a) :  = g^{-1}(E_{ij}a) \in K_{\epsilon_i - \epsilon_j}.$$ 
$$x_{ji}(a) : = g^{-1}(E_{ji}a) \in  K_{\epsilon_j - \epsilon_i}.$$

Also, (\ref{collinear_zero_hom_eqn})
$L_0 = \sum_{j \in J }[E_{1j}D, E_{j1}D]$ and $K_0 = \sum_{j \in J}[x_{1j}(D), x_{j1}(D) ].$

\end{defi}

\begin{lem} \label{alt_on _deg_pairs_lem} 
There are $2^6$ maps $s : R \times R \rightarrow \{-1, 0, 1\}$ with the following two properties
\begin{itemize}
\item[\rm (i)] $ s(\alpha, \beta) \neq 0 \iff (\alpha, \beta) \in \mathbb{DP},  $
\item[\rm (ii)]$s(\alpha, \beta) = - s(\beta, \alpha) $ for all $(\alpha, \beta) \in R \times R.$
\end{itemize}
\end{lem}
\begin{proof}
There are $6$ degenerate sums. Each can be written in exactly two ways $\gamma = \alpha + \beta = \beta + \alpha$ for $\alpha, \beta \in R^\times$ and we get in total $2^6$ different ways to assign a value to $s(\alpha,\beta)$, $\alpha + \beta$ a degenerate sum and each of them fulfills the conditions set out in the lemma.
\end{proof}

\begin{defi} With the notation as in Lemma~\ref{alt_on _deg_pairs_lem}, define
 $$s(ij, kl) := s(\alpha, \beta) \;  \mbox{ for } \alpha  = \epsilon_i - \epsilon_j, \beta =\epsilon_k - \epsilon_l \in A_2. $$
\end{defi}
\begin{defi}Let $\mathcal Z = \bigoplus_{\gamma \in \mathbb{DS}}(D_3)_\gamma$ be the direct sum of six copies of $D_3$, viewed as $\mathcal Q(R)$ -graded module  with $\supp \mathcal Z = \mathbb{DS}$. For $s: R \times R  \rightarrow \{\pm 1, 0\}$ a map as given in Lemma~\ref{alt_on _deg_pairs_lem}, 
define a bilinear map $\psi : L \times L \rightarrow \mathcal Z $ by 
 \begin{eqnarray}
\psi(E_{ij} a, E_{kl} b) &=& \begin{cases} s(ij, kl)(\overline {ab})_{ \epsilon_i - \epsilon_j+   \epsilon_k - \epsilon_l} & \text{if }(\epsilon_i - \epsilon_j,  \epsilon_k - \epsilon_l) \in \mathbb{DP},\\
0&\text{else. } \end{cases} \label{cocyc_A_2_def} ,\\
\psi(L_0, L) &=& \psi(L, L_0) = \{0\} . \label{cocyc_A_2_def2}
\end{eqnarray}
\label{cocyc_recall}
\end{defi}


\begin{lem}\label{Lemma_1_for_A_2} 
\begin{itemize} \item[\rm (i)]
The bilinear map $\psi$ is a $\mathcal{Q}(R)$-graded $2$-cocycle on $L$ with coefficients in $\mathcal Z$. 
\item[\rm (ii)] The  central extension $L \oplus_{\psi} \mathcal Z \rightarrow L$ (see Definition~\ref{central_ext_by_cocyle}) is a $\mathcal{Q}(R)$-graded covering.
\item[\rm (iii)] By the universal property of $\uce L $ there is a a unique graded Lie algebra epimorphism $\pi : \uce L \rightarrow L \oplus_{\psi} \mathcal Z$  such that $$\pi(\uce {L}_\gamma) = \mathcal Z_\gamma \cong D_3$$ for all $\gamma \in \mathbb{DS}.$
\end{itemize}
\end{lem}
\begin{proof}
The map $\psi$ is a well-defined $k$-bilinear map. If $\gamma$  is  a degenerate sum then there is (up to switching the summands) only one way of expressing $\gamma$ as sum $\alpha + \beta$ where $\alpha, \beta \in R^\times$. Furthermore, it is obvious that $\psi$ is graded.\\
For $\psi$ to be a $2$-cocycle it is necessary  and sufficient to show that $\psi(x,x) = 0$ and $\psi([x,y], z) + \psi([y,z], x) + \psi([z,x], y) = 0$ for all $x,y, z  \in L.$ \\
Pick elements $x = E_{ij}a$ and $y = E_{kl}b$. Then $\psi(x, x) = 0$ as $2R \cap \mathbb{DS}(A_2) = 0.$ Moreover,  $$\psi(E_{ij}a, E_{kl}b) = s(ij, kl)\overline{ab} = -s(kl, ij)\overline{ab} = -s(kl, ij)\overline{ba} =-\psi(E_{kl}b,E_{ij}a),$$
hence $\psi(x,y) = - \psi(y,x)$ for $x = E_{ij}a $ and $y = E_{kl}b.$ Hence $\psi$ is alternating. \\
For the second identity it suffices to consider  homogeneous elements whose degrees sum up to a degenerate sum. Assume first that $z  \in L_0$, w.l.o.g. $z =[E_{13}c, E_{31}d]$ and $x = E_{12}a$ and $y = E_{13}b$ with $ \gamma = 2 \epsilon_1 - \epsilon_2 -\epsilon_3$ a degenerate sum. All other cases where only one of the elements is of degree $0$ can be obtained from this one by permuting the indices. Then for any $x,y \in L$ the element $\psi([x,y], z)$ equals zero by (\ref{cocyc_A_2_def2}). We compute the remaining two summands: 
 $$[y,z] = - E_{13}(c(db) + b(dc)), \quad \psi([y,z], x) = - s(13, 12) \overline{a(c(db) + b(dc))},$$ 
$$[z,x] =  E_{12}(c(da)), \quad \psi([z,x], y) = s(12, 13) \overline{(c(da))b}.$$
For $\psi$ to be a cocycle, we need to show that 
\begin{eqnarray*} 0 &=&  s(12, 13) \overline{(c(da))b} - s(13, 12) \overline{a(c(db) + b(dc))}  \\ &=& s(12, 13) \overline{(c(da))b} + s(12, 13) \overline{a(c(db) + b(dc))}.
\end{eqnarray*}
 Observe that by definition $D_3$ is a module, so 
 $$  \overline{(c(da))b} +  \overline{a(c(db) + b(dc))} = \overline{(c.da)b + a(c(db) + b(dc))}. $$
By  (\ref{D_3_5}) and (\ref{D_3_7})
$$ \overline{a(c.db + b.dc) +(c.da)b} = \overline{(c.db + cd.b + cb.d)a}  = 0.$$
This gives us
$$ s(12, 13) (\overline{a(c.db + b.dc)} + s(12, 13) \overline{(c(da))b})  = 0. $$
Hence for $x, y, z$  as above $\psi([y,z], x) +  \psi([z,x], y)  + \psi([x, y], z) = 0.$
 Let $\alpha = \deg (x)$, $\beta = \deg(y)$, $\gamma = \deg(z)$ be all non-zero. After possible reordering of the summands we may also assume that $\psi(x, [y,z]) \neq 0$ and therefore that $(\alpha, \beta + \gamma)$ is a degenerate pair. Then at most one of $\psi([y,z], x)$ and $\psi([z,x], y)$ is non-zero which we see as follows: if $\psi([y,z], x) \neq 0$ then $(\beta + \gamma, \alpha)$ is a degenerate pair
 and since there is up to switching summands only one way to express $\alpha + \beta + \gamma$ as degenerate sum, it follows that $\alpha = \gamma$ or $\alpha = \alpha + \beta$. The latter cannot occur, since $\beta \neq 0$. Therefore $\alpha = \gamma$ and $[x, z] = 0.$  A similar argument proves that if $\psi([z,x], y) \neq  0$ then $\psi([y,z], x) = 0.$ Therefore at most $2$ of the summands can be non-trivial and it suffices to check that $\psi([E_{ij} a, E_{jk} b], E_{ij} c) + \psi([E_{jk} b, E_{ij} c], E_{ij} a) = 0.$
 By (\ref{D_3_6}) we indeed have
\begin{eqnarray*}
\lefteqn{ \psi([E_{ij}a, E_{jk} b], E_{ij} c) + \psi([E_{jk} b, E_{ij} c], E_{ij} a) }\\
&=& \pm s(ik, ij) \overline{(ab)c} + \mp s(ik, ij)\overline{(bc)a} \\
&=&\pm s(ik, ij)(\overline{0})\\
\end{eqnarray*}
Thus $\psi$ is a $2$-cocycle and 
we conclude that there is a graded central extension $L' = L \oplus_{\psi} \mathcal Z$ of $L$ by $(\mathcal Z, \psi).$ It remains to prove that $L'$ is perfect. We have that $(L ,0)\subset [L', L']$: if $L_\alpha$, $0 \neq \alpha \in R$ is the $\alpha$-component of $L$ then by Lemma~\ref{perfect_TKK_lem} there is $D_\alpha \in L_0$ such that $[D_\alpha,L_\alpha] = L_\alpha $ and moreover $\psi(D_\alpha, L_\alpha)  = \{0\}$. Thus 
$(L_\alpha, 0) \subset [L',L'].$ Since $L_0 =  \sum_{\alpha \in R^\times}[L_{\alpha}, L_{-\alpha}]$ and $\psi(L_\alpha, L_{-\alpha}) = \{0\}$ it follows that $(L_0, 0) \subset [L', L'].$ Thus $(L, 0) \subset [L', L'].$ Let $z \in \mathcal Z$ be a homogeneous element of degree $\gamma = \epsilon_i - \epsilon_j + \epsilon_k - \epsilon_l $. Then $(0, z) = [E_{ij}a, E_{kl}] \in [L', L']$ for some $a \in D$ and thus $(0, \mathcal Z) \subset [L', L'].$ This proves that $L'$ is perfect. 
\end{proof}

\begin{lem} 
\label{Lemma_2_for_A_2}
Let $g : K \rightarrow L$ be a graded covering and $\gamma \in S := \supp_{\mathcal{Q}(R)}K\setminus R.$ Then \begin{itemize}
\item[\rm (i)] $\gamma = \alpha+ \beta$ is a degenerate sum.
\item[\rm (ii)] The map $D_3 \rightarrow K_\gamma$ given by
  $\overline{a} \rightarrow [ x_{\alpha}(a), x_{\beta}(1) ] $  where $\gamma = \alpha + \beta$, $a\in D$ is a $k$-module epimorphism. 
  \item[\rm (iii)] If $ u: \uce L \rightarrow L$ is the graded universal central extension, then there is a $k$-module epimorphism $\rho : D_3 \rightarrow \mathfrak{uce}(L)_\gamma$ given by $\overline{a} \rightarrow [ x_{\alpha}(a), x_{\beta}(1) ] $ . 
\end{itemize}
\end{lem}
\begin{proof} 

	Part (i) follows from Proposition~\ref{torsion_bijection_prop}. The universal property of $\uce$ allows us to conclude (ii) from (iii), so from now on $K = \uce{L}.$
Let $\gamma$ be a degenerate sum. Then there are unique (up to order) $\alpha, \beta \in R$ such that $\alpha + \beta = \gamma$ and it follows that $\mathfrak{uce}(L)_\gamma = [\mathfrak{uce}(L)_\alpha,\mathfrak{uce}(L)_\beta ].$ Without loss of generality $\alpha = \epsilon_i -\epsilon_j$ and $\beta = \epsilon_i - \epsilon_k$ with $j < k$. 
We denote the elements of $\mathfrak{uce}(L)_\alpha$ by $x_{ij}(a)$ and those of $\mathfrak{uce}(L)_\beta$ by $x_{ik}(b)$ where $a, b \in D$. The assignments $a \rightarrow x_{ij}(a)$ and $b \rightarrow x_{ik}(b)$ are $k$-module isomorphisms with domain $D$ by Lemma~\ref{root_space_iso_lem}.
The map $z_\gamma : D \times D \rightarrow \mathfrak{uce}(L)_\gamma$ given by 
$$ z_\gamma(a, b) := [x_{ij}(a), x_{ik}(b)]$$
is bilinear and surjective because $[\uce{L}_\alpha, \uce{L}_\beta] = \uce{L}_\gamma$.
For part (ii) we will show the following 
\begin{enumerate}
\item[(1)] $3 z_{\gamma}(D,D)  = 0 ,$
\item[(2)] If $a, b \in D$ , then $z_\gamma(ab, 1) = z_\gamma(a, b) = z_\gamma(b,a),$
\item[(3)] If $a, b, c \in D,$ then $z_\gamma(ab, c) =  z_\gamma(ba,c),$
\item[(4)] If $a, b, c \in D,$ then $z_\gamma((a, b, c), 1) = 0.$
\item[(5)] If $a, b, c, d \in D,$ then $z_\gamma(ad.c + a.dc + a.cd, b) = 0.$
\end{enumerate}

Since $n_\gamma = 3$ by Proposition~5.2.18, $3(z_\gamma(D,D)) = 0$,\\
 and therefore \begin{equation} 3z_\gamma(a, b) = 0 \quad \forall a, b \in D \label{3_torsion}.\end{equation} 
 This proves (1).  \\
  The image of $(a, b)$ only depends on the product $ba$: 
Pick $a, b, c \in D$ and consider the element
$$[x_{ij}(a), x_{ik}(bc)] =-[x_{ij}(a), [[x_{kj}(c), x_{jk}(1)], x_{ik}(b)]].$$ By the Jacobi identity this equals 
 \begin{eqnarray*}
z_{\gamma}(a, bc) &=& - [x_{ij}(a), [[x_{kj}(c), x_{jk}(1)]x_{ik}(b)]]\\
 &=& - [[[x_{kj}(c), x_{jk}(1)], x_{ij}(a)], x_{ik}(b)]  - [[x_{kj}(c), x_{jk}(1)], \underbrace{[x_{ij}(a), x_{ik}(b)]}_{\text{central}}] \\  &=& [x_{ij}(ac),x_{ik}(b) ] = z_\gamma(ac, b).\end{eqnarray*}   Thus 
\begin{equation} z_{\gamma}(a, bc) = z_{\gamma}(ac, b).\label{almost_assoc} \end{equation}
From this we get (2). \\
Setting $b= 1$ yields:
\begin{equation} z_\gamma(ac, 1) =  z_\gamma(a,c),\label{almost_comm} \end{equation}
and if $a = 1,$ then we obtain 
$$z_\gamma(1, bc) = z_\gamma(c,b) = z_\gamma(cb, 1). \label{z_gamma_is_mult_compatible}$$

\
The next calculation uses again only the Jacobi identity. 

\begin{eqnarray*}
z_\gamma((ab)c, 1) + z_\gamma((cb)a, 1)
&=&  [[[x_{ij}(a), x_{ji}(b)] x_{ij}(c)], x_{ik}(1)]\\
&=& -\left [x_{ij}(c), [[x_{ij}(a), x_{ji}(b)], x_{ik}(1)] \right]\\
&=& -[x_{ij}(c), x_{ik}(ab)]\\
&=&  -z_\gamma(c,ab)\\
&=& -z_\gamma(ab,c).
\end{eqnarray*}
By (\ref{z_gamma_is_mult_compatible}), we have that $z_\gamma(ab, c) = z_\gamma(ab.c, 1).$ So if we subtract this elements on both sides of the equation
$$z_\gamma((ab)c, 1) + z_\gamma((cb)a, 1) = - z_\gamma(ab , c),  $$ this yields
 that $z_\gamma(cb,a) = - 2z_\gamma(ab,c) .$ 
 We showed before that $3 z_\gamma(ab , c) = 0$ ($3$-torsion), thus $z_\gamma(cb,a) = z_\gamma(ab,c).$  
Hence, by (\ref{almost_assoc}) and (\ref{z_gamma_is_mult_compatible}), $z_\gamma(a, cb) = z_\gamma(ac,b).$ 
We also have
 \begin{equation}
z_\gamma(ab, c) =  z_\gamma(a, bc) = z_\gamma(ac, b) = z_\gamma(b, ac) = z_\gamma(ba, c). \label{commutative_assoc}
 \end{equation}
 This proves (2) and (3) and (4).
 Next consider the element 
 $$z_\gamma((cd)a, b) + z_\gamma((ad)c, b) = \left[[[x_{ij}(c), x_{ji}d ], x_{ij} a], x_{ik}b  ] \right].$$ By the Jacobi identity, this is equal to  
 $$ \left[[x_{ij}(c), x_{ji}d ], \underbrace{[z_{ij} a, z_{ik}b]}_{\text{central}}  ] \right]  - \left[ x_{ij} a, [ [x_{ij}(c), x_{ji}d ], x_{ik}b  ]] \right]  = -z_\gamma(a, c(db)).$$
Since $z_\gamma(a, c(db)) = z_\gamma(ca, db) = z_\gamma(ca.d, b)$ this allows to conclude that
$$ z_\gamma((cd)a, b) + z_\gamma((ad)c, b) + z_\gamma(ac.d, b) = z_\gamma(cd.a + c.da + ca.d, b) = 0. $$
   Therefore, we have proven (5). 
  Together with  (2) and (3) this shows that we have a well-defined epimorphism $D_3 \rightarrow z_\gamma(D, D)$ which is given by
  $ \overline {a } \mapsto z_\gamma(a, 1).$ This proves (iii).  

\end{proof}

\begin{theo}Let $\gamma \in \supp_{\mathcal{Q}(R)}\mathfrak{uce}(L)\setminus R.$
\begin{itemize} 
\item[\rm{(i)}]The element $\gamma$ is a degenerate sum $\gamma = \alpha + \beta$, i.e., $\alpha$ and $\beta$ are non-zero distinct roots and their sum is not a root. 
\item[\rm(ii)] For every degenerate sum $\gamma =\alpha + \beta$, there is a $k$-module isomorphism
$$\mathfrak{uce}(L)_\gamma \cong D_3, $$
given by 
$[x_\alpha(a), x_\beta(b)] \rightarrow \overline{ab}.$
\end{itemize}\label{non_roots_theo}
\label{deg_sum_theo}
\end{theo}
\begin{proof}

Putting Lemma~\ref{Lemma_1_for_A_2} and Lemma~\ref{Lemma_2_for_A_2} together gives us a $k$-module epimorphism 
$\pi : \mathfrak{uce}(L)_\gamma \rightarrow D_3$ for every degenerate sum $\gamma$ (Lemma~\ref{Lemma_1_for_A_2}) and also a $k$-module epimorphism $\rho : D_3 \rightarrow \mathfrak{uce}(L)_\gamma$ (Lemma~\ref{Lemma_2_for_A_2}).  It remains to show that they are inverse to each other. Let $[x_{ij}(a), x_{kl}(b)] \in \mathfrak{uce}(L)_\gamma$, $\gamma = \epsilon_i - \epsilon_j + \epsilon_k - \epsilon_l$, a degenerate sum. Then $\rho \circ \pi([x_{ij}(a), x_{kl}(b)]) = \rho((\overline{ab})_\gamma) = z_\gamma(a, b) = [x_{ij}(1), x_{kl}(ab)] = [x_{ij}(a), x_{kl}(b)]$. If $\bar a \in D_3$ then $\pi \circ \rho (\bar a ) = \pi ([x_{ij}(1), x_{kl}(a)]) = 1 \bar a = \bar a.$ Thus
$$\mathfrak{uce}(L)_\gamma \cong D_3. $$ 
\end{proof}
\begin{cor} Let $L$ and $D$ be as above. Then: 
\begin{itemize}
\item[\rm (i)] If $1/3 \in k$ then $\supp_{\mathcal Q (R)}\mathfrak{uce}(L) = R$ and $\ker u \subset \mathfrak{uce}(L)_0.$
\item[\rm (ii)] If $D$ is associative then $\mathfrak{uce}(L)_\gamma \cong  D/(D[D,D], 3D)$ for every degenerate sum $\gamma.$
\end{itemize}
\end{cor}
\begin{proof}This is an easy consequence of Theorem~\ref{deg_sum_theo} and Proposition~\ref{torsion_bijection_prop}.
\end{proof}
\begin{rem} 
Part (ii) was obtained in \cite{GS05} under the assumption that $k$ is a field. Our proof uses some of their ideas.
\end{rem}
\subsubsection*{The kernel of $\uider(V) \rightarrow \instr(V)$}
Recall that we denote by $u: \mathfrak{uce}(L) \rightarrow L$ a $\mathcal Q(R)$-graded central extension of $L = \TKK(V).$ In the case where $V$ is covered by a collinear grid, we saw that $(\ker u)_0 = \ker \ud = (\ker \ud)_0.$
\subsubsection{Known cases}
There are a number of results concerning the structure of $(\ker u)_0$, most of them for the special case where $D$ is associative. We briefly recall some of the known facts. 
Let $\mathcal B$ be an associative unital algebra and $\mathfrak{gl}_3(\mathcal B) = \mathfrak{gl}_3(k) \otimes_k \mathcal B$ the Lie algebra of $3 \times 3$ matrices with entries in $\mathcal B$. Define
$$ \mathfrak{sl}_3(\mathcal B) := \{M \in \mathfrak{gl}_3(\mathcal B) : \mathrm{trace}(M) \in [\mathcal B, \mathcal B]\}$$
where the trace is the usual matrix trace, i.e., the sum of the diagonal elements. 
Then $ \mathfrak{sl}_3(\mathcal B)$ is a perfect Lie algebra. It can also be described as the subalgebra of $\mathfrak{gl}_3(\mathcal B)$ which is generated by $\{E_{ij} \otimes a: i \neq j, a \in \mathcal B\}.$ The centre of $\mathfrak{sl}_3(\mathcal B)$ consists of the matrices $Z(\mathfrak{sl}_3(\mathcal{B})) = \{a I_3: a \in Z(\mathcal B), 3a \in [\mathcal B, \mathcal B]\}.$ By definition $\mathfrak{psl}_3(\mathcal B ) = \mathfrak{sl}_3(\mathcal B)/Z(\mathfrak{sl}_3(\mathcal{B})).$ 
\begin{rem} We always have $\mathfrak{psl}_3(\mathcal B ) \cong \TKK (V)$ if $V$ is defined as above with $D$ replaced by the associative algebra $\mathcal B$. But if $D$ is non-associative there is no known embedding of $\TKK(V)$ into a matrix algebra $\mathfrak{gl}{I}(D).$ 
\end{rem}
\begin{rem}Unfortunately the centre of $\mathfrak{sl}_3(D)$ is not always a direct summand, so that already the module structure of the quotient can be fairly complicated. 
\end{rem}
Over a field the kernel of $u: \mathfrak{uce}(L) \rightarrow L$ coincides with the second homology group of the Lie algebra $L,$ denote by $H_2(L).$
In \cite{KasLo}, the authors give a description of $H_2(\mathfrak{sl}_3(\mathcal B))$ for an associative unital $k$-algebra $\mathcal B$ over a field $k$: $H_{2}(\mathfrak{sl}_3(\mathcal{B})) \cong HC_1(\mathcal B)$. Note that this  is not a description of $H_2(\mathfrak{psl}_3(\mathcal B)).$ However, if for all $c \in Z(\mathcal B)$ the element $3c$ is not contained in $[\mathcal B, \mathcal B] = \{0\}$ then 
$H_2(\mathfrak{psl}_3(\mathcal B)) = H_2(\mathfrak{sl}_3(\mathcal B)).$
We give a short proof for this result: 
 Since $\mathfrak{psl}_3(\mathcal B)$ is centreless, equality holds if and only if $\mathfrak{sl}_3(\mathcal B)$ is also centreless. Let $z \in Z(\mathfrak{sl}_3(\mathcal B))$. Then $[z, E_{ij}a] = 0$ for all $1 \leq i \neq  j\leq 3$ and $a \in \mathcal B$, thus $z = c I_3$ for some $c \in Z(D)$. Since $\mathrm{trace}(z) = 3c \in [\mathcal B, \mathcal B]$ we see that 
$$ z \in Z(\mathfrak{sl}_3(\mathcal B)) \iff z = cI_3, c \in Z(\mathcal B) \cap \{ c: 3c \in [\mathcal B, \mathcal B]\}.$$
\\

To better understand the action of $\instr(V)_0$ we go back to Lemma~\ref{H_Action_alt_lem}. The action of the spanning elements of $\uider(V)_0$ on some of the homogeneous spaces  of non-zero degree is described by the table below. We only include in the table what is needed for the proofs in the present section. 
\begin{center}
\begin{table*}[h]
\centering
	\begin{tabular}[h]{l|l|l|l|l}
	& $E_{12}D$ & $E_{13}D$ & $E_{32}D$ & $E_{23}D$  \\
	\hline
	$ h_{2}(a)$ & $L_a + R_a$ & $L_a$ & $R_a$ & $-L_a$ \\
	\hline
	$h_3(b)$ & $L_b$ & $L_b + R_b$ & $- L_b$ & $R_b$\\
	\hline
		$H(a, b)$ & $L_{[a, b]} - [L_a, R_b]$ & $L_{[a, b]} - [L_a, R_b]$ & $-[L_a, R_b]$ & $[L_a, R_b]$\\
		\hline
		$T(a, b)$ & $-SD(a, b)$  & $-SD(a, b)$  & $-SD(a, b)$ & $SD(a, b)$
	\end{tabular}
\end{table*}
\end{center}

\begin{prop} If $\card J = 2,$ then  $ \tau : \instr(V)_0 \rightarrow \mathcal T_0(D)$ given by 
$ \tau  : \ud(h_3(a)) \mapsto \lambda(a)$ and $\tau: \ud(h_2(a)) \mapsto \rho(a) + \lambda(a)$
is a Lie algebra isomorphism. 
\label{Lieisoidertrial}
\end{prop}
\begin{proof} See Definition~\ref{inner_trial_def} for notation.  Define the following elements in $\instr(V)_0:$
$\Lambda(a) = \ud(h_3(a)), $ $P(b) =\ud(h_2(b) - h_3(b)).$ Then it is not difficult to show (see for example \cite[p. 466]{Neh1996}) that there is  a unique Lie algebra homomorphism $\tau : \instr(V)_0 \rightarrow \mathcal T_0$ given by  $\Lambda(a) \mapsto \lambda(a)$ and $P(b) \mapsto \rho(b)$ and that $\tau$ is an isomorphism. 
\end{proof}

\begin{rem} This can also be found in \cite{Faulkner89}. 
\end{rem}

\begin{rem}
If $1/3 \in k,$ then we have proven in Corollary~\ref{assoc_coord_A_2_cor} that $\uider(V)_0 \cong \langle D, D \rangle \oplus D \oplus D.$ It follows from the proof of Corollary~\ref{assoc_coord_A_2_cor} that under $\ud,$ the corresponding decomposition of $\mathcal T_0 \cong \instr(V)_0$ is given  by $\mathcal T_0  \cong \mathrm{StanDer}(D) \oplus \Lambda(D) \oplus P(D)$ where $\tau(\ud(\langle a, b \rangle )) = SD(a, b).$ 
\label{one_third_instr_decomp_remark}
\end{rem}

\begin{prop}
\label{A_2_uider_kernel}
An element $h_2(a) + h_3( b) + \sum H(a_i, b_i) $ is in $\ker \ud$ if and only if $a= b$, $3a + \sum [a_i, b_i] =0$ and the inner derivation $L_a - R_a +  \sum [L_{a_i}, R_{b_i}]$ is trivial.
\end{prop}
	\begin{proof} Let $h_1(1, a) + h_2(1, b) - \sum H(a_i, b_i)$ be an  element of $\uider(V).$
	Using the table above the element is central if and only if 
	\begin{eqnarray*}
	L_a + R_a + L_b + \sum (- L_{[a_i, b_i]} + [L_{a_i}, R_{b_i}]) &=& 0,\\
	L_a + L_b + R_b + \sum (- L_{[a_i, b_i]} + [L_{a_i}, R_{b_i}]) &=& 0,\\
	-R_a + L_b - \sum  [L_{a_i}, R_{b_i}] &=& 0,\\
	-L_a + R_b + \sum  [L_{b_i}, R_{a_i}] &=& 0.\\
	\end{eqnarray*}
	By  applying either of the last two lines to $1$ we obtain $a = b$ and the equations can be rewritten as 
	\begin{eqnarray*}
	2L_a + R_a  + \sum (-L_{[a_i, b_i]} + [L_{a_i}, R_{b_i}]) &=& 0,\\
   -L_a + R_a - \sum  [L_{a_i}, R_{b_i}] &=& 0.\\
	\end{eqnarray*}
	Subtracting the last from the first line leaves the following system: 
		\begin{eqnarray*}
	 3L_a  + \sum L_{[a_i, b_i]} &=& 0,\\
   -L_a + R_a - \sum  [L_{a_i}, R_{b_i}] &=& 0.\\
	\end{eqnarray*}
	The first line is equivalent to $3a + \sum [a_i, b_i] = 0$ thus we have shown
\begin{eqnarray*}
	a&=& b,\\
	3a + \sum [a_i, b_i] &=& 0,\\
	L_a  - R_a + \sum  [L_{a_i}, R_{b_i}] &=& 0.
	\end{eqnarray*}

For the sake of completeness we will prove that such an element  $Z$ is central. 	With the conditions $3a + \sum [a_i, b_i] = 0$  and $a =b$ we know that 
$L_a + R_a + L_b + \sum (L_{[a_i, b_i]} + [L_{a_i}, R_{b_i}]) = -L_a + R_a +  \sum  [L_{a_i}, R_{b_i}] = 0. $ Therefore the element under consideration annihilates all of $e_{12}D$ and $e_{13}D.$ Similarly one proves (by passing to the opposite algebra) that also $e_{21}D$ and $e_{31}D$ are annihilated. Since those submodules generate the Lie algebra $\mathfrak{uce}(L)$ it follows that $Z \in  \ker u.$
	\end{proof}
	
	Before we prove the next result, we would like to recall the crucial fact that $\ker \ud = \ker (\ud)_0 = (\ker u)_0$ where $(\ker u)_0$ and $(\ker \ud)_0$ are the $0$-components with respect to the $\dot A_I$-grading which is induced by the collinear grid of idempotents( Proposition~\ref{rect_grid_la_prop}).
\begin{cor}
\label{A_2_kernel_ud_0_cor}
\begin{itemize}
\item[\rm{(i)}]
If $1/3  \in k,$ then $\uider(V)_0$ is isomorphic to $\langle D,D  \rangle  \oplus D \oplus D$ and
$$(\ker u)_0 = \left \{\sum \langle  a ,b \rangle : \sum SD(a_i, b_i) = 0 \right\} \label{zero_space_one_third} $$
where $SD(a,_i, b_i) = L_{[a_i, b_i]} - 3[R_{a_i}, L_{b_i}] - R_{[a_i, b_i]}$ is the standard inner derivation of $D$ determined by $(a_i, b_i).$
\item[\rm{(ii)}] If  $3 D = 0$, then 
\begin{equation} (\ker u)_0 = \big \{ \sum H(a_i, b_i) + h_2(1, c) + h_3(1, c): \sum [L_{a_i}, R_{b_i}] = 0, \sum [a_i, b_i] =0 \big \}.\label{zero_space_3_tor}\end{equation} 
\item[\rm{(iii)}] If $D$ is commutative, then 
\begin{equation} (\ker u)_0 = \{ \sum H(a_i, b_i) + h_2(1, c) + h_3(1, c):\sum [L_{a_i}, R_{b_i}] = 0, 3c = 0 \}.\end{equation}
If, in addition, $D$ does not have $3$-torsion, then 
 \begin{equation}(\ker u)_0 = \{ \sum H(a_i, b_i) : a_i, b_i \in D \}.\end{equation}
\item[\rm (iv)] If $D$ is associative, then 
$$(\ker u)_0 = \big \{ \sum H(a_i, b_i) + h_2(1, a) + h_3(1, a), 3a + \sum [a_i, b_i] =0, a \in \mathrm{Cent}(D) \big \}.$$
\end{itemize}
\end{cor}
\begin{proof} Throughout assume $Z =h_2(1, a) + h_3(1, a) + \sum H(a_i, b_i)$ is in the kernel of $\ud : \uider(V) \rightarrow \instr(V)$. \\
(i) If $1/3$ in $k,$ then the image of $\sum \langle a_i, b_i \rangle + a_\alpha + b_\beta$ in the triality algebra (which is isomorphic to $\instr(V)_0$) is 
$\sum (SD(a_i, b_i), SD(a_i, b_i), SD(a_i, b_i)) + \lambda(a) + \rho(b)$, see Remark \ref{one_third_instr_decomp_remark}. 
Assume that this is zero. Since $\mathcal T_{0} \cong \mathrm{StanDer}(D) \oplus (\lambda(a) \oplus \rho(b))$ by Lemma~\ref{inner_Der_inner_trial}, this is equivalent to  $a = b = 0$ and  $\sum SD(a_i. b_i) = 0.$ \\ 
(ii) If $3 D = 0$ then $3a = 0$ for all $a$ and the condition is $a = b$, $\sum [L_{a_i}, R_{b_i}] = 0$ and  $\sum [a_i, b_i] =0$. \\
(iii) Clearly,  if $D$ is commutative, then any sum over commutators is zero. Moreover, in the absence of $3$-torsion, every alternative, commutative algebra is automatically associative, so that the conditions are all void with the exception of $a = 0$. Thus the elements of the kernel are of the stated form. \\
(iv) If $D$ is associative then $\sum [L_{a_i}, R_{b_i}] = 0$ and we need $L_a = R_a$, i.e., $a \in \mathrm{Cent}(D)$ and $3a + \sum [a_i, b_i] =0.$
	\end{proof}
\begin{rem}
Case (i) was shown in  \cite{BGKN} for $k$ a field, case (iii) is well-known, see for instance \cite{loopgroups}. The associative case is treated in \cite{Neh1996} and also in \cite{KasLo}.  
\end{rem}


\begin{rem} Proposition~\ref{A_2_uider_kernel} and  Theorem~\ref{deg_sum_theo} completely describe the kernel of 
$\mathfrak{uce}(L_k(D)) \rightarrow L_k(D)$ for all base rings $k$ and all alternative $k$-algebras $D.$
\end{rem}

 \

\subsection{Example: Octonion algebras}
\label{Octonion_alg}
\begin{center}
\textbf{Until the end of this subsection let $k$ be a ring containing $1/3$.}
\end{center}
 Following Corollary~\ref{A_2_kernel_ud_0_cor}, the task of computing the kernel of $\ud : \uider(V) \rightarrow \instr(V)$ is reduced to describing the kernel of the map 
$$f: \langle D, D \rangle \rightarrow \mathrm{StanDer}(D), \quad \langle a, b  \rangle  \mapsto SD(a, b) $$
where $\langle D, D \rangle$ is the quotient of $D \otimes D$ modulo the submodule $U''$ generated by $a \otimes b + b \otimes a$ and $ab \otimes c + bc \otimes a + ca \otimes b. $

\begin{defi} For the remaining part, if $D$ is an alternative $k$-algebra and $k \rightarrow K$ an extension, then
$$L_k(D) := \mathrm{TKK}(\mathbf{M}(1, 2, D)), $$ 
$$D_K = D \otimes_k K .$$
\end{defi}

\begin{rem} Proposition~\ref{A_2_uider_kernel} and  Theorem~\ref{deg_sum_theo} completely describe the kernel of 
$\mathfrak{uce}(L_k(D)) \rightarrow L_k(D)$ for all base rings $k$ and all alternative $k$-algebras $D.$
\end{rem}

Moreover, since $1/3 \in k,$  we know from Proposition~\ref{torsion_bijection_prop} and (\ref{A_2_in_description}) that, if $ u : \mathfrak{uce}(L_k(D)) \rightarrow L_k(D)$, then $$\ker u = \ker \ud = \ker f.  $$ 
\begin{rem}
We always have $\langle 1, D \rangle = \langle D, 1 \rangle  = 0. $ 
\label{remark_1_d_is_trivial}
\end{rem} 

By \cite[Prop 2.9]{lopera} and \cite[Cor. 5.2]{lopera}, if $D$ is an octonion algebra and $1/3 \in k$ then for every flat base change $k \rightarrow K$ 
 \begin{equation} \mathrm{Der}_k( D) \otimes_k K=  \mathrm{Der}_K( D \otimes_k K), \label{Der_base_change}\end{equation}
and all the derivations are standard.

\subsubsection{The smallest example: $\mathfrak{sl}_3(k)$ }

\begin{lem} Assume $1/3 \in k$, then $L(k)$ is isomorphic to the Lie algebra of traceless $3 \times 3$ matrices with entries in $k.$ Moreover, $L(k)$ is centrally closed. \label{sl_3_baby_lem}
\end{lem}
\begin{proof} We will first show that $L(k)$ is centrally closed. We know that $L(k)$ is perfect. Let $u : \mathfrak{uce}(L(k)) \rightarrow L(k)$ be a universal central extension. Since $1/3 \in k$, $\ker u =\ker  f$ and it easy to see that $\langle k,k \rangle  = 0.$ This is evident, since $\langle \alpha, \beta \rangle  = \alpha\beta \langle 1, 1\rangle = 0$ for all $\alpha, \beta \in k.$ Let $ L = \mathfrak{sl}_3(k).$ Then $L$ is centreless and of Jordan type with associated Jordan pair $\mathbf{M}(1, 2, k).$ Hence the central extension $\mathfrak{sl}_3(k) \rightarrow L(k)$ which exists by Corollary~\ref{uTKK_universal_prop_JP} is an isomorphism.
\end{proof}
\begin{rem}Of course, Lemma \ref{sl_3_baby_lem} is well known, see for instance \cite[Cor 3.14]{vdK}. 
\end{rem}

\subsubsection{The classical setting}
Let $\mathbb O$ be an octonion algebra over a field containing $1/6$. 
The reader is referred to \cite[Section 2.4]{TasteOfJA} to check the finer details of the following claims:  All associators and all commutators are skew with respect to the involution which is central and the centre is spanned by $1_{\mathbb O}.$ 
Using for instance the argument in \cite[p. 158]{TasteOfJA} one can easily show that: 
If $1/3 \in k,$ then
\begin{equation}(\mathbb O, \mathbb O, \mathbb O) = [\mathbb O, \mathbb O] = \frac{\mathbb O}{k1_{\mathbb O} } \label{octonion_field_comm_ass_equation} \end{equation} so that the commutator and associator space coincide and both have dimension $7.$

We will first prove that over a field $k$ not of characteristic $2, 3$, the map $\langle a, b \rangle \mapsto SD(a, b)$ is an isomorphism. 
Recall the definition of $\langle D,  D \rangle$ for an alternative algebra $D$ from Corollary~\ref{assoc_coord_A_2_cor}. 

\begin{prop} Let $\mathbb O$ be an octonion algebra over a field $k$, $1/6 \in k$. 
Then $\dim(\langle \mathbb O , \mathbb O\rangle ) = 14$. 
\end{prop}
\begin{proof} By Remark~\ref{remark_1_d_is_trivial}, $\langle 1, \mathbb O \rangle = \langle \mathbb O, 1 \rangle = 0$.  
If $\mathbb O$ is an octonion algebra over a field containing $1/3$, then  it is easy to verify (see above) that $[\mathbb O, \mathbb O]$ is $7$-dimensional and $[\mathbb O, \mathbb O] = (\mathbb O, \mathbb O, \mathbb O).$ \\
Let $\{\mathbb O, \mathbb O\} = \mathbb O \wedge \mathbb O/(1 \wedge \mathbb O).$ Then $\langle \mathbb O , \mathbb O\rangle $ is a quotient of $\{ \mathbb O , \mathbb O\}$ and $\dim \{ \mathbb O , \mathbb O\} = 21$. Define $\mathbb O \wedge \mathbb O  \rightarrow \mathbb O$ by 
$a \wedge b \mapsto [a, b].$

The image of this map has dimension $7.$ Clearly, the element $1$ commutes with all elements in $\mathbb O$ and thus the image of the composition map from $\{ \mathbb O , \mathbb O\}$ to $\mathbb O$ is also $7$-dimensional : 
$$\{ \mathbb O, \mathbb O \} = \mathbb O \wedge \mathbb O/(1 \wedge \mathbb O) \rightarrow \langle  \mathbb O, \mathbb O \rangle \rightarrow \mathbb{O}$$ since the map on the left is surjective.
Denote by $\tilde U$ the image of $U''$ in $\{\mathbb O, \mathbb O\}.$ Note that $\{\mathbb O, \mathbb O\}/\tilde U = \langle \mathbb O, \mathbb O \rangle.$
Then  $\tilde U$ is spanned by  the elements $\{ab,  c \}  + \{bc, a\} + \{ca, b\}$.

 The image of such an element is $[ab, c] + [bc, a] + [ca, b] = 3(a, b,c).$ Since $(\mathbb O, \mathbb O, \mathbb O) = [\mathbb O, \mathbb O],$  the submodule $\tilde U \subset \{\mathbb O, \mathbb O\}$ maps onto a $7$-dimensional space and must therefore have dimension at least $7.$  By the rank theorem $\dim \langle \mathbb O, \mathbb O \rangle \leq \dim \{\mathbb O, \mathbb O\} - 7 = 14.$
We also have a map  from  $\langle \mathbb O, \mathbb O \rangle$ onto $\mathrm{Der}(\mathbb O)$ by mapping $\langle a, b \rangle \mapsto SD(a,b)$ since in the presence of $1/3$ all derivations are standard. Therefore $\dim \langle \mathbb O, \mathbb O \rangle \geq 14 = \dim (\mathrm{Der}(\mathbb O))$ (see \cite{lopera} or also \cite{sprinveldbook}).
Hence $\dim \langle \mathbb O, \mathbb O \rangle = 14.$  
\end{proof}
\begin{cor}
\label{cor_classical_octonion_iso} Over a field $k$ not of characteristic $2, 3$, the map $f: \langle a, b \rangle \mapsto SD(a, b)$ is an isomorphism
$$f: \langle \mathbb O, \mathbb O \rangle \rightarrow \mathrm{StanDer}(\mathbb O). $$
\end{cor}
\begin{proof} Since every derivation of $\mathbb O$ is a standard derivation, this is obviously a surjection. Moreover, $\dim (\mathrm{Der}(\mathbb O)) = 14$, so that by the rank theorem for finite dimensional vector spaces, it follows that it is also bijective. 
\end{proof}

\begin{cor} If $k$ is a field containing $1/6$, then $L_k(\mathbb O)$ is simply connected. \label{octon_simple_field_cor}
\end{cor}
\begin{proof}  For $L  = L_k(\mathbb O)$ let  let $u : \mathfrak{uce}(L) \rightarrow  L$ be the universal central extension. Then $\ker u = \ker f$ where
$f : \langle \mathbb{O}, \mathbb O \rangle \to \mathrm{StanDer}(\mathbb O), \langle a, b \rangle \mapsto SD(a, b).$  The Lie algebra $L_k(Z)$ is perfect and by Corollary~\ref{cor_classical_octonion_iso}, $\ker u = \ker f = \{0\}.$ Hence $L$ is perfect and centrally closed, i.e., simply connected.
\end{proof}
\begin{rem}
Of course, this is well-known and not our own result. Usually, one shows that $\mathrm{Der}_k(\mathbb O)$ is isomorphic to a Lie algebra of type $G_2$ and then simple connectedness follows. 
\end{rem}
 \begin{defi} 
 \label{octonion_alge_Defi}
 Let $k$ be an arbitrary base ring and $M^+ = M$ a finitely generated projective module of rank $3$ over $k$, $\theta : \bigwedge^3 M^* \rightarrow k$ an isomorphism, referred to as  \emph{volume element}. We say that $\mathbb O$ is a \emph{(reduced) octonion algebra} over $k$ if $\mathbb O$ is a $k$-algebra and isomorphic to the alternative algebra
$$Z := Z(k, \theta) = \left (\begin{array}{cc} k & M^+ \\ M^- & k \end{array} \right ) $$
with multiplication given by 
$$ \left ( \begin{array}{cc} \alpha_1 & u \\ x & \alpha_2 \end{array} \right )\left ( \begin{array}{cc} \beta_1 & v \\ y & \beta_2 \end{array} \right ) = \left ( \begin{array}{cc} \alpha_1 \beta_1 -  u \ast y  & \alpha_1 v + \beta_2 u + x \times y \\ \beta_1 x + \alpha_2 y + u \times v & \alpha_2\beta_2 -  x \ast v  \end{array} \right ),$$
where $M^- = M^*$ and $\ast: M^\sigma \times M^{-\sigma} \to k$ are the natural pairings induced by $\theta$ and $\theta^{-1}$. Likewise $M^\sigma \times M^\sigma \rightarrow M^{-\sigma}$ is the ``vector product '' induced by those pairings. More explicitly, for $x,y \in M^+$ the element $ z = x \times y$ is uniquely determined by the condition 
$$\theta( x \wedge y \wedge z ) = (x \times y) \ast z $$ and for $u, v \in M^{\sigma}$, $w = u \times v$ is likewise determined by $$\theta^{-1}(w \wedge u \wedge v) = w \ast(u \times v) .$$
If $M$ is free over $k$ and $\theta$ is the determinant, the algebra $Z$ is a \emph{split octonion algebra.} An \emph{octonion algebra} is an algebra  $Z$ over $k$ such that under some faithfully flat extension $k \to K$, $Z \otimes_k K$ is isomorphic to a split octonion algebra over $K.$  
\end{defi}
\begin{rem} Usually we define an octonion algebra as a unital non-associative algebra whose underlying $k$-module is finitely generated and projective of constant rank $8$ and which admits a norm (or composition). By \cite[Cor 4.11]{lopera}, equivalently $\mathbb O$ is an octonion algebra if and only if there is a faithfully flat base change  $ k \rightarrow K$ such that $\mathbb O \otimes_k K$ is isomorphic to a split octonion algebra over $K$. This characterization is absolutely crucial for the main result of this section, see Theorem~\ref{simply_Connected_oct_theo}. 
\end{rem}


\begin{defi}
\label{type_g_defi}
 \label{G_2_ring} 
 
 If $k$ is an algebraically closed field, a Lie $k$-algebra is said to be of \emph{type $G_2$} if it is isomorphic to the derivation algebra of an octonion algebra $\mathbb O$ over $k$ (which is necessarily split). \\
 In general, if $k$ is a ring, we define a \emph{Lie algebra of type $G_2$} as a  Lie $k$-algebra $L$ such that for every $\mathfrak{p} \in \mathrm{Spec}(k),$ the prime spectrum of $k,$ the Lie algebra $L \otimes_{\overline{Q(k_{\mathfrak p})}} {\overline{Q(k_{\mathfrak p})}}$ is of type $G_2,$ where $Q(k_{\mathfrak p})$ is the quotient field of $k_{\mathfrak p}$ and $\overline{Q(k_{\mathfrak p})}$ its algebraic closure. 
 \end{defi}

\begin{prop} If $\mathbb O$ is an octonion algebra, then $\mathrm{Der}_k(\mathbb O)$ is of type $G_2.$   \label{oct_der_G_2_prop}
\end{prop}
\begin{proof}Let $\mathfrak p \in \mathrm{Spec}(k)$ By  \cite[Ch.2 Thm 1]{AC}, we know that $Q(k_{\mathfrak p})$ is flat over $k_{\mathfrak p}$ and $k_{\mathfrak p}$ is flat over $k$ for all $\mathfrak p \in \mathrm{Spec}(k).$  Let $F = \overline{Q(k_{\mathfrak p})}.$ Since $Z$ is finitely presented, this implies that $\mathrm{Der}_F(\mathbb O \otimes_k F) \cong \mathrm{Der}_F(\mathbb O) \otimes_k F.$ By $\cite[4.1]{lopera},$ octonion algebras are invariant under base changes. Hence $\mathbb O \otimes_k F$ is an octonion algebra over the algebraically closed field $F$, $\mathrm{char} F \neq 3$ and therefore split. Now the claim follows from Definition~\ref{G_2_ring}. 
\end{proof}

\begin{lem}Let $\mathbb O$ be an octonion algebra over $k$.  The following are equivalent
\begin{itemize}
\item[\rm (i)] $L_k(\mathbb O)$ is centrally closed.
\item[\rm (ii)] $L_K(Z)$ is centrally closed for some split octonion algebra $Z= \mathbb O \otimes_k K$ and $k \rightarrow K$ a faithfully flat extension.
\end{itemize}
  \label{octonionfundobservatationlem} 
\end{lem}
\begin{proof}
Let $\mathbb O$ be an octonion algebra over $k$ and $k/K$ flat. Then the Jordan pair $\mathbf{M}(1, 2, \mathbb O)$ is a finitely generated $k$-module and by flatness, $\instr(\mathbf{M}(1, 2, \mathbb O_K)) = \instr(\mathbf{M}(1, 2, \mathbb O)) \otimes_k K, $ therefore $L_K(\mathbb O_K) = L_k(\mathbb O) \otimes_k K.$ \\
It known (\cite{prep2009EN}) that the following are equivalent for any $k$-Lie algebra $L$:
\begin{itemize}
\item[(a)] $L$ is centrally closed,
\item[(b)] $L_k \otimes _k K$ is centrally closed for some $K/k$ faithfully flat, 
\item[(c)] $L_k \otimes _k K$ is centrally closed for all $K/k$ faithfully flat.
\end{itemize}
Therefore, (i) $\implies$ (ii) follows from the equivalence of (a) and (c). Also, (ii) $\implies$ (i)  by (a) $\iff$ (b). 
\end{proof}

If $Z$ is a reduced octonion algebra, then $e := e_{{1}} := e_{11}$ and $1- e = e_{{2}}= e_{22}$ are complementary idempotents and we obtain a Peirce decomposition with respect to $e$: 
$$Z_{{21}} = M^+, \quad Z_{{12}} = M^-, Z_0 : =  Z_{11} + Z_{22} =  k(e_{{1}}) + k (e_{{2}}), $$
with the usual Peirce multiplication rules. 
We already know that $\langle e_{1} + e_{2}, Z\rangle = \langle 1_Z, Z\rangle \{0\}$ by Remark~\ref{remark_1_d_is_trivial}. 
\begin{lem}The module  $\langle Z, Z \rangle $ is spanned by 
$$\langle e_{1}, Z_0 \rangle \cup  \langle  e_{2}, Z_0 \rangle \cup \langle Z_{{12}}, Z_{{21}}\rangle . $$  
\end{lem} 
\begin{proof} First observe that for $i \neq j:$
$$ 0 = \langle e_i^2, e_j\rangle + \langle e_i e_j, e_i\rangle + \langle e_je_i, e_i \rangle = \langle e_i, e_j \rangle.$$
Since $\langle e_{{1}} + e_{{2}}, z \rangle = \langle e_{{1}}, z \rangle + \langle e_{{2}}, z \rangle = 0$ we have
$ 0 = \langle e_i + e_j, e_i \rangle = \langle e_i, e_i \rangle$ and therefore $\langle k e_1 + k e_2, k e_1 + k e_2  \rangle = 0.$
Let $a, b \in Z_{ij}$, then
$$0= \langle e_i a, b\rangle + \langle ab , e_i\rangle + \langle be_i , a\rangle  = \langle a, b\rangle + \langle ab , e_i\rangle. $$ Since $ab \in Z_{ji}$ it follows that $\langle Z_{ij}, Z_{ij}\rangle \in \langle e_i, Z_0 \rangle.$
We showed that that $\langle e_1, z \rangle = -\langle e_2, z \rangle $, so
consequentially $$\langle e_{1}, Z_{12} \rangle \cup  \langle  e_{2}, Z_{21} \rangle \cup \langle Z_{{12}}, Z_{{21}}\rangle  $$  
spans all of $\langle Z, Z \rangle.$
\end{proof}
Our ultimate goal is to establish for $1/3 \in k$, that the map
$$ f : \langle Z, Z\rangle \to \mathrm{Der}_k(Z) $$ is an isomorphism. In \cite{lopera} the authors established a decomposition
$\mathrm{Der}_k(Z) = \mathfrak g_0 \oplus \mathfrak g_1 \oplus \mathfrak{g}_2$
where $\mathfrak g_i$, $i = 1,2$ are projective of rank $3$ and isomorphic to $M$ and $\mathfrak{g}_0 = \{D  \in \mathrm{Der}_k(Z) : De = 0\}$ is isomorphic as Lie algebra to $\mathfrak{sl}(M)$ and thus projective of rank $8$. We will imitate this approach.
\begin{lem} Assume that $Z$ is a reduced octonion algebra.
Under the map $ f: \langle Z, Z \rangle \to \mathrm{Der}(Z)$  given by $\langle a, b\rangle \mapsto SD(a, b),$ we have for $u, v \in M, x, y \in M^*$
$$f(\langle e_{1}, u\rangle) = L_{u} - R_{u},$$
$$f(\langle e_{2}, x \rangle) = L_{- x} - R_{-x},$$
$$f(\langle u, x\rangle)v  = (u \ast x)v - 3(v \ast x)u, $$
$$f(\langle u, x\rangle)y  = - (x \ast u)y + 3(y \ast u)x, $$
$$f(\langle u, x\rangle)e = 0, $$
and the following induced maps
\begin{eqnarray}
f: \langle e_1, M \rangle & \stackrel{\cong}{\to}& \mathfrak{g}_1, \\
f: \langle e_2, M^* \rangle &\stackrel{\cong}{\to}& \mathfrak{g}_2, \\
f :\langle M, M^* \rangle & \stackrel{epi}{\rightarrow} & \mathfrak{g}_0 \cong \mathfrak{sl}(M) \label{iso_onto_sl_m}.
\end{eqnarray}
\end{lem}
\begin{proof}
The first two isomorphisms can be proven in the same way. So let without loss of generality $x \in M.$ The image of $\langle e_1, x \rangle $ is the standard derivation $SD(e_1,x)$ which is equal to  $L_{x} - R_{x}$. Since $M \oplus M^*$ does not contain  non-zero  central elements of $Z$, the restriction of $f$ to $\langle e_1, M \rangle$ is therefore injective. 
 By \cite[Proposition 5.5]{lopera}, $\mathfrak{g}_1$ is spanned by standard derivations of the form $- SD(e_1, x) = SD(e_2, x)$ with $x \in M.$ Therefore $f(\langle e_1, M \rangle)= \mathfrak{g}_1 $ is an isomorphism.\\
We compute the action of $f(\langle u, x \rangle ) = SD(u,x) = L_{[u,x]} - R_{[u,x]} - 3[L_u, R_x].$
First note $[u,x] = u \ast x(e_2 - e_1)$ and this element commutes by Peirce rules with $e = e_1.$ In addition $(ue)x - u(ex) = ux -ux = 0$, thus $SD(u,x)e = 0= f(\langle u, x\rangle)e = 0.$\\
  Let $v\in M$, then $SD(u,x)v = u\ast x((e_2 - e_1)v - v(e_2 - e_1))- 3(u(vx) - (uv)x) = -2(u\ast x)v + 3(u \times v)\times x = (u \ast x)v - 3(v \ast x)u. $ The computation for $y \in M^*$ is analogous (or dual if one wishes).\\
For every $u \in M, x \in M^*$, $f(\langle u,x \rangle)$ annihilates $e$. Thus $f(\langle M, M^* \rangle ) \subset \mathfrak{g}_0.$  Since $f$ is surjective and $f(\langle k e_1 \oplus ke_2, M \oplus M^* \langle) = \mathfrak{g}_1 \oplus \mathfrak{g}_2$, it follows that $f(\langle M, M^* \rangle ) = \mathfrak{g}_0.$   
\end{proof}
\subsubsection{The Lie algebra structure on $\langle M, M^*\rangle $}
As observed above in (\ref{iso_onto_sl_m}), $f$ induces an epimorphism of Lie algebra $\langle  M, M^* \rangle$ onto $\mathfrak{sl}(M)$.

\begin{lem} Let $M = k^3$ be free of rank $3$. Let $i, j, k, l \in \{1,2,3\},$ then 
\begin{equation}[\langle x_i, x^j \rangle, \langle x_k, x^l \rangle]  = 3(\delta_{il} \langle x_k, x^j \rangle - \delta_{jk}\langle x_i, x^l \rangle) \label{MMstarLiebracket} \end{equation}
where $\{x_1, x_2, x_3\}$ and $\{x^1, x^2, x^3\}$ are ordered dual bases of $M$ and $M^*$ respectively. 
Therefore, $\langle M, M^* \rangle$ is perfect as Lie algebra.  \label{MMstarsplitperfectlem}
\end{lem}
\begin{proof}
By definition \begin{eqnarray*}
&& [\langle x_i, x^j \rangle, \langle x_k, x^l \rangle] \\ 
&=& \langle SD(x_i, x^j)x_k, x^l \rangle + \langle x_k, SD(x_i, x^j)x^l \rangle\\
&=& \langle \delta_{ij}x_k - 3 \delta_{jk}x_i, x^l \rangle  + \langle x_k, -\delta_{ij}e^l + 3 \delta_{il}x^j \rangle \\
&=& 3(\delta_{il}\langle x_k, x^j \rangle - \delta_{jk} \langle x_i, x^l \rangle).
\end{eqnarray*}
We first show that the element $\sum \langle x_i, x^i \rangle$ is $0$. Observe that for pairwise distinct (cyclically ordered) indices $i,j$ and $k$:  $x^ix^j = x_k.$
Therefore  $$0 = \sum_{cyc} \langle x^i x^j, x^k\rangle = \langle x_k, x^k \rangle + \langle x_i, x^i \rangle + \langle x_j, x^j \rangle, $$ which is the asserted identity.  \\
Claim: $\langle  M, M^*\rangle $ is finitely generated (as module) by the set $\{ \langle x_i , x^j \rangle ,  \langle  x_i, x^i \rangle - \langle x_j, x^j \rangle, 1 \leq i < j \leq 3\}.$   
It suffices to show that every element in $ M \wedge M^*/ \sum_{i = 1}^3  x_i \wedge x^i$ can be expressed as linear combination of those elements. This is only a question for images of $x_i \wedge x^i.$ Since $ \sum_{i = 1}^3  x_i \wedge x^i$ lies in the kernel, the image of $x_i \wedge x^i$ is in the same coset as the image of $1/3 (2 x_i \wedge x^i - x_j \wedge x^j  -x_k \wedge x^k )$ for $i,j,k$ pairwise distinct.  Note that we used $1/3$ here. \\
The formula (\ref{MMstarLiebracket}) shows  that $\{ \langle x_i, x^j \rangle: i\neq j \}$ generates $\langle M, M^* \rangle$ as Lie algebra and that $\langle M, M^* \rangle$ is perfect: 
$$[\langle x_i, x^j \rangle, \langle x_j, x^i\rangle] = 3( \langle x_j, x^j \rangle - \langle x_i, x^i \rangle ), $$ hence all module generators lie in the Lie algebra generated by $\{ \langle x_i, x^j \rangle : i\neq j \}$. Moreover, every generator $\langle x_i, x^k \rangle$ can be expressed as
$$\langle x_i, x^k \rangle = 1/3[\langle x_i, x^j \rangle, \langle x_k, x^i \rangle] $$
for $\{i,j,k\} = \{1,2,3\}.$ 
\end{proof}
\begin{prop} Assume that $Z$ is a split octonion algebra. The map $f$ induces an isomorphism $\langle M, M^* \rangle \rightarrow \mathfrak{sl}(M). $ \label{MMstarslisoprop}
\end{prop}
\begin{proof}
Since $f: \langle Z, Z \rangle \to \mathrm{Der}_k(Z)$ is a central extension, it follows that the restriction
$$f : \langle M, M^* \rangle \rightarrow \mathfrak{sl}(M) $$ must be a central extension as well. If $M = k^3$ is free, then 
by a theorem of van der Kallen (see for instance \cite[Cor 3.14]{vdK}) or by Lemma \ref{sl_3_baby_lem}, $\mathfrak{sl}(M)= \mathfrak{sl}_3(k)$ is simply connected if $1/3 \in k.$  \\
Hence, (by Lemma~\ref{MMstarsplitperfectlem}) $f: \langle M, M^* \rangle \rightarrow \mathfrak{sl}(M)$ is a covering of a simply connected Lie algebra and thus an isomorphism. 
\end{proof}
\begin{cor} If $Z$ is a split octonion algebra over a ring containing $1/3$, then
$$ f: \langle Z, Z \rangle \to \mathrm{Der}_k(Z) $$
is an isomorphism and $L(Z)$ is simply connected.\label{splitsimplyconncor}
\end{cor}
\begin{theo} Let $Z$ be an octonion algebra over a ring $k$ containing  $1/3.$ Then:
\begin{itemize}
\item[\rm(i)]
\label{simply_Connected_oct_theo}The map
$$f: \langle Z , Z \rangle \stackrel{\cong}\longrightarrow \mathrm{Der}_k(Z), \quad \langle a, b\rangle \mapsto SD(a,b),$$
is an isomorphism. 
In particular $\langle M, M^* \rangle \cong \mathfrak{sl}(M).$ 
\item[\rm(ii)] The Lie algebra $L_k(Z)$ is simply connected.
\end{itemize}
\end{theo}
\begin{proof} This follows from Corollary \ref{splitsimplyconncor}, Proposition \ref{MMstarslisoprop} and Lemma \ref{octonionfundobservatationlem}.
\end{proof}

\section{Lie algebras graded by $A_n$, $n \geq 3$}
\label{associative_coordinatization}
Let $D$ be an associative unital $k$-algebra.


We fix some notation for the rest of the section: 
\begin{itemize}
\item $K$ is a set, $\card K \geq 4$ and we fix a partition of $K = I \dot \cup J$ where $I = \{1\}$ is  a singleton, 
\item $R$ is the root system $\dot A_K,$ together with the collinear $3$-grading given  by 
$R_1 = \{\epsilon_1 - \epsilon_j : j \in J\}.$
 \item  $\mathbb{DS}:= \mathbb{DS}(A_3),$ $\mathbb{DP}:= \mathbb{DP}(A_3),$
\item  $D$ is an associative unital algebra,
\item $V  = \mathbf{Mat}(I ,J , D)$ is an associative matrix Jordan pair  with coordinates in $D,$
\item $L = \TKK(V).$
\end{itemize}

All of the followig has already been established at the beginning of the section:

If $u: \mathfrak{uce}(L) \rightarrow L$ is the universal central extension then

\begin{equation} \label{A_3_in_description_recal;} \ker(u) = \bigoplus_{\alpha \in \mathbb{DS}(A)} \mathfrak{uce}(L)_\alpha \oplus (\ker u)_0 \end{equation}

Therefore: 
\begin{itemize}
\item[-] If $\card J = 3$, then 
$\ker(u) = \bigoplus_{\alpha \in \mathbb{DS}(A_3, 2)} \mathfrak{uce}(L)_\alpha \oplus (\ker u)_0$
where $ \mathbb{DS}(A_3, 2)$ is the set of degenerate sums of divisor $2$ in $Q(A_3).$
\item[-] If $\card J > 3$, then
$\ker u =  (\ker u)_0.$
\end{itemize}

\begin{cor} \label{root_space_iso_lem_A_3}Let $u: \mathfrak{uce}(L) \rightarrow L$ be a universal central extension.
 Then $u: \mathfrak{uce}(L)_\alpha \rightarrow L_\alpha$ is a bijection for $ \alpha \in R^\times$. 
\end{cor}
\begin{proof}
This is an immediate consequence of Proposition~\ref{torsion_bijection_prop}. 
\end{proof} 
\begin{defi}Let $i \neq j \in K$, $a \in D$ and $g : K \rightarrow L$ a  $\mathcal Q(R)$-graded covering. By Corollary~\ref{root_space_iso_lem_A_3} the following is a well-defined map
$$x_{ij}(a) :  = g^{-1}(E_{ij}a).$$ 
$$x_{ji}(a) : = g^{-1}(E_{ji}a).$$
Also, by Corollary~\ref{collinear_zero_hom_Cor}
$L_0 = \sum_{j \in J }[E_{1j}D, E_{j1}D]$ and $K_0 = \sum_{j \in J}[x_{1j}(D), x_{j1}(D) ].$
\end{defi}

\subsection{Degenerate sums $A_3$}

There are three root systems of type $A$ which admit degenerate sums: $A_1,$ $A_2$ and $A_3.$ The root system $A_1$ can be considered a special case of type $C$ and will be dealt with in Section \ref{Section C_n_coordinatization}. For type $A_2$ see Section \ref{alternative_coordinatization}. The upcoming section~\ref{associative_coordinatization_sec} applies to type $A_3$ since the coordinates of a typical $A_3$-graded Lie algebra  are associative. However, we yet have to describe $\mathfrak{uce}(L)_{\alpha + \beta}$ when $\alpha + \beta$ is a degenerate sum. By \ref{tab:DegSums}, $\mathbb{DS}(A_3)=  \pm \{\varepsilon_1 - \varepsilon_2 + \varepsilon_3 - \varepsilon_4, \varepsilon_1 - \varepsilon_2 - \varepsilon_3 + \varepsilon_4, \varepsilon_1 + \varepsilon_2 - \varepsilon_3 - \varepsilon_4 \}.$ We also note that every degenerate sum  can be expressed (up to switching the summands) in two ways as sum of two roots and that with respect to any $3$-grading on $A_3$, the degenerate sums have degree $\pm 1$.

\begin{defi}
If $D$ is an associative  $k$-algebra, then define $D_2$ to be the following quotient in the category of $k$-algebras
$$D_2 = D /(2 D,  [D, D]) .$$  The coset of $a \in D$ in  $D_2$ is denoted by $\bar a$.
\end{defi}
\begin{lem} $D[D, D] = [D, D]D = D[D, D]D$
\end{lem}
\begin{proof}
First it is to show that $[D, D]D \subset D[D, D].$ In all associative unital algebras $D$
$[D,D]D = [[D, D],D] + D[D, D] = D[D, D], $. 
With the same argument $D[D, D] = [D, D]D$ and $[D,D] D \subset D[D, D]D \subset [D, D]D^2 \subset [D,D]D$ whence equality everywhere. 
\end{proof}
We assume from now on $\card J = 2.$
\begin{defi}Let $\mathcal Z = \bigoplus_{\gamma \in \mathbb{DS}}(D_2)_\gamma$ be a direct sum of six copies of $D_2$, labeled by pairwise distinct degenerate sums. 
Define a bilinear map $\psi : L \times L \rightarrow \mathcal Z $ by 
 \begin{eqnarray}
\psi(E_{ij} a, E_{kl} b) &=& \begin{cases} (\overline {ab})_\gamma & \text{ if }  (\epsilon_i - \epsilon_j,  \epsilon_k - \epsilon_l) \in \mathbb{DP}(A_3, 2),\\
0& \text{ else. } \end{cases} \label{cocyc_A_3_def} \\
\psi(L_0, L) &=& \psi(L, L_0) = \{0\} , \label{cocyc_A_3_def2}
\end{eqnarray}
\label{A_3_cocyle_Defi}. 
\end{defi}
\begin{lem}
\label{Lemma_1_for_A_3}
\begin{itemize} 
\item[\rm (i)]
The bilinear map $\psi$ is a $\mathcal{Q}(R)$-graded $2$-cocycle on $L$ with coefficients in $\mathcal Z$. 
\item[\rm (ii)] With the definition of $L \oplus_{\psi} \mathcal Z $ as in Definition~\ref{central_ext_by_cocyle}, the  central extension $L \oplus_{\psi} \mathcal Z \rightarrow L$ is a $\mathcal{Q}(R)$-graded covering.
\item[\rm (iii)] By the universal property of $\uce L $ there is a a unique graded Lie algebra epimorphism $\pi : \uce L \rightarrow L \oplus_{\psi} \mathcal Z$  such that $$\pi(\uce {L}_\gamma) = \mathcal Z_\gamma \cong D_2$$ for all $\gamma \in \mathbb{DS}.$
\end{itemize}
\end{lem}
\begin{proof}
It is clear that $\psi$ is well-defined and bilinear.
\begin{itemize}
\item[\rm(i)]  
\begin{itemize}
\item
Let $y \in L$. Then $y = \sum_{ 1 \leq i \neq j \leq 4}E_{ij}(a) + y_0$ with $y_0 \in L_0,$ We need to show $\psi(y,y) = 0.$ Since we have (\ref{cocyc_A_3_def2}) we may assume that $y_0 = 0.$  
Since $(\epsilon_i - \epsilon_j, \epsilon_i - \epsilon_j ) \notin \mathbb{DP},$ the element $\psi(E_{ij}a, E_{ij}a)$ equals zero. Further, if $\{i,j,k,l\} = \{1,2,3,4\},$ then $(\epsilon_i - \epsilon_j, \epsilon_k - \epsilon_l )$ is  degenerate pair and 
$$\psi(E_{ij}a, E_{kl}b) = \overline{ab}_{(\epsilon_i - \epsilon_j + \epsilon_k - \epsilon_l )} = \overline{ba}_{(\epsilon_i - \epsilon_j + \epsilon_k - \epsilon_l )} = -\overline{ba}_{(\epsilon_i - \epsilon_j +\epsilon_k - \epsilon_l )} = -\psi( E_{kl}b, E_{ij}a).$$
If  $\{i,j,k,l\} \neq \{1,2,3,4\}$, then $(\epsilon_i - \epsilon_j, \epsilon_k - \epsilon_l ) \notin \mathbb{DP}(A_3, 2),$ thus $\psi(E_{ij}a, E_{kl}b) = 0.$ We have shown that $\psi$ is alternating.  
\item  For  three homogeneous elements $x,y,z$  in $L$ consider the expression $\psi(x, [y,z]) + \psi(y, [z,x]) +  \psi(z, [x,y]).$ First assume that none of the three elements has degree $0$. Assume that  $x = E_{ij} a$, $y = E_{kj}b$ and $z = E_{jl}c$ for $i,j,k,l$ pairwise distinct. By assumption $[x, y] = 0.$ It remains to look at 
$$ \psi(x, [y,z]) + \psi(y, [z,x]) = \psi(E_{ij}a, E_{kl}(bc)) - \psi(E_{kj}b, E_{il}ac).$$
Since $\gamma =  \epsilon_i - \epsilon_j  + \epsilon_k - \epsilon_l = \epsilon_k - \epsilon_j + \epsilon_i - \epsilon_l,$ this
is equal to $(\overline{abc - bca})_\gamma= (\overline{abc - abc})_\gamma = 0.$ This shows that $\psi(x, [y,z]) + \psi(y, [z,x]) +  \psi(z, [x,y]) = 0.$ Since this expression is invariant under cyclically permuting the arguments, we may next assume that 
$x = E_{ji} a$, $y = E_{kj}b$ and $z = E_{jl}c$ for $i,j,k,l \neq.$ Then $[x, z] = 0,$ we are left with the claim
$$  0 = \psi(x, [y,z]) + \psi(z, [x,y]) = \psi(E_{ij}a, E_{kl}(bc)) - \psi(E_{jl}c, E_{kj}ba)$$
Since $\gamma =  \epsilon_i - \epsilon_j  + \epsilon_k - \epsilon_l = \epsilon_j - \epsilon_l + \epsilon_k - \epsilon_j$ the same argument as above gives $\psi(x, [y,z]) + \psi(y, [z,x]) +  \psi(z, [x,y]) = 0.$ 
All other cases where all elements have non-zero degree, are equivalent to one of the two that we have just considered. 
If $2$ of the $3$ arguments  have degree $0,$ then by (\ref{cocyc_A_3_def2}), $\psi(x, [y,z]) + \psi(y, [z,x]) +  \psi(z, [x,y]) = 0.$ \\
If one of the three arguments is of degree $0$, say $x \in L_0,$ then without loss of generality $ y = E_{ij}a$ and 
$ z = E_{kl}b$ and $\{i,j,k,l\} = \{1,2,3,4\}.$ It is enough to consider the two  cases $x = [E_{ij}c, E_{ji}d]$ or $x = [E_{il}c, E_{li}d].$ Clearly, in all these cases $\psi(x, [y,z]) = 0.$ 
In the first case, $[x,y] = \{ E_{ij}c, E_{ji}d ,E_{ij}a\} = E_{ij}(cda + adc)$ and $[z,x] = - \{ E_{ij}c, E_{ji}d, E_{kl}b\} = 0.$ Then, $\psi(z,[x,y]) = (\overline{bcda + badc})_{\gamma} = 2(\overline{abcd})_\gamma = 0$ where $\gamma = \epsilon_i - \epsilon_j  + \epsilon_k - \epsilon_l. $ \\
In the second case, $[x,y] = \{ E_{il}c, E_{li}d, E_{ij}a\} = E_{ij}(cda)$ and $[z, x] = - \{ E_{il}c, E_{li}d, E_{kl}b\} = -E_{kl}(bdc).$ Then $\psi(z, [x,y]) + \psi(y, [z,x]) = (\overline{bcda + abdc})_{\gamma} = 2(\overline{abcd})_{\gamma} = 0$ with $\gamma$ as above.
\end{itemize}
\item[\rm(ii)] The Lie algebra $L$ is perfect and generated by its homogeneous components of non-zero degree. Let $x= \sum [x_i, y_i] \in L$ where all $x_i, y_i$ are non-zero and homogeneous, $\sum [(x_i, 0), (y_i, 0)] = (x, 0) \in L \oplus_{\psi} \mathcal Z.$ So $(L, 0) \subset [ L \oplus_{\psi} \mathcal Z,  L \oplus_{\psi} \mathcal Z].$ Let $(a)_\gamma \in \mathcal Z$ be homogeneous of degree $\gamma.$ By definition of $\psi$, there is  a non degenerate pair $(\alpha, \beta)$ with $[(x_\alpha(a), 0), (x_\beta(1), 0)] = \psi(x_\alpha(a), x_\beta(1)) = (a)_\gamma$ and therefore $\mathcal Z \subset [ L \oplus_{\psi} \mathcal Z, L \oplus_{\psi} \mathcal Z].$ Thus proves that $L \oplus_{\psi} \mathcal Z$ is perfect.  
\item[(iii)] The last part follows from the universal property of $\mathfrak{uce}(L).$ 
\end{itemize}
\end{proof}
\begin{lem}Let $ f: L' \rightarrow L$ be a $\mathcal Q(R)$-graded covering and $\gamma =\alpha + \beta  \in (R+R)\setminus R$.\label{lemma_deg_sums}
\begin{itemize}
\item[\rm{(i)}] There is a $k$-module epimorphism $D_2\rightarrow [L'_\alpha, L'_\beta]$ given by $\bar{a} \mapsto [x_\alpha(a), x_\beta(1)].$
\item[\rm(ii)]If $\gamma = \alpha + \beta$ is not a degenerate sum, then $[L'_\alpha, L'_\beta] = \{0\}$. If $\gamma = \alpha + \beta$ is  a degenerate sum, then $L'_\gamma = [L'_\alpha, L'_\beta].$ In particular: 
\begin{itemize}
\item[\rm(1)]$[L'_\alpha, L'_\alpha] = \{0\}$ for all $\alpha \in R$,
\item[\rm(2)] $R \subset \supp L' \subset R \cup \mathbb{DS} \subset R+R. $
\end{itemize}
\item[\rm (iii)] If $1/2 \in k,$ then $R = \supp L'.$ 
\end{itemize}
\label{Lemma_2_for_A_3}
\label{module_structure}
\end{lem}
\begin{proof}
Note that every $\gamma \in 
\mathbb{DS}(A_3)$ is of the form $\gamma = (\epsilon_i - \epsilon_j) + (\epsilon_k - \epsilon_l)$ with $\{i, j\} \cap \{k,l\} = \emptyset$. Fix such a $\gamma.$
For convenience we will use the notations $x_{ij}(a)$ and $x_{\epsilon_i- \epsilon_j}(a),$ $i \neq j$  for the same elements. Also, we set $t_{ij}(a) = [x_{ij}(a), x_{ji}(1)]$ where $i, j \in \{1,2,3,4\}$ and $a\in D.$ For $i \neq j, $ 
$$\ad t_{ij}(a).x_{kl}(b)  = (\delta_{ik} - \delta_{kj})x_{kl}(ab) + (\delta_{jl}- \delta_{il})x_{kl}(ba).$$ 
Choose  $m \in \{k,l\}$; then the element $[x_{ij}(b), x_{kl}(1)]$ lies in the centre of $L'$ and applying the Jacobi identity yields:
\begin{eqnarray*}
\left [x_{ij}(ab), x_{kl}(1)\right] &=& \big[[t_{im}(a), x_{ij}(b)], x_{kl}(1)\big] \\
&=& \big[t_{im}(a), [x_{ij}(b), x_{kl}(1)]\big] - \big[ x_{ij}(b), [t_{im}(a),x_{kl}(1)]\big]\\
&=& -\langle( \epsilon_{k}- \epsilon_l), (\epsilon_i - \epsilon_m)\che \rangle  [x_{ij}(b), x_{kl}(a)].
\end{eqnarray*}
The calculation uses that $[x_{ij}(b), x_{kl}(1)]$ is central in $L'.$ 
The value of $-\langle( \epsilon_{k}- \epsilon_l), (\epsilon_i - \epsilon_m)\che \rangle$ is $-\delta_{km} + \delta_{ml} $ is congruent to $1 \mod 2k$. The module $L'_\gamma$ has $2$-torsion, hence 
this proves 
\begin{equation}[x_{ij}(ab), x_{kl}(1)] = [x_{ij}(b), x_{kl}(a)], \quad k\neq i \label{eqn_A_3_lem11}.\end{equation}
We obtain with the same argument 
\begin{equation}[x_{ij}(ab), x_{kl}(1)] = [x_{ij}(a), x_{kl}(b)] \label{eqn_A_3_lem12}.\end{equation}
It follows  from (\ref{eqn_A_3_lem11}) and (\ref{eqn_A_3_lem12}) that the $k$-module $[L'_\alpha, L'_\beta]$ is spanned by elements of the form 
$[x_{\alpha}(a), x_\beta(1)]$ where $ a \in D$ and that there is a well-defined  $k$-module epimorphism $D_2  \rightarrow [L'_\alpha, L'_\beta]$ given by mapping $\bar a$ to $[x_\alpha(a), x_\beta(1)].$ This proves (i).\\
Consider two different ways to write $\gamma$ as sum of two roots, $\gamma = \alpha + \beta = \delta + \varepsilon$
where the roots $\alpha, \beta, \delta$ and $\varepsilon$ are all distinct. After possibly re-ordering the roots  $N_{\delta - \beta, \beta} \equiv N_{\alpha -\varepsilon , \varepsilon} \equiv N_{\beta, \varepsilon - \beta} \equiv 1 \mod 2$ as well as  $\alpha + \delta, \beta + \varepsilon \notin A_2$ as shown in \cite[3.7(16)]{vdK}  .
We want to conclude that $[L'_\alpha, L'_\beta] = [L'_\delta, L'_\epsilon].$ It is clearly sufficient to show for all $a \in D$ that
$$[x_\alpha(a), x_\beta(1)] = [x_\delta(a), x_\epsilon(1) ]. $$
 Using the Jacobi identity and $\alpha = \delta + \varepsilon - \beta$ and $N_{\delta- \beta, \beta} \equiv 1 \mod 2k,$ we calculate
\begin{eqnarray*}
[x_\varepsilon(1), x_\delta(a)] &=&  N_{\delta- \beta, \beta} \left[ x_\varepsilon(1), [x_{\delta - \beta}(a),x_\beta(1) ]\right ] \\
&\equiv & [x_\beta(1), [x_{\varepsilon}(1), x_{\delta - \beta}(a)]] +[x_{\delta - \beta}(a),[x_\beta(1), x_\varepsilon(1)]] \mod 2k\\
&=  & N_{\varepsilon, \delta - \beta} [x_\beta(1), x_\alpha(a)] \mod 2k.  
\end{eqnarray*}
The second equality holds because $[x_\varepsilon(1), x_\beta(1)]$ is central in $L'$. Now $- N_{\varepsilon, \delta - \beta} \equiv 1 \mod 2k,$
thus $[x_\delta(a), x_\varepsilon(1)] \equiv [x_\alpha(a), x_\beta(1)] \mod 2k $. \\
Since for $\alpha \in R^\times$ there is no way to express $2\alpha$ as a degenerate sum, part(ii) follows immediately from Proposition~\ref{torsion_bijection_prop}. 
\end{proof}

\begin{theo}Let $\gamma \in \supp_{\mathcal{Q}(R)}\mathfrak{uce}(L)\setminus R.$
\begin{itemize} 
\item[\rm{(i)}]The element $\gamma$ is a degenerate sum $\gamma = \alpha + \beta$, i.e., $\alpha$ and $\beta$ are non-zero distinct roots and their sum is not a root. 
\item[\rm(ii)] For every degenerate sum $\gamma =\alpha + \beta$, there is a $k$-module isomorphism
$$\mathfrak{uce}(L)_\gamma \cong D_2, $$
given by 
$[x_\alpha(a), x_\beta(b)] \rightarrow \overline{ab}.$
\end{itemize}\label{non_roots_theo_A_3}
\label{deg_sum_theo_A_3}
\end{theo}
\begin{proof}

Putting the two lemmas together gives us a $k$-module epimorphism 
$\pi : \mathfrak{uce}(L)_\gamma \rightarrow D_2$ for every degenerate sum $\gamma$ (Lemma~\ref{Lemma_1_for_A_3}) and also a $k$-module epimorphism $\rho : D_2 \rightarrow \mathfrak{uce}(L)_\gamma$ (Lemma~\ref{Lemma_2_for_A_3}).  It remains to show that they are inverse to each other. Let $[x_{ij}(a), x_{kl}(b)] \in \mathfrak{uce}(L)_\gamma$, $\gamma = \epsilon_i - \epsilon_j + \epsilon_k - \epsilon_l$, a degenerate sum. Then $\rho \circ \pi([x_{ij}(a), x_{kl}(b)]) = \rho((\overline{ab})_\gamma) = z_\gamma(a, b) = [x_{ij}(1), x_{kl}(ab)] = [x_{ij}(a), x_{kl}(b)]$. If $\bar a \in D_2$ then $\pi \circ \rho (\bar a ) = \pi ([x_{ij}(1), x_{kl}(a)]) = 1 \bar a = \bar a.$ Thus
$$\mathfrak{uce}(L)_\gamma \cong D_2. $$ 
\end{proof}

\begin{cor} Let $L$ and $D$ be as above. Then: 
\begin{itemize}
\item[\rm(i)] If $1/2 \in k$ then $\supp_{\mathcal Q (R)}\mathfrak{uce}(L) = R$ and $\ker u \subset \mathfrak{uce}(L)_0.$
\item[\rm (ii)] If $D$ is commutative  then $\mathfrak{uce}(L)_\gamma \cong  D/2 D$ for every degenerate sum $\gamma.$
\end{itemize}
\end{cor}
\begin{proof}This is an easy consequence of the fact that $\mathfrak{uce}(L)_\gamma \cong D_2 $ and Proposition~\ref{torsion_bijection_prop}.
\end{proof}

\subsection{The kernel of $\uider(V) \rightarrow \instr(V)$}
By Lemma~\ref{identies_h_lem}, the Lie algebra $\uider(V)_0$ has the following decomposition
$$\uider(V)_0 = H(D,D) \oplus  \sum_{j \in J}(h_j(D)) $$
where $\ud(H(a, b))$ and $\ud(h_j(c))$ are given by Lemma~\ref{uider_to_Liemult_lem}. Recall also that
$\ker \ud = \ker \ud \cap \uider(V)_0 = (\ker u)_0$
where $u  : \mathfrak{uce}(L) \rightarrow L$ is the $\mathcal Q(R)$-graded central extension and $(\ker u)_0 = \ker u \cap \mathfrak{uce}(L)_0$ where the grading is the $Q(R)$-grading.

\begin{prop} \label{A_n_assoc_coord_uni_centre_prop}
If $K$ is a finite set then
\begin{eqnarray*}
 &\sum_{k \in J}h_{k}(a_{k}) + \sum_i H(a_i, b_i) \in \ker \ud & \\
 &\iff& \\
  & \exists A \in \mathrm{Cent}(D), \quad a_k = A \forall \; k \in J, \quad \card K \cdot A + \sum [a_i, b_i] = 0.&
  \end{eqnarray*}

  If $K$ is infinite then 
  \begin{eqnarray*}
 &\sum_{k \in J}h_{k}(a_{k}) + \sum_i H(a_i, b_i) \in \ker \ud & \\
 &\iff& \\
  & a_k = 0 \; \forall k \in J, \quad  \sum [a_i, b_i] = 0.&
  \end{eqnarray*}

\end{prop}
\begin{proof} Let $x = E_{1n} c \in V^{\sigma}_\gamma$, $\gamma = \epsilon_1 - \epsilon_n \in R_1$ and $X \in \uider(V)_0$. We have to determine when the image of $X$ under $\ud$ annihilates all such $x$, since $\ker \ud = Z(\uTKK(V)). $ There is no restriction in assuming that $\sigma = +,$ since the calculations are ``symmetric'' in $\sigma.$ Recall that by
Lemma~\ref{uider_to_Liemult_lem}

\begin{eqnarray*}
[h_{\alpha_k}(a_{k}) ,E_{1n} c] &=& \sum_{k \in J}h_{\alpha_k}(a_{k}).e_\gamma c  \\ &=& E_{1n} \left( \left (\sum_{k \in J}a_k  \right) c + c a_n \right)
 \end{eqnarray*}
where the coefficient of $E_{1n}$ equals
$$C_{n}= \left( \sum_{k \in J}a_k  \right) c + c a_n$$ 
The action of $H(a, b)$ is less complicated: 
$$[H(a, b), x]  =E_{1n} ([a, b]c).$$
Since $\ud(X)$ must annihilate all $E_{1,n}c$, regardless of the choice of $c,$ this leads to the following system of linear equations: For all $n \in J$ and $c \in D$
\begin{eqnarray*}
0 &=& \sum_{k \in J}a_k  c + c a_n + \sum[a_i, b_i]c . 
\end{eqnarray*}
First setting $c = 1$, yields that in particular $\sum_{j \in J} a_k + \sum[a_, b_i] = - a_n$ for all $n \in J.$ Thus there is $A \in D$ such that $A = a_n$ for all $n \in J.$  At this point we can conclude that $A = 0$ for $\card J = \infty.$ The linear systems can be re-written as
\begin{eqnarray*}
0 &=& \sum_{k \in J}A  c + c A + \sum[a_i, b_i]c . 
\end{eqnarray*}
 Then $R_A = L_{-\sum_{k \in J}A  c - \sum[a_i, b_i]},$ and since a right multiplication $L_B$ agrees with a left multiplication  $R_A$ if  and only if $A = B \in \mathrm{Cent}(A),$ this yields $A \in \mathrm{Cent}(A).$ Thus
 \begin{eqnarray*}
0 &=& \sum_{k \in J}A  c + c A + \sum[a_i, b_i]c  \\
& \iff & \\
0 &=& \sum_{k \in K}A   + \sum[a_i, b_i].  
\end{eqnarray*}
Therefore if, $\card K < \infty$, 
$\sum[a_i, b_i]  = - \card K  A$  and $A$ is arbitrary. If $\card K$ is infinite, then  $(\card K - 1)$ is infinite. But then it follows that $A = 0,$ since otherwise the sums are not well-defined and thus $a_{k'} = 0$ for all $k' \in K'$ and $ 0  = \sum[a_i, b_i]$ is necessary for every element in the kernel of $\ud.$\\ 
It remains to be checked that those conditions are sufficient. The only operation applied to the linear system which was not a priori an equivalence transformation was setting $c$ equal to $1$. Going back we see that, if  $a_n$ had been central, this would have yielded an equivalent system. However, those elements are indeed central (even zero if $\card K = \infty$), thus the condition is also sufficient. 
 \end{proof} 

We will now look at the connection between the Lie algebras $\uTKK(V)$, $\mathrm{TKK}(V)$ and the matrix Lie algebra $\mathfrak{sl}_{K}(D).$ 
\begin{defi} Let $K$ be a set. The special linear Lie algebra  $\mathfrak{sl}_K(D)$ with coefficients in $D$ of size $\card K$ is the subalgebra of  $\mathrm{gl}_K(D)$ which is generated by
$$\{ E_{ij}a : i \neq j \in K, a \in D\}.$$
It is well-known (see for example \cite{KasLo} for the case $\card K < \infty$), that $L = \mathfrak{sl}_K(D)$ is $\dot A_K$ graded with root spaces
\begin{eqnarray*}
L_\alpha &=& E_{ij}D, \alpha = \epsilon_i - \epsilon_j \\
L_0 &=&  \left \{\sum E_{ii} a_i : \sum a_i \in [D, D] \right \}. 
\end{eqnarray*} 
\label{special_lin_alg_ass_Coord_defi}
\end{defi}

\begin{lem} There are central extensions $\hat f : \uTKK(V) \rightarrow \mathfrak{sl}_K(D)$ and $\hat g: \mathfrak{sl}_K(D) \rightarrow \TKK(V)$ such that $ f \circ d = \hat \ud$ (see Definition~\ref{universal_TKK_ce_def}). 
\end{lem}
\begin{proof} Pick an element $1 \in K.$
 Set $L^+ = \sum_{j \in J}E_{1j}D$, $L^{-} =  \sum_{j \in J} E_{j1}D$ and $L^0 = \sum_{i \neq 1, j \neq 1}E_{ij}D.$ Then $L^0 = [L^+, L^-],$ since $E_{ij}D = [E_{i1}, E_{1j}D]$ for any $i,j \in J.$
We can define a quadratic operator on $(L^+, L^-)$ by the same formulas as in Definition~\ref{rect_defi} and then $\ad[x, y]|_{L^+ \oplus L^{-}} = \delta(x, y)$ for $(x,y) \in (L^+, L^-).$ Therefore, $\mathfrak{sl}_K(D)$ is Jordan $3$-graded and the associated Jordan pair is isomorphic to $V = \mathbf{Mat}(\{1\}, J).$ The remaining part of the statement follows from Proposition~\ref{uTKK_universal_prop}. 

\end{proof}


\begin{cor} Consider the diagram in the category of central extensions
$$\uTKK(V) \stackrel{\hat f}{\rightarrow} \mathfrak{sl}_K(D)  \stackrel{\hat g}{\rightarrow} \mathrm{TKK}(V) \rightarrow 0 .$$ 
For $X = \sum_{j \in J}h_j(A) + \sum_{i = 1}^n H(a_i, b_i) \in \uider(V) $ the following holds:
\begin{itemize}
\item[\rm(i)]
Let $K$ be a set of finite cardinality: 
\begin{enumerate}
 \item $X \in \ker(\hat f) \iff ( A = 0,  \sum[a_i, b_i] = 0)$  and $\ker \hat f \cong \mathrm{HC}_1(D)$.
 \item The centre of $\mathfrak{sl}_K(D)$ is $\{ A \cdot \mathrm{id}_{\card K}: A \in \mathrm{Cent}(D), \card K A  \in [D,D]\}.$ 
\end{enumerate}
\item[\rm(ii)]
\begin{enumerate}
\item 
If $\card K \geq \infty$ then 
$$X \in \ker(\hat f) \iff (A= 0, X = \sum H(a_i, b_i), \sum[a_i, b_i] = 0).$$ 
\item
The map $\hat g : \mathfrak{sl}_K(D) \rightarrow \mathrm{TKK}(V)$ is an isomorphism and $\ker (\hat f) = \mathrm{HC}_1(D)$.
\end{enumerate}

\end{itemize}
\end{cor}
\begin{proof}
 The image of $h_j(A)$ under $\hat f$ in $\mathfrak{sl}_K(D)$ is $E_{1, 1}A - E_{j,j}A$ and the image of $H(a, b)$  under $\hat f$ is $E_{1,1}[a,b].$ 
 Since $\uTKK(V) \stackrel{\hat f}{\rightarrow} \mathfrak{sl}_K(D)  \stackrel{\hat g}{\rightarrow} \mathrm{TKK}(V) \rightarrow 0$ is a  diagram in the category of central extensions and $\hat g \circ \hat f = \hat \ud,$  the centre of $\mathfrak{sl}_k(D)$ has the following descriptions
 $$Z(\mathfrak{sl}_K(D))  = \ker \hat g = f(\mathrm{HC}_1(V)). $$
 
 We have already obtained a description of $\mathrm{HC}_1(V)$ in Proposition~\ref{A_n_assoc_coord_uni_centre_prop}. 
 First assume that $\card K $ is finite. Then by Proposition~\ref{A_n_assoc_coord_uni_centre_prop}, an element $X$ in the centre of $\uTKK(V)$ is of the form $X = \sum h_j(A) + \sum H(a_i, b_i), A \in \mathrm{Cent}(D), \card K \cdot A  + \sum[a_i,  b_i] = 0$. The image of $X$ in $\mathfrak{sl}_K(D)$ is 
 $$ -\sum_{j \in J}E_{jj}A + E_{11}(\card J \cdot A + \sum[a_i, b_i]). $$
 Note that since $\card K  A = (\card J + 1)A = 0$ we have $\card J A = -A$. 
 By assumption, $ (\card J + 1) \cdot A + \sum[a_i, b_i] = \card K \cdot A + \sum[a_i, b_i] = 0$ and thus the $(1,1)$ entry of the matrix $\hat f(X)$ is equal to $-A.$ Therefore, an element in the centre of $\mathfrak{sl}_K(A)$ is of the form
  $$- A \cdot \mathrm{I}_{\card K}, \quad  A \in \mathrm{Cent} (A), \card K \cdot A \in[D,D] .$$
  where $\mathrm{I}_{\card K}$ is the identity matrix of size $\card K.$
  It follows that $f(X) = 0$ if and only if $A = 0$ and $X \in Z(\uTKK(V))$, i.e., $X = \sum H(a_i, b_i)$ with $\sum[a_i, b_i] = 0.$ \\
If $\card K  = \infty,$ then the elements in the centre of $\uTKK(V)$ are of the form 
$X  = \sum H(a_i, b_i), \sum[a_i, b_i] = 0.$ In this case $f(X) = E_{ii}(\sum [a_i, b_i]) = 0 ,$ hence $\mathfrak{sl}_k(D)$ is centreless. It follows that $\hat g:  \mathfrak{sl}_K(D) \rightarrow \mathrm{TKK}(V)$ is an isomorphism. \\
Combining  Proposition~\ref{uider_alt_decomp_prop}, (i) and (ii) of this corollary, we see that the kernel of $\hat f : \uTKK(V) \rightarrow \mathfrak{sl}_K(D)$ equals $\mathrm{HC}_1(D).$ 
\end{proof}

\begin{cor} \begin{itemize}
 \item[\rm (i)] If $\card K \geq 5,$ then $\hat f: \uTKK(V) \rightarrow \mathfrak{sl}_K(D)$ is the universal central extension and $\ker \hat f = \mathrm{HC}_1(D).$
\item[\rm (ii)]  If $\card K = 4$ and $v: \mathfrak{uce}(\mathfrak{sl}_4(D)) \rightarrow \mathfrak{sl}_4(D)$ is  a universal central extension, then 
$\ker v \cong \mathrm{HC}_1(D) \oplus (D_2)^6$ 
\end{itemize}
\end{cor}

\begin{proof}
If $\card K \geq 5,$ then $\dot A_{K}$ does not admit any degenerate sums and thus $\hat f: \uTKK(V) \rightarrow \mathfrak{sl}_K(D)$ is universal. By the previous proposition, $\ker \hat f = \mathrm{HC}_1(D).$\\
If $\card K = 4,$ then $\ker v = \sum_{\gamma \in \mathbb{DS}} (\ker v)_{\gamma} + (\ker v)_0.$ By the universal property of $\uTKK(V),$ $(\ker v)_0 \cong \mathrm{HC}_1(D)$ and by Theorem~\ref{deg_sum_theo_A_3}, $(\ker v)_\gamma \cong D_2$ for $\gamma \in \mathbb{DS}.$
\end{proof}
\begin{rem}
Part(i) was shown in \cite{KasLo} for the case where $\card K < \infty.$ 
Part (ii) is also in \cite{GS05} for $k$ a field.  
\end{rem}
\label{associative_coordinatization_sec}

\section{$C_n$- Coordinatization}
 \label{Section C_n_coordinatization}
 \subsection{Jordan algebra homology and Jordan pair homology}
 \label{Jordan_hom_Subsection}
 Let $V = (V^+, V^-)$ be a Jordan pair.  Recall that in Definition~\ref{uider_JkP_Def}  the Lie algebra $V  \diamond V = \uider(V)$ was defined as the quotient of $ V^+ \otimes  V^- $ modulo the submodule $I(P)$ generated by 
 \begin{eqnarray}
&x \otimes \{yxy\} - \{xyx\} \otimes y,& \label{HCPA11}\\ 
&\{xyu\} \otimes w - u \otimes \{yxw\} - x\otimes \{wuy\} + \{uwx\} \otimes y,& \label{HCPA12}  
\end{eqnarray}
where $x, u \in  V^+$ and $y, w \in V^-$ 
and the Lie bracket was given by $[x \diamond y, u \diamond v] = \delta(x,y)u \diamond  v + u \diamond \delta(x,y)v. $
We also defined in the same section 
$$ \ud_{JP} : \uider(V) \rightarrow \instr(V), \;x \diamond y \rightarrow \delta(x, y). $$

In the following we will study the central extensions of the usual models of $C_n$-graded Lie algebras when these are defined over a ring $k$. In order to guarantee that the arising Lie algebras are perfect and thus admit a universal central covering, we require that $$1/2 \in k.$$
  Let $J$ be a unital quadratic Jordan algebra with quadratic map $U$. Then the $k$-module $ \mathbf J = (J, J)$ is a Jordan pair with quadratic map defined by $Q_x = U_x$.   We will refer to $\mathbf J$  as the Jordan pair of $J.$ Note that this defines a functor from the category of Jordan algebras into the category of Jordan pairs. However, the functor ``forgets'' the unit of $J$  and going back we can only recover an isotope of $J.$ From now on \begin{center}
  $\mathbf J $ is the Jordan pair associated to a unital Jordan algebra $J.$
  \end{center}
  Since $1/2 \in k$, the Jordan product is completely determined by the linearization of the quadratic operator: By \cite[p.147] {TasteOfJA}
\begin{equation}\{abc\} = 2(a(bc) + c(ba) - b(ca))  \label{circ_prod_Defi}\end{equation} and specializing $b= 1$, this gives $\{a1c\} = 2ac.$ \\
  By Corollary~\ref{JP_uider_cor}, the module $\mathfrak{uider}(\mathbf J) = \mathbf J^+ \otimes \mathbf J^{-}/I(\mathbf J) $ is always  a Lie algebra  and $\ud: \mathfrak{uider}(\mathbf J) \rightarrow \mathfrak{ider}(\mathbf J)$ becomes a central extension by linearly extending $\ud: a \diamond b \mapsto D(a, b)$, as we also saw in  Lemma~\ref{instr_aleg_is Lie_lage_lem}. As noted in Remark~\ref{one_half_remark}, $I(\mathbf J)$ is already generated by 
  \begin{eqnarray}x \otimes \{yxy\} - \{xyx\} \otimes y \label{HCPA11onehalf} 
\end{eqnarray}
since $1/2 \in k.$
  
  \begin{defi}  \label{Jordan_homology_defi_one_half} Let $\mathbf C$ be a category of algebras  with homomorphisms as morphisms and let $D$ be an object of $\mathbf C.$
  Define  $D * D$ to be the quotient of $D \otimes D$ modulo the submodule $M$ generated by 
  \begin{eqnarray}
  &a \otimes (b  c) + b \otimes (c  a) + c \otimes (a  b),& \label{HCJ1}\\
  &a \otimes b + b \otimes a.& \label{HCJ2}
  \end{eqnarray}
  Denote $a * b = a\otimes b + M$, $a, b \in D.$ This obviously defines a functor from $C$ to the category of $k$-modules.  
   The following observation is very easy to make, but will be quite important:
   Linearization yields that (\ref{HCJ2}) is equivalent to $a * a$ in $D*D.$
   Let $c = b = 1$ in (\ref{HCJ1}), then $a * 1 + 1 * a + a * 1 = 0$ and (\ref{HCJ2}) permits to cancel $1 * a + a * 1 = 0,$ thus
   \begin{equation} a* 1 = 1 * a = 0, \quad\forall a \in D. \label{one_kills_allunider} \end{equation}
   Similarly, if $D$ is commutative,  setting $a = c$ in (\ref{HCJ1}) gives:
    \begin{equation} a^2 * b - 2a * (ba) = 0, \quad\forall a,b \in D. \label{unider_square_shift} \end{equation}
   
   \end{defi}
   
    \begin{rem} We obviously have that $D*D = \langle D, D \rangle$ as defined in Definition~\ref{cyc_hom_sans_one_half_defi}. But we would like to have a different notation for Jordan algebras at hand. 
    \end{rem}
    \begin{lem}  The linear map
 $ \ud_{JA} : (J * J) \rightarrow \mathrm{IDer}_k(J)$ given by 
 $\ud_{JA}: a * b \mapsto 2[L_a, L_b]$ or $(a * b).c  = 2(a(bc) - b(ac))$ is well-defined. 
 Moreover, the derivation algebra of $J$ acts on $J*J$ by 
 $\Delta.(a* b) = \Delta(a)  * b + a * \Delta(b).$ \label{well_Def_JA_der_act}
    \end{lem}
    \begin{proof} 
    The map $\ud : J \otimes J \rightarrow \mathrm{IDer}(J),$ sending $a \otimes b$ to $2[L_a, L_b]$ is well-defined. Since $\ud(a \otimes a) = 0,$ well-definedness of $\mathfrak{ud}_{JA}$ will follow from $\ud(a \otimes bc+ b \otimes ca + c \otimes ab) = 0,$ i.e, $[L_a, L_{bc} ] + [L_b, L_{ca}] + [L_c, L_{ab}] =0,$ but this is a well-known identity for Jordan algebras (e.g. view \cite[ (JAX2''), p.148]{TasteOfJA} as linear transformation applied to the middle variable).\\
    Next we show that the derivation algebra acts on $J*J.$ Let $\Delta \in \mathrm{Der}(J)$ and consider the canonical action of $J \otimes J,$ then 
        $$\Delta.(a \otimes (b  c)) = \Delta(a) \otimes bc + a \otimes \Delta(bc) = \Delta(a) \otimes bc + a \otimes \Delta(b)c + a \otimes b\Delta(a). $$
    Cyclically permuting $a, b , c$ and summing up gives
   \begin{eqnarray*}
   && \Delta(a) \otimes bc + a \otimes \Delta(b)c + a \otimes b\Delta(c) \\
   &&+ \Delta(b) \otimes ca + b \otimes \Delta(c)a + b \otimes c\Delta(a)\\
   &&+ \Delta(c) \otimes ab + c \otimes \Delta(a)b + c \otimes a\Delta(b) \\
   &=& \sum_{cyc.}\left(\Delta(a) \otimes bc + b \otimes c\Delta(a) + c \otimes \Delta(a)b \right)  \in M.\\ 
   \end{eqnarray*}
   This proves that the action factors through the relation (\ref{HCJ1}). Lastly, observe that
   $\Delta(a \otimes a) = \Delta (a) \otimes a + a \otimes \Delta(a).$
    and this is equal to $(\Delta(a) + a) \otimes (\Delta(a) + a) - (\Delta(a) \otimes \Delta(a)) - a \otimes a$
    and therefore an element of $M.$
    It follows that $\mathrm{Der}(J)$ acts on $J*J. $ 
   \end{proof}
   \begin{rem} The last calculation in the proof works also for $J$ replaced by an arbitrary (linear) algebra.  
   \end{rem}
   \begin{defi}
   We define
 $$\mathrm{HC}(J) = \ker( \ud_{JA}) = \{\sum a_k * b_k |  \sum [L_{a_k}, L_{b_k} ] = 0\}.$$ 
  We will also refer to $J*J$ as $\mathfrak{uider}_{JA}(J).$ The justification for this will become apparent later on. 
   \end{defi}
   
    \begin{lem} Let $\mathbf J$ be the Jordan pair of a unital Jordan algebra $J.$ The map 
    \begin{eqnarray*}
    \alpha : \mathbf J \diamond  \mathbf J & \rightarrow & J * J\\
    a \diamond b &\mapsto& a * b
    \end{eqnarray*}
    is an epimorphism. A $k$-linear section of $\alpha$ is given by $\beta : a* b \rightarrow a \diamond b - 1 \diamond ab,$ so that we obtain a direct sum decomposition $\mathbf J \diamond \mathbf J = \beta (J*J) \oplus 1 \diamond \mathbf J.$ 
    \label{jordan_der_splitting_off_lemma}
    \end{lem}
\begin{proof}
By the universal property of the tensor product, we may consider the maps 
\begin{eqnarray*}
\overline{\alpha} : J \otimes J \rightarrow J \otimes J\\
a \otimes b \mapsto a \otimes b
\end{eqnarray*}
and
\begin{eqnarray*}
\overline{\beta} : J \otimes J \rightarrow J \otimes J\\
a \otimes b \mapsto a \otimes b - 1 \otimes ab 
\end{eqnarray*}
We show that $\bar {\alpha}(N)\subset M.$ Of course, $\overline{\alpha}$ is just the identity on the tensor product. First use identity  (\ref{HCJ2}), then identity (\ref{unider_square_shift}) on both terms,  then again (\ref{HCJ2}) (this time adding a term $xy \otimes xy$),(\ref{HCJ2}), and identity (\ref{unider_square_shift}):
\begin{eqnarray*}
\overline{\alpha}(x\otimes \{yxy\} - \{xyx\}\otimes y) & \equiv & x \otimes \{yxy\} + y\otimes \{xyx\} \mod  M\\
& \equiv & 2 x\otimes (2 y(xy) - y^2x) + 2 y\otimes (2 x(yx) - x^2y)  \mod  M\\
&\equiv &- 2(x\otimes y^2x + y\otimes x^2y ) \\ &&+ 4(x \otimes y(xy) +  y \otimes x(yx))  \mod M\\
& \equiv &- 2(x\otimes y^2x + y\otimes x^2y )  + 4(x \otimes y(xy))\\ &&+  4(y \otimes xyx + xy \otimes xy) - 4(xy \otimes xy)  \mod  M	\\
  & \equiv  & -x^2 \otimes y^2 - y^2 \otimes x^2   \mod  M  \equiv 0 \mod M.
  \end{eqnarray*}
Because  of (\ref{HCPA11onehalf}), $\alpha$ is a well-defined homomorphism and onto.\\
We next show that $\beta$ is a well-defined homomorphism of $k$-modules, i.e., $\overline{\beta}$ factors through the submodule generated by $a \otimes a$  and (\ref{HCJ1}): 
\begin{eqnarray*}
\overline{\beta}(a \otimes a)
									&= & (1/2)(\{1a1\} \otimes a + a \otimes\{a11\})  \\
									& \equiv& a \otimes a + a \otimes a \mod N\\
									& \equiv & 0 \mod N,
\end{eqnarray*}
and
\begin{eqnarray*}
\overline{\beta}(ab \otimes c + bc \otimes a - b\otimes ac) & = & ab \otimes c + bc \otimes a -  b \otimes ac \\ &&- 1\otimes (ab)c -1\otimes (bc)a + 1 \otimes b(ac) \\
									& =  & (1/2)(\{ab1\} \otimes c + \{bc1\} \otimes a \\ &&- b \otimes \{ac1\} - 1 \otimes \{abc\}) \\
									& \equiv & 0   \mod N ,
\end{eqnarray*}
which follows from (\ref{HCPA12}) with $x= b$, $y = a$, $u = 1$ and $w= c.$
Hence $\beta$ is well-defined. 
Composing $\beta$ with $\alpha$ yields
$$\alpha(\beta(a * b)) = \alpha(a \diamond b - 1 \diamond ab) = a* b - 1 * ab = a*b.$$
Thus the map $\alpha \circ \beta$ is the identity on $J*J$ or, equivalently, $\beta$ is a section of the epimorphism $\alpha$ and
$$\mathbf J \diamond \mathbf J = \ker \alpha \oplus \beta (J *J). $$ 

It remains to show that $\ker \alpha = 1 \diamond \mathbf J.$ The inclusion $\supset$ follows from (\ref{one_kills_allunider}). For the other inclusion, suppose $x = \sum a_i \diamond b_i \in \ker \alpha.$ Then $0 = \alpha(x) = \sum_i a_i \star b_i,$ whence also $ 0 = \beta(\alpha(x)) = (\sum_i a_i \diamond b_i) - 1 \diamond (\sum_i a_ib_i)$, which proves $x \in 1 \diamond J.$
\end{proof}
\begin{cor} The module $J*J$ is a Lie algebra with bracket defined by 
$$[a* b, x*y] = D(a, b)x*y + x*D(a, b)y, $$ where $D(a, b) = 2[L_a, L_b],$
and the map 
\begin{eqnarray*}
\ud_{JA} : \uider_{JA}(J) & \rightarrow & \mathrm{IDer}(J)\\
a* b & \mapsto & D(a, b)
\end{eqnarray*} \label{ud_JA_action_ce_cor}
is a  central extension of Lie algebras. 
\end{cor}
\begin{proof}
By Lemma~\ref{well_Def_JA_der_act}, $\mathrm{IDer}(J)$ acts on $J*J$ by $D(a, b)(x*y) = D(a, b)x * y + x * D(a, b)y$ and $\ud_{JA} : J*J \rightarrow \mathrm{IDer}(J)$ is a well-defined $k$-module epimorphism.  
The product defined in the statement evidently fulfills that $\ud_{JA}(\ud_{JA}(a*b).(x*y)) = [D(a* b), D(x*y)].$ We calculate $\ud(a*b)(a*b) = \ud(a*b)a*b + a*\ud(a*b)b $ using that $\{aba\} = D(a,b)a  + 2 (aa)b.$ Thus
\begin{eqnarray*}
 \ud(a*b)(a*b) &=& \{aba\} * b - a * \{bab\}  - 2\left((a^2)b *b  - a*  (b^2)a \right).\\
 \end{eqnarray*}
 By (\ref{HCJ1}), 
 $(a^2)b *b  + b^2 * a^2 + (a^2)b *b  = 0$ and  $a*  (b^2)a + b^2 * a^2 + a * b^2a  = 0 . $
 Thus $-a^2 b * b = 1/2(a^2 * b^2) $ and  $b^2 a* a = 1/2(b^2 * a^2)$ and  $(a^2)b *b  - a*  (b^2)a = 0$ by $(\ref{HCJ2}).$ It follows that 
 $\ud(a*b)(a*b) = \{aba\} * b - a * \{bab\}.$
 Applying $\beta$ to this, which is an injective map yields
\begin{eqnarray*}
&&(\{aba\} \diamond b - a \diamond \{bab\} ) - 1 \diamond(\{aba\}b - a\{bab\} ).
\end{eqnarray*}
The first summand is $0$ by (\ref{HCPA11onehalf}) and the second one is $0$ by applying a well-known identity in Jordan algebras, see for example \cite[p.202, FFIIe]{TasteOfJA} with $z = 1.$ It follows that  $ \ud(a*b)(a*b) = 0$ and, linearizing this identity first in $a$ and then in $b$, gives us $\ud(a*b)(b*c) + \ud(b*c)(a*b) = 0$. Therefore $\ud(X).X= 0$ for all $X \in J*J$. Since $\ud_{JA}$ is surjective, it follows from Lemma~\ref{Lie_alg_from_mod} that $J*J$ is a  Lie algebra and $\ud_{JA}$ is a central extension. 
\end{proof}
\begin{prop}\label{one_half_jordan_hom_H_2_prop}If $L = \mathrm{TKK}(\mathbf J)$ then
$$\ker(\ud_{JP}) = \mathrm{HC}(J) = \ker{\ud_{JA}} = \ker (u ),$$
where $u : \mathfrak{uce}(L) \rightarrow L$ is the universal central extension. 
\end{prop}
\begin{proof}
First note that in the presence of $1/2$ the universal central extension of the $A_1$-graded algebra $L$ is also $A_1$-graded and thus by Corollary \ref{JKP_central_ext_kernel_cor}, $\ker(\ud_{JP}) = \ker (u)$.\\
By Lemma~\ref{jordan_der_splitting_off_lemma}, $\ud_{JP}(\mathbf J \diamond \mathbf J)  = \ud_{JA}(J * J) + \ud_{JP}(1 \diamond \mathbf J)$.  
Under $\ud_{JP},$ the image of $1 \diamond a \in 1 \diamond \mathbf J$ is $2L_a$
 and the image of $ a * b $ is $[L_a, L_b].$ For any  linear Jordan algebra, the sum $L_J \oplus [L_J,  L_J] \subset \mathrm{End}_k(J) $ is direct, thus $\ud_{JP}(\mathbf J \diamond \mathbf J)  = \ud_{JA}(J * J) \oplus \ud_{JP}(1 \diamond \mathbf J)$ . Therefore, $\sum a_i * b_i + 1 \diamond a \in \ker \ud_{JP}$, if and only if $ \sum [L_{a_i}, L_{b_i}] + L_a = 0$ if and only if $ \sum [L_{a_i}, L_{b_i}]  = 0$ and $L_a = 0$, if and only if $a = 0$ and $\sum [L_{a_i}, L_{b_i}]  = 0.$ So $\ker \ud_{JP} \subset J*J$ and
 the calculation above then shows that
$\sum a_i * b_i \in \ker \ud_{JP} $ if and only if  $\sum [L_{a_i}, L_{b_i}] = 0$ if and only if $\sum a_i * b_i \in \mathrm{HC}(J) = \ker(\ud_{JA})$ (per Definition~\ref{Jordan_homology_defi_one_half}).
\end{proof}
Before we continue, the result above needs some interpretation: If $J$ is a linear Jordan algebra then the module $J*J$ is the analogue of $\uider(V)$ where $V$ is a pair. This interpretation is supported by the fact that for $\mathbf J$ the pair of the Jordan algebra $J$, we have an isomorphism of kernels
$$\ker (\ud_{JP} ) =  \ker( \ud_{JA} ).$$
This justifies in particular why we called the map on the left hand side $\ud$ and used the name $\uider_{JA}(J)$ for $J*J.$

\begin{prop} Let $J$ be a unital Jordan algebra over $k$ and  let $k \rightarrow K$ be a flat base change. Then 
$$\mathrm{IDer}_k(J) \otimes_k K = \mathrm{IDer}(J \otimes_k K). $$ \label{herm_jord_der_flat_base_change_cor_rev}
\end{prop}
\begin{proof}
Define $J_K := J_k \otimes K.$ Then by Lemma~\ref{JA_cycle_base_change_lem}, $ J_K *  J_K  = ( J, J ) \otimes_k K.$ By Lemma~\ref{well_Def_JA_der_act} there are well-defined epimorphisms $ \ud_K :  J_K *  J_K  \rightarrow \mathrm{IDer}(J_K)$ and $\ud \otimes K : ( J * J ) \otimes_k K \rightarrow \mathrm{IDer}(J) \otimes_k K.$ By flatness of the base change $\mathrm{Im}(\ud_k) = \mathrm{Im}(\ud) \otimes_k K.$ This proves the Lemma.  
\end{proof}

\begin{prop}[Need reference here]
Let $J$ be a Jordan algebra which is finitely presented as a $k$-module. If $k \rightarrow K$ is a flat base change, then 
$$\mathrm{Der}_k(J) \otimes K = \mathrm{Der}_K(J \otimes_k K) .$$
In particular, if $J$ is finitely generated projective and  if $k \rightarrow K$ is a flat basechange, then 
$$\mathrm{Der}_k(J) \otimes K = \mathrm{Der}_K(J \otimes_k K). $$
\label{flat_basechange_Jordan_Der}
\end{prop}

\subsubsection{Idempotents and Jordan homology}
\label{idempotent_homology}
\begin{defi} Let $e \in J$ be an idempotent. Hence, putting
$$ J_2 =  Q_e(J), \quad J_1 = D_{e, e-1} (J), \quad  J_0 = Q_{1-e}(J)  $$
we have a direct sum decomposition
$$J = J_0 \oplus J_1 \oplus J_2.$$
\end{defi}

\begin{defi}For $\mathcal E = \{e_i : i \in I\}$ a set of idempotents, we call $\mathcal E$ \emph{orthogonal}, if $e_ie_j = 0$ for all $i \neq j.$ 
If $\{e_i\}_{i \in I}$ is a set of orthogonal  idempotents, we put
$$J_{ij} :=  D_{e_i, e_j}(J), J_{ii} := Q_{e_i}(J),\quad (i\neq j). $$
Note that always $J_{ij} = J_{ji}.$
\end{defi}

Let $ \{e_i\}_{i \in I}$ be a set of pairwise orthogonal idempotents and let $i, j , k, l$ different elements of $I$. We defined $a \circ b = \{a1b\} = 2ab$ for $a, b \in J$ (see (\ref{circ_prod_Defi})). Then the following multiplication rules hold:   
\begin{eqnarray}
J_{ii } \circ J_{ij} & \subset & J_{ij}, i \neq j \label{Peirce_rule_1},\\
J_{ii }  \circ J_{ii} & \subset & J_{ii} \label{Peirce_rule_2},\\
J_{ij} \circ J_{jk}  & \subset & J_{ik}, i,j,k \neq \label{Peirce_rule_3}, \\
J_{ij } \circ  J_{kl} &=& \{0\}, i, j, k, l \neq  \label{Peirce_rule_4},\\
J_{ij} \circ J_{ij} &\subset & J_{ii} + J_{jj}, i, j \neq . \label{Peirce_rule_5}
\end{eqnarray}
From now on, $\mathcal E = \{e_i: i \in I\}$ is an orthogonal family in a Jordan algebra $J$. We choose a total ordering $< $ on $I$ and in addition we assume
\begin{equation} J = \bigoplus_{i \leq j;i,j \in I} J_{ij} .  \label{supplementary_cond} \end{equation}
\begin{rem} If $I$ is finite and complete, i.e., $\sum_{i \in I}e_i = 1_J,$ then (\ref{supplementary_cond}) always holds and in this case the Jordan algebra is automatically unital with unit $1_J = \sum_{i \in I}e_i,$ but if $I$ is infinite, we cannot conclude that $J$ is unital, as the following example shows.
\end{rem}

\begin{expl}
Let $J = \mathcal H_I(k)$ be the Jordan algebra of symmetric matrices over $k$ of size $\card I$ (where the Jordan product is $1/2(AB + BA)$). If $I$ is infinite, we assume that the elements of $\mathcal H_I(k)$ have only finitely many non-zero entries. Then the set $\{E_{ii} : i \in I\}$ is a set of orthogonal idempotents and $J = \bigoplus_{i \leq j}J_{ij}$ with $J_{ij} = k(E_{ij} + E_{ji}).$ This Jordan algebra is unital, if and only if $\card I < \infty.$
\end{expl}

\begin{lem} The Peirce multiplication rules imply: \label{Peirce_mult_rules_hom}
\begin{equation} \label{Peirce_mult_rules_hom_eq_1}
J_{ik} * J_{kj} =  e_{i} * J_{ik}J_{kj} \subset e_i * J_{ij}, i, j, k \neq  
\end{equation}
\begin{equation} \label{Peirce_mult_rules_hom_eq_2}
J_{jj} *J_{ij} \subset e_i * J_{ij} \supset  J_{ii} * J_{ij}, i\neq j
\end{equation}
If $\{ij\} \cap \{kl\}$ is empty then
\begin{equation} J_{ij} * J_{kl} = 0.  \label{Peirce_mult_rules_hom_eq_3}\end{equation}
If    $J_{ii} \subset \sum_{i \neq k}J_{ik}^2, $ then $ J* J = \sum_{i <j } J_{ij}  * J_{ij} +  \sum_{i < j}e_i * J_{ij}.$ 

 \end{lem}
 \begin{proof}
 We will repeatedly use $J_{ij} = J_{ji}$ for all $i,j \in I$
and also $J_{ij} * J_{kl} = J_{kl} * J_{ij} $ for any two Peirce spaces. \\
Let $x_{ik} \in J_{ik}$ and $x_{kj} \in J_{kj}$  where $i,j,k$ are pairwise distinct then 
$$1/2x_{ik} * x_{kj} = e_i x_{ik} * x_{kj} = - x_{ik}x_{kj} * e_i - x_{kj}e_i * x_{ik} = e_i * x_{ik}x_{kj} \in e_i * J_{ik}J_{kj} \subset e_i * J_{ij},$$
whence   (\ref{Peirce_mult_rules_hom_eq_1}). 
If $x_{ii} \in J_{ii}$ and $x_{ij} \in J_{ij},$ then 
$$  x_{ii} * x_{ij} =  e_i x_{ii} * x_{ij} = - x_{ii}x_{ij} * e_i - x_{ij}e_i * x_{ii} = - x_{ii}x_{ij} * e_i + (1/2)(x_{ii}*x_{ij}) $$ $$ \iff (1/2) x_{ii} * x_{ij} = e_i * x_{ii}x_{ij} \in e_i * J_{ij}. $$
For the inclusion on the left of (\ref{Peirce_mult_rules_hom_eq_2}), 
$$ x_{jj} * e_i x_{ij} = - e_i * x_{ij}x_{jj} \in e_i * J_{ij} $$
and since $e_i J_{ij} = J_{ij}$ this implies (\ref{Peirce_mult_rules_hom_eq_2}).\\ 
Let $x_{ij} \in J_{ij} $ and $x_{kl} \in J_{kl}$ where $\{i,j\} \cap \{k,l\} = \emptyset.$
$$e_i x_{ij} * x_{kl}  = - x_{ij}x_{kl} * e_i - x_{kl}e_i * x_{ij} = 0   + 0  = 0$$
hence $J_{ij} * J_{kl} =\{0\}.$ In particular $J_{ii} * J_{kk} = \{0\}$ if $i \neq k.$
By \eqref{Peirce_rule_3} and \eqref{Peirce_mult_rules_hom_eq_3}, 
$$\sum_{i \neq k} J_{ik}^2 * J_{ii} \subset J_{ii} * J_{ii} $$
and we have equality if $\sum_{k \neq i}J_{ik}^2 \supset  J_{ii}.$
In this case, by \eqref{Peirce_mult_rules_hom_eq_2}, $J_{ii} * J_{ii} \subset \sum_{i \neq k} J_{ik}^2 * J_{ii} \subset \sum_{i < k} e_i * J_{ik} + \sum_{i> k}e_k *J_{ki}.$

Let 
\begin{eqnarray*} J* J &=& \left (\sum_{i, j \neq} J_{ij} * J_{ij}  + (J_{ii} * J_{jj}) 
 + (J_{ii} * J_{ij})\right) \\ &&+  \left ( \sum_{i}(J_{ii} * J_{ii})\right) +  \left( \sum_{i,j,k \neq }(J_{ij} * J_{jk} ) \right)  + \left( \sum_{i,j,k,l \neq}(J_{ij} * J_{kl}) \right) \end{eqnarray*} be the decomposition of $J* J$ induced by the Peirce decomposition of $J.$
In general, we have just proven that in fact,  
$$J* J = \left (\sum_{i< j} J_{ij} * J_{ij} \right) +  \left ( \sum_{i< j }(e_i * J_{ij})\right) + \left ( \sum_{i}(J_{ii} * J_{ii})\right) $$
and  
$$J* J = \left (\sum_{i< j} J_{ij} * J_{ij} \right) +  \left ( \sum_{i< j }(e_i * J_{ij})\right)  $$ 
if $\sum_{k \neq i}J_{ik}^2 \supset J_{ii}.$

 \end{proof}
{ \begin{equation} \label{non_degeneracy_assumption_peirce_decomp}\sum_{i\neq k }J_{ik}^2 = \sum J_{ii}  , \quad \card{I} \geq 3.   \quad  \quad  \quad \end{equation} }

\begin{prop} \label{Uni_der_Decomp_prop_ja} Let $\mathcal D_0 = \{X \in J*J : \ud(X).e_i = 0 \; \forall i \in I\}$ and $\mathcal D = \sum_{i< j}e_i * J_{ij},$
 Then $$J* J = \mathcal D_0 \oplus \mathcal D $$ and 
 $$\mathcal D \cong \bigoplus_{i< j}J_{ij}, \; \mathcal D_0 = \sum_{i< j} J_{ij}* J_{ij}. $$
    Moreover, with $\ud :(J *J) \rightarrow \mathrm{IDer}(J)$ as in Definition~\ref{Jordan_homology_defi_one_half} 
    $$\ker (\ud ) \subset \mathcal D_0 \mbox{ and } \ud(\mathcal D) \cong \bigoplus_{i< j}J_{ij}. $$
 \end{prop}
 \begin{proof}
 \begin{enumerate}
 \item Since we assume \eqref{non_degeneracy_assumption_peirce_decomp}, Lemma~\ref{Peirce_mult_rules_hom} implies that 
 $$J* J =  \sum_{i< j} (e_i * J_{ij}) + \sum_{i< j}(J_{ij} * J_{ij}). $$
 In particular $J * J = \mathcal D + \mathcal D_0' $, where $\mathcal D_0' =  \sum_{i< j}J_{ij} * J_{ij}. $
 \item Let $x_{ij} \in J_{ij},$ then $\frac{1}{2}\ud(e_i * x_{ij}).e_k = [L_{e_i}, L_{x_{ij}}].e_k= e_i(x_{ij}e_k) - x_{ij}(e_i e_k) = (1/2)e_i(\delta_{ik} + \delta_{jk})x_{ij} - \delta_{ik}x_{ij}(e_i) = 1/4(\delta_{ik} + \delta_{jk})x_{ij} - (1/2)\delta_{ik}x_{ij} = (1/4)(\delta_{jk} - \delta_{ik})x_{ij}. $
 
  Now assume that $ \sum_{i < j} e_i * x_{ij}  \in \mathcal D_0.$ Applying this to an arbitrary idempotent yields
  $$2\ud \left (\sum_{i < j} e_i * x_{ij}\right).e_k = \sum_{i < j} (\delta_{jk} - \delta_{ik})x_{ij}= \sum_{i< k}x_{ik} - \sum_{k< j}x_{kj} = 0.$$ 
  Since the decomposition into Peirce spaces is direct, this implies that $x_{ik} = x_{jk} = 0$ for all $i,j, k \neq.$ Varying $k$ over $I$ yields that $x_{ij} = 0$ for all $i < j.$ It also implies that $\mathcal D_0 \cap \mathcal D = \{0\},$ by definition of $\mathcal D_0$ as the pre-image of all those which annihilate all idempotents. \\
  As a special case,  $\ud(e_i * x_{ij}) = 0$ if and only if $x_{ij} = 0.$ This proves that $e_i * J_{ij} \cong \ud(e_i * J_{ij}) \cong J_{ij},$ for all $i < j.$
   Thus, in particular, 
  $ \sum(e_i * x_{ij}) = 0$ if and only if $x_{ij} = 0$ for all $i \neq j,$ and therefore we have a direct sum decomposition $\bigoplus_{i< j} \ud(e_i * J_{ij}).$ 
  So $\ud$ induces an isomorphism
  $$ \mathcal D = \ud \left( \sum_{i< j}e_i* J_{ij} \right) = \bigoplus_{i< j} \ud(e_i * J_{ij}) \cong \bigoplus_{i< j} J_{ij}.$$
   
   \item In 2. we saw that $\ud(X) \in \mathcal \ud(\mathcal D)$ annihilates all idempotents if and only if $X = 0.$ Thus by definition of $\mathcal D_0$ 
   $$\ud(\mathcal D) \cap \mathcal \ud (\mathcal D_0) = \{0\}.$$
   \item Let $X = x_{ij} * y_{ij} \in J_{ij} * J_{ij}$. Clearly, elements of this form span $\mathcal D_0'.$
   Since $J$ is commutative
   $$2 \ud(X).e_k = 4 x_{ij}(y_{ij}e_k) - 4 y_{ij}(x_{ij}e_k) = (\delta_{ij} + \delta_{jk})(x_{ij}y_{ij} - y_{ij}x_{ij}) =0 .$$
   Hence $\mathcal D_0' \subset \mathcal D_0.$
  
   \item From 4. we have $\mathcal D \cap \mathcal D_0 ' \subset \mathcal D \cap \mathcal D_0 = \{0\}$ and therefore 
   $$J* J =  \mathcal D + \mathcal D_0' = \mathcal D \oplus \mathcal D_0'. $$
   Let $X \in \mathcal D_0 \subset J*J.$ Then there are unique elements $X_{\mathcal D} \in \mathcal \mathcal D$ and $X_{\mathcal D_0'} \in \mathcal D_0'$ such that $X = X_{\mathcal D} + X_{\mathcal D_0'}.$ Since $X \in \mathcal D_0$, 
   $ \ud(X).e_k = 0$ for all idempotents $e_k$. Thus $\ud( X_{\mathcal D} + X_{\mathcal D_0'}).e_k = \ud(X_{\mathcal D}).e_k + \ud(X_{\mathcal D_0'}).e_k = \ud(X_{\mathcal D}).e_k  = 0.$  Thus $ X_{\mathcal D} \in \mathcal D_0\cap \mathcal D = \{0\}.$ Hence
   $X = X_{\mathcal D_0'}$ and it follows that $\mathcal D_0 = \mathcal D_0'.$
   \item Repeating the argument in the previous step shows that if $X = X_{\mathcal D} + X_{\mathcal D_0} \in \ker \ud$, then 
   $X_{\mathcal D} \in \mathcal D_0 \cap \mathcal D = \{0\}$ and $X_{\mathcal D_0} \in \ker \ud.$ Thus $\ker \ud \subset \mathcal D_0.$  
 \end{enumerate}
 \end{proof}

By Proposition~\ref{Uni_der_Decomp_prop_ja} the task of computing the kernel of $\uider(V) \rightarrow \instr(V)$ has in the special case under consideration been reduced to study the map
\begin{equation} \label{ud_0_def} \ud_0 : \mathcal D_0 \rightarrow \sum_{i < j}[L_{J_{ij}}, L_{J_{ij}}] .\end{equation}

\begin{lem}
The map $\ud_0 : \mathcal D_0 \rightarrow \sum_{i < j}[L_{J_{ij}}, L_{J_{ij}}]$ given by $a_{ij} * b_{ij} \mapsto  2[L_{a_{ij}} , L_{b_{ij}}]$ is a Lie algebra epimorphism.
\end{lem}
\begin{proof} By definition, $\mathcal D_0$ is  a Lie algebra and $\ud_{JA}$ is a Lie algebra morphism into $\mathrm{IDer}(J) \subset \mathrm{End}(J).$ Thus the restriction of $\ud_{JA}$ to $\mathcal D_0$ is a Lie algebra epimorphism onto the image. It is easy to check that the image is $\sum_{i < j}[L_{J_{ij}}, L_{J_{ij}}].$ 
\end{proof}

\begin{cor} The Lie algebra $\mathrm{IDer}(J)$ is generated by $[L_{e_i}, L_{J_{ij}}]$, $i < j$ and $J*J$ is generated by $e_i * J_{ij}, i \neq j.$ \label{matrix_jord_der_generation_cor} In particular

 $$[e_i * x_{ij}, e_i * y_{ij}] = -1/2(x_{ij} * y_{ij}). $$
\end{cor}
\begin{proof} Since $\ud$ is an epimorphism and $\ud(e_i * J_{ij}) = 2[L_{e_i}, L_{J_{ij}}]$ it suffices according to Lemma~\ref{Peirce_mult_rules_hom} to show that  $J_{ij} * J_{ij} = [e_i * J_{ij}, e_i * J_{ij}]$ for all $i< j.$
Pick $x_{ij}, y_{ij} \in J_{ij}.$
Then $[e_i * x_{ij}, e_i * y_{ij}] = \ud(e_i * x_{ij})e_i * y_{ij} + e_{i} *\ud(e_i * x_{ij})y_{ij}.$
By 2. in the proof of Proposition~\ref{Uni_der_Decomp_prop_ja}, $\ud(e_i * x_{ij}).e_i= -1/2 x_{ij}.$ 
Thus the first summand equals $-(1/2)(x_{ij} * y_{ij}).$ 
The second summand is  equal to $e_i * 2[L_{e_i}, L_{x_{ij}}].y_{ij} = 2e_i* (e_i(x_{ij} y_{ij}) - (1/2)x_{ij}y_{ij }).$

Note that $x_{ij}y_{ij}= 2(e_i + e_j)(x_{ij}y_{ij})$. With this observation, $e_i(x_{ij} y_{ij}) - (1/2)x_{ij}y_{ij } = e_j(x_{ij}y_{ij}) \in J_{jj}$
and it follows that $e_i * 2[L_{e_i}, L_{x_{ij}}].y_{ij} \in e_i * J_{jj} = \{0\}.$
Therefore, 
$ [e_i * x_{ij}, e_i * y_{ij}] = -(1/2)(x_{ij} * y_{ij}), $ 
which proves the claim.  
\end{proof}
\label{general_peirce_result_homology}
\subsection{Hermitian matrix Jordan algebras}
The results of Section~\ref{general_peirce_result_homology} can be applied to the case where $J$ is the Jordan algebra of hermitian matrices. 
\begin{defi} \label{hermitian_matrix_jord_alg}
Fix an index set $I$, $\card I \geq 3$ and
a unital alternative $k$-algebra $D$, which is associative if $\card I \geq 4$ and which has a nuclear involution $\bar{\;}$, that is $d= \bar d$ implies $d \in \mathrm{Nuc}(D)$ (see Definition~\ref{nucleus_defi}). If $\mathbb M_I(D) = \bigoplus_{i, j \in I}DE_{ij}$ is the set of $I \times I$ matrices with finitely many non-zero entries in $D$, then the involution $\bar{\;}$ can be extended  to a self-inverse linear  map on  $\mathbb M_I(D):$
$$\bar{\;} : \mathbb M_I(D) \rightarrow \mathbb M_I(D)$$ given by $\overline{dE_{ij}} = \bar d E_{ji}.$
With these settings, the fixed point set of $\bar{\;} : \mathbb M_I(D) \rightarrow \mathbb M_I(D)$ is spanned by the 
\emph{elementary symmetric matrices}
$$d[ij] = dE_{ij} + \bar dE_{ji}, \;d \in D,  i \neq j \in I \mbox{  and  } d[ii] = d E_{ii}, \, \bar d = d, \; i \in I .$$
Note that with these definitions $\bar d [kl] = d[lk]$ for all $k, l \in I.$\\
We denote the set of fixed points  by $\mathcal H_{I}(D, \bar{\;})$, called the \emph{hermitian matrices} with entries in $D$. With the conditions as above, 
$J = \mathcal H_{I}(D, \bar{\;})$ is a Jordan algebra with circle  product (see (\ref{circ_prod_Defi})): 
\begin{eqnarray}
 a[ij] \circ b[jk] &=& (ab)[ik], i, j, k \neq, \\
 a[ii] \circ b[ik] &=& (ab)[ik], i \neq k, \\
  d[ii] \circ d'[ii] &=& (dd' + d'd)[ii],\\
  a[ij] \circ b[ji] &=& (ab +  \bar b \bar a )[ii] + ( b a+    \bar a \bar b)[jj], i\neq j, \\
  a[ij] \circ b[kl] &=& d[ii] \circ d'[kk] = 0, \{i,j\} \cap \{kl\} = 0.
\end{eqnarray}
Also, the elements $E_{ii}, i \in I$ are a set of orthogonal idempotents and 
$$J_{ij} = (E_{ij}D + E_{ji}D)\cap J, \; J  =  \bigoplus_{i \leq j}J_{ij} $$
for any total ordering on $I.$ 
Also, to avoid occurrences of $1/2$, we will  mostly use the $\circ$ product as the product on $J$. 
\end{defi}
 \begin{rem} It is convenient to use the circle product as product on the Jordan algebra. So for this section, when we write $ab$ we actually mean $a \circ b.$ Also, when we write $D(a, b).c$ this shall be equal to $a \circ ( b \circ c) - b \circ (a \circ c).$  Since $1/2$ is invertible, these substitution are harmless. 
 \end{rem}
 \begin{defi} Let $J = \mathcal H_I(D, \bar{\;})$. 
Define a $\mathrm{Fix}(D)$-valued \emph{trace} on $D$ by 
 \begin{eqnarray*}
 t: D &\rightarrow& \mathrm{Fix}(D)\\
 t(z)1 &=& z + \bar z,
 \end{eqnarray*}
 and a \emph{norm} form by 
 \begin{eqnarray*}
 N: D &\rightarrow & \mathrm{Fix}(D) \\
 N(z)1 &=& z \cdot \bar z.
 \end{eqnarray*}
 By linearization we obtain  a bilinear form $$N(a, b) = a \bar b + b \bar a$$
 which is symmetric
 \begin{equation} 
 N(a, b) = N(b,a).  
 \label{bil_from_trace_is_symm}
 \end{equation} 
 \label{alt_norm_trace_Defi}
   \end{defi} 
 \begin{lem} \label{der_alt_explicit_form_lem}
 Recall (see Corollary~\ref{ud_JA_action_ce_cor}) that we have a central extension $ \ud: J* J \rightarrow \mathrm{IDer}(J)$ which induces an action by derivations $(x * y).z = 4 D(x,y).z$ for all $x,y, z \in J.$ Under this action,  for $a,b,c \in D$,
  \begin{eqnarray}
   (a[ij] * b[ij]).c[ij] &=& (2(a, b,c) + L_{a \bar b - b \bar a}c - R_{\bar a b  -  \bar b a}c)[ij],  \label{der_alt_explicit_form_lem1}\\
  (a[ij] * b[ij]).c[jl]  &=& (\bar a(bc) - \bar b(ac))[jl] \quad i,j, l \neq, \label{der_alt_explicit_form_lem2} \\
   (a[ij] * b[ij]).c[li]  &=& ((cb)\bar  a - (ca)\bar b)[li] \quad i,j, l \neq,  \label{der_alt_explicit_form_lem3}\\
   (a[ij] * b[ij]).c[kl]  &=& 0\quad i,j,k,l \text{ pairwise distinct}.
   \end{eqnarray} 
  If the involution is not only nuclear, but also central, then 
  \begin{equation}
  (a[ij] * b[ij]).c[ij] =2 (N(a, c) b - N(b,c)a )[ij].
  \label{diagonal_action_1}
  \end{equation}
  \end{lem}
  \begin{proof}
 Since the involution is nuclear, an element $x + \bar x$ associates with all other elements $y,z \in D.$ Thus
 $$ 0 = (x + \bar x, y, x) = (x, y, z) + (\bar x, y, z) $$
 or equivalently \begin{equation}(\bar x, y, z) = - (x, y, z). \label{conj_associatior} \end{equation} Combined with the alternative law, this means that replacing an arbitrary  element in an associator by its conjugate changes the sign of the associator. Also, for all choices of $a, b$ and $c$, $\overline{(a, b, c)} = -(a, b, c).$\\
The coefficient of $(a[ij] * b[ij]).c[ij] = (a[ij](b[ij] c[ij] ) - b[ij](a[ij] c[ij] ))$ is (using for example the multiplication rules on p.174 in \cite{TasteOfJA}) 
 \begin{eqnarray*} \lefteqn{
 (c \bar b )a + (b \bar  c)a - (c \bar a)b - (a \bar c) b  + a(\bar b c) + a(\bar c b) - b (\bar a c) - b (\bar c a) } \\
 &=& (b \bar  c)a  - b (\bar c a)  - (a \bar c) b + a(\bar c b) + (c \bar b )a   - (c \bar a)b   + a(\bar b c)  - b (\bar a c)  \\
  &=& (b, \bar c, a) - (a, \bar c, b)  + (c, \bar b ,a) - (c, \bar a, b) - (a, \bar b, c) +(b, \bar a, c)\\
 &&+ c(\bar  b a ) - c(\bar a b) + (a \bar b )c-  (b \bar  a)c \\
 &=& -(a, b, c) - (a, b, c)  + (a,  b ,c) + (a, b, c) + (a,  b, c) +(a, b, c)\\
 &&+ c(\bar  b a ) - c(\bar a b) + (a \bar b )c-  (b \bar  a)c \\
 &=&2(a, b , c) +  L_{a \bar b - b \bar a}c - R_{a \bar b - b \bar a}c .
 \end{eqnarray*}
In the second last step, (\ref{conj_associatior}) was used. \\
The coefficient of $(a[ij] * b[ij])c[jl] = a[ij](b[ij] c[jl]) - b[ij](a[ij] c[jl])$  $= \bar a[ji]  (b[ij] c[jl]) - \bar b[ji]  (a[ij] c[jl])$ is
$\bar a( bc) - \bar b(ac)$ which proves (\ref{der_alt_explicit_form_lem2}), and similarly, the coefficient of $(a[ij] * b[ij])c[li] = a[ij](b[ij] c[li]) - b[ij](a[ij] c[li]) =  (c[li] b[ij] )\bar a[ji] - ( c[li]  a[ij]) \bar b[ji]  $ is
$ (cb)\bar  a - (ca)\bar b$ which proves (\ref{der_alt_explicit_form_lem3}). 
Under the assumption that the involution is central, the commutator has the following property: Let $x, y \in D,$ then
$[x + \bar x , y] = 0$ thus $[x, y] = -[\bar x, y] = [y, \bar  x].$ 
The first observation is $\overline{[x, y]} =-[x, y] $, so all commutators are skew. Next 
$[x, \bar y] = -[x,y] = [\bar x, y] $, which implies that \begin{equation} x \bar y + y \bar x = \bar x y + \bar y x  \label{in_proof_pre-triality_id_2}\end{equation}for all $x, y \in D.$
By definition of the norm, this is also equivalent to $N(x,y) = N(\bar x, \bar y),$ whence also to
$$N(\bar x, y) = N(x, \bar y). \label{invariant_in_one argument_if _central_inv}$$
Since  $a \bar b + b\bar a \in Z(D),$ the left and right multiplication operators of these elements agree:
 $L_{a \bar b + b\bar a} - R_{a \bar b + b \bar a} = 0.$
 Therefore, $L_{a \bar b - b \bar a} - R_{\bar a b  -  \bar b a} = L_{a \bar b - b \bar a} - R_{\bar a b  -  \bar b a} + L_{a \bar b + b\bar a} - R_{a \bar b + b \bar a} = 2 L_{a \bar b } - R_{\bar a b  -  \bar b a + a \bar b + b \bar a}.$ Use (\ref{in_proof_pre-triality_id_2}) to obtain $\bar a b  -  \bar b a + a \bar b + b \bar a  = \bar a b \  - \bar b a + \bar b a + \bar a b = 2(\bar a b)$, so that  
 \begin{equation}L_{a \bar b - b \bar a} - R_{\bar a b  -  \bar b a} =  2(L_{a \bar b } - R_{\bar a b}) \label{in_proof_pre-triality_id_3}. \end{equation}
 This in particular implies that 
 \begin{equation} L_{a \bar b } - R_{\bar a b}  = - ( L_{b \bar a } - R_{\bar b a}) \label{switching_signs_inin_proof_pre-triality} \end{equation} since the left hand side changes the sign when the roles of $a$ and $b$ are interchanged. 
 Now, 
 \begin{eqnarray*} N(b,c)a  - N(a, c) b &=& N(b,c)a - N(c,a) b \\
 &=& N(c,b)a = N(\bar a, \bar c) b\\
 &= & (c \bar b + b \bar c)a - ( \bar a c + \bar c a)b \\
  &= & (c \bar b + b \bar c)a - b( \bar a c + \bar c a) \\
  &= & (c \bar b)a + (b \bar c)a - b(\bar a c) - b(\bar c a) \\
  &=& (c, \bar b, a) + c(\bar ba ) + (b \bar c)a + (b, \bar a, c) - (b \bar a)c - b(\bar c a) \\
  &=& (c, \bar b, a)  + (b \bar c)a + (b, \bar a, c)  - b(\bar c a) \\   
  && + c(\bar ba ) - (b \bar a)c \\
  &= & (a, b, c) + (a, b,c) + (b, \bar c, a) + R_{ b \bar a}c - L_{b \bar a}c \\
  &= & (a, b, c) + (a, b,c) - (a, b, c) + R_{ b \bar a}c - L_{b \bar a}c \\
  &= & (a, b, c) + R_{\bar b  a}c - L_{b \bar a}c \\
 \end{eqnarray*}
 from which it follows by (\ref{in_proof_pre-triality_id_3}) that
 $$2(a, b,c) + L_{a \bar b - b \bar a}c - R_{\bar a b  -  \bar b a}c =  2N(b,c)a  -2 N(a, c) b .$$
 Thus (\ref{diagonal_action_1}).
  \end{proof}
  
 \begin{defi} Define an  operator $\mathbf g(a \otimes b)$ from $D \otimes D$ into $\End{D} \times \End{D} \times \End{D}$ where the components $(g_1(a \otimes b),g_2(a \otimes b), g_3(a \otimes b) )$ are given by
 \begin{eqnarray*} g_1 (a\otimes b)(c) &=&  2(a, b,c) + L_{a \bar b - b \bar a}c - R_{\bar a b  -  \bar b a}c\\
  g_2  (a \otimes b)(c) &=& \bar a(bc) - \bar b(ac)\\
  g_3  (a \otimes b)(c) &=& (cb)\bar  a - (ca)\bar b
  \end{eqnarray*}
  for $a, b \in D$, $c \in D.$
  \label{g_i_operators_def}
  \end{defi} 
  
  \begin{lem} Let $\Delta \in \mathrm{Der}(J)$ such that $\Delta(J_{ij}) \subset J_{ij}$ for $i \neq j.$   Fix three distinct elements  $i,j, k \in I$ denote $\Delta|_{J_{ij}}$ by $g_k$, $\Delta|_{J_{jk}}$ by $g_i$, $\Delta|_{J_{ki}}$ by $g_j$.  Then
\begin{itemize}
 \item[\rm(i)] The triple $$\mathbf{g}_\Delta(k,i,j) = (\overline{g_k}, g_i, g_j)$$ is a triality of the alternative algebra $D.$ 
  \item[ \rm(ii)]  Every cyclic permutation $\pi$ of the indices yields a triality  $(\overline {g_{\pi k}}, g_{\pi i}, g_{\pi j}). $ 
  \item[\rm (iii)] If $\Delta \in \ud(\mathcal D_0)$ is an inner derivation of $J$ and $\mathbf{g}_\Delta(k,i,j)$ the corresponding triality, then there are inner derivations $\Delta'$ and $\Delta''$ of $J$ such that 
  \begin{eqnarray*}
  \mathbf{g}_{\Delta'}(k,i,j) &=& (g_k, \overline{g_i}, \overline{g_j}),\\
   \mathbf{g}_{\Delta''}(k,i,j) &=& (\overline{g_{\sigma k}}, g_{\sigma i}, g_{\sigma j})\\
  \end{eqnarray*}
  where $\sigma$ is any cyclic permutation of $(k, i, j).$
 \end{itemize}
 \label{general_trial_lemma}
   \end{lem}
   
  \begin{proof} (i) For $i,j, k \neq$, $x[ij]y[jk] =(xy)[ik],$ (with the product of $x$ and $y$ taken in $D$). Since $D$ is a unital  algebra,  $J_{ik} = J_{ij}J_{kj}.$ 
Choose an index pair $(mn)$ among $(ij)$, $(jk)$ and $(ki)$ and denote the third index by $l$ so that $(mnl)$ is a cyclic permutation of $(ijk).$ Define, $z[mn]$ by $x[nl] y[lm] = \bar z [mn] \in J_{mn} $ for $x[nl] \in J_{nl}$, $y[lm] \in J_{lm}$ and  $\Delta_l = \Delta |_{J_{mn}}, \Delta_m = \Delta |_{J_{nl}}, \Delta_n = \Delta |_{J_{lm}}.$ Then, since $\Delta$ is a derivation $J$: 
\begin{eqnarray*}  \Delta(z[mn]) &=& \Delta_l(( \bar y \bar x)[mn]) = {\Delta }( (\bar y)[ml] \bar x[ln]) \\ &=&  
 \Delta( y[lm] )x[nl] + y[lm] \Delta(x[nl]) \\ &=& (x \Delta_3(y))[nm] + (\Delta_2(x) y)[nm]  
 \end{eqnarray*}
and thus $(\overline {\Delta_l}, \Delta_m, \Delta_n )$ is a triality.  \\ 
(ii) Since the index pair $(mn)$ was arbitrary, we can choose another one and this will result in a cyclic permutation of $(mnl).$ \\
(iii) We may without loss of generality assume that $\Delta = [L_{a[ij]}, L_{b[ij]}]$ since that statement, once proven, will allow us to cyclically permute the indices. Then by Lemma~\ref{der_alt_explicit_form_lem}
$\mathbf{g}_{[L_{\bar a[ij]}, L_{\bar b[ij]}]} = \overline{\mathbf{g}}_{\Delta}$ and if $\sigma = ( j,l, i)$ is a cyclic permutation then $\Delta '' = [L_{a[jl]}, L_{b[jl]}]$ has the property that 
$ \mathbf{g}_{\Delta''} = (\overline{g_{\sigma 1}}, g_{\sigma 2}, g_{\sigma 3}).$
\end{proof}

\begin{lem} \label{matrix_jord_alg_lem} For $J = \mathcal H_I(D, \bar{\;}),$ $ \card I \geq 3$ with idempotents $e_i = E_{ii}$: 
 $\sum_{i\neq k }J_{ik}^2 = \sum J_{ii} .$
\end{lem}
\begin{proof}As noted before
$$J_{ij} = \{ d E_{ij}  + \bar d E_{ji } : d \in D\}, \quad J_{ii} = D_0E_{ii}$$ where $D_0$ is the set of hermitian elements of $D.$
 Since $d[ik] b[ik] = (d \bar b  + b\bar d)[ii] + (\bar d  b  + \bar b d)[kk] )$ and $ D_0$ is spanned by elements of the form $\{d + \bar d : d \in D \}$ we have to write $(b + \bar b)[jj]$ as a linear combination of elements of the form $(d [ii] + \bar d[kk]).$ 
This is a standard trick (see for example \cite[4.2]{Neh1996}): let $\beta = b + \bar b$ and pick three distinct indices $i,j,k$, then 
$$(b + \bar b)[ii]  = b[ii] + \bar b[jj]+ \bar b [ii] + b [kk] - \bar b [jj] - b [kk] \in J_{ij}^2 + J_{ik}^2 + J_{jk}^2.$$
\end{proof}
\begin{cor} Let $L = \mathrm{TKK}(\mathcal H_I(D, \bar{\;}))$, $u  : \mathfrak{uce}(L) \rightarrow L$ the universal central extension. Then 
$$\ker u = \ker(\ud_0 ) $$ where $\ud_0: \mathcal D_0 \rightarrow \sum_{i \neq j}[L_{J_{ij}}, L_{J_{ij}}].$
\label{hermitian_D_0_cor}
\end{cor}
\begin{proof} Proposition~\ref{one_half_jordan_hom_H_2_prop} states that $\ker(u) = \ker(\ud_{JA}). $ Lemma~\ref{matrix_jord_alg_lem} combined with  Proposition~\ref{Uni_der_Decomp_prop_ja}  and the fact that $e_i$ spans $J_{ii}$ in the case of the hermitian matrix Jordan pair gives that 
$\ker(\ud_{JA}) = \ker(\ud_0 ).$
\end{proof}
\subsection{Alternative coordinates: $n= 3$}
\label{alt_coordinates}
\begin{center}
From now on assume that, in addition to $1/2$, also $1/3 \in k.$

\end{center}
For this subsection we consider the case where $n = 3,$ $D$ is an alternative unital  $k$-algebra with nuclear involution and $J = \mathcal H_3(D, \bar{\;})$. Denote the hermitian elements of $D$ by $D_+$ and the skew hermitian elements by $D_-.$ 
We are interested in the relation between (inner) trialities of the alternative algebra $D$ and the (inner) derivations of the Jordan algebra $\mathcal H_3(D, \bar{\;}).$ \\
We will see  that this relation can give us information about the module $\mathcal D_0$ and ultimately about the kernel of $\ud_0$ (see Proposition~\ref{Uni_der_Decomp_prop_ja}). 
The space of all inner trialities of the alternative algebra $D$ is ``too big'' to describe $\ud(\mathcal D_0)$.   It turns out that the following object is the correct one: 
\begin{defi} With the definition of $h$ as in Lemma~\ref{trial_der_alt_lemma} define
$$\mathcal T_{00} := \im(h)((\mathrm{IDer} (D), D_-, D_-)) \subset \mathcal T_0.$$
\end{defi}
\begin{prop} \label{inner_trial_prop}
Assume that all derivations of $D$ are inner or the involution on $D$ is central. 
Then the map $$ \theta : \ud(\mathcal D_0) \mapsto \mathcal T_{00} $$
given by 
$$ \theta(\ud(X)) =  (\overline{g_1}, g_2, g_3) $$
is injective and well-defined, 
where the $g_1$,$g_2$ and $g_3$ are the restriction of $\ud(X)$ to $J_{12},$ $J_{23}$ and $J_{31}$ respectively. \\
If the involution on $D$ is central, then $\theta : \ud (\mathcal D_0) \rightarrow \mathcal T_{00}$ is an isomorphism.
\end{prop}
\begin{proof} Since $1/3 \in k$, standard inner derivations and inner derivations coincide (Proposition~ \ref{lopera_der_proposition}).
Assume that all derivations are inner.\\
Thus by  \ref{trial_der_alt_lemma}, all trialities are inner as well. Therefore the map $\theta$ is clearly well-defined as map into $\mathcal T_0$.  It is immediate that $\theta$ is a homomorphism of algebras. Also, since $\ud(X)$, $X \in \mathcal D_0$ is zero, if each of its restrictions  $\ud(X)|_{J_{ij}}$ is $0$,  $\theta$ is seen to be a monomorphism into $\mathcal T_{0}$. It remains to show that $\theta(\mathcal D_0) \subset\mathcal  T_{00},$  if all derivations are inner and that $\theta$ is  an isomorphism  onto $\mathcal T_{00}$ if the involution is central.\\
We know that $\mathcal D_0 = J_{12} * J_{12} + J_{23} * J_{23} + J_{31} * J_{31}.$
First assume that $X = a[12] * b[12] \in J_{12} * J_{12}.$
  The preimage of $\theta(\ud(X))$ under the isomorphism $h$ defined in Lemma~\ref{trial_der_alt_lemma} is $ h^{-1} (\bar g_1, g_2, g_3) = (\Delta, 2/3(\bar a b - \bar b a) -1/3(b \bar a- a \bar b), 1/3(\bar a b - \bar b a) + 2/3(b \bar a - a \bar b))$ for some (inner) derivation $\Delta$ which proves (since  $2/3(\bar a b - \bar b a) -1/3(b \bar a- a \bar b), 1/3(\bar a b - \bar b a) + 2/3(b \bar a - a \bar b)\in D_-$ ) that
  $h^{-1}(\theta(\ud(J_{12}* J_{12}))) \subset h^{-1}(\mathcal T_{00})$ and thus $(\theta(\ud(J_{12}*J_{12}))) \subset \mathcal T_{00}.$
  Let $\sigma$ be a cycle in the symmetric group on three elements. Then 
  $ \theta(\ud(a[\sigma 1, \sigma 2]* b[\sigma 1, \sigma 2])) = (\overline g_{\sigma 1}, g_{\sigma 2}, g_{\sigma 3} )$ which is a triality by Lemma~\ref{general_trial_lemma}(i).
 Clearly, if all derivations are inner, then in $h^{-1}((\overline{g_ \sigma 1}, g_{\sigma 1 }, g_{\sigma 3})) = (\Delta^\sigma, a^\sigma, b^\sigma )$ the 
 derivation $\Delta^\sigma$ will still be inner.  \\ (**) \\
  Assume $\sigma = (1, 2, 3).$ It is also not difficult to see that $a^\sigma$ and $b^\sigma$ are elements of $D_{-},$ since they are defined by 
 \begin{eqnarray*}
 3a^{\sigma^2} &=& 2 g_1(1) + g_2(1) =  2(a \bar b - b \bar a ) +  \bar a b - \bar b a, \\
  3b^{\sigma^{2}} &=&  -2 g_1(1) - g_2(1) =  -(a \bar b - b \bar a) - 2(\bar a b - \bar b a).  
 \end{eqnarray*}
 Therefore $\theta(\ud(J_{\sigma 1, \sigma 2} * J_{\sigma 1, \sigma 2})) \subset \mathcal T_{00}.$ 
 The argument for the cycle $\sigma^2$ is almost the same. Again it is given that in $(\Delta^{\sigma^2}, a^{\sigma^2}, b^{\sigma^2}),$ the derivation $\Delta^{\sigma^2}$ is inner. The elements $a^{\sigma^2}$ and $b^{\sigma^2}$ are given by
 \begin{eqnarray*}
 3a^{\sigma^2} &=& 2 g_3(1) + g_1(1) =  2(b \bar a - a \bar b ) +   a \bar b -  b  \bar a, \\
  3b^{\sigma^{2}} &=& - 2 g_3(1) - g_1(1 ) =  -(b \bar a - a \bar b ) - 2( a \bar b -  b \bar a).  
 \end{eqnarray*}
 and it is again easy to see that they lie in $D_{-}.$\\
 We have shown that for all cyclic permutations $\theta(\ud(J_{\sigma 1, \sigma 2} * J_{\sigma 1, \sigma 2})) \subset \mathcal T_{00}.$ Hence  $\theta(\ud(\mathcal D_0)) \subset \mathcal T_{00}.$ This shows that $\theta \circ \ud $ maps $\mathcal D_0$ into $\mathcal T_{00}.$\\
  Now assume also that the involution is central.
  It is left to show  that the map is surjective, i.e., for every element in  $\mathcal T_{00}$ there is a pre-image $\ud(\mathcal D_0).$ We already know that every element in $\mathcal T_{00}$ can be uniquely expressed as $(\Delta, \Delta, \Delta) + \lambda(a) + \rho(b)$ where $\Delta$ is an inner derivation and $a, b \in D_-$ and that under the isomorphism $h,$ this is the image of $(\Delta,a, b ).$ Therefore it suffices that  $\theta(\ud(\mathcal D_0))$ contains all $\lambda(a)$ and $\rho(b)$, $a, b \in \mathcal D_-$  and that for every inner derivation $\Delta$ there is an elements $X \in \mathcal D_0$ such that    
  $h^{-1} \circ \theta(\ud(X))$ is of the form $(\Delta,  A, B)$ for some $A$ and $B.$\\
 Let $\lambda(a) = (L_a, L_a + R_a, -L_a)$. If $\lambda(a)$ corresponds to an element in $\ud(\mathcal D_0)$, then equivalently so does the cyclic permutation $(-L_a, L_a, L_a + R_a)$ (by Lemma~\ref{general_trial_lemma}, since  the trialities associated to the set $\ud(\mathcal D_0)$ are invariant under cyclically permuting the entries). Furthermore, if $\bar a = -a$, then  Lemma~\ref{general_trial_lemma} and the identity $\overline{L_a} = R_{\bar a}$ imply that $ (R_a, -R_a, -L_a - R_a ) = -\rho(a)$ is an inner triality. In fact, $(R_a, -R_a, -L_a - R_a )$ is an inner triality if and only if $(L_a, L_a + R_a, -L_a)$ is an inner triality.
 Therefore we only have to show that some cyclic permutation of $\lambda(a)$ lies in $\theta(\ud(\mathcal D_0)).$ 
 Let $X = a[ij] * 1[ij]$ for fixed $i< j$ and $a = -\bar a.$
 Then $\theta(\ud(X)) = 2( \overline{L_a + R_a},  L_a, -R_a )$ by Lemma~\ref{der_alt_explicit_form_lem}. 
 Note  that, when $\bar a = -a ,$  then $\overline{L_a} = -R_a$, since $\bar{\;}$ is an involution. 
 Therefore,   $$ \theta(\ud(X)) = -2(L_a + R_a, -L_a, R_a  )  $$ which is a  cyclic permutation of $-2\lambda(a)$, hence all $\lambda(a)$ and $\rho(b)$ for $\bar a = -a$ and $\bar b = -b$ are in the image of $\theta \circ \ud.$
 Let $\Delta$ be the standard inner derivation $SD(a, b)$. We claim that $h(\theta(\ud(a[12] * b[12])))$ is of the form $2/3 (SD(a, b), A, B).$ Let $(\bar {t_1}, t_2, t_3)$ be the inner triality determined by $\ud(a[12] * b[12])$ and $(\Delta, A, B) = h(\theta(\ud(a[12] * b[12])))$. It follows from Lemma~\ref{der_alt_explicit_form_lem} that 
 $\Delta = (t_1 - L_A + R_B)$ (see below for the explicit formula for $A$ and $B$) where $\overline{t_1}(z) = 2(a, b,z) -  L_{a\bar b - b \bar a} +  R_{\bar a b - \bar b a}z $. Since for $x \in D_{-}$ we have $\overline{L_x} = -R_{x}$ and by centrality of the involution, 
$\overline{(a, b, \bar z)} = (a, b, z)$, the map $t_1$ is actually equal to $2[L_a, R_b] + R_{a\bar b - b \bar a} -L_{\bar a b - \bar b a} $.  By definition: $$A = -1/3t_2(1) - 2/3t_3(1) = -1/3(\bar a b - \bar b a) - 2/3 (b \bar a - a \bar b),$$ 
$$ B = 2/3 t_2(1) + 1/3 t_3(1) = 2/3 (\bar a b - \bar b a) + 1/3 (b \bar a - a \bar b).$$
Since the involution is central $[\bar a, b] + [a, \bar  b] =  -2[a, b]$,   and
thus $$-L_{\bar a b - \bar b a} - L_A = L_{-\bar a b + \bar b a + 1/3(\bar a b - \bar b a) + 2/3 (b \bar a - a \bar b)} = 2/3L_{[a,b]}$$ and likewise $$R_{a\bar b - b \bar a} + R_B = R_{a\bar b - b \bar a + 2/3 (\bar a b - \bar b a) + 1/3 (b \bar a - a \bar b) } = -2/3R_{[a, b]},$$
\begin{eqnarray}
\Delta &=& 2[L_a, R_b] + 2/3L_{[a,b]} -2/3R_{[a, b]} \\
 &=& 2/3SD (a, b), \label{image_is_inner_der}
\end{eqnarray}
and since $1/2, 1/3 \in k$ this proves that $3/2 h(\theta(\ud(a[12] * b[12]))) = (SD(a, b), A, B).$
Therefore $\theta$ is surjective. \\

If the involution is central, then we can drop the assumption that all derivations are inner and $\theta$ is still well-defined.
We have seen above (\ref{image_is_inner_der}) that the map is well-defined, if we restrict it to $J_{12} * J_{12} \subset \mathcal D_0$ Choose a cyclic permutation $\sigma$ of the indices $(1, 2, 3)$. By the argument used for the proof of  Lemma~\ref{general_trial_lemma}, we have that $\theta(\ud(a[\sigma 1 \sigma 2 ] * b[\sigma 1 \sigma 2])) = (\overline {g_{\sigma 1}}, g_{\sigma 2}, g_{\sigma 3})$ where $(\overline{g_1}, g_2, g_3 ) = \theta(\ud (a[12] * b[12])) \in \mathcal T_{00}$. We have already shown that  $(\overline {g_1}, g_2,  g_3 ) = (\Delta, \Delta, \Delta) + \lambda(A) + \rho(B)$ for the standard derivation $2/3SD(a, b)$ see  (\ref{image_is_inner_der})  and some $A, B \in D_{-}.$ It remains therefore to show that mapping an element $(\overline{g_1}, g_2, g_3 ) \in \mathcal T_{00}$ to $(\overline{g_{\sigma 1}}, g_{\sigma 2}, g_{\sigma 3})$ is a well-defined map from $\mathcal{T}_{00} $ to $\mathcal{T}_{00}$. \\
Consider $(\overline{g_{\sigma 1}}, g_{\sigma 2}, g_{\sigma 3}) = \sigma((\Delta, \Delta, \Delta)) + \sigma (\lambda(A)) + \sigma (\rho(B))$ where $\Delta$ is a standard derivation and $A, B \in  D_{-}.$ Since the involution is central, 
\begin{eqnarray*}
\overline{SD}(a, b)(c) &=& \overline{(a, \bar c, b) + L_{[a, b]}\bar c - R_{[a, b]}\overline c} \\ 
 &=& (a , c, b) - L_{[a, b]} c + R_{[a, b]}c = SD(a, b)c.
\end{eqnarray*}
  Thus $\sigma((\Delta, \Delta, \Delta)) = (\Delta, \Delta, \Delta). $ Then we can proceed with the argument from (**) on. 
\end{proof}

 \begin{cor} \label{inner_trialities_cor} Let $D$ be an alternative algebra with central involution. Then the  map
$\theta \circ \ud: \mathcal D_0 \mapsto \mathcal T_{00}$ is a central extension of Lie algebras. 
\end{cor}
\begin{proof}
Since $\theta$ is an isomorphism and $\ud$ a surjection, it follows that $\theta \circ \ud$ is surjective and $\ker(\theta \circ \ud) = \ker (\ud) \subset Z(\mathfrak{uce}(L)) \cap \mathcal D_{0} \subset Z_{\mathcal D_0}(\mathcal D_0)$, see Definition ~\ref{Liealgebra_centre_definition}.
\end{proof}
The following is worthwhile noting:
\begin{cor} Let $D$ be an alternative algebra with central involution.   The inner derivation algebra of the Jordan algebra $J = \mathcal H_3(D, \bar{\;})$ is isomorphic as a $k$-module to
$$\mathrm{StanDer}(D) \oplus  (D_-)^2 \oplus D^3 . $$
\label{stam_der_jrod_der_cor} 
\end{cor}
\begin{proof}By Proposition~\ref{Uni_der_Decomp_prop_ja}
$\mathrm{IDer}(J) \cong \ud(\mathcal D_0) \oplus  \bigoplus_{i< j} J_{ij} .$
since $\ud(\mathcal D ) \cong \bigoplus_{i< j} J_{ij}.$ Each of the three modules $J_{12}$, $J_{23}$ and $J_{13}$ is isomorphic to $D$. By \ref{inner_trial_prop}, $\ud(\mathcal D_0) \cong (\mathrm{IDer} (D), D_-, D_-)$ under the isomorphism $h$. Since, if $1/3 \in k$, $\mathrm{IDer}(D) = \mathrm{StanDer}(D),$ the claim follows. 
 \end{proof}
\begin{cor}Let $D$ be an alternative algebra with central involution and  $J = \mathcal H_3(D, \bar{\;}).$
If a base change $k \rightarrow K$ preserves the centre, i.e.,  
 $\mathrm{Cent}(D \otimes_k K) = \mathrm{Cent}(D) \otimes_k K$, then 
 $$\mathrm{IDer}(J) \otimes_k K \cong \mathrm{IDer}(J \otimes_k K)  \iff \mathrm{StanDer}(D) \otimes_k K = \mathrm{StanDer}(D \otimes_k K). $$
  In particular, if $D$ is finitely generated as a $k$-module and $k \rightarrow K$ is a flat base change, then  
  $$\mathrm{IDer}(J) \otimes_k K = \mathrm{IDer}(J \otimes_k K)  \mbox{ and } \mathrm{StanDer}(D) \otimes_k K = \mathrm{StanDer}(D \otimes_k K). $$
 
\label{herm_jord_der_flat_base_change_cor}
\end{cor}
\begin{proof} 
Since the base change preserves the centre, the canonically given involution on $D \otimes_k K$ is central if $(D_+ \otimes_k K) = (D \otimes_k K)_+.$  We claim
$$D_-  \otimes_k K=  (D \otimes_k K)_- \mbox{ and } D_+  \otimes_k K=  (D \otimes_k K)_+.$$
 We only prove the first of the two statements because the calculations are identical. One inclusion is clear $D_-  \otimes_k K \subset  (D \otimes_k K)_-$. Conversely,  every element in $D \otimes_k K$ is a combination of elements of the form $a \otimes \alpha$ where $a \in K,$ $\alpha \in K.$
Then $(D \otimes_k K)_-$ is spanned  by the skew parts of those elements, namely $a \otimes \alpha - \overline{a \otimes \alpha} = a \otimes \alpha - \bar  a \otimes \alpha = (a - \bar a) \otimes \alpha  \in D_- \otimes_k K.$ It follows that $D_-  \otimes_k K \supset  (D \otimes_k K)_-.$  
Hence, by Corollary~\ref{stam_der_jrod_der_cor}, 
$$\mathrm{IDer}(J \otimes_k K) \cong \mathrm{StanDer}(D\otimes_k K) \oplus (D_-^2 \otimes K) \oplus (D \otimes_k K)^3 .$$
Since we always have canonical epimorphisms,
\begin{eqnarray*}
\mathrm{IDer}(D) \otimes_k K & \rightarrow & \mathrm{IDer}(D \otimes_k K),\\ 
\mathrm{StanDer}(D) \otimes_k K & \rightarrow & \mathrm{StanDer}(D \otimes_k K), 
\end{eqnarray*}
the isomorphism above proves the claimed equivalence. 
Suppose that $k \rightarrow K$ is a flat base change and that $D$ is finitely generated over $k$ , then by (\cite[Proposition 1.14]{lopera})
$\mathrm{Cent}(D \otimes_k K) = \mathrm{Cent}(D) \otimes_k K.$ In the same paper, (\cite[Proposition 2.9]{lopera}), it is proven that under those assumptions $\mathrm{StanDer}(D) \otimes_k K = \mathrm{StanDer}(D \otimes_k K).$
\end{proof}

 \begin{defi} Assume that $M$ is a $k$-module  which is free of rank $n$ and that $M$ carries a quadratic form defined by a symmetric matrix $B$ by $N(x,y) = x^T B y.$ Then $\mathfrak{so}(M, B) = \{ X \in \mathrm{Mat}_n(k) : BX + X^TB = 0\} . $
 It is a standard exercise that $\mathfrak{so}(M, B)$ is a Lie algebra. 
 \end{defi}

\begin{prop} Assume that the alternative algebra $D$ is free of rank $2n$ and that $D$ has a basis $\{x_n, \ldots, x_1, x^1, \ldots x^n \}$ such that the norm form in Definition  \ref{alt_norm_trace_Defi}
is given by the matrix  $$\left [\begin{array}{cc} 0 & I_n \\ I_n & 0  \end{array} \right].$$
 Let $g_1(a \otimes b) \in \mathrm{End}(D)$ be defined as in Definition~\ref{g_i_operators_def}. Then
$\spa \{g_1(a \otimes b)\}  = \mathfrak{so}(D, N).$ \label{so_first_entry_of_trial_prop}
\end{prop}
\begin{proof} The Lie algebra $\mathfrak{so}(D, N)$ is spanned by the matrices $E_{-i, -j} - E_{j, i}$, $E_{-i, j} - E_{-j, i}$, $E_{i, -j} - E_{j, i}$, $1 \leq i, j \leq n.$ It is easy to see that for $i \neq j,$ 
$g_1(x_i \otimes  x_j) =  2(E_{-i, j} - E_{-j, i})$, $g_1(x_i \otimes x^j) =  2(E_{-i, -j} - E_{j, i})$ and $g_1(x^i \otimes x^j) = 2(E_{i, -j} - E_{j, i})$ and $g_1(x_i \otimes x^i) = 2(E_{-i,  -i} + E_{i, i}).$
\end{proof}

\subsection*{Example: Albert Algebras}
\label{Albert_example}
We will discuss the special case where the alternative algebra is an octonion algebra and $n= 3.$ 
For the rest of the  present subsection we will denote by $z \rightarrow \bar z$  the involution on a split octonion algebra $Z$ as defined in Definition~\ref{octonion_alge_Defi}:
$$\bar z = t(z)1_z - z $$ where $t(z)$ is the matrix trace. This involution is central, the only fixed points being multiples of the identity. The norm form is
$$N  \left ( \begin{array}{cc} \alpha_1 & u \\ x & \alpha_2 \end{array} \right ) = \alpha_1\alpha_2 -  (x,u) ,$$
where $(x,u) = x* u$ is defined as in Definition \ref{octonion_alge_Defi}. Since $*$ has already been used in this section to denote the spanning elements of $J*J$, we prefer to denote the inner product by simple parentheses. 
For this norm form, $Z$ decomposes as $Z_u \oplus Z_l$,  
where $Z_u = \spa\{x_1 ,x_2, x_3, x_4\},$ $Z_l = \spa\{x^1, x^2, x^3, x^4\},$ and $x_1 = e_{11},$ $x^1 = 1- e_{11} = e_{22}$ and $\{x_2, x_3, x_4\},$ $\{x^2, x^3, x^4\} $ are hyperbolic bases of $M^{\pm}$ respectively. Thus $N(x_i, x^j) = \delta_{ij}.$
\begin{defi} Let $Z$ be the split octonion algebra over $k$ with involution $\bar{z}$ as defined above and norm form $N(\cdot, \cdot)$ as above. Then 
$$\mathcal H_3(Z, \bar{\,}) := \left \{\left( \begin{array}{ccc} \alpha 1 & z & \bar y \\
\bar z & \beta 1 & x \\
y & \bar x & \gamma 1 \\
  \end{array} \right) : \alpha, \beta, \gamma \in k, x, y, z \in Z \right\}.$$
  We also abbreviate 
$$  \left( \begin{array}{ccc} \alpha 1 & z & \bar y \\
\bar z & \beta 1 & x \\
y & \bar x & \gamma 1
  \end{array} \right) := \alpha e_1 + \beta e_2 + \gamma e_3 + P_1(x) +P_2(y)+ P_3(z).$$
  
  Then $H_3(Z, \bar{\,})$ is a linear Jordan algebra with commutative product defined as follows, where $(i,j,k)$ is a cyclic  permutation of $(1, 2, 3)$, 
  \begin{eqnarray}
  e_i  \circ e_j &=& 2\delta_{ij} e_i, \\
  e_i  \circ P_j(x) &=& (1 - \delta_{ij})P_j(x),\\
  P_i(x)  \circ P_i(y) &=& N(x,y) (e_j + e_k),\\
  P_i(x) \circ P_j(y) &=&  P_k(\bar y \bar x).
  \end{eqnarray}
  In particular the set $\{e_1, e_2, e_3\}$ is a complete orthogonal set of idempotents and the Peirce spaces are 
  $$J_{ij} = P_k(Z), \; \{i,j,k\} = \{1,2,3\}, \quad J_{ii} = ke_{i}.$$
 Jordan algebras of this type are called \emph{split Albert algebras.}  A Jordan algebra $J$ over $k$ is called an \emph{Albert algebra} if there is a faithfully flat base change $k \rightarrow K$ such that $J \otimes_k K$ is isomorphic to a split $K$-Albert algebra.     
 \label{Albert_defi} 
 This is really just a fancy way to write the Jordan product as it was introduced on $\mathcal H_I(D, \bar{\;})$ in Definition~\ref{hermitian_matrix_jord_alg}. It works well for our purposes here. 
\end{defi}
\begin{rem} This definition of an Albert algebra is justified by the results of Petersson and Racine in  \cite{pera1996}.
\end{rem}

\paragraph{Orthogonal trialities}
\begin{center}
From now on $J$ is the Albert algebra with coordinates in the split octonion algebra $Z.$ 
\end{center}

\begin{defi} \label{orth_trial_defi}  
A triality $(T_1, T_2, T_3)$ is called \emph{orthogonal} if for each $T_i$ and all $a, b \in Z:$
$$N(T_ia, b ) = N(a, T_{i}(b)) $$  
\end{defi}

\begin{rem} The condition $N(T_ia, b ) = N(a, T_{i}(b)) $ is equivalent to $N(T_ia, a) = 0$, i.e., $T_i\in  \mathfrak{so}(D, N).$
\end{rem}

Recall that by Proposition~\ref{so_first_entry_of_trial_prop}
$\spa \{g_1(a \otimes b) : Z \rightarrow Z | a, b \in D\}  = \mathfrak{so}(Z, N).$ 
\begin{lem} 
Under this identification, the Lie algebra $\mathfrak{so}(Z, N)$ is generated by $L_a$ and $R_b$ for $a,b \in D_{-}.$ 
\end{lem}
\begin{proof}  For $a = a_+  + a_-$, $b = b_+ + b_- \in D,$  we get, using $D_+ \subset Z(D)$, that  $SD(a, b) = [3L_{a_-}, R_{b_-}] + L_{[a, b]} + R_{[a, b]}.$ Since $[a, b] \in D_{-}$  it follows that 
  $\{L_{D_-},  R_{D_-}\}$ generates $\mathfrak{so}(Z, N)$ as a Lie algebra.  
\end{proof}

For $k$ a field not of characteristic $2$ or $3$ the following is stated as the ``Principle of local triality'' in \cite{jacobson1971}: 
\begin{prop} Let $T_1 \in \mathfrak{so}(Z, N).$ Then there exist unique $T_2, T_3 \in \mathrm{End}_k(Z)$ such that $(T_1, T_2, T_3)$ is a triality and this triality is orthogonal.  \label{jacobson_prop}
\end{prop}
\begin{proof} See \cite[p.8 -9]{jacobson1971} for the proof where $k$ is a field. We first show that the set of $T_1 \in \mathfrak{so}(Z, N)$ such that $T_2$ and $T_3$ as stated exist, is a subalgebra of $\mathfrak{so}(Z, N).$ Let $\mathbf T = (T_1, T_2, T_3)$ and $\mathbf S = (S_1, S_2, S_3)$ be trialities such that all components are in $ \mathfrak{so}(Z, N).$ Then $ [\mathbf T, \mathbf S] = ([T_1, S_1], [T_2, S_2], [T_3, S_3])$ has again components in $\mathfrak{so}(D, N)$.
Because $L_{D_0}$ and $R_{D_0}$ generate $\mathfrak{so}(D, N)$, the fact that  $\lambda (a)$ and $\rho(b)$, $a, b \in D_{-}$ are  trialities implies that $T_2, T_3$ exist for every $T_1 \in \mathfrak{so}(Z, N)$ as they exist for $L_a$ and $R_b.$
To prove uniqueness, it suffices to show that $(T_2 x)y + x T_3(y) = 0$ for all $x,y \in D$ implies $T_2 = T_3 = 0.$ Let $x= y = 1$. Then there is $u \in D$ such that $T_2(1) = u = -T_3(1).$ Let $x = 1$ and let $y$ be arbitrary. This gives $T_3(y) = -uy$ and likewise we obtain $T_2(x) = xu.$ Hence $$(xu)y = x(uy) $$ for all $x,y \in D$. Since the centre is spanned by $1,$ it follows that $u = \alpha 1$. But moreover $x \mapsto T_2(x) = \alpha x$ is skew, and therefore $0  = N(T_2(x_i), x^i) =  N(\alpha x_i, x^i)= \alpha(N(x_i, x^i)) = \alpha \delta_{ii} = 0.$ This implies $\alpha = 0.$
\end{proof}

\begin{prop}Let $J$ be an Albert algebra over $k$.  The following are equivalent
\begin{itemize}
\item[\rm (i)] $\TKK(J)$ is centrally closed.
\item[\rm (ii)] $\TKK(\mathcal A)$ is centrally closed for some split octonion algebra $\mathcal A=  J \otimes_k K$ and $k \rightarrow K$ a faithfully flat extension.
\end{itemize}
  \label{albertfundobservatationlem} 
\end{prop}
\begin{proof}
Let $J$ be an octonion algebra over $k$ and $k/K$ flat. Then the Jordan algebra $J$ is a finitely generated $k$-module and by flatness, $\instr(J \otimes _k K) = \instr(J) \otimes_k K, $ therefore $\TKK(J \otimes_k K) = \TKK(J) \otimes_k K.$ \\
It known (\cite{prep2009EN}) that the following are equivalent for any $k$-Lie algebra $L$:
\begin{itemize}
\item[(a)] $L$ is centrally closed,
\item[(b)] $L_k \otimes _k K$ is centrally closed for some $K/k$ faithfully flat, 
\item[(c)] $L_k \otimes _k K$ is centrally closed for all $K/k$ faithfully flat.
\end{itemize}
This (i) $\implies$ (ii) follows from the equivalence of (a) and (c). The direction (ii) $\implies$ (i) is given by (a) $\iff$ (b). 
\end{proof}
When dealing with questions regarding the central closure of $\TKK(J),$ $J$ an Albert algebra we can therefore without loss of generality assume that 
\begin{center}
$J$ is a split Albert algebra with coordinates in a split octonion algebra $Z.$
\end{center}

\begin{cor} There is a central extension $\phi: \mathcal D_{0} \rightarrow \mathfrak{so}(Z, N)$,  and this central extension is isomorphic to $\ud_0$. \label{D_null_so_ce_cor}
\end{cor}
\begin{proof}
Note that $Z$ has only inner derivations (\cite[Cor. 5.2]{lopera}) and that the involution is central. Thus Proposition~\ref{inner_trial_prop} can be applied and $\theta(\ud(\mathcal D_0)) \cong \mathcal T_{00}.$
It follows  by Proposition~\ref{so_first_entry_of_trial_prop} and Proposition~\ref{jacobson_prop} that as Lie algebras: 
$$ \mathcal T_{00}  \cong \ud_0(\mathcal D_0) \cong \mathfrak{so}(Z, N).$$ 
Call the isomorphism $f.$
Since $ \ud_0 : \mathcal D_0 \rightarrow \mathcal T_{00} $ is a central extension, we obtain a central extension $\phi = \ud_0 \circ f. $
\end{proof}
\begin{prop} 
Under the assumptions of this section ($1/2$, $1/3 \in k$ and $J$ an Albert algebra), the Lie algebra $\TKK(J)$ is centrally closed \label{TKK_matri_jord_cent_closed_theo}.
\end{prop}
\begin{proof}
 Let $L = \mathrm{TKK}(\mathbf J)$ and let $u: \mathfrak{uce}(L) \rightarrow L$ be a universal central extension. As explained above, we can assume that $J$ is split. Then $\ker u \cong \ker( f)$  where $f$ is the  central extension $\phi: \mathcal D_0 \rightarrow \mathfrak{so}(Z, N)$ of Corollary~\ref{D_null_so_ce_cor}. Since $\mathfrak{so}(Z, N)$ is centrally closed for $1/2 \in k$ (see \cite{vdK}), it suffices to prove
that $\mathcal D_0$ is perfect.\\
Since $\mathcal D_0 = \sum_{1 \leq i , j \leq 3}J_{ij} * J_{ij}$, it suffices to prove that for every two basis elements $a$ and $b$ of $J_{12} \cong Z$, $a * b \in [J_{12} * J_{12}, J_{12} * J_{12}].$ First assume that $a, b$ are basis elements such that $N(a, b) = 0$, in particular, $a * b \neq  \pm x_i * x^i$.  There are elements $c$ and $d$ linearly independent from $a$ and $b$ such that $N(a, d) = N(b,c) = N(a, c) = 0$ and $N(c,d) = 1$. We can even choose $c$ and $d$ to be basis elements. Thus using (\ref{diagonal_action_1}):
$$[a  * c, d * b] = 2\big((N(c,d)a - N(a, d)c)\big) * b + d * \big( N(c,b)a - N(a,bc) \big) = 2(a * b), $$ which proves that $a* b$ is in $[\mathcal D_0, \mathcal D_0]. $
If $N(a, b) = 1$ we can assume $a = x_i$, $b = x^i$. Let $i \neq j$, then 
$$[x_i * x^j, x_j * x^i] = 2(x_i * x^i -x_j * x^j)$$
and
$$[x_i * x_j, x^i * x^j] = 2(- x_j * x^j +  x^i * x_i).$$
Adding up those two elements in $[\mathcal D_0, \mathcal D_0]$ yields
$$-4(x_j * x^j) \in [\mathcal D_0, \mathcal D_0] $$ and thus
$\mathcal D_0 = [\mathcal D_0, \mathcal D_0]$ as needed. 
\end{proof}

\begin{defi}
 \label{Albert_Alg_Def_ring}If $k$ is an algebraically closed field, a Lie $k$-algebra is said to be of \emph{type $F_4$} if it is isomorphic to the derivation algebra of an Albert algebra $\mathcal A$ over $k$ (which is necessarily split). \\
 In general, if $k$ is a ring, we define a \emph{Lie algebra of type $F_4$} as a  Lie $k$-algebra $L$ such that for every $\mathfrak{p} \in \mathrm{Spec}(k),$ the prime spectrum of $k,$ the Lie algebra $L \otimes_{\overline{Q(k_{\mathfrak p})}} {\overline{Q(k_{\mathfrak p})}}$ is of type $F_4,$ where $Q(k_{\mathfrak p})$ is the quotient field of $k_{\mathfrak p}$ and $\overline{Q(k_{\mathfrak p})}$ its algebraic closure.
\end{defi}

We need one more fact:
\begin{prop}[\cite{jacobson1971}] Let $F$ be a field of characteristic $\neq 2$ and $J$ an Albert algebra over $F$. Then 
$$\mathrm{Der}_F(J) = \mathrm{IDer}_F(J) $$
is a  Lie algebra of type $F_4.$
\label{all_der_inner_albert}
\end{prop}
\begin{proof} The inner derivations $\mathrm{IDer}(J)$ always form a non-zero ideal in $\mathrm{Der}(J).$ Since by \cite[p .21, Thm.3]{jacobson1971} the derivation algebra of $J$ is simple over a field not of characteristic $2$, it follows that  $\mathrm{Der}_F(J) = \mathrm{IDer}_F(J).$
\end{proof}

\begin{theo}Let $J = \mathcal H_3(\mathbb O, \bar{\;})$ for an octonion algebra $\mathbb O$ over a ring $k$ containing $1/2$ and $1/3$. Then 
$\mathrm{Der}_k(J)$ and $\mathrm{IDer}_k(J)$ are Lie algebras of type $F_4.$
\label{Der_Mat_alg_F_4_der}
\end{theo}
\begin{proof} By \cite[Ch.2 Thm 1]{AC} we know that $F = \overline{Q(k_{\mathfrak p}})$ is flat over $k_{\mathfrak p}$ and $k_{\mathfrak p}$ is flat over $k$ for all $\mathfrak p \in \mathrm{Spec}(k).$  So $F/k$ is flat. 
Since $J$ is finitely generated projective and finitely presented
$\mathrm{Der}_{F }(J \otimes _k  F ) = \mathrm{Der}_k(J) \otimes _k F $ and $\mathrm{IDer}_{F }(J \otimes _k  F ) = \mathrm{IDer}_k(J) \otimes _k F .$ 
We have $J_F = \mathcal H_3(\mathbb O, \bar{\;})  \otimes_k F =  \mathcal H_3(\mathbb O \otimes_k F, \bar{\;})$ and by \cite[4.1]{lopera} $\mathbb O \otimes_k F$ is a split octonion algebra, hence $J_F$ is a split Albert algebra.  
 By Proposition~\ref{all_der_inner_albert}, $\mathrm{Der}_F(J) = \mathrm{IDer}_F(J),$ thus both $\mathrm{Der}_k(J)$ and $\mathrm{IDer}_k(J)$ are of type $F_4.$
\end{proof}

\begin{cor} Let $J$ be any Albert algebra over k. Then 
$\mathrm{IDer}_k(J)$ is of type $F_4.$
\end{cor}
\begin{proof}
In view of Corollary ~\ref{stam_der_jrod_der_cor}, the  inner derivation algebra of $\mathcal H_3(Z_K, \bar{\;})$ is isomorphic to  $\mathrm{StanDer}_K(Z_K) \oplus  ((Z_K)_-)^2 \oplus (Z_K)^3$ for any ring $K$ containing $1/2$ and $1/3.$  \\ Let $J$ be any Albert algebra and assume that $\mathcal H_3(Z_K, \bar{\;}) \cong J \otimes_k K$ and that $K/k$ is faithfully flat. Then $\mathrm{IDer}(J) \otimes_k K = \mathrm{StanDer}_K(Z_K) \oplus  ((Z_K)_-)^2 \oplus (Z_K)^3. $ 
But $Z_K$ is split, thus there is an octonion $k$-algebra $Z_k$ such that $Z_k \otimes K = Z_K.$ As suggested by the notation, $Z_k$ can be chosen to be split. Hence, $\mathrm{StanDer}_K(Z_k) \oplus  ((Z_K)_-)^2 \oplus (Z_K)^3 = (\mathrm{StanDer}_k(Z_k) \oplus  ((Z_k)_-)^2 \oplus (Z_k)^3) \otimes_k K .$
By faithful flatness, $\mathrm{IDer}_k(J) = \mathrm{StanDer}_k(Z_k) \oplus  ((Z_k)_-)^2 \oplus (Z_k)^3$ and $\mathrm{StanDer}_k(Z_k) \oplus  ((Z_k)_-)^2 \oplus (Z_k)^3 = \mathrm{IDer}(\mathcal H_3(Z_k, \bar{\;})).$ Now the claim follows from Theorem~\ref{Der_Mat_alg_F_4_der}.
\end{proof}

\subsection{Associative coordinates: $n > 3$}
The setting for this section is the following: 

We continue to assume $1/2 \in k,$ but not necessarily $1/3 \in k.$
Furthermore $J$ is the hermitian Jordan algebra $\mathcal H_n(D, \bar{\;})$ where $n\geq 4$ and $D$ is an associative algebra with involution $\bar{\;} : D \rightarrow D$. Since $1/2 \in k$, we can uniquely write every element as sum of a symmetric and a skew-symmetric element:
$$ z = \frac{z+ \bar z}{2} + \frac{z -\bar z}{2} $$
and 
$D_+ = \{z: z = \bar z \}$ is spanned by the ``traces'' $z+ \bar z.$
and $D_-$ is spanned by $z - \bar z.$
The hermitian matrix Jordan algebra decomposes then into simultaneous Peirce spaces: 
if $1 \leq i < j \leq n$, then 
$$J_{ij} = \{ dE_{ij} + \bar d E_{ji}: d \in D\} = J_{ji}.$$
As usual, we will denote:
$d[ij] = dE_{ij} + \bar d E_{ji}$ for $i \neq j$ and  for $1 \leq i \leq n:$
$ J_{ii} = D_+E_{ii}= \{d_0E_{ii} : d_0 \in D_+\}.$

Many of the observations made for alternative coordinates will simplify when $D$ is associative. However, the trade-off is of course, that the number  of idempotents can be greater than $3.$

It is very easy to prove the following simplification of
\ref{der_alt_explicit_form_lem}.

\begin{lem} Let $D$ be associative. \label{ass_alt_explicit_form_lem}
 Recall (see Corollary~\ref{ud_JA_action_ce_cor}) that we have a central extension $ \ud: J* J \rightarrow \mathrm{IDer}(J)$ which induces an action by derivations $(x * y).z = D(x,y).z$ for all $x,y, z \in J.$ Under this action,  for $a,b,c \in D$,
  \begin{eqnarray}
   (a[ij] * b[ij]).c[ij] &=&  ((a \bar b - b \bar a)c - c(\bar a b  -  \bar b a))[ij],  \label{der_ass_explicit_form_lem1}\\
  (a[ij] * b[ij]).c[jl]  &=& (\bar abc  - \bar bac )[jl] \quad i,j, l \neq, \label{der_ass_explicit_form_lem2} \\
   (a[ij] * b[ij]).c[li]  &=& (cb\bar  a - ca\bar b)[li] \quad i,j, l \neq,  \label{der_ass_explicit_form_lem3}\\
   (a[ij] * b[ij]).c[kl]  &=& 0\quad \{i,j\} \cap \{kl\} = \emptyset.
   \end{eqnarray} 
  \end{lem}

\begin{proof} This follows immediately from Lemma~\ref{der_alt_explicit_form_lem}. One just has to remember that left and right multiplication operators commute in this case.  
\end{proof}

\begin{lem}
In $\mathcal D_0 = \sum_{i < j} J_{ij} * J_{ij}$ the following hold:
\begin{itemize}
\item[\rm (i)] For any $i,j \neq$ and neither equal to $1$:
\begin{equation} ab[ij] * \bar c[ij]  + bc[1i] * \bar a[1i] + ca[j1] * \bar b[j1] = 0. \label{cyclic_with_involution}\end{equation}
\item[\rm(ii)] $T(a, b) := a[1j] *  b[1j]  - 1[1j]* ( \bar a b)[1j] $ does not depend on $j.$
\item[\rm(iii)] $1 [ij] * a[ij] = 0, \quad \forall a \in D_+ $
\end{itemize}
\end{lem}

\begin{proof}Let $a, b, c \in D$ and consider the element $a[ik]b[kj]  * \bar c[ij]$ where $i,j, k\neq.$ We have by (\ref{HCJ1})
\begin{eqnarray}
ab[ij] * \bar c[ij] &=& a[ik]b[kj]  * \bar c[ji] \\
 &=& - b[kj] c[ji] * a[ik] -  c[ji] a[ik] * b[kj] \label{C_n_graded_eqn1_k_gen}\\
&=& -b[kj] c[ji] * \bar a[ki] - c[ji]a[ik] * b[kj] \\
&=& -bc[ki] * \bar a[ki] -ca[jk] * \bar b[jk] 
 \end{eqnarray}
whence (\ref{cyclic_with_involution}). 

Setting $ k= 1$ gives (\ref{cyclic_with_involution}). Assume  $b = 1$ and $k = 1$, then
\begin{eqnarray}
a[ij] * \bar c[ij] &=& -c[1i] * \bar a[1i] - ca[j1] * 1[1j]  \nonumber \\
&\in & \sum_{k \neq 1} J_{1k} * J_{1k} \label{C_n_graded_eqn1}.
\end{eqnarray}
In the above identity (\ref{C_n_graded_eqn1_k_gen}), let  now $a = b = 1,$ then 
$$1[ij] *  \bar c[ij] + c[ki] * 1[ki] + c[jk] * 1[jk].$$
If we switch $j$ and $k$, then a second identity can be obtained:  
$$1[ik] *  \bar c[ik] + c[ji] * 1[ji] + c[kj] * 1[kj].$$
Adding these two gives 
$ (\bar c + c)[jk] * 1[jk] = 0 .$ This is equivalent to 
\begin{equation}c[jk] * 1[jk] = -\bar c[jk] * 1[jk] \label{conjugation_switch_sgn_hom_id}. \end{equation}
Thus if $c = \bar c$, $c[jk] * 1[jk]= 0$ which gives (iii).  
Let $ k= 1$ and $c = 1$ in (\ref{C_n_graded_eqn1_k_gen}), then 
\begin{eqnarray*}
ab[ij] * 1[ij] &=& - b[1i] * \bar a[1i] - \bar a[1j] * b[1j] \\
&=& \bar a[1i] * b[1i] - \bar a[1j] * b[1j].
\end{eqnarray*}
Similarly, 
\begin{eqnarray*}
1[ij] * ab[ij] = - \bar b \bar a[1i] * 1[1i] +  \bar b \bar a[1j] * 1[1j]. \\
\end{eqnarray*}
If we add those two equations, the left hand side vanishes and we obtain 
$$ 0 =  \bar a[1i] * b[1i] - \bar a[1j] * b[1j] - \bar b \bar a[1i] * 1[1i] +  \bar b \bar a[1j] * 1[1j] $$
or equivalently
$$\bar a[1j] * b[1j] - \bar b \bar a[1j] * 1[1j] =  \bar a[1i] * b[1i]  - \bar b \bar a[1i] * 1[1i] +  \bar b \bar a[1j] * 1[1j]. $$
Since $\bar b \bar a[1j] * 1[1j] = 1[1j]  * a b[1j]$ and $\bar b \bar a[1i] * 1[1i] = 1[1i]  * a b[1i],$
this proves that $T(\bar a, b)$ does not depend on the choice of $j.$ If we replace $\bar a$ by $a,$ this is (ii).
\end{proof}

\begin{rem} 
 By the  result above, the following elements span $\mathcal D_0:$
\begin{eqnarray}
T(a,b) &&a, b \in D, \\
h_j(a):= 1 [1j] * a[1j] && a \in D_{-}.
\end{eqnarray}
\end{rem}

\begin{prop}
\label{C_N_prop_Center}
An element $\sum T(a_i, b_i) + \sum_{j \geq 2}1[1j] * c_j[1j]$ is in the kernel of $\ud: J* J \rightarrow \mathrm{IDer}(J) $ if and only if
there is $c \in Cent(D) \cap D_-$ and
$$c_j = c ,\quad \sum_{i}[\bar a_i, b_i] + {[a_i, \bar b_i]} = -2nc $$
where $n$ is as in the definition of $J.$ 
\end{prop}
\begin{proof}
The element $\ud(T(a, b))$ acts on $J_{1k}$ as $L_{[a, \bar b]  + [\bar a, b]}$ as one easily verifies using \linebreak[3] Lemma~\ref{ass_alt_explicit_form_lem} and annihilates all other Peirce spaces. 
The inner derivation $\ud(1[1j] * c[1j])$ is $L_c - L_{\bar c} + R_{c - \bar  c}$ on $J_{1j}$ and on $J_{jl}$, $l \neq 1,$ it acts by $L_{\bar c -c} = -2L_{c}.$ Analogously, it acts on $J_{lj}$ by $-2R_{c}.$
The element $\sum T(a_i, b_i) + \sum_{j \geq 2}1[1j] * c_j[1j]$ is in the kernel of $\ud : J* J \rightarrow \mathrm{IDer}(J) $ if and only if its image under $\ud$ annihilates all Peirce spaces. First consider a Peirce space $J_{il}$ where $i \leq j$ and $i \neq 1.$ Then the only summands which act on this space are $\ud ([1l] * c_l[1l])$  and $\ud(1[1i] * c_i[1i]),$ and so the condition to annihilate $J_{il}$ is that $-2L_{c_i} = -2R_{c_l}$ which implies $c_i = c_l \in \mathrm{Cent}(D).$ 
Thus $L_c - L_{\bar c} + R_{c - \bar  c} = 4L_c$ and the action of $\sum\ud(1[1j] * c[1j])$ on $J_{1l}$ is simply multiplication by $2nc$ where $c$ and $n$ are as in the assumption. (Indeed, we multiply once by $4c$ and then $n-1$ times by $2c$ since we only consider the cases where $j \neq 1$). Since $T(a_i, b_i)$ is simply left multiplication by $\sum_{i}[\bar a_i, b_i] + {[a_i, \bar b_i]}$ we obtain that this element has to be equal to $-2nc$ for some $c \in \mathrm{Cent}(D) \cap D_{-}$
\end{proof}

\begin{rem}
If the involution is trivial, then $D$ must be commutative and  there are no non-zero central elements in $D_{-},$ so that always $c= 0.$ Also in this case 
$[ \bar a, b] = -[a, b] = 0$ and thus $\sum_{i}[\bar a_i, b_i] + [a_i, \bar b_i] = - 2 \sum [a_i, b_i] = 0,$ so that  $T(D,D) = \ker \ud.$ This is in agreement with result obtained for $C_n$-graded Lie algebras with commutative coordinate algebra (see for instance \cite{Kas}).
\end{rem}

\begin{rem} The reader should compare our results to those of Allison and Gao in \cite{AllisonGao}. Their analysis goes much further and gives more insight in the structure of $\uider(V)_0,$ but under the assumptions that $k$ is a field and that the characteristic of the field does not divide $n-1$.  
\end{rem}

\subsubsection{Example}

Let $D$ be the algebra of square matrices with entries in $k$ and involution given by transposition.  The centre of $D$ is spanned by the multiples of the identity and we assume that the centre of $D$ intersects $[D, D]$ trivially. Also all the central elements are symmetric.  Thus 
an element $\sum T(a_i, b_i) + \sum_{j \geq 2}1[1j] * c_j[1j]$ is in the kernel of $f: J* J \rightarrow \mathrm{IDer}(J) $ if and only if
there is $c \in Z(D)$ which is also skew such that
$$c_j = c ,\quad \sum_{i}[\bar a_i, b_i] + {[a_i, \bar b_i]} = -2nc $$
where $n$ is the size of the matrices. Then
$c \in Z(D)$, $c = - \bar c$ and $-2nc \in [D, D]$ if and only if $c = \alpha 1_D$ where $2 \alpha = 0.$ In particular since $k$ does not have $2$-torsion,  this simply means $c= 0.$ 
The condition on $\sum T(a_i, b_i)$ is that $\sum_{i}[\bar a_i, b_i] + {[a_i, \bar b_i]} = 0.$
\backmatter

\bibliographystyle{alpha}

\def\cprime{$'$} \def\cprime{$'$} \def\cprime{$'$} \def\cprime{$'$}
  \def\cprime{$'$} \def\cprime{$'$}

\end{document}